\documentclass[11pt]{article}
\usepackage{paper,asb}

\usepackage{caption}
\usepackage{subcaption}

\usepackage[sort&compress,numbers]{natbib}

\newtheorem{theorem}{Theorem}[section]

\newtheorem{corollary}[theorem]{Corollary}
\newtheorem{assumption}[theorem]{Assumption}%
\newtheorem{lemma}[theorem]{Lemma}
\newtheorem{remark}[theorem]{Remark}%

\numberwithin{equation}{section}

\newcommand{\SVRSQP}{\texttt{SVR-SQP}}
\newcommand{\SVRSQPCONST}{\texttt{SVR-SQP-C}}
\newcommand{\SVRSQPADAPT}{\texttt{SVR-SQP-A}}
\newcommand{\StoSQP}{\texttt{Sto-SQP}}

\newcommand{\StoSubVR}{\texttt{Sto-Subgrad-VR}}

\DeclareMathOperator*{\argmin}{arg\,min}
\newcommand{\xbest}{x_{\texttt{best}}}

\usepackage{booktabs}
\usepackage{caption}
\usepackage{float}
\usepackage{titlesec}
\usepackage{capt-of}

\usepackage{array}
\usepackage{arydshln}
\newcommand{\papertitle}{Accelerating Stochastic Sequential Quadratic Programming for Equality Constrained  Optimization using Predictive Variance Reduction}
\newcommand{\paperauthora}{Albert S. Berahas}
\newcommand{\paperauthoraaffiliation}{Department of Industrial and Operations, University of Michigan}
\newcommand{\paperauthoraemail}{albertberahas@gmail.com}
\newcommand{\paperauthorb}{Jiahao Shi}

\newcommand{\paperauthorbemail}{jiahaos@umich.edu}
\newcommand{\paperauthorc}{Zihong Yi}
\newcommand{\paperauthorcaffiliation}{Department of Computer Science and Engineering, University of Michigan}
\newcommand{\paperauthorcemail}{zihongy@umich.edu}
\newcommand{\paperauthord}{Baoyu Zhou}
\newcommand{\paperauthordaffiliation}{Booth School of Business, University of Chicago}
\newcommand{\paperauthordemail}{baoyu.zhou@chicagobooth.edu}
				
\begin{document}
\title{\papertitle}

\author{\paperauthora\footnotemark[1]\ \footnotemark[2]
   \and \paperauthorb\footnotemark[2]
   \and \paperauthorc\footnotemark[3]
   \and \paperauthord\footnotemark[4]}

\maketitle

\renewcommand{\thefootnote}{\fnsymbol{footnote}}
\footnotetext[1]{Corresponding author.}
\footnotetext[2]{\paperauthoraaffiliation. (\url{\paperauthoraemail},\url{\paperauthorbemail})}
\footnotetext[3]{\paperauthorcaffiliation. (\url{\paperauthorcemail})}
\footnotetext[4]{\paperauthordaffiliation. (\url{\paperauthordemail})}
\renewcommand{\thefootnote}{\arabic{footnote}}

\begin{abstract}{
In this paper, we propose a stochastic method for solving equality constrained optimization problems that utilizes predictive variance reduction. Specifically, we develop a method based on the sequential quadratic programming paradigm that employs variance reduction in the gradient approximations. Under reasonable assumptions, we prove that a measure of first-order stationarity evaluated at the iterates generated by our proposed algorithm converges to zero in expectation from arbitrary starting points, for both constant and adaptive step size strategies. Finally, we demonstrate the practical performance of our proposed algorithm on constrained binary classification problems that arise in machine learning.

}
\end{abstract}


\section{Introduction}\label{sec.intro}

We consider the design of algorithms for solving equality constrained finite-sum problems of the form 
\begin{align}\label{prob.f_finitesum}
  \min_{x\in\mathbb{R}^n}\ f(x)\ \text{ s.t. }\ c(x) = 0,\ \ \text{with}\ \ f(x) = \frac{1}{N} \sum_{i=1}^N f_i(x),
\end{align}
where $f: \mathbb{R}^n \rightarrow \mathbb{R}$, $f_i: \mathbb{R}^n \rightarrow \mathbb{R}$ for all $i \in \{1,\dots,N\}$, and $c: \mathbb{R}^n \rightarrow \mathbb{R}^m$ are smooth nonlinear (possibly nonconvex) functions. Such problems arise in a plethora of areas such as machine/deep learning \cite{zhu2019physics,NandPathAbhiSing19,RaviDinhLokhSing19,achiam2017constrained,marquez2017imposing,roy2018geometry}, statistics \cite{chatterjee2016constrained,geyer1991constrained}, and stochastic optimal control \cite{malikopoulos2013stochastic,lioutikov2014sample}, as well as other science and engineering applications such as optimal power flow \cite{vrakopoulou2014stochastic,wood2013power,summers2015stochastic}, multi-stage modeling \cite{shapiro2021lectures}, and porfolio optimization \cite{uryasev2013stochastic,ziemba2014stochastic}.

Numerous algorithms have been developed over the last half century for solving deterministic equality constrained optimization problems, such as that in \eqref{prob.f_finitesum}. A few classical examples are penalty methods, projection methods and sequential quadratic programming (SQP), each of which have their associated merits and limitations \cite{NoceWrig06}. Penalty methods are intuitive and simple to implement, however, their performance critically relies on the choice of the penalty function and penalty parameter, and, in practice, often suffers from ill-conditioning issues and subproblems' nonsmoothness. On the other hand, projection methods are powerful \emph{feasible} methods, however, they assume that projections can be efficiently computed at every iteration, something that is often not the case with general nonlinear constraints. SQP methods attempt to alleviate these issues by solving a sequence of subproblems that minimize a quadratic model of the objective function subject to a linearization of the constraints, and, as such can handle general nonlinear constraints, however, this comes at the cost of more expensive iterations (SQP methods require solving a linear system at every iteration). That being said, all aforementioned deterministic methods require the computation of the true gradient of the objective function (and constraints) at every iteration, which can be prohibitively expensive in settings in which $n$ and/or $N$ are large.

Rather than minimizing the finite-sum optimization problem \eqref{prob.f_finitesum} 
with a  deterministic method, one can employ stochastic methods that utilize a stochastic approximation of the gradient in lieu of the true gradient in order to reduce the per iteration computational cost. In this direction, several stochastic penalty and projection methods have been proposed \cite{chen2018constraint,NandPathAbhiSing19,lan2020first,ghadimi2016mini,lan2012optimal,reddi2016stochastic2,negiar2020stochastic,shi2021sqp}.
Another line of research considers stochastic alternating direction method of multipliers (ADMM) algorithms \cite{ouyang2013stochastic,zhong2014fast}, and variants of ADMM that utilize variance reduction~\cite{bian2021stochastic,bai2022inexact}. Following the SQP paradigm, recent work \cite{berahas2021sequential} proposed a stochastic SQP method with an adaptive step size selection rule for solving equality constrained stochastic optimization problems endowed with theoretical guarantees (convergence in expectation) analogous of those of the stochastic gradient (SG) method for unconstrained problems, and empirical performance superior to that of the stochastic subgradient method. Several extensions of this work have been developed; namely, in \cite{berahas2021stochastic} the authors relax requirements on the constraints (relax constraint qualifications), in \cite{curtis2021inexact} the authors develop an inexact stochastic SQP method (linear system solved inexactly at every iteration), and in \cite{curtis2021worst} the authors analyze the complexity of the stochastic SQP algorithm proposed in \cite{berahas2021sequential}. Along a slightly different direction, under the assumption that the error in the stochastic gradient approximations employed can be diminished as needed, the authors in \cite{na2021inequality,na2021adaptive} proposed stochastic line search SQP methods for equality and inequality constrained stochastic optimization problems, respectively, that utilize a differentiable exact augmented Lagrangian function as its merit function. 

In the last decade, several stochastic first-order algorithms have been proposed for solving unconstrained finite-sum optimization problems. One such class of algorithms are \emph{variance-reduction} methods, that attempt to reduce the the variance in the stochastic gradient approximation employed as the optimization progresses. Examples of popular variance reduction methods include, but are not limited to, the Stochastic Average Gradient (SAG/SAGA) method \cite{schmidt2017minimizing,defazio2014saga}, the Stochastic Variance Reduced Gradient (SVRG) method \cite{johnson2013accelerating}, the Stochastic Recursive Gradient Algorithm (SARAH) method \cite{nguyen2017sarah}, and the Stochastic Dual Coordinate Ascent (SDCA) method \cite{shalev2013stochastic}. As a result of the variance reduction, these methods enjoy improved convergence results as compared to their classical counter-parts (e.g., SG method \cite{robbins1951stochastic,bottou2018optimization}), and these benefits are very often also observed in practice. Motivated by this fact, we design and analyze a stochastic SQP method that employs variance reduced gradients for solving \eqref{prob.f_finitesum}.

\subsection{Contributions}

The contributions of our work can be categorized as follows:
\begin{itemize}
    \item \textbf{Algorithmic.} We present a stochastic sequential quadratic optimization algorithm that uses variance reduced gradients. Specifically, inspired by the theoretical and empirical advantages of \emph{variance reduced methods} (unconstrained finite-sum problems) and  SQP methods (deterministic equality constrained problems), we propose a stochastic SQP method that uses variance reduced gradients (SVRG-type, \cite{johnson2013accelerating}) in lieu of the true gradient (\SVRSQP). We propose one algorithm with two possible step size selection strategies; a constant step size scheme (similar to that in SVRG \cite{johnson2013accelerating,reddi2016stochastic}), and an adaptive step size scheme (similar to that in \cite{berahas2021sequential}). 
    
    Our proposed algorithm is based on a stochastic  SQP framework, similar to the stochastic algorithm proposed in \cite{berahas2021sequential}, but with several distinguishing algorithmic and theoretical differences.
    In particular, our proposed algorithm is of nested form, due to the nature of the construction of the SVRG gradients, and operates with two types of iterations (inner and outer). As a direct consequence of the use of variance reduced gradients, the proposed step size selection strategy only requires minimal safe-guarding, as compared to the safe-guards imposed in  \cite{berahas2021sequential}, e.g., safe-guarding parameters or sequences and projections. 

    We should note that while in this work we chose to employ SVRG-type gradient approximations, others, e.g., \cite{shalev2013stochastic,nguyen2017sarah,schmidt2017minimizing,defazio2014saga}, may also be employed. We chose SVRG because it is based on an intuitive idea (control variates \cite{ross2013simulation}), has no additional storage requirements, has proven robust and efficient in practice, and, perhaps most importantly because the unbiasedness of the SVRG gradients allows for simple convergence analysis. 
    \item \textbf{Theoretical.} We provide convergence guarantees for the \SVRSQP{} method with the two different step size strategies (constant and adaptive). 
    For both strategies, we present strong theoretical guarantees in the sense that a measure of first-order stationarity evaluated at the iterates generated by \SVRSQP{} vanishes in expectation with \emph{non-diminishing} step size sequences. 
    This is in contrast with the results in \cite{berahas2021sequential} where a {\sl diminishing} step size sequence is required to ensure exact convergence in expectation. Our result (equality constrained finite sum setting) for the \SVRSQP{} method with a constant step size can be viewed as analogues of the results that can be proven for the SVRG method \cite{johnson2013accelerating,reddi2016stochastic} in the unconstrained finite sum setting. Similar convergence guarantees are established for the more flexible variant with adaptive step sizes. Table \ref{tab.summary} summarizes our results. 
    \item \textbf{Empirical.} 
    We illustrate the performance of our proposed method on  constrained binary classification problems, and we compare our proposed algorithm against other popular methods, such as the adaptive stochastic SQP method proposed in \cite{berahas2021sequential} and a stochastic subgradient method that utilizes variance reduction. We provide evidence illustrating the benefits of using variance reduced gradients within the stochastic SQP framework for solving equality constrained finite sum optimization problems.
\end{itemize}
\begin{table}[]
\renewcommand\arraystretch{1.1}
    \centering
    \caption{Summary of asymptotic results for different problem settings and different methods (and step size choices) for nonconvex functions. For unconstrained finite sum problems the stationarity measure 
    is $\| \nabla f(x)\|_2^2$, whereas for equality constrained finite sum problems the stationarity measure
    is $\| \nabla f(x) + \nabla c(x)y\|_2^2 + \| c(x) \|_2 $ (where $y$ are least-squares Lagrange
multipliers). In the table, ``exact'' and ``neighborhood'' denote 
 convergence to the stationarity 
 measure in expectation and convergence to a neighborhood of the stationarity 
 measure in expectation, respectively, and ``-'' denotes that no result exists. For brevity and ease of exposition, we do not state the exact constants in the results below.}
\label{tab.summary}
    \begin{tabular}{
        ccccc}
        \toprule
        \multirow{2}{*}{\textbf{Setting}} & \multirow{2}{*}{\textbf{Method}} &  \multicolumn{3}{c}{\textbf{Step size}}\\
        \cmidrule(lr){3-5}
       \multicolumn{1}{c}{}& \multicolumn{1}{c}{} & \multicolumn{1}{c}{\textbf{Diminishing}} & \multicolumn{1}{c}{\textbf{Constant}} & \multicolumn{1}{c}{\textbf{Adaptive}}
       \\
       \cmidrule(lr){1-1} 
       \cmidrule(lr){2-2}
       \cmidrule(lr){3-3}
       \cmidrule(lr){4-4}
       \cmidrule(lr){5-5}
       \multirow{2}{*}{\begin{tabular}[c]{@{}c@{}}Unconstrained\\ 
       Finite Sum\end{tabular}}          & SG~\cite{bottou2018optimization}                                    & \multicolumn{1}{c}{exact}       & \multicolumn{1}{c}{neighborhood} & -            \\ \cdashline{2-5} 
        & SVRG~\cite{reddi2016stochastic}                                   & \multicolumn{1}{c}{-}           & \multicolumn{1}{c}{exact}        & -            \\ \hdashline
        \multirow{3}{*}{\begin{tabular}[c]{@{}c@{}}Equality \\ Constrained\\
        Finite Sum\end{tabular}}          & Stoch. SQP \cite{berahas2021sequential}                                      & \multicolumn{1}{c}{exact}       & \multicolumn{1}{c}{neighborhood} & \multicolumn{1}{c}{neighborhood}            \\ \cdashline{2-5} 
        & \multirow{2}{*}{\begin{tabular}[c]{@{}c@{}}\SVRSQP{}\\ (this paper)  \end{tabular}}                             & \multirow{2}{*}{\begin{tabular}[c]{@{}c@{}} -  \end{tabular}} 
        & \multirow{2}{*}{\begin{tabular}[c]{@{}c@{}} exact  \end{tabular}}        & \multirow{2}{*}{\begin{tabular}[c]{@{}c@{}} exact  \end{tabular}} \\ \\ \bottomrule
    \end{tabular}
\end{table}

\subsection{Organization}

The paper is organized as follows. We conclude this section by setting the notation that will be used throughout the paper. In Section \ref{sec.sec2} we introduce the assumptions and main algorithmic components of our proposed method. The stochastic variance reduced sequential quadratic optimization method is presented in Section \ref{sec.svrsqp}, and its associated convergence guarantees are presented in Section \ref{sec.convergence}. In Section \ref{sec:experiments}, we demonstrate the empirical performance of the proposed algorithm. Finally, in Section \ref{sec:final_remarks} we make some concluding remarks.

\subsection{Notation}

Let $\mathbb{N}$ denote the set of natural numbers, $\mathbb{R}$ denote the set of real numbers and  $\mathbb{R}_{>0}$ denote the set of positive real numbers. For any $m \in \mathbb{N}$, let $[m]$ denote the set of integers $\{1,\dots,m\}$, and $[\bar{m}]$ denote the set of integers $\{0,1,\dots,m-1\}$.  Let $\mathbb{R}^n$
denote the set of $n$-dimensional real vectors, $\mathbb{R}^{m \times n}$ denote the set of $m$-by-$n$-dimensional real matrices, and $\mathbb{S}^{n}$ denote the set of $n$-by-$n$-dimensional
symmetric matrices.

The algorithms described in this paper will either operate with a single type of iteration and produce sequences of iterates $\{x_{k}\}$ where $k \in \mathbb{N}$ is the index of iterations, or will operate with two types of iterations (e.g., inner and outer) and produce sequences of iterates $\{x_{k,s}\}$, where $k \in \mathbb{N}$ is the index of outer iterations and $s \in \left[\bar{S}\right]$ is the index of inner iterations. The index of iteration number is also appended as a subscript to other quantities corresponding to each iteration; e.g., $f_{k,s} = f(x_{k,s})$, respectively for the single iteration algorithms. Throughout the paper, we use the overline to denote stochastic quantities; e,g., $\bar{g}_{k,s}$ is an estimate of  ${g}_{k,s}:= \nabla f(x_{k,s})$.

\section{Assumptions and Algorithm Preliminaries}\label{sec.sec2}

Throughout the paper, we assume that the constraint function and its associated first-order derivatives can be computed exactly. With regards to the objective function and its associated derivatives, we assume that those quantities are prohibitively expensive to compute at every iteration, however, exact evaluations can be accessed as required by the algorithm. We formalize our assumptions  with regards to \eqref{prob.f_finitesum} and the iterates generated by our algorithm $\{x_{k,s}\}$ below. 
\begin{assumption}
   Let $\mathcal{X} \subseteq \mathbb{R}^{n}$ be an open convex set containing the iterates $\{x_{k,s} \}$ generated by any run of the algorithm. The objective function $f : \mathbb{R}^{n} \to \mathbb{R}$ and its gradient $g:=\nabla f : \mathbb{R}^{n} \to \mathbb{R}^{n}$ are bounded over $\mathcal{X}$. For each $i\in \left[N\right]$, the component objective function $f_i: \mathbb{R}^n \to \mathbb{R}$ is continuously differentiable, and each component gradient $ \nabla f_i: \mathbb{R}^n \to \mathbb{R}^n$ is Lipschitz continuous with constant $L$. For each $i\in \left[m\right]$, the constraint function $c_i:\mathbb{R}^n\to\mathbb{R}$ is continuously differentiable and bounded over $\mathcal{X}$, and its gradient $\nabla c_i:\mathbb{R}^n\to\mathbb{R}^n$ is Lipschitz continuous with constant $\gamma_i\in\mathbb{R}_{>0}$. We define $\Gamma := \sum_{i=1}^m \gamma_i$. The Jacobian function $J := \nabla c^T : \mathbb{R}^{n} \to \mathbb{R}^{m \times n}$ is bounded over $\mathcal{X}$, 
   and has singular values bounded away from zero over $\mathcal{X}$.
\label{ass.function}
\end{assumption}

\begin{remark} Assumption~\ref{ass.function} implies that the objective function $f : \mathbb{R}^{n} \to \mathbb{R}$ is continuously differentiable, and that its gradient $g:=\nabla f : \mathbb{R}^{n} \to \mathbb{R}^{n}$ is Lipschitz continuous with constant $L$. Under Assumption~\ref{ass.function}, it follows that, for all $(k,s) \in \mathbb{N} \times \left[\bar{S}\right]$, there exist positive real numbers $(f_{\inf},f_{\sup},\kappa_{g},\kappa_c,\kappa_J,\kappa_\sigma) \in \mathbb{R} \times \mathbb{R} \times \mathbb{R}_{>0} \times \mathbb{R}_{>0} \times \mathbb{R}_{>0}\times \mathbb{R}_{>0}$ such that
\begin{equation}\label{eq.bounds}
\begin{aligned}
  &f_{\inf} \leq f_{k,s} \leq f_{\sup},\ \ \  \|g_{k,s}\|_2 \leq \kappa_{g},&\\
  \text{and} \ \ \ & \|c_{k,s}\|_1 \leq \kappa_c, \ \ \  \|J_{k,s}\|_2 \leq \kappa_J, \ \ \  \|(J_{k,s}J_{k,s}^T)^{-1}\|_2\leq \kappa_{\sigma}^{-2}.&
\end{aligned}
\end{equation}
Assumption \ref{ass.function} ensures the smoothness of the objective function and constraint functions. Unlike many projection methods aimed to solve stochastic optimization problems \cite{nemirovski2009robust,shi2021sqp}, we do not assume that $\mathcal{X}$ is bounded. We remark that the boundedness assumption of the singular values of $\nabla c^T$ guarantees the linear independence constraint qualification (LICQ). Note that it is generally not ideal to assume that the objective and constraint function and derivative values are bounded over $\mathcal{X}$ containing stochastic iterates $\{x_{k,s} \}$. However, this assumption is reasonable in our problem setting if we assume the component functions $ \{ f_i\} $ have bounded derivatives  over $\mathcal{X}$.  In addition,  this assumption can be loosen if one chooses to use constant step sizes. This assumption is similar to those in \cite{berahas2021sequential,na2021adaptive}. 
\end{remark} 

We define the Lagrangian, $\mathcal{L}: \mathbb{R}^n \times \mathbb{R}^m \rightarrow \mathbb{R}$, of \eqref{prob.f_finitesum} as $\mathcal{L}(x, y) := f(x) + y^T c(x)$, where
$y \in \mathbb{R}^m$ represents a vector of Lagrange multipliers.  Under Assumption \ref{ass.function} (and as a result of LICQ),
necessary conditions for first-order stationarity with respect to \eqref{prob.f_finitesum} are given by
\begin{align*}
    0 = \begin{bmatrix}
        \nabla_x \mathcal{L}(x,y) \\
        \nabla_y \mathcal{L}(x,y)
    \end{bmatrix} = \begin{bmatrix}
        \nabla f(x) + \nabla c(x)y  \\
         c(x)
    \end{bmatrix}.
\end{align*}

Next, we formalize our assumption on the gradient approximation employed by the \SVRSQP{} method. Given an iterate $x_{k,s} \in \mathbb{R}^n$ (where $(k,s) \in \mathbb{N} \times \left[\bar{S}\right]$), let $\tilde{g}_{k,s} \in \mathbb{R}^n$ be defined as
\begin{align}\label{eq.sg}
    \tilde{g}_{k,s} := \frac{1}{b} \sum_{i \in {I}_{k,s}} \nabla f_i (x_{k,s}),
\end{align}
where ${I}_{k,s} \subset [N]$ of size $b$ is a mini-batch (subset) of all the data. Throughout the paper, we refer to the gradient approximation in \eqref{eq.sg} as the \emph{stochastic gradient}.
\begin{assumption}\label{ass.unbias}
The gradient approximation \eqref{eq.sg} is an unbiased estimator of the true gradient of the objective function, i.e., we have that $\mathbb{E}_{k,s} [\tilde{g}_{k,s}] = g_{k,s}$, 
where $\mathbb{E}_{k,s}$ denotes the expectation taken conditioned on the event that the algorithm has reached $x_{k,s} \in \mathbb{R}^n$ in iteration $(k,s) \in \mathbb{N} \times \left[\bar{S}\right]$. $($We impose an additional condition on this expectation in subsequent sections of the paper; see Lemma~\ref{lemma_variance}.$)$  
The unbiasedness assumption of $\tilde{g}_{k,s}$ can be easily satisfied, e.g., when each sample in the mini-batch  ${I}_{k,s}  \subset [N]$ is selected uniformly at random. 
\end{assumption}


 Finally, the variance reduced gradient approximation employed by the \SVRSQP{} method $\bar{g}_{k,s} \in \mathbb{R}^n$ is defined as
\begin{equation}
\begin{aligned}\label{eq.svrg_grad}
    \bar{g}_{k,s}  :&= \frac{1}{b} \sum_{i \in I_{k,s}} \left(\nabla f_{i}(x_{k,s} ) - \left(\nabla f_{i}(x_{k,0}) - \nabla f(x_{k,0})\right)\right)\\
    &= \tilde{g}_{k,s} - \tilde{g}_{k,0} + g_{k,0},
\end{aligned}
\end{equation}
where $I_{k,s} \subset [N]$ of size $b$, and $x_{k,0}$ is known as the reference point (the initial point for the inner iterations of the  $k$th outer iteration). Throughout the paper, we refer to the gradient approximation in \eqref{eq.svrg_grad} as the \emph{SVRG gradient}. Under Assumption~\ref{ass.unbias}, it follows that the SVRG gradient is an unbiased estimate of the true gradient, i.e.,  $\mathbb{E}_{k,s} [\bar{g}_{k,s}] = g_{k,s}$. 

\section{Stochastic Variance Reduced Sequential Quadratic Programming}\label{sec.svrsqp}

Our proposed algorithm (\SVRSQP{}) is based on the Sequential Quadratic Programming (SQP) paradigm. A high level description of the \SVRSQP{} method is as follows: \SVRSQP{} operates with two types of iterations (inner and outer), employs variance reduced approximations of the gradient of the objective function following \eqref{eq.svrg_grad} in lieu of the true gradient, and updates the iterates in SQP fashion. 

Given an iterate $x_{k,s} $ for all $(k,s) \in \mathbb{N} \times \left[\bar{S}\right]$, the \SVRSQP{} methods proceeds to compute a search direction $\bar{d}_{k,s} \in \mathbb{R}^n$ by solving the following subproblem 
\begin{align}
\min_{\bar{d} \in \mathbb{R}^n}  \bar{g}_{k,s}^T \bar{d} + \tfrac{1}{2} \bar{d}^T H_{k,s} \bar{d} \quad \text{s.t}\quad c_{k,s}+ J_{k,s} \bar{d} = 0,
\label{subproblem:SQP}  
\end{align}
where $\bar{g}_{k,s}$ is defined in \eqref{eq.svrg_grad} and  
$H_{k,s} \in \mathbb{S}^n$ 
satisfies Assumption \ref{ass.H} below.
\begin{assumption}\label{ass.H}
    The sequence $\{H_{k,s} \}$ is independent of $\{\bar{g}_{k,s}\}$ and is bounded in norm by $\kappa_H \in \mathbb{R}_{>0}$.  In addition, there exists a constant $\zeta \in \mathbb{R}_{>0}$ such that, for all $(k,s) \in \mathbb{N} \times \left[\bar{S}\right]$, the matrix $H_{k,s} $ has the property that $u^TH_{k,s} u \geq \zeta \|u\|_2^2$ for all $u \in \mathbb{R}^{n}$ such that $J_{k,s} u = 0$.
\end{assumption}

Under Assumptions~\ref{ass.function} and \ref{ass.H}, the solution of \eqref{subproblem:SQP}, denoted by $\bar{d}_{k,s}  \in \mathbb{R}^n$, can be equivalently computed by solving the following linear system of equations
\begin{align}
\begin{bmatrix}
      H_{k,s}  & J_{k,s}^T \\ J_{k,s}  & 0
\end{bmatrix}
\begin{bmatrix}
      \bar{d}_{k,s}  \\ \bar{y}_{k,s} 
\end{bmatrix} = - \begin{bmatrix}
      \bar{g}_{k,s}  \\ c_{k,s} 
\end{bmatrix},
\label{eq.system_stochastic}
\end{align}
where $\bar{y}_{k,s} \in \mathbb{R}^m$ is the vector of associated Langrange multipliers of \eqref{subproblem:SQP}. The linear system in  \eqref{eq.system_stochastic} has a unique solution under Assumptions~\ref{ass.function} and \ref{ass.H}; see \cite{NoceWrig06}.

With a search direction $\bar{d}_{k,s}  \in \mathbb{R}^n$ in hand, \SVRSQP{} proceeds to utilize a merit function, $\phi : \mathbb{R}^n \times \mathbb{R}_{>0} \to \mathbb{R}$, to judge the quality of the computed  step (in terms of stationarity 
and feasibility), and then compute a positive step size $\bar\alpha_{k,s} \in \mathbb{R}_{>0}$ in order to update the current iterate $x_{k,s} \in \mathbb{R}^{n}$ via 
\begin{align}
    x_{k,s+1} = x_{k,s} + \bar\alpha_{k,s}\bar{d}_{k,s}.
\end{align}
Similar to \cite{berahas2021sequential}, our algorithm makes use of, possibly the most common merit (penalty) function, the $l_1$-norm merit function, defined as 
\begin{equation}
    \phi(x,\tau) := \tau f(x) + \|c(x)\|_1, \label{def.merit}
\end{equation}
where $\tau \in \mathbb{R}_{>0}$ is known as the merit (penalty) parameter and whose value is set adaptively as the optimization progresses. Before we proceed, we introduce two quantities that are used in our proposed algorithm, and that are vital to the analysis. First, a local linear model of the merit function $l : \mathbb{R}^{n} \times \mathbb{R}_{>0} \times \mathbb{R}^{n} \times \mathbb{R}^{n} \to \mathbb{R}$ is defined by 
\begin{equation}
  l(x,\tau,g,d) := \tau (f(x) + g^T d) + \|c(x) + \nabla c(x)^Td\|_1.
\label{def.localmodel}
\end{equation}
Second, the reduction function  of the local linear model of the merit function $\Delta l : \mathbb{R}^{n} \times \mathbb{R}_{>0} \times \mathbb{R}^{n} \times \mathbb{R}^{n} \to \mathbb{R}$, given $d \in \mathbb{R}^{n}$ with $c(x) + \nabla c(x)^Td = 0$, is defined by 
\begin{equation}
\begin{aligned}
        \Delta l(x,\tau,g,d) := l(x,\tau,g,0) - l(x,\tau,g,d) = -\tau g^Td + \|c(x)\|_1.
\label{def.merit_model_reduction}
\end{aligned}
\end{equation}

Given a search direction $\bar{d}_{k,s}  \in \mathbb{R}^n$, the merit parameter update strategy goes as follows. To begin with, a trial merit parameter is defined as
\begin{align}
\label{eq.merittrial}
  \bar\tau_{k,s}^{trial} \gets \begin{cases} \infty & \text{if } \bar{g}_{k,s}^T\bar{d}_{k,s} + \max\{ \bar{d}_{k,s}^TH_{k,s}\bar{d}_{k,s},0\} \leq 0; \\ \tfrac{(1 - \sigma)\|c_{k,s} \|_1}{\bar{g}_{k,s}^T\bar{d}_{k,s} + \max\{ \bar{d}_{k,s}^TH_{k,s}\bar{d}_{k,s},0\}} & \text{otherwise},
\end{cases}
\end{align}
where the parameter $\sigma \in (0,1)$ is user-defined. It follows that $\bar\tau_{k,s}^{trial} > 0$ since if $\|c_{k,s} \|_1 = 0$, by Assumption \ref{ass.H} and \eqref{eq.system_stochastic}, $\bar{g}_{k,s}^T\bar{d}_{k,s} + \max\{ \bar{d}_{k,s}^TH_{k,s}\bar{d}_{k,s},0\} = \bar{g}_{k,s}^T\bar{d}_{k,s} + \bar{d}_{k,s}^TH_{k,s}\bar{d}_{k,s} = -\bar{d}_{k,s}^TJ_{k,s}^T\bar{y}_{k,s} = c_{k,s}^T\bar{y}_{k,s} = 0$. Next, for some user-defined parameter 
$\epsilon_{\tau} \in (0,1)$, $\bar\tau_{k,s}$ is computed via 
\begin{align}
\label{eq.meritupdate}
\bar\tau_{k,s} \gets 
\begin{cases} \bar\tau_{k,s-1} & \text{if $\bar\tau_{k,s-1} \leq \bar\tau_{k,s}^{trial}$} \\ (1-\epsilon_\tau) \bar\tau_{k,s}^{trial} & \text{otherwise.}
\end{cases}
\end{align}
Note, the above rule ensures that $\bar\tau_{k,s} \leq \bar\tau_{k,s}^{trial}$. Moreover, and more importantly, the updates \eqref{eq.merittrial}--\eqref{eq.meritupdate} ensure that
\begin{equation}
  \Delta l(x_{k,s} ,\bar\tau_{k,s},\bar{g}_{k,s}  ,\bar{d}_{k,s} ) \geq  \bar\tau_{k,s} \max\{\bar{d}_{k,s}^TH_{k,s} \bar{d}_{k,s} ,0\} + \sigma \|c_{k,s} \|_1.
\label{eq.reduction_lb}    
\end{equation}
The above inequality plays a critical role in our algorithm and analysis.

Finally, the \SVRSQP{} method computes a positive step size. We propose two different step size selection strategies; a constant step size strategy and an adaptive step size strategy. The constant step size strategy, similar to that in \cite{reddi2016stochastic}, specifies an upper bound on acceptable step sizes (see Theorem~\ref{theorem:constant_step size_prescribed} for the exact specification). 

The adaptive step size strategy, inspired by \cite{berahas2021sequential}, is  motivated by the desire to select a step size that minimizes 
an upper bound on the change in the merit function. By the definition of the merit function \eqref{def.merit} and under Assumption~\ref{ass.function}, the upper bound on the change in the merit function is a convex (strongly-convex when $\|\bar{d}_{k,s}\|\neq 0$), piece-wise quadratic function in $\bar{\alpha}_{k,s} \in \mathbb{R}_{> 0}$, 
\begin{equation}\label{eq.merit_reduction_ub1}
\begin{aligned}
     &\ \phi(x_{k,s+1},\bar{\tau}_{k,s}) - \phi(x_{k,s},\bar{\tau}_{k,s})    \\
     \leq&\    \bar{\alpha}_{k,s} \bar{\tau}_{k,s}   g_{k,s}^T   \bar{d}_{k,s}  + (\lvert1-\bar{\alpha}_{k,s}\rvert - 1) \| c_{k,s} \|_1 + \tfrac12 (\bar{\tau}_{k,s} L +\Gamma )\bar{\alpha}_{k,s}^2  \| \bar{d}_{k,s} \|_2^2.
\end{aligned}
\end{equation}
(See \cite[Lemma 3.1]{byrd2008inexact} for derivation of above inequality.) Our adaptive stategy attempts to select a step size that approximately minimizes this upper bound. To this end, at iteration $(k,s) \in \mathbb{N} \times \left[\bar{S}\right]$, two trial step sizes are computed, specifically,  
\begin{align}
    \bar{\widehat\alpha}_{k,s} &\gets \min\left\{ \tfrac{ \Delta l(x_{k,s} ,\bar\tau_{k,s},\bar g_{k,s} ,\bar d_{k,s} ) }{(\bar\tau_{k,s} L_{k,s} + \Gamma_{k,s} ) \|\bar{d}_{k,s} \|_2^2}, \alpha_u \right\} \beta \label{eq:step size1}\\
    \text{and}\quad \bar{\widetilde\alpha}_{k,s} &\gets \bar{\widehat\alpha}_{k,s} - \tfrac{4\|c_{k,s} \|_1}{(\bar\tau_{k,s} L_{k,s} +  \Gamma_{k,s})\|\bar{d}_{k,s} \|_2^2}, \label{eq:step size2}
\end{align}
where $\alpha_u \in \mathbb{R}_{>0}$ is a user-defined parameter that is introduced here to avoid the step size being arbitrarily large (the precise specification of $\alpha_u$ is given in Section~\ref{sec.adaptive}).  Due to the nonsmoothness of the upper bound (notice, nonsmooth point at $\bar{\alpha}_{k,s} = 1$), the approximate minimizer, and the step size used by the \SVRSQP{}, is set as
\begin{equation}\label{eq.step size}
          \bar{\alpha}_{k,s}\gets
		    \begin{cases}
		      \bar{\widehat\alpha}_{k,s} & \text{if $\bar{\widehat\alpha}_{k,s} < 1$} \\
		      1 & \text{if $\bar{\widetilde\alpha}_{k,s} \leq 1 \leq \bar{\widehat\alpha}_{k,s}$} \\
		      \bar{\widetilde\alpha}_{k,s} & \text{if $\bar{\widetilde\alpha}_{k,s} > 1$}
		    \end{cases}
\end{equation}

Our proposed algorithm \SVRSQP{} is fully described in Algorithm~\ref{alg.sqp_svrg}. Similar to the SVRG method \cite{johnson2013accelerating,reddi2016stochastic}, \SVRSQP{} operates with inner and outer iterations. Each outer iteration commences with the computation of the full gradient of the objective function at the \emph{reference point} $x_{k,0}$, i.e., $g_{k,0}$. Given the $g_{k,0}$ at every inner iteration, a stochastic variance reduced gradient is computed via \eqref{eq.svrg_grad}, and then paralleling the SQP paradigm, the search direction is computed by solving the linear system given in \eqref{eq.system_stochastic}. Finally, similar to the stochastic SQP algorithm proposed in \cite{berahas2021sequential}, the merit parameter is updated following \eqref{eq.merittrial}--\eqref{eq.meritupdate}, a step size is computed, and the current iterate is updated. The algorithm allows for two different step size choices: constant step size (\textbf{Option I}) and adaptive step size (\textbf{Option II}) via the equations \eqref{eq:step size1}--\eqref{eq.step size}.

\begin{algorithm}[ht]
  \caption{Stochastic Variance Reduced SQP (\SVRSQP{})}
  \label{alg.sqp_svrg}
  \begin{algorithmic}[1]
    \Require $x_{-1,S} \in \mathbb{R}^{n}$ (initial iterate); $ \bar\tau_{-1,S-1}  \in \mathbb{R}_{>0}$ (initial merit parameter value); $\epsilon_\tau \in (0,1)$ (merit decrease parameter); $\sigma \in (0,1)$ (model reduction parameter), $b \in [N-1] 
    $ (batch size) 
    
    \textbf{Require (Option I: Constant step size algorithm):} $\alpha \in (0,1]$ (constant step size parameter)
    
    \textbf{Require (Option II: Adaptive step size algorithm):} $\beta \in (0,1]$ (adaptive step size parameter), $\alpha_u \in \mathbb{R}_{>0}$ (adaptive step size bound)
    \vspace{0.25cm}
    \For{$k=0,1,\dots,$}
    \State  Set $x_{k,0} = x_{k-1,S}$; $\bar\tau_{k,-1} = \bar\tau_{k-1,S-1}$
    \State Compute gradient $g_{k,0} = \nabla f(x_{k,0})$
    \For{$s=0, 1,\dots,S-1$} 
      \State Choose a mini-batch $I_{k,s} \subset [N]$ of size $b$, and compute $\bar{g}_{k,s}$ via (\ref{eq.svrg_grad}) 
	  \State Compute $(\bar{d}_{k,s} ,\bar{y}_{k,s} )$ as the solution of \eqref{eq.system_stochastic} 
	  \If{$\bar{d}_{k,s}  = 0$} \textbf{set} $x_{k,s+1} \gets x_{k,s} $, $\bar\tau_{k,s} \gets \bar\tau_{k,s-1}$; go to Line 5 
      \EndIf 
    \State Set $\bar\tau_{k,s}^{trial}$ via (\ref{eq.merittrial}) and $\bar\tau_{k,s}$ via (\ref{eq.meritupdate})
    \State Set step size parameter $\bar{\alpha}_{k,s}$
    
    \hspace{0.75cm}\textbf{Option I:} Set $\bar{\alpha}_{k,s} = \alpha$
    
    \hspace{0.75cm}\textbf{Option II:} Compute $\bar{\widehat\alpha}_{k,s}$ and $\bar{\widetilde\alpha}_{k,s}$ via \eqref{eq:step size1}--\eqref{eq:step size2}
    
    \hspace{2.8cm} Set  $\bar{\alpha}_{k,s}$ via \eqref{eq.step size}
      \State Set $x_{k,s+1} \gets x_{k,s}  + \bar{\alpha}_{k,s} \bar{d}_{k,s} $
    \EndFor
    \EndFor
  \end{algorithmic}
\end{algorithm}

\begin{remark} 
Due to the nature of the SVRG gradient estimate, our proposed algorithm is of nested nature (inner and outer iterations), and the full batch gradient is computed once every outer iteration in order to reduce the variance of the gradient estimate. Our proposed algorithm has two options for selecting the step size.  Algorithm \ref{alg.sqp_svrg} with \textbf{Option I} (constant step size) can be considered a natural extension of \cite{reddi2016stochastic} to the equality constrained setting. Algorithm \ref{alg.sqp_svrg} with \textbf{Option II} (adaptive step size) can be considered a natural extension of \cite{berahas2021sequential} where the stochastic gradient estimate is replaced by the SVRG gradient estimate.
\end{remark}

\section{Convergence Analysis}\label{sec.convergence}
In this section, we present convergence guarantees for \SVRSQP{}  (Algorithm~\ref{alg.sqp_svrg}) under the two step size regimes. We begin with some general technical lemmas (Section~\ref{sec.general_results}), then discuss the behavior of the merit parameter (Section~\ref{sec.merit_parameter_behaviors}), and finally present our main theoretical results for constant and adaptive step size choices (Sections~\ref{sec.constant} and \ref{sec.adaptive}, respectively). Throughout this section we assume that Assumptions~\ref{ass.function}, \ref{ass.unbias} and \ref{ass.H} hold; for brevity, we do not remind the reader of this fact within the statement of each result.


For the purposes of the analysis, we define several deterministic quantities that are never explicitly computed in Algorithm~\ref{alg.sqp_svrg}. First, $d_{k,s} \in \mathbb{R}^n$ and $y_{k,s} \in \mathbb{R}^m$ for all $(k,s) \in \mathbb{N} \times \left[\bar{S}\right]$ are defined as
\begin{align}
\begin{bmatrix}
      H_{k,s}  & J_{k,s}^T \\ J_{k,s}  & 0
\end{bmatrix}
\begin{bmatrix}
      {d}_{k,s}  \\ {y}_{k,s} 
\end{bmatrix} = - \begin{bmatrix}
      {g}_{k,s}  \\ c_{k,s} 
\end{bmatrix}.
\label{eq.system_det}
\end{align}
We note that the only difference between~\eqref{eq.system_stochastic} and \eqref{eq.system_det} is the right-hand-side, where the gradient approximation $\bar{g}_{k,s}$ is replaced by the true gradient ${g}_{k,s}$. Moreover, for all $(k,s) \in \mathbb{N} \times \left[\bar{S}\right]$,  $\tau_{k,s}^{trial}$ and $\tau_{k,s}$ are the deterministic analogues of stochastic merit parameters values $\bar{\tau}_{k,s}^{trial}$ and $\bar{\tau}_{k,s}$, respectively, where $\bar{g}_{k,s}$ and $\bar{d}_{k,s}$ are replaced by $g_{k,s}$ and $d_{k,s}$ in \eqref{eq.merittrial} and \eqref{eq.meritupdate}.


\subsection{General results}
\label{sec.general_results}

The first lemma of this section consists of several technical conditions that are used for the convergence analysis of the \SVRSQP{} method. 
These conditions are analogues of those in \cite[Lemma 3.4]{berahas2021sequential}\footnote{For lemmas with proofs equivalent to those in \cite{berahas2021sequential}, we refer interested reader to the appropriate sections.}. 


\begin{lemma} \label{lemma:21'}
There exists a constant $\kappa_l \in \mathbb{R}_{>0}$ such that the following statements hold true for all $(k,s) \in \mathbb{N} \times \left[\bar{S}\right]$:
\begin{enumerate}
    \item[(a)] $\Delta  l(x_{k,s} , \bar \tau_{k,s},g_{k,s},d_{k,s} ) \geq \kappa_l  \bar \tau_{k,s} \|{d}_{k,s} \|_2^2$;
    \item[(b)] $\Delta  l(x_{k,s} , \bar\tau_{k,s},\bar{g}_{k,s},\bar{d}_{k,s} ) \geq \kappa_l  \bar\tau_{k,s} \|\bar{d}_{k,s} \|_2^2$; 
    \item[(c)] and, 
\begin{equation}
\begin{aligned}\label{eq.merit_reduction_ub}
    &\phi(x_{k,s+1},\bar{\tau}_{k,s})  - \phi(x_{k,s},\bar{\tau}_{k,s})\\
    \le &\ \bar{\alpha}_{k,s} \bar{\tau}_{k,s}   g_{k,s}^T   \bar{d}_{k,s}  + \lvert1-\bar{\alpha}_{k,s}\rvert  \| c_{k,s} \|_1 -  \| c_{k,s} \|_1 + \tfrac12 (\bar{\tau}_{k,s} L +\Gamma )\bar{\alpha}_{k,s}^2  \| \bar{d}_{k,s} \|_2^2.
\end{aligned}
\end{equation}
\end{enumerate}
\end{lemma} 
\begin{proof}First note that condition $(b)$ is the stochastic analogue of condition $(a)$. Conditions $(a)$ and $(b)$ can be derived directly from \cite[Lemmas 2.11, 2.12 \& 3.4(c),(d)]{berahas2021sequential}. 
Inequality $(c)$ is identical to that in \cite[Lemma 3.4]{berahas2021sequential}, but accounts for the double indices of the \SVRSQP{} algorithm. 
\end{proof}

In the next lemma, we bound the error in the gradient approximation $\bar{g}_{k,s}$ employed by the \SVRSQP{} algorithm.

\begin{lemma}\label{lemma_variance} 
Let  $\bar{g}_{k,s} \in \mathbb{R}^n$ 
be the gradient approximation computed by Algorithm~\ref{alg.sqp_svrg} via \eqref{eq.svrg_grad}. Then, for all $(k,s) \in \mathbb{N} \times \left[\bar{S}\right]$,
\begin{align}\label{eq.variance}
  \mathbb{E}_{k,s}\left[\|\bar{g}_{k,s}  - \nabla f(x_{k,s} ) \|_2^2\right]\leq  M_{k,s},
\end{align}
where $M_{k,s} = \tfrac{L^2}{b}\|x_{k,s}  - x_{k,0}\|_2^2$, and $\mathbb{E}_{k,s}$ denotes the expectation taken conditioned on the event that the algorithm has reached $x_{k,0} \in \mathbb{R}^n$ in (outer) iteration $k \in \mathbb{N}$ and $x_{k,s}$ in (outer-inner) iteration $(k,s) \in \mathbb{N} \times \left[\bar{S}\right]$.
\end{lemma}
\begin{proof}
For the ease of exposition, we introduce the following notation,
\begin{align}\label{eq.zeta}
    \zeta_{k,s} =\tfrac{1}{b} \sum_{i \in I_{k,s}} \left(\nabla f_{i}\left(x_{k,s}\right)-\nabla f_{i}\left(x_{k,0}\right)\right).
\end{align}
It follows by Assumption~\ref{ass.unbias} and \eqref{eq.zeta} that $\mathbb{E}_{k,s}\left[\zeta_{k,s}\right]=\nabla f(x_{k,s})-\nabla f(x_{k,0})$. By the definition of $\bar{g}_{k,s}$ \eqref{eq.svrg_grad} and Assumption~\ref{ass.function}, and the facts that $ \mathbb{E}[\|z-\mathbb{E}[z]\|^{2}] \leq \mathbb{E}[\|z\|^{2}]$ for random variable $z$ and $\mathbb{E}[\|z_1+\cdots+z_r\|^2] = \mathbb{E}[\|z_1\|^2 + \cdots + \|z_r\|^2]$ for independent mean zero random variables $z_1,\dots,z_r$ \cite{reddi2016stochastic},  
it follows that
\begin{align*}
        \mathbb{E}_{k,s}\left[\|\bar{g}_{k,s} -\nabla f(x_{k,s} )\|_2^2 \right] &= \mathbb{E}_{k,s}\left[\|\zeta_{k,s} + \nabla f(x_{k,0}) - \nabla f(x_{k,s} ) \|_2^2\right]\\
        &= \mathbb{E}_{k,s}\left[\|\zeta_{k,s} - \mathbb{E}_{k,s}[\zeta_{k,s}] \|_2^2\right]\\
        & = \tfrac{1}{b^2}  \mathbb{E}_{k,s} \left[ \left\|\sum_{i \in I_{k,s}} \left(\nabla f_{i}\left(x_{k,s}\right)-\nabla f_{i}\left(x_{k,0}\right)- \mathbb{E}_{k,s}[\zeta_{k,s}]\right) \right\|_2^2  \right] \\
        & = \tfrac{1}{b^2}  \mathbb{E}_{k,s} \left[ \sum_{i \in I_{k,s}} \left\| \nabla f_{i}\left(x_{k,s}\right)-\nabla f_{i}\left(x_{k,0}\right)- \mathbb{E}_{k,s}[\zeta_{k,s}] \right\|_2^2  \right] \\
        & \le \tfrac{1}{b^2}  \mathbb{E}_{k,s} \left[ \sum_{i \in I_{k,s}} \left\| \nabla f_{i}\left(x_{k,s}\right)-\nabla f_{i}\left(x_{k,0}\right) \right\|_2^2  \right] \\
        & \leq \tfrac{L^2}{b}\mathbb{E}_{k,s}\left[\|x_{k,s}  - x_{k,0}\|_2^2\right] = \tfrac{L^2}{b}\|x_{k,s}  - x_{k,0}\|_2^2.
    \end{align*}
\end{proof}

Lemma~\ref{lemma_variance} is one of the major differences between this work and that in \cite{berahas2021sequential}, and in \cite{berahas2021stochastic,curtis2021inexact,curtis2021worst}. Specifically, in \cite{berahas2021sequential} (and in \cite{berahas2021stochastic,curtis2021inexact,curtis2021worst}) it is assumed that the variance in the stochastic gradients employed is bounded uniformly by a constant $M \in \mathbb{R}_{>0}$ (i.e., this would be equivalent to having $M_{k,s} = M$ for all $(k,s) \in \mathbb{N} \times \left[\bar{S}\right]$). This is a classical assumption for the convergence analysis of the SG method \cite{bottou2018optimization,robbins1951stochastic}, which leads to the fact that the algorithm can only converge to a neighborhood  depending on $M$ in expectation when a constant step size is employed. By employing variance reduced gradients, this allows us to control the variance, and diminish it as needed, in order to prove exact convergence of first-order stationary measure in expectation.

In the next two lemmas, we present some useful bounds pertaining to the solutions of the linear system \eqref{eq.system_stochastic}.

\begin{lemma}\label{lemma:25'}
For all $(k,s) \in \mathbb{N} \times \left[\bar{S}\right]$, we always have $\mathbb{E}_{k,s}[\bar{d}_{k,s} ] = d_{k,s} $  
and $\mathbb{E}_{k,s}[\bar{y}_{k,s} ] = y_{k,s}$. In addition, 
there exists some constant $\kappa_d \in \mathbb{R}_{>0}$, independent of $(k, s)$ and any run of the algorithm, with  $\mathbb{E}_{k,s}[\|\bar{d}_{k,s}  - d_{k,s} \|_2] \leq \kappa_d \sqrt{M_{k,s}}$. 
\end{lemma} 
\begin{proof}
The proof of this lemma is similar to that in \cite[Lemma 3.8]{berahas2021sequential}. The first statement follows from the facts that $(i)$ conditioned on $x_{k,s} $, the matrix on the left-hand-side of \eqref{eq.system_stochastic} is deterministic; $(ii)$ under Assumption \ref{ass.function}, the matrix is invertible; $(iii)$ under Assumption \ref{ass.unbias}, $\mathbb{E}_{k,s}[\bar{g}_{k,s} ] = \nabla f(x_{k,s} )$; and $(iv)$ expectation is a linear operator. By \eqref{eq.system_stochastic} for any realization $\bar{g}_{k,s} $, it follows that
\begin{align}\label{eq.imply2}
\begin{bmatrix}
        \bar{d}_{k,s}  - d_{k,s}  \\ \bar{y}_{k,s}  - y_{k,s}
\end{bmatrix} &= - \begin{bmatrix}
        H_{k,s}  & J_{k,s}^T \\ J_{k,s}  & 0
\end{bmatrix}^{-1} \begin{bmatrix}
        \bar{g}_{k,s}  - \nabla f(x_{k,s} ) \\ 0
\end{bmatrix}.
\end{align}
The second result follows by Jensen’s inequality, the concavity of the
square root, and Lemma \ref{lemma_variance}, and where $\kappa_d \in \mathbb{R}_{>0}$ is an upper bound on the norm of the matrix in \eqref{eq.imply2}.
\end{proof}


\begin{lemma} 
For all $(k,s) \in \mathbb{N} \times \left[\bar{S}\right]$, it follows that 
\begin{align}
    g_{k,s}^T d_{k,s}  \geq \mathbb{E}_{k,s}[{\bar{g}_{k,s} }^T\bar{d}_{k,s} ] \geq g_{k,s}^T d_{k,s}  - \zeta^{-1} M_{k,s}, 
\end{align}
\label{lemma:26'}
\end{lemma}
\begin{proof}
The proof is identical to  \cite[Lemma 3.9 (proof)]{berahas2021sequential} with $M$ replaced by $M_{k,s}$ ($M_{k,s}$ defined in Lemma~\ref{lemma_variance}).
%
\end{proof}


We conclude this subsection by defining a Lyapunov function $R : \mathbb{R}^{n} \times \mathbb{R}^{n} \times \mathbb{R}_{>0} \times \mathbb{R}_{>0} \to \mathbb{R}$ that will be used in the analysis. Specifically, 
\begin{align}\label{eq.lyapunov_1}
    R_{k,s}:=R(x_{k,s},x_{k,0},\bar{\tau}_{k,s},\lambda_s) = \mathbb{E}_{k,s}\left[\phi(x_{k,s},\bar{\tau}_{k,s}) + \lambda_s \|x_{k,s}  - x_{k,0}\|_2^2\right],
\end{align}
where $x_{k,s} \in \mathbb{R}^n$ and $\bar{\tau}_{k,s} \in \mathbb{R}_{>0}$ are the iterate and merit parameter at outer-inner iteration $(k,s) \in \mathbb{N} \times \left[\bar{S}\right] $, respectively, $x_{k,0} \in \mathbb{R}^n$ is the reference point at the $k$th outer iteration and $\lambda_s \in \mathbb{R}_{>0}$ is a parameter (defined explicitly later in the analysis). The Lyapunov function is defined as the expected value of the merit function plus the distance squared between any inner iterate and the reference iterate parameterized by a constant. When $s=0$ the Lyapunov function only involves the merit function. Moreover, the last term in the Lyapunov function is similar to that of the upper bound in the variance of the SVRG gradient (see Lemma~\ref{lemma_variance}), and, if the iterates converge, the Lyapunov function reduces to the expected value of the merit function. This is by construction, and will allow us to prove strong theoretical guarantees for \SVRSQP{}.

\subsection{Merit Parameter behavior}\label{sec.merit_parameter_behaviors}

The behavior of the merit parameter $\bar\tau_{k,s}$ requires careful considerations as it is a crucial component of the \SVRSQP{} method and the analysis. Specifically, what is important is the behavior of $\bar\tau_{k,s}$ for large $k \in \mathbb{N}$. As described in \cite{berahas2021sequential}, there are three possible outcomes for $\bar\tau_{k,s}$: ($\romannumeral1$)  converges to zero (vanishes); 
($\romannumeral2$) remains constant at a large positive value; ($\romannumeral3$) remains constant at a sufficiently small positive value. We argue that in the finite-sum setting \eqref{prob.f_finitesum} and under reasonable assumptions, outcome ($\romannumeral1$) is not possible, and outcome ($\romannumeral2$) occurs with probability zero. To show the former, i.e., outcome ($\romannumeral1$) is not possible, we make the following assumption.
\begin{assumption}
  Each component $\{ f_i \}$ of the objective function $f$ in \eqref{prob.f_finitesum} has bounded derivatives over $\mathcal{X}$ (defined in Assumption \ref{ass.function}). 
\label{ass.bounded_gradient}
\end{assumption}
Under 
Assumption~\ref{ass.bounded_gradient}, 
the merit parameter cannot vanish.
\begin{lemma}
Suppose Assumption \ref{ass.bounded_gradient} holds, then there exists $\bar{k}_{\tau}\in \mathbb{N} $ and $\bar\tau_{const} \in \mathbb{R}_{>0}$ such that  $\bar\tau_{k,s} =\bar\tau_{const}$ for all $k \ge \bar{k}_{\tau}$ and $s \in \left[\bar{S}\right]$. 
\label{lemma:bounded_grad_novanish}
\end{lemma}
\begin{proof}
Under Assumption \ref{ass.bounded_gradient}, there exists $g_{\max}\in\mathbb{R}_{>0}$ such that $\|\bar{g}_{k,s} - {g}_{k,s}  \| \le g_{\max}$ for all  $k \ge \bar{k}_{\tau}$ and $s \in \left[\bar{S}\right]$. The desired conclusion follows using similar arguments as in \cite[Proposition 3.18]{berahas2021sequential}. 
\end{proof}

Following a similar argument as that in \cite{berahas2021sequential}, we argue the latter, i.e., ($\romannumeral2$) occurs with probability zero. 

\begin{lemma} 
  Suppose event $E_{\tau\uparrow}$ occurs in the sense that there exists infinite $ (\bar{\mathcal{K}}_{\tau},\bar{\mathcal{S}}_{\tau}) \subseteq \mathbb{N}\times \left[\bar{S}\right]  $ and $\bar\tau_{big} \in \mathbb{R}_{>0}$ such that 
 \begin{align}
   \bar\tau_{k,s} =\bar\tau_{big}   > \tau_{k,s}^{trial}    \text{ for all }   (k,s) \in  (\bar{\mathcal{K}}_{\tau},\bar{\mathcal{S}}_{\tau}).
 \end{align} 
 Moreover, suppose that $ \bar{d}_{k,s}^TH_{k,s}\bar{d}_{k,s} \ge 0 $  for all   $(k,s) \in \mathbb{N} \times \left[\bar{S}\right]$.
Then, $E_{\tau\uparrow}$ occurs with probability zero. 
\end{lemma}
\begin{proof} 
By \eqref{eq.system_stochastic} and $ \bar{d}_{k,s}^TH_{k,s}\bar{d}_{k,s} \ge 0 $ it follows that $\bar{g}_{k,s}^T\bar{d}_{k,s} + \max\{ \bar{d}_{k,s}^TH_{k,s}\bar{d}_{k,s},0\} = \bar{g}_{k,s}^T\bar{d}_{k,s} + \bar{d}_{k,s}^TH_{k,s}\bar{d}_{k,s}  = c_{k,s}^T \bar y_{k,s}$. Similarly, by \eqref{eq.system_det} it follows that ${g}_{k,s}^T {d}_{k,s} + \max\{ {d}_{k,s}^TH_{k,s}{d}_{k,s},0\} = c_{k,s}^T  y_{k,s}$. Since we are considering an objective function composed of a finite number of components, there are a finite number of realizations for $\bar{g}_{k,s}$. Among ${N \choose b}$ possible  realizations of $\bar{g}_{k,s}$, there should at least be one  realization such that $c_{k,s}^T \bar y_{k,s}$ is no smaller than $c_{k,s}^T  y_{k,s}$, since $\mathbb{E}_{k,s}[c_{k,s}^T \bar y_{k,s} ] = c_{k,s}^T  y_{k,s} $ by Lemma \ref{lemma:25'}. Hence, it follows that 
\begin{align*}
     &\mathbb{P}[\bar{g}_{k,s}^T\bar{d}_{k,s} + \max\{ \bar{d}_{k,s}^TH_{k,s}\bar{d}_{k,s},0\} 
     \ge {g}_{k,s}^T {d}_{k,s} + \max\{ {d}_{k,s}^TH_{k,s} {d}_{k,s},0\}  ] \\
     = &\mathbb{P}[  c_{k,s}^T \bar y_{k,s}
     \ge c_{k,s}^T  y_{k,s}   ] \\   \ge& \tfrac{1}{{N \choose b}} 
\end{align*}
The desired conclusion then follows from  \cite[Proposition 3.16]{berahas2021sequential}. 
\end{proof}

If Assumption~\ref{ass.bounded_gradient} holds and $\bar{d}_{k,s}^TH_{k,s}\bar{d}_{k,s} \geq 0$ for all $(k,s) \in \mathbb{N} \times \left[\bar{S}\right]$, then $ \bar\tau_{k,s}$ is guaranteed to remain constant at a sufficiently small positive value eventually with probability 1. While one could prove such a corollary, we instead assume that the merit parameter remains constant at a sufficiently small positive value because the above only provides sufficient conditions, and this merit parameter behavior can potentially be exhibited on a wider class of problems. For the remainder of the paper, we will assume that the merit parameter remains constant at a sufficiently small positive value, and formalize this assumption below. 
\begin{assumption}\label{ass.small_tau}
 Suppose event $E_{\tau_{\min}}$ occurs in the sense that there exists an iteration number $ \bar{k}_{\tau} \in \mathbb{N}$ and a merit parameter value $\bar\tau_{\min} \in \mathbb{R}_{>0}$ such that,
 \begin{align}
   \bar\tau_{k,s} =\bar\tau_{\min}   \le \tau_{k,s}^{trial} \text{ for all }   k \ge \bar{k}_{\tau} \text{ and } s \in \left[\bar{S}\right].
 \end{align} 
  In addition, we further assume that the stochastic gradient sequence $\{\bar g_{k,s} \}_{k \geq \bar{k}_{\tau}, s \in \left[\bar{S}\right]}$ satisfies $\mathbb{E}_{k,s,\tau_{\min}} [\tilde{g}_{k,s}] = g_{k,s}$, where $\mathbb{E}_{k,s,\tau_{\min}}$ denotes the expectation taken conditioned on the event that $E_{\tau_{\min}}$ occurs and that the algorithm has reached $x_{k,0} \in \mathbb{R}^n$ in (outer) iteration $k \in \mathbb{N}$ and $x_{k,s}$ in (outer-inner) iteration $(k,s) \in \mathbb{N} \times \left[\bar{S}\right]$.
\end{assumption}

Assumption \ref{ass.small_tau} is a critical assumption in proving the convergence of the \SVRSQP{} method, and will be assumed to hold throughout the remainder of this section. For ease of exposition, we use $\mathbb{E}_{k,s}$ to denote $\mathbb{E}_{k,s,\tau_{\min}}$, and we define the following quantity
\begin{align*}
    \mathbb{E}_{\tau_{\min}}[\cdot] := \mathbb{E} [\cdot \vert \text{Assumption \ref{ass.small_tau}} ],
\end{align*}
i.e., the total expectation conditioned on the event  $E_{\tau_{\min}}$. Moreover, we define a constant $\phi_{\inf} > -\infty$ as
\begin{equation*}
    \phi_{\inf} := \inf_{x\in\mathcal{X}}\phi(x,\bar\tau_{\min}),
\end{equation*}
and whose existence is guaranteed under Assumptions~\ref{ass.function} and~\ref{ass.small_tau}.


Before we proceed, we state and prove one more technical lemma that will be used in the analysis in Sections~\ref{sec.constant} and \ref{sec.adaptive}.
\begin{lemma} 
Suppose that  Assumption \ref{ass.small_tau} holds. For all $k \ge \bar{k}_{\tau}$ and $s \in \left[\bar{S}\right]$, it follows that 
\begin{align*}
    \mathbb{E}_{k,s}[\Delta l(x_{k,s} ,\bar\tau_{k,s},\bar{g}_{k,s} ,\bar{d}_{k,s} )] \leq \Delta l(x_{k,s} ,\bar\tau_{\min},{g}_{k,s},d_{k,s} ) + \bar\tau_{\min} \zeta^{-1}M_{k,s}.
\end{align*}
\label{lemma:29'}
\end{lemma} 
\begin{proof}
For $k \ge \bar{k}_{\tau}$ and $s \in \left[\bar{S}\right]$, it follows by \eqref{def.merit_model_reduction} and Lemma~\ref{lemma:26'} that 
\begin{align*}
    \mathbb{E}_{k,s}[\Delta l(x_{k,s} ,\bar\tau_{k,s},\bar{g}_{k,s} ,\bar{d}_{k,s} )] &=  \mathbb{E}_{k,s}\left[-\bar\tau_{\min} \bar{g}_{k,s}^T \bar{d}_{k,s} + \| c_{k,s}\|_1 \right] \\
    & =  \mathbb{E}_{k,s}\left[-\bar\tau_{\min} \bar{g}_{k,s}^T \bar{d}_{k,s} + \bar\tau_{\min} {g}_{k,s}^T {d}_{k,s} \right.  \\
    & \left.\qquad \qquad - \bar\tau_{\min} {g}_{k,s}^T {d}_{k,s}  + \| c_{k,s}\|_1 \right] \\
     & \leq \Delta l(x_{k,s} ,\bar\tau_{\min}, {g}_{k,s} ,{d}_{k,s} ) + \bar\tau_{\min} \zeta^{-1}M_{k,s}.
\end{align*}
\end{proof}

\subsection{Constant step size analysis}\label{sec.constant}

In this subsection, we present convergence results for Algorithm \ref{alg.sqp_svrg} with
the constant step size strategy (\textbf{Option I}) under Assumption \ref{ass.small_tau}. The first lemma provides a useful upper bound for the difference in merit function after a step.
\begin{lemma} 
Suppose that  Assumption~\ref{ass.small_tau} holds and $\alpha \in (0,1]$.  For all $k \ge \bar{k}_{\tau}$ and $s \in \left[\bar{S}\right]$, it follows that 
\begin{align*}
     &\ \phi(x_{k,s+1},\bar{\tau}_{k,s}) - \phi(x_{k,s},\bar{\tau}_{k,s})    \\
     \leq&\    - \alpha   \Delta  l(x_{k,s} ,\bar\tau_{\min},g_{k,s} ,d_{k,s} )  + \tfrac{\alpha^2(\bar{\tau}_{\min} L +\Gamma)}{ 2 \bar{\tau}_{\min}\kappa_l}  \Delta l(x_{k,s} ,\bar\tau_{\min},\bar{g}_{k,s} ,\bar{d}_{k,s} ) \\
     & \ + 
    \alpha \bar{\tau}_{\min}  g_{k,s}^T (\bar{d}_{k,s} -{d}_{k,s} ) .
\end{align*}
\label{lemma:24'}
\end{lemma} 
\begin{proof}
For $\alpha \in (0,1]$, $k \ge \bar{k}_{\tau}$ and $s \in \left[\bar{S}\right]$,  by  \eqref{def.merit_model_reduction} and Lemma \ref{lemma:21'}$(b)$ it follows that
\begin{align*}
     &\ \phi(x_{k,s+1},\bar{\tau}_{k,s}) - \phi(x_{k,s},\bar{\tau}_{k,s})    \\
     \leq&\    \alpha (\bar{\tau}_{\min}   g_{k,s}^T   \bar{d}_{k,s}  - \| c_{k,s} \|_1) + \tfrac12 (\bar{\tau}_{\min} L +\Gamma )\alpha^2  \| \bar{d}_{k,s} \|_2^2 \\
     =&\   \alpha (\bar{\tau}_{\min}   g_{k,s}^T   {d}_{k,s}  - \| c_{k,s} \|_1) + \alpha \bar{\tau}_{\min}   g_{k,s}^T (\bar{d}_{k,s} -{d}_{k,s} ) + \tfrac12 (\bar{\tau}_{\min} L +\Gamma )\alpha^2  \| \bar{d}_{k,s} \|_2^2 \\ 
      \le &\    - \alpha   \Delta  l(x_{k,s} ,\bar\tau_{\min},g_{k,s},d_{k,s} )  + 
    \alpha \bar{\tau}_{\min}   g_{k,s}^T (\bar{d}_{k,s} -{d}_{k,s} )  \\  
    & \quad + \tfrac{\alpha^2(\bar{\tau}_{\min} L +\Gamma)}{ 2 \bar{\tau}_{\min}\kappa_l}  \Delta l(x_{k,s} ,\bar\tau_{\min},\bar{g}_{k,s}  ,\bar{d}_{k,s} ). 
\end{align*}
\end{proof}


The next lemma is the central lemma of this subsection; it provides a useful upper bound on the expected value of the sum (over all inner and outer iterations) of the model reduction function of the merit function. 
\begin{lemma}
Suppose that  Assumption \ref{ass.small_tau} holds. Let $\lambda_S = 0$, and 
\begin{equation}
\begin{aligned}
    \lambda_s &= \lambda_{s+1} (\tfrac{  \alpha^2 L^2 }{\kappa_l b \zeta} + \alpha z +1)  +   \tfrac{\alpha^2 (\bar{\tau}_{\min} L +\Gamma )  L^2}{2\kappa_l  b \zeta}, \label{lambda_update} \\
   \Lambda_s  &= \alpha -  \tfrac{ \alpha^2 (\bar{\tau}_{\min} L +\Gamma) }{2 \bar{\tau}_{\min}\kappa_l} - \lambda_{s+1} \tfrac{\alpha }{ \bar{\tau}_{\min}\kappa_l}(\alpha+\tfrac{1}{z}), 
   \end{aligned}
   \end{equation}
   for $s \in \left[\bar{S}\right]$, where 
   $\alpha \in (0,1]$, $z \in \mathbb{R}_{>0}$, $\lambda_s \in \mathbb{R}_{>0}$ are chosen such that $\Lambda_s \in  \mathbb{R}_{>0}$, and $\Lambda_{min} = \min_{s \in \left[\bar{S}\right]} \Lambda_s$. Then, for all $k \ge \bar{k}_{\tau}$ and $s \in \left[\bar{S}\right]$, the sequence of iterates $\{x_{k,s}\}$ generated by Algorithm \ref{alg.sqp_svrg} (\textbf{Option I}) satisfy
\begin{align}
    \mathbb{E}_{\tau_{\min}} \left[\tfrac{1}{(k-\bar{k}_{\tau}+1)S} \sum_{j =  \bar{k}_{\tau}}^{k} \sum_{s=0}^{S-1}  \Delta l(x_{j,s} ,\bar\tau_{\min},g_{j,s},d_{j,s})\right] \le  \tfrac{\mathbb{E}_{\tau_{\min}}[ \phi(x_{\bar{k}_{\tau},0},\bar\tau_{\min}) ] -  \phi_{\inf}}{(k-\bar{k}_{\tau}+1)S\Lambda_{\min}}.
\label{thm1:bound}
\end{align}
\label{lemma:1'}
\end{lemma} 

\begin{proof}
Consider arbitrary $k \ge \bar{k}_{\tau}$ and $s \in \left[\bar{S}\right]$. By Lemmas \ref{lemma:25'}, \ref{lemma:29'} and \ref{lemma:24'}, we have
\begin{align*}
 \mathbb{E}_{k,s}[ \phi(x_{k,s+1},\bar{\tau}_{k,s}) ] & \le  \mathbb{E}_{k,s}[ \phi(x_{k,s},\bar{\tau}_{k,s}) ]  - \alpha    \Delta l(x_{k,s} ,\bar\tau_{\min},g_{k,s},d_{k,s})  \\ &\quad+   \tfrac{(\bar{\tau}_{\min} L +\Gamma )\alpha^2}{2\bar{\tau}_{\min}\kappa_l } \left(  \Delta l(x_{k,s} ,\bar\tau_{\min},g_{k,s},d_{k,s}) + \bar\tau_{\min} \zeta^{-1}M_{k,s}\right) \\
  & = \mathbb{E}_{k,s}[ \phi(x_{k,s},\bar{\tau}_{k,s}) ] - \left( \alpha - \tfrac{\alpha^2(\bar{\tau}_{\min} L +\Gamma )}{2\bar{\tau}_{\min}\kappa_l} \right)  \Delta l(x_{k,s} ,\bar\tau_{\min},g_{k,s},d_{k,s}) \\ 
  & \quad  +    \tfrac{ \alpha^2(\bar{\tau}_{\min} L +\Gamma )   L^2}{2 \kappa_l  b \zeta} \|x_{k,s}  - x_{k,0}\|_2^2.
\end{align*} 
Moreover, by Lemmas \ref{lemma:21'}, \ref{lemma:25'} and \ref{lemma:29'}, and the fact that $2XY = 2 (\sqrt{z} X) (Y/\sqrt{z}) \le z X^2 + Y^2/z $ for $\{X,Y\}\subset \mathbb{R}$ and $z \in \mathbb{R}_{>0}$,  it follows that 
\begin{align*}
    \mathbb{E}_{k,s}\left[\|x_{k,s+1} - x_{k,0}\|_2^2 \right] & =   \mathbb{E}_{k,s}\left[\|x_{k,s+1} -  x_{k,s} + x_{k,s} - x_{k,0}\|_2^2 \right] \\ 
    & = \mathbb{E}_{k,s}\left[ \alpha^2 \| \bar{d}_{k,s}  \|_2^2 \right] +  \|x_{k,s} - x_{k,0}\|_2^2 + 2\alpha  {d}_{k,s}^T (x_{k,s} - x_{k,0})  \\
    & \le \mathbb{E}_{k,s}\left[ \alpha^2 \| \bar{d}_{k,s}  \|_2^2 \right] +  \|x_{k,s} - x_{k,0}\|_2^2   \\ & \quad + 2\alpha \left(\tfrac{1}{2z} \|{d}_{k,s} \|_2^2 + \tfrac{z}{2}   \|x_{k,s} - x_{k,0}\|_2^2\right)  \\
    & \le \mathbb{E}_{k,s}\left[  \tfrac{\alpha^2 }{\bar{\tau}_{\min}\kappa_l}  \Delta l(x_{k,s} ,\bar\tau_{\min},\bar{g}_{k,s} ,\bar{d}_{k,s} )\right] + \|x_{k,s} - x_{k,0}\|_2^2 \\ 
    & \quad +   \tfrac{\alpha  }{z \bar{\tau}_{\min}\kappa_l}  \Delta l(x_{k,s} ,\bar\tau_{\min},g_{k,s},d_{k,s}) + \alpha z  \|x_{k,s} - x_{k,0}\|_2^2 \\
     & \le \tfrac{\alpha }{\bar{\tau}_{\min}\kappa_l}\left(\alpha+\tfrac{1}{z}\right)  \Delta l(x_{k,s} ,\bar\tau_{\min},g_{k,s},d_{k,s}) \\ & \quad + \left(\tfrac{\alpha^2   L^2}{\kappa_l b \zeta} + \alpha z +1 \right) \|x_{k,s} - x_{k,0}\|_2^2.
\end{align*}

Taking total expectation conditioned on the event  $E_{\tau_{\min}}$, for all $k \ge \bar{k}_{\tau}$ and $s \in \left[\bar{S}\right]$, combining the results above and by the definitions of $\lambda_s$, $R_{k,s}$ and  $ \Lambda_s$ and the fact that $\bar{\tau}_{k,s+1}=\bar{\tau}_{k,s}= \bar\tau_{\min} $, it follows that 
\begin{align*}
    \mathbb{E}_{\tau_{\min}}\left[R_{k,s+1}\right] &= \mathbb{E}_{\tau_{\min}} \left[\phi(x_{k,s+1},\bar{\tau}_{k,s+1}) + \lambda_{s+1} \|x_{k,s+1} - x_{k,0}\|_2^2\right] \\
    & \le \mathbb{E}_{\tau_{\min}}\left[R_{k,s}\right] - \Lambda_s \mathbb{E}_{\tau_{\min}}\left[ \Delta l(x_{k,s} ,\bar\tau_{\min},g_{k,s},d_{k,s})\right] \\
    & \le \mathbb{E}_{\tau_{\min}}\left[R_{k,s}\right] - \Lambda_{\min} \mathbb{E}_{\tau_{\min}}\left[ \Delta l(x_{k,s} ,\bar\tau_{\min},g_{k,s},d_{k,s}) \right]. 
\end{align*}
Summing over all inner iterations  ($s \in \left[\bar{S}\right]$), we have 
\begin{align*}
    \sum_{s=0}^{S-1} \mathbb{E}_{\tau_{\min}}\left[ \Delta l(x_{k,s} ,\bar\tau_{\min},g_{k,s},d_{k,s}) \right] &\le \tfrac{\mathbb{E}_{\tau_{\min}}\left[R_{k,0}-R_{k,S} \right] }{\Lambda_{\min}}  \\
    &= \tfrac{\mathbb{E}_{\tau_{\min}}\left[\phi(x_{k,0},\bar\tau_{\min})-\phi(x_{k+1,0},\bar\tau_{\min}) \right] }{\Lambda_{\min}}.
\end{align*} 
The equality follows from the fact that $\lambda_S = 0$ and $x_{k,S} = x_{k+1,0}$.
Summing this inequality for $j \in \{ \bar{k}_{\tau}, \bar{k}_{\tau}+1,\dots, k \}$, we have
\begin{align*}
     \sum_{j = \bar{k}_{\tau}}^{k} \sum_{s=0}^{S-1} \mathbb{E}_{\tau_{\min}}[ \Delta l(x_{j,s} ,\bar\tau_{\min},g_{j,s},d_{j,s})] 
      & \le \tfrac{\mathbb{E}_{\tau_{\min}}[ \phi(x_{\bar{k}_{\tau},0},\bar\tau_{\min})] -  \phi_{\inf}}{\Lambda_{\min}},
\end{align*}
for which the desired conclusion \eqref{thm1:bound} follows.
\end{proof}

As a consequence of Lemma~\ref{lemma:1'}, in Theorem~\ref{theorem:constant_step size_prescribed} we present the main convergence result of this subsection, along with a specification of the controlled parameters (e.g., step size, inner iteration length, etc).

\begin{theorem}\label{theorem:constant_step size_prescribed} Suppose   Assumption \ref{ass.small_tau} holds.
Let $\lambda_s$, $\Lambda_s$ and $\Lambda_{\min}$ be defined as in Lemma~\ref{lemma:1'}. 
Suppose 
$\alpha = \tfrac{\mu_0 b}{(\bar\tau_{\min}L + \Gamma)N^{\gamma}} \in (0,1] $ with $\mu_0 \in (0,1]$,  $z = \tfrac{\bar\tau_{\min}L + \Gamma}{N^{\gamma/2}}$, $\gamma \in (0,1]$, $b < N^{\gamma}$, and $S \le \left\lfloor \tfrac{N^{3\gamma/2}}{\mu_0\left(b+\tfrac{ L^2}{(\bar{\tau}_{\min} L +\Gamma )^2 \kappa_l \zeta} \right)} \right\rfloor$. Then, for all $k \ge \bar{k}_{\tau}$ and $s \in \left[\bar{S}\right]$, there exist universal constants $\mu_0$ and  $\nu_0 \in (0,1)$ such that  $\Lambda_{\min} \ge \tfrac{\nu_0 b}{ (\bar{\tau}_{\min} L +\Gamma ) N^{\gamma}}$ and, 
\begin{align*}
    \mathbb{E}_{\tau_{\min}}  &\left[\tfrac{1}{(k - \bar{k}_{\tau} + 1)S} \sum_{j =  \bar{k}_{\tau}}^{  k} \sum_{s=0}^{S-1}  \Delta l(x_{j,s}, \bar\tau_{\min}, g_{j,s}, d_{j,s})\right] 
    \\ &\qquad \le  \tfrac{(\bar{\tau}_{\min} L +\Gamma ) N^{\gamma}(\mathbb{E}_{\tau_{\min}}[ \phi(x_{\bar{k}_{\tau},0},\bar\tau_{\min}) ] -  \phi_{\inf})}{(k - \bar{k}_{\tau} + 1)S\nu_0 b }.
\end{align*}
\end{theorem} 
\begin{proof}
    By the recursive definition of  $\lambda_s$ \eqref{lambda_update} and the fact that $\lambda_S= 0$, we have that
    \begin{align}\label{eq.svrg_eta0}
    \lambda_0 = \tfrac12 (\bar{\tau}_{\min} L +\Gamma ) \tfrac{\alpha^2   L^2}{ \kappa_{l}b \zeta}\tfrac{((1+\rho)^S - 1)}{\rho},
    \end{align}
    with 
    \begin{align*}
        \rho &= \alpha z + \tfrac{\alpha^2     L^2}{ \kappa_{l}b \zeta}  \\ &= \tfrac{\mu_0 b}{N^{3\gamma/2}}+\tfrac{\mu_0^2 b    L^2}{ (\bar{\tau}_{\min} L +\Gamma )^2 N^{2\gamma} \kappa_{l} \zeta} \\ &\le 
        \mu_0 N^{-3\gamma/2}\left(b+ \tfrac{ b   L^2}{ (\bar{\tau}_{\min} L +\Gamma )^2 \kappa_{l} \zeta} \right),
    \end{align*} 
    where $\mu_0 \in (0,1]$ and $ N \ge 1$.  (Note, without loss of generality, we assume that the user defined constants are chosen such that $\rho \in (0,1)$.) Plugging $\alpha$ and $\rho$ into equation \eqref{eq.svrg_eta0}, it follows that
    \begin{align*}
        \lambda_0 &= \tfrac12 \tfrac{ L^2 \mu_0^2 b}{ \kappa_{l} (\bar\tau_{\min}L + \Gamma)N^{2\gamma} \zeta}  \tfrac{(1+\rho)^S - 1}{\tfrac{\mu_0 b}{N^{3\gamma/2}}+\tfrac{\mu_0^2 b    L^2}{ (\bar{\tau}_{\min} L +\Gamma )^2 N^{2\gamma}\zeta}} \\
        &\leq \tfrac{ L^2 \mu_0 (e - 1) }{2 \kappa_{l} (\bar\tau_{\min}L + \Gamma) \zeta} N^{-\gamma/2},
    \end{align*}
    where the inequality is obtained by noticing that for $l > 0$,  $(1 + \tfrac{1}{l})^l$ is an increasing function and $(1 + \tfrac{1}{l})^l \to e$ as $l \to \infty$. Hence, $(1+\rho)^S \le e$ by the definition of $S$. 
    Now, with the upper bound of $\lambda_0$,  the fact that $\lambda_s$ is decreasing as $s$ increases from 0 to $S$, and $\mu_0 \in (0,1]$, $b<N$ and $N \ge 1$, it follows that $\Lambda_{\min}$ can be lower bounded by
    \begin{align*}
        \Lambda_{\min} &= \min_{0\leq s \leq S-1}\left\{-  \tfrac{ (\bar{\tau}_{\min} L +\Gamma)\alpha^2 }{2 \bar{\tau}_{\min}\kappa_l} + \alpha - \lambda_{s+1} \tfrac{\alpha }{ \bar{\tau}_{\min}\kappa_l}(\alpha+\tfrac{1}{z})\right\}\\
        &> -\tfrac{ (\bar{\tau}_{\min} L +\Gamma)\alpha^2 }{2 \bar{\tau}_{\min}\kappa_l}  + \alpha - \tfrac{\lambda_0\alpha }{\bar{\tau}_{\min}\kappa_l}(\alpha+\tfrac{1}{z})\\
        &\ge - \tfrac{ \mu_0 b}{2\bar{\tau}_{\min}\kappa_l N^{\gamma} } \alpha + \alpha  -\tfrac{ L^2 \mu_0^2  (e - 1)b }{2\kappa_l^2 (\bar\tau_{\min}L + \Gamma)^2 \bar{\tau}_{\min}\zeta } N^{-3\gamma/2} \alpha - \tfrac{ L^2 \mu_0  (e - 1) }{2\kappa_l^2 (\bar\tau_{\min}L + \Gamma)^2 \bar{\tau}_{\min}\zeta }  \alpha \\
        &\geq \alpha \left[1- \tfrac{ \mu_0 b}{\bar{\tau}_{\min}\kappa_l}- \tfrac{ L^2 \mu_0  (e - 1)}{2\kappa_l^2 (\bar\tau_{\min}L + \Gamma)^2 \bar{\tau}_{\min}\zeta }  \right].
    \end{align*}
    Let $\nu_0 = 1- \tfrac{ \mu_0 b}{2\bar{\tau}_{\min}\kappa_l}- \tfrac{ L^2 \mu_0 (e - 1) }{\kappa_l^2 (\bar\tau_{\min}L + \Gamma)^2 \bar{\tau}_{\min} \zeta }$. By  choosing $\mu_0$ (independent of $N$) such that  $\nu_0>0$, it follows that $\Lambda_{\min} \geq \tfrac{b\nu_0}{(\bar\tau_{\min}L + \Gamma) N^\gamma}$. Combining this lower bound with Lemma \ref{lemma:1'} yields the desired result.
\end{proof}

Finally, we conclude this section by presenting a corollary to Theorem~\ref{theorem:constant_step size_prescribed}; this result shows that \SVRSQP{} generates a sequence of iterates whose first order stationary measure (corresponding to \eqref{prob.f_finitesum}) converges to zero.

\begin{corollary}\label{cor.constant}
Under the conditions of Theorem  \ref{theorem:constant_step size_prescribed}, there exists universal constants $\mu_0$, $\nu_0$ such that  $\Lambda_{\min} \ge \tfrac{\nu_0 b}{ (\bar{\tau}_{\min} L +\Gamma ) N^{\gamma}}$ and
\begin{align*}
    \mathbb{E}_{\tau_{\min}} &\left[ \tfrac{1}{(k - \bar{k}_{\tau} + 1)S}   \sum_{j =  \bar{k}_{\tau}}^{k} \sum_{s=0}^{S-1}   \tfrac{\|g_{j,s}  + J_{j,s}  y_{j,s}\|_2^2 }{\kappa_H^2} + \|c_{j,s}  \|_2  \right] \\ 
    &\qquad \le \tfrac{(\bar{\tau}_{\min} L +\Gamma ) N^{\gamma}(\mathbb{E}_{\tau_{\min}}[ \phi(x_{\bar{k}_{\tau},0},\bar\tau_{\min}) ] -  \phi_{\inf})}{(k - \bar{k}_{\tau} + 1)S\nu_0 b}.
\label{cor:complexity_adaptive_max}
\end{align*}

Moreover, if for some $(k,s) \in \mathbb{N} \times \left[\bar{S}\right]$, $\| x_{k,s} - x^* \|_2 \leq \delta_x$, $\| x_{k,0} - x^* \|_2 \leq \delta_{x,0}$ and $\| c_{k,s} \|_2 \leq \delta_c$, for $(\delta_x,\delta_{x,0},\delta_c) \in  \mathbb{R}_{>0} \times \mathbb{R}_{>0} \times \mathbb{R}_{>0}$,  and some stationary point $(x^*,y^*) \in \mathbb{R}^n \times \mathbb{R}^m$ of \eqref{prob.f_finitesum}, then, there exists $\kappa_{g_*} \in \mathbb{R}_{>0}$, such that
\begin{equation*}
    \mathbb{E}_{k,s} \left[
    \left\| \begin{bmatrix}
        \bar{g}_{k,s} - \nabla f(x_*) \\    c_{k,s}       
    \end{bmatrix}\right\|_2 \right]
    \leq \delta_y\;\; \text{ and } \;\; \mathbb{E}_{k,s} \left[
    \left\| \bar{y}_{k,s} - y_{k,s} \right\|_2 \right] \leq \kappa_d \delta_{y} + 2 \kappa_d^2 \Gamma \delta_{x} \|g_*\|
\end{equation*}               
where $\delta_y = \delta_c + L(\delta_{x,0} +\delta_x)/\sqrt{b} +  L\delta_x \in \mathbb{R}_{>0}$ and $\kappa_{d} \in \mathbb{R}_{>0}$ is an upper bound for $
\left\|\left[\begin{array}{cc}
H_{k,s} & J_{k,s}^{T} \\
J_{k,s} & 0
\end{array}\right]\right\|^{-1}
$.
\end{corollary}
\begin{proof} The first part follows by Lemma~\ref{lemma:21'} and Theorem~\ref{theorem:constant_step size_prescribed}, and the fact that by the deterministic variant of \eqref{eq.system_stochastic} for all $(k,s) \in \mathbb{N} \times \left[\bar{S}\right]$
\begin{align*}
     \|g_{k,s}  + J_{k,s}  y_{k,s}\|_2  \leq \| H_{k,s} d_{k,s}\|_2 \leq \| H_{k,s}\|_2 \| d_{k,s}\|_2 \leq \kappa_H\|d_{k,s}\|_2.
\end{align*}

The second part follows  by Assumptions  \ref{ass.function} and \ref{ass.unbias}, the definitions of $(\delta_x,\delta_{x,0},\delta_c)$, and the triangle inequality that $\|  x_{k,s} -x_{k,0}  \|_2 \le \| x_{k,0} - x^* \|_2 + \| x_{k,s} - x^* \|_2 \le \delta_{x,0} + \delta_x$
\begin{align*}
     \mathbb{E}_{k,s} \left[
     \left\| \begin{bmatrix}
        \bar{g}_{k,s} - \nabla f(x_*) \\    c_{k,s}       
    \end{bmatrix}\right\|_2 \right] &\leq \| c_{k,s} \|_2 +\mathbb{E}_{k,s} \left[ \| \bar g_{k,s} - \nabla f(x_{k,s}) \|_2 \right]+ \|  \nabla f(x_{k,s}) - \nabla f(x_*) \|_2 \\ 
    &\leq  \| c_{k,s} \|_2 + \sqrt{\mathbb{E}_{k,s} \left[ \| \bar g_{k,s} - \nabla f(x_{k,s}) \|_2^2 \right]}+  \|  \nabla f(x_{k,s}) - \nabla f(x_*) \|_2  \\
    &\leq \delta_c + L(\delta_{x,0} +\delta_x)/\sqrt{b} + L \delta_x = \delta_y.
\end{align*}
Let $g_{*}:=\nabla f\left(x_{*}\right)$ and $c_{*}:=c\left(x_{*}\right)=0$, for $x_{k,s}$ sufficiently close to $x_{*}$ there exists $\kappa_{g_*} \in \mathbb{R}_{>0}$ such that 
\begin{align*}
&\mathbb{E}_{k,s}\left[\left\|\bar{y}_{k,s}-y_{*}\right\|_{2} \right]\\
 \leq &\mathbb{E}_{k,s}\left[ \left\|\left[\begin{array}{cc}
H_{k,s} & J_{k,s}^{T} \\
J_{k,s} & 0
\end{array}\right]^{-1}\left[\begin{array}{c}
\bar{g}_{k,s} \\
c_{k,s}
\end{array}\right]-\left[\begin{array}{cc}
H_{k,s} & J_{*}^{T} \\
J_{*} & 0
\end{array}\right]^{-1}\left[\begin{array}{c}
g_{*} \\
c_{*}
\end{array}\right]\right\|_{2} \right]\\
=&\mathbb{E}_{k,s}\left[\left\|\left[\begin{array}{cc}
H_{k,s} & J_{k,s}^{T} \\
J_{k,s} & 0
\end{array}\right]^{-1}\left[\begin{array}{cc}
\bar{g}_{k,s}-g_{*} \\
c_{k,s}
\end{array}\right]+\left(\left[\begin{array}{cc}
H_{k,s} & J_{k,s}^{T} \\
J_{k,s} & 0
\end{array}\right]^{-1}-\left[\begin{array}{cc}
H_{k,s} & J_{*}^{T} \\
J_{*} & 0
\end{array}\right]^{-1}\right)\left[\begin{array}{c}
g_{*} \\
0
\end{array}\right]\right\|_{2}\right] \\
\leq & \left\|\left[\begin{array}{cc}
H_{k,s} & J_{k,s}^{T} \\
J_{k,s} & 0
\end{array}\right]^{-1} \right\|_2  \mathbb{E}_{k,s}\left[\left\|\left[\begin{array}{cc}
\bar{g}_{k,s}-g_{*} \\
c_{k,s}
\end{array}\right] \right\|_2 \right] \\  & \quad+ 2 \left\|\left[\begin{array}{cc}
H_{k,s} & J_{k,s}^{T} \\
J_{k,s} & 0
\end{array}\right]^{-1} \right\|_2  \left\|\left[\begin{array}{cc}
0 & (J_* - J_{k,s})^{T} \\
J_* - J_{k,s} & 0
\end{array}\right]^{-1} \right\|_2 \left\|\left[\begin{array}{cc}
H_{k,s} & J_{k,s}^{T} \\
J_{k,s} & 0
\end{array}\right]^{-1} \right\|_2 \left\| \left[\begin{array}{c}
g_{*} \\
0
\end{array}\right] \right\|_2 
\\
\leq & 
\kappa_d \delta_{y} + 2 \kappa_d^2 \Gamma \delta_{x} \|g_*\|
\end{align*}
where the last inequality is satisfied since $(A+\Delta)^{-1} = A^{-1} -  A^{-1} \Delta A^{-1} + \mathcal{O}(\|\Delta\|^2)$, and we assume that   $\|\Delta\|_2 = \|J_* - J_{k,s}\|_2 \le \Gamma \delta_{x}$ is small enough such that $\mathcal{O}(\|\Delta\|^2) \le \|A^{-1} \Delta A^{-1}\|_2$. 
This completes the proof.
\end{proof}

Corollary~\ref{cor.constant} characterizes the behavior of optimality measure $\|g_{k,s}  + J_{k,s}  y_{k,s}\|_2^2$ and feasibility measure $\|c_{k,s}  \|_2$ for all $k \ge \bar{k}_{\tau}$ and $s \in \left[\bar{S}\right]$. The result of Corollary~\ref{cor.constant} reveals that, under the assumption that merit parameter $\bar{\tau}_k$ has stabilized at a sufficiently small value, both measures converge to zero in expectation, 
which justifies our summary in Table~\ref{tab.summary}. It is important to note the difference in nature of the results of Corollary~\ref{cor.constant} and the analogues proven in the unconstrained setting for the SVRG method \cite{reddi2016stochastic}. The first result in Corollary~\ref{cor.constant} is with respect to the expectation of the averaged optimality/feasibility measure across iterations, whereas in \cite{reddi2016stochastic} the results are with respect the the minimal optimality measure (\cite{reddi2016stochastic} considers the unconstrained setting, and so the optimality measure is the norm of the gradient) over the iterations. One can easily derive similar convergence results for \SVRSQP{}. Moreover, if the output of Algorithm \ref{alg.sqp_svrg} is  uniformly chosen from $(k,s)$, where $k \ge \bar{k}_{\tau}$, one can derive a bound for $\mathbb{E}_{\tau_{\min}}\left[\tfrac{\|g_{k,s}  + J_{k,s}  y_{k,s}\|_2^2 }{\kappa_H^2}  + \|c_{k,s}  \|_2\right]$. Finally, we provide an upper bound for the error of the Lagrange multiplier estimates which is dependent on a feasibility measure and  the distances from $x_{k,s}$ and  $x_{k,0}$ to the optimal solution. As a result, if the primal iterates converge to a feasible point and in expectation the SVRG gradient approximation converges to the true gradient of the objective function at the optimal solution, then the Lagrange multipliers also converge. 

We conclude this section with a remark about the iteration complexity of \SVRSQP{}. In the unconstrained setting, by employing variance reduced gradients SVRG is able to improve upon the iteration complexity of the stochastic gradient (SG) method in term of the dependence on $\epsilon$ (the termination tolerance). Specially, in the nonconvex setting the iteration complexity for SVRG is $\mathcal{O}(n + \tfrac{n^{2/3}}{\epsilon})$ \cite{reddi2016stochastic} whereas the iteration complexity for SG is $\mathcal{O}(\tfrac{1}{\epsilon^2})$. In the constrained setting, deriving such results is significantly more difficult due to the fact that one needs to consider two measures of optimality (feasibility and stationarity) and the fact that the merit function  (measure of progress) is changing over the course of the optimization (the merit parameter changes adaptively). Under the assumption that the merit parameter has stabilized at a sufficient small positive value, as a result of Corollary~\ref{cor.constant} one can show that the number of iteration to achieve $\epsilon$-optimality (where the optimality measure is a combination of stationarity and feasibility, i.e.,  $\max\{\| g_{k,s} + J_{k,s}y_{k,s} \|_2^2,\|c_{k,s}\|_2\} \leq \epsilon$) is $\mathcal{O}(n + \tfrac{n^{2/3}}{\epsilon^2})$. To contrast this result, the algorithm in~\cite{berahas2021sequential} (the analogue of the SG method in the equality constrained setting), after the merit parameter has stabilized,  requires $\mathcal{O}(\tfrac{1}{\epsilon^4})$. As a result, it is clear that variance reduction does have an effect, albeit the limited setting under which the result has been derived and the fact that this result does not say anything about the iterations before the merit parameter stabilizes. To the best of our knowledge, the only work that has analyzed the iteration complexity with regards to the whole sequence of merit parameter is~\cite{curtis2021worst}. One can certainly extend that analysis for our algorithm. We defer such analysis to a different study since it would require extending the paper significantly.

\subsection{Adaptive step size} \label{sec.adaptive}

In this subsection, we present convergence results for Algorithm \ref{alg.sqp_svrg} with
the adaptive step size strategy (\textbf{Option II} in the Algorithm \ref{alg.sqp_svrg}) under Assumption \ref{ass.small_tau}. The analysis in this section is significantly more involved than the analysis in Section~\ref{sec.constant} primarily due to the stochastic nature of the step size rule (\eqref{eq:step size1}--\eqref{eq.step size}). Paralleling the analysis of the constant step size strategy, we first provide an upper bound for the difference in merit function after a step.




\begin{lemma} 
Suppose that  Assumption \ref{ass.small_tau} holds. 
For all $k \ge \bar{k}_{\tau}$ and $s \in \left[\bar{S}\right]$, it follows that 
\begin{align*}
     &\ \phi(x_{k,s+1},\bar{\tau}_{k,s}) - \phi(x_{k,s},\bar{\tau}_{k,s})    \\
     \leq&\    - \bar\alpha_{k,s}  \Delta  l(x_{k,s} ,\bar\tau_{\min},g_{k,s},d_{k,s})  + \tfrac12 \bar\alpha_{k,s} \beta  \Delta  l(x_{k,s} ,\bar\tau_{\min},\bar g_{k,s},\bar d_{k,s}) \\
     & \ + 
    \bar\alpha_{k,s} \bar\tau_{\min} g_{k,s}^T (\bar{d}_{k,s} -{d}_{k,s} )  
\end{align*}
\label{lemma:24'_adaptive}
\end{lemma}
\begin{proof}
For $k \ge \bar{k}_{\tau}$ and $s \in \left[\bar{S}\right]$, we consider three cases depending on how the step size is set in Algorithm \ref{alg.sqp_svrg} (Option II). 

\textbf{Case 1:} Suppose in Algorithm \ref{alg.sqp_svrg} (Option II) that $\bar{\widehat\alpha}_{k,s}< 1$, meaning that $\bar\alpha_{k,s}\gets \bar{\widehat\alpha}_{k,s} \le \tfrac{\beta \Delta  l(x_{k,s} ,\bar\tau_{\min},\bar{g}_{k,s},\bar{d}_{k,s}) }{(\bar\tau_{\min} L + \Gamma ) \|\bar{d}_{k,s} \|_2^2}$. 
It then follows from (\ref{eq.reduction_lb}) and Lemma~\ref{lemma:21'} that 
\begin{align*}
     &\ \phi(x_{k,s} + \bar\alpha_{k,s}\bar{d}_{k,s} , \bar\tau_{k,s}) - \phi(x_{k,s} , \bar\tau_{k,s}) \\
      \leq&\ \bar\alpha_{k,s}(\bar{\tau}_{\min} g_{k,s} ^T \bar{d}_{k,s}  - \|c_{k,s} \|_1) + \tfrac12 (\bar{\tau}_{\min} L + \Gamma) \bar\alpha_{k,s}^2 \|\bar{d}_{k,s} \|_2^2 \\
      =&\ \bar\alpha_{k,s}(\bar{\tau}_{\min} g_{k,s} ^T d_{k,s}  - \|c_{k,s} \|_1) + \tfrac12 (\bar{\tau}_{\min} L + \Gamma) \bar\alpha_{k,s}^2 \|\bar{d}_{k,s}\|_2^2 + \bar\alpha_{k,s}\bar{\tau}_{\min} g_{k,s}^T (\bar{d}_{k,s}  - d_{k,s} ) \\
      =&\ -\bar\alpha_{k,s}\Delta  l(x_{k,s} ,\bar\tau_{\min},g_{k,s},d_{k,s}) + \tfrac12 (\bar{\tau}_{\min} L + \Gamma) \bar\alpha_{k,s}^2 \|\bar{d}_{k,s}\|_2^2 \\
      & + \bar\alpha_{k,s}\bar{\tau}_{\min} g_{k,s}^T (\bar{d}_{k,s}  - d_{k,s}) \\
      \le &\  -\bar\alpha_{k,s}\Delta  l(x_{k,s} ,\bar\tau_{\min},g_{k,s},d_{k,s}) + \tfrac12 \bar\alpha_{k,s}\beta \Delta  l(x_{k,s} ,\bar\tau_{\min},\bar g_{k,s},\bar d_{k,s}) \\
      & + \bar\alpha_{k,s}\bar{\tau}_{\min} g_{k,s}^T (\bar{d}_{k,s}  - d_{k,s}) 
\end{align*}

\textbf{Case 2:} Suppose in Algorithm \ref{alg.sqp_svrg} (Option II) that
$\bar{\widetilde\alpha}_{k,s}\leq 1 \leq \bar{\widehat\alpha}_{k,s}$, meaning that $\bar\alpha_{k,s}\gets 1 \leq \tfrac{\beta \Delta  l(x_{k,s} ,\bar\tau_{\min},g_{k,s},d_{k,s}) }{(\bar\tau_{\min} L + \Gamma ) \|\bar{d}_{k,s} \|_2^2}$. Similar to Case 1, it follows that 
\begin{align*}
     &\ \phi(x_{k,s} + \bar\alpha_{k,s}\bar{d}_{k,s} , \bar\tau_{k,s}) - \phi(x_{k,s} , \bar\tau_{k,s}) \\
      \leq&\ -\bar\alpha_{k,s}\Delta  l(x_{k,s} ,\bar\tau_{\min},g_{k,s},d_{k,s}) + \tfrac12 (\bar\tau_{\min} L + \Gamma) \bar\alpha_{k,s}^2 \|\bar{d}_{k,s} \|_2^2 \\
      & + \bar\alpha_{k,s}\bar\tau_{\min} g_{k,s}^T (\bar{d}_{k,s}  - d_{k,s} ) \\
      \le &\  -\bar\alpha_{k,s}\Delta  l(x_{k,s} ,\bar\tau_{\min},g_{k,s},d_{k,s}) + \tfrac12 \bar\alpha_{k,s}\beta \Delta  l(x_{k,s} ,\bar\tau_{\min},\bar g_{k,s},\bar d_{k,s})\\
      & + \bar\alpha_{k,s}\bar\tau_{\min} g_{k,s}^T (\bar{d}_{k,s}  - d_{k,s} ) 
\end{align*}

\textbf{Case 3:} Suppose in Algorithm \ref{alg.sqp_svrg} (Option II) that
$\bar{\widetilde\alpha}_{k,s}> 1$, meaning that $\bar\alpha_{k,s}\gets \bar{\widetilde\alpha}_{k,s} \le \tfrac{\beta \Delta  l(x_{k,s} ,\bar\tau_{\min},g_{k,s},d_{k,s})-4\|c_{k,s} \|_1 }{(\bar\tau_{k,s} L + \Gamma ) \|\bar{d}_{k,s} \|_2^2}$.  It follows from (\ref{eq.reduction_lb}) and Lemma~\ref{lemma:21'} that 
\begin{align*}
      &\ \phi(x_{k,s} + \bar\alpha_{k,s}\bar{d}_{k,s} , \bar\tau_{k,s}) - \phi(x_{k,s} , \bar\tau_{k,s}) \\
      \leq&\ \bar\alpha_{k,s}\bar\tau_{\min} g_{k,s} ^T \bar{d}_{k,s}  + (\bar\alpha_{k,s}- 1)\|c_{k,s} \|_1 - \|c_{k,s} \|_1 + \tfrac12 (\bar\tau_{\min} L + \Gamma) \bar\alpha_{k,s}^2 \|\bar{d}_{k,s} \|_2^2 \\
      =&\ \bar\alpha_{k,s}(\bar\tau_{\min} g_{k,s} ^T \bar{d}_{k,s}  - \|c_{k,s} \|_1) + 2(\bar\alpha_{k,s}- 1)\|c_{k,s} \|_1 + \tfrac12 (\bar\tau_{\min} L + \Gamma) \bar\alpha_{k,s}^2 \|\bar{d}_{k,s} \|_2^2 \\
      \leq&\ \bar\alpha_{k,s}(\bar\tau_{\min} g_{k,s} ^T d_{k,s}  - \|c_{k,s} \|_1) + 2 \bar\alpha_{k,s}\|c_{k,s} \|_1 + \tfrac12 (\bar\tau_{\min} L + \Gamma) \bar\alpha_{k,s}^2 \|\bar{d}_{k,s} \|_2^2 \\
      & \quad + \bar\alpha_{k,s}\bar\tau_{\min} g_{k,s} ^T (\bar{d}_{k,s}  - d_{k,s} ) \\
     \le &\  -\bar\alpha_{k,s}\Delta l(x_{k,s} ,\bar\tau_{\min},g_{k,s},d_{k,s}) + \tfrac12 \bar\alpha_{k,s}\beta \Delta  l(x_{k,s} ,\bar\tau_{\min},\bar g_{k,s},\bar d_{k,s}) \\
      & \quad + \bar\alpha_{k,s}\bar\tau_{\min} g_{k,s} ^T (\bar{d}_{k,s}  - d_{k,s} ) 
\end{align*}

The result follows by combining the three cases.
\end{proof}

While the upper bounds for the difference in merit function after a step for the two step size strategies are very similar (Lemmas~\ref{lemma:24'} and \ref{lemma:24'_adaptive}, respectively), a key difference pertains to the fact step sizes computed by the adaptive algorithm (\textbf{Option II}) are stochastic, and as such the last term in the bound in Lemma~\ref{lemma:24'_adaptive} is nonzero in expectation. Moreover, due to the adaptive and stochastic nature of the step size strategy, an additional user-defined parameter $\alpha_u\in\mathbb{R}_{>0}$ is required. Before we proceed, we make the following remark with regards to the selection of  $\alpha_u$.
\begin{remark}
Under Assumption~\ref{ass.small_tau} and by Lemma~\ref{lemma:21'}$(b)$, it follows that for all $k \ge \bar{k}_{\tau}$ and $s \in \left[\bar{S}\right]$
\begin{align*}
	\tfrac{ \Delta  l(x_{k,s} ,\bar\tau_{\min},\bar g_{k,s},\bar d_{k,s}) }{(\bar\tau_{\min} L + \Gamma ) \|\bar{d}_{k,s} \|_2^2} \geq \tfrac{\kappa_l \bar{\tau}_{\min}}{\bar{\tau}_{\min} L+\Gamma}\in\mathbb{R}_{>0}.
\end{align*}
When the user-defined parameters $\alpha_u\in\mathbb{R}_{>0}$ and $\beta\in (0,1]$ are chosen such that $\alpha_u \beta \leq \tfrac{\beta\kappa_l \bar{\tau}_{\min}}{\bar{\tau}_{\min} L+\Gamma} \in (0,1]$, it follows that \textbf{Option II} in Algorithm~\ref{alg.sqp_svrg} always selects a constant step size $\alpha_u\beta\in (0,1]$, whose analysis has already been discussed in Section~\ref{sec.constant}. Therefore, under Assumption~\ref{ass.small_tau}, for the rest of this subsection we only consider the case where $\alpha_u > \tfrac{\kappa_l \bar{\tau}_{\min}}{\bar{\tau}_{\min} L+\Gamma}$ and $\beta \in \mathbb{R}_{>0}$ is chosen such that $\tfrac{\beta\kappa_l \bar{\tau}_{\min}}{\bar{\tau}_{\min} L+\Gamma} \in (0,1]$.
\end{remark}

Next, we provide upper and lower bounds for the step sizes $\bar\alpha_{k,s} \in \mathbb{R}_{>0}$ chosen by \SVRSQP{}. 

\begin{lemma} 
\label{lemma:bound_for_step size}
Suppose that Assumption \ref{ass.small_tau} holds. 
Let $\bar{\alpha}_{k,s}$ be defined as in \eqref{eq:step size1}--\eqref{eq.step size}, and consider $\alpha_l:= \tfrac{\kappa_l \bar{\tau}_{\min}}{\bar{\tau}_{\min} L+\Gamma} \in \mathbb{R}_{>0}$ with $\alpha_l < \alpha_u$. 
Suppose $\beta \in (0,1]$ is chosen such that $ \tfrac{\beta \kappa_l \bar\tau_{\min}}{ \bar\tau_{\min} L + \Gamma} \in (0,1]$, then for all $k \ge \bar{k}_{\tau}$ and $s \in \left[\bar{S}\right]$, it  follows that $\bar\alpha_{k,s}\in [\alpha_l \beta,\alpha_u \beta ]$.
\end{lemma} 
\begin{proof}
By Lemma \ref{lemma:21'}$(b)$, it follows that $ \tfrac{ \Delta  l(x_{k,s} ,\bar\tau_{\min},\bar g_{k,s},\bar d_{k,s}) }{(\bar\tau_{\min} L + \Gamma ) \|\bar{d}_{k,s} \|_2^2}  \ge \tfrac{\kappa_l \bar{\tau}_{\min}}{\bar{\tau}_{\min} L+\Gamma}$
for all $k \ge \bar{k}_{\tau}$ and $s \in \left[\bar{S}\right]$. By \eqref{eq:step size1}--\eqref{eq.step size}, the desired conclusion follows by considering the following three cases.

\textbf{Case 1}: Suppose that $\bar{\widehat\alpha}_{k,s} =  \min\left\{ \tfrac{ \Delta l(x_{k,s} ,\bar\tau_{k,s},\bar g_{k,s} ,\bar d_{k,s} ) }{(\bar\tau_{k,s} L_{k,s} + \Gamma_{k,s} ) \|\bar{d}_{k,s} \|_2^2}, \alpha_u \right\} \beta < 1$, in which case the
algorithm sets $\bar\alpha_{k,s} = \bar{\widehat\alpha}_{k,s}$. It follows that 
\begin{align*}
    \alpha_l \beta = \tfrac{\kappa_l \bar{\tau}_{\min}}{\bar{\tau}_{\min} L+\Gamma} \beta \le \min\left\{\tfrac{\kappa_l \bar{\tau}_{\min}}{\bar{\tau}_{\min} L+\Gamma},\alpha_u  \right\}\beta \le  \bar\alpha_{k,s} \le \alpha_u \beta .
\end{align*}

\textbf{Case 2}: Suppose that  $ \bar{\widetilde\alpha}_{k,s} = \bar{\widehat\alpha}_{k,s}- \tfrac{4\|c_{k,s} \|_1}{(\bar\tau_{k,s} L_{k,s} +  \Gamma_{k,s})\|\bar{d}_{k,s} \|_2^2} \le 1 \le \bar{\widehat\alpha}_{k,s}$, in which case the
algorithm sets $\bar\alpha_{k,s} = 1$. It follows that 
\begin{align*}
  \alpha_l \beta \le  1 = \bar\alpha_{k,s}  \le \bar{\widehat\alpha}_{k,s} \le \alpha_u \beta.
\end{align*}

\textbf{Case 3}: Suppose that $\bar{\widetilde\alpha}_{k,s} > 1 $, in which case the
algorithm sets $\bar\alpha_{k,s} = \bar{\widetilde\alpha}_{k,s}$. It follows that 
\begin{align*}
  \alpha_l \beta \le  1<  \bar\alpha_{k,s} = \bar{\widehat\alpha}_{k,s}- \tfrac{4\|c_{k,s} \|_1}{(\bar\tau_{k,s} L_{k,s} +  \Gamma_{k,s})\|\bar{d}_{k,s} \|_2^2} \le \bar{\widehat\alpha}_{k,s} \le \alpha_u \beta
\end{align*}
\end{proof}

As mentioned above, due to the adaptive (and stochastic) nature of the step size strategy, the third term on the right-hand-side of the bound in Lemma~\ref{lemma:24'_adaptive} is nonzero in expectation. We provide an upper bound for this quantity in the next lemma.

\begin{lemma} 
Suppose that  Assumption \ref{ass.small_tau} holds. 
For all $k \ge \bar{k}_{\tau}$ and $s \in \left[\bar{S}\right]$, it follows that
\begin{align*} 
   &\mathbb{E}_{k,s}\left[\bar\alpha_{k,s}\bar{\tau}_{k,s}   g_{k,s}^T (\bar{d}_{k,s} -{d}_{k,s} )\right] \\
   \leq \ &\tfrac{\alpha_u\kappa_H \kappa_d  \beta^2 }{2 \kappa_l}   \Delta l(x_{k,s} ,\bar\tau_{\min},g_{k,s},d_{k,s}) + \tfrac{\alpha_u \bar\tau_{\min}\kappa_H \kappa_d L^2 }{2b} \|x_{k,s} - x_{k,0}\|_2^2.
\end{align*}
\label{lemma:adaptive_phi_reduction3}
\end{lemma}
\begin{proof}
For all $k \ge \bar{k}_{\tau}$ and $s \in \left[\bar{S}\right]$, by Lemmas~\ref{lemma:21'}, \ref{lemma:25'} and \ref{lemma:bound_for_step size}, Cauchy–Schwarz inequality and the fact of $2XY = 2 (\sqrt{\beta} X) (Y/\sqrt{\beta}) \le \beta X^2 + Y^2/\beta $ for any $\{X,Y\}\subset\mathbb{R}$, it follows that
\begin{align*}
    &\mathbb{E}_{k,s}\left[\bar\alpha_{k,s} \bar\tau_{k,s} g_{k,s}^T (\bar{d}_{k,s} - d_{k,s})\right] \\
    = \ &\mathbb{E}_{k,s}\left[  \bar\alpha_{k,s} \bar\tau_{\min} (g_{k,s} + J_{k,s}^T y_{k,s} )^T (d_{k,s} -  \bar{d}_{k,s})  \right]  \\
    = \ &\mathbb{E}_{k,s}\left[  \bar\alpha_{k,s} \bar\tau_{\min} (-H_{k,s} d_{k,s} )^T (d_{k,s}- \bar{d}_{k,s})  \right]  \\
    \leq \ &\alpha_u \beta \bar\tau_{\min}\kappa_H   \| {d}_{k,s} \|_2 \mathbb{E}_{k,s}[\|\bar{d}_{k,s} - {d}_{k,s} \|_2]\\
    \leq \ &\alpha_u \beta \bar\tau_{\min}\kappa_H \kappa_d  \| {d}_{k,s} \|_2 \sqrt{M_{k,s}} \\
    \leq \ &\alpha_u \beta \bar\tau_{\min}\kappa_H \kappa_d \left( \tfrac{\| {d}_{k,s} \|_2^2 \beta }{2} +\tfrac{M_{k,s}}{2\beta}\right) \\
    \leq \ &\tfrac{\alpha_u  \kappa_H \kappa_d \beta^2}{2 \kappa_l}   \Delta l(x_{k,s} ,\bar\tau_{\min},g_{k,s},d_{k,s}) + \tfrac{\alpha_u  \bar\tau_{\min}\kappa_H \kappa_d L^2 }{2b} \|x_{k,s} - x_{k,0}\|_2^2.
\end{align*}
\end{proof}

Lemma~\ref{lemma:core_adaptive} and Theorem~\ref{theorem:constant_adaptive_prescribed} (below) are the analogues of Lemma~\ref{lemma:1'} and Theorem~\ref{theorem:constant_step size_prescribed}, respectively, for the adaptive step size case.

\begin{lemma}
\label{lemma:core_adaptive}
Suppose that  Assumption \ref{ass.small_tau} holds. Let $\alpha_l$ and $\alpha_u$ be defined as Lemma \ref{lemma:bound_for_step size}, and $\beta \in (0,1]$ chosen such that $\tfrac{\beta \kappa_l \bar\tau_{\min} }{(\bar\tau_{\min} L + \Gamma)} \in (0,1]$. In addition, let $\lambda_S = 0$, and 
\begin{equation}
\begin{aligned}
    \lambda_s &= \lambda_{s+1}(1+z)(1+ \tfrac{ a_u^2 \beta ^2   L^2}{\kappa_l z b \zeta})  +  \tfrac{\alpha_u  \bar\tau_{\min} L^2}{2b} ( \kappa_H \kappa_d + \tfrac{\beta^2}{\zeta} )\\
   \Lambda_s  &=   \alpha_l \beta-\tfrac12 \alpha_u \beta^2    - \tfrac{\alpha_u\kappa_H \kappa_d  \beta^2}{2 \kappa_l} - \lambda_{s+1}(1+z) \tfrac{ \alpha_u^2 \beta ^2 }{\bar{\tau}_{\min}\kappa_lz}
\end{aligned}
\label{def.lambda_def_adaptive} 
\end{equation}
for $s \in \left[\bar{S}\right]$, where $\beta$, $z \in \mathbb{R}_{>0}$, $\lambda_s \in \mathbb{R}_{>0}$ are chosen such that $\Lambda_s \in  \mathbb{R}_{>0}$, and $\Lambda_{min} = \min_{s \in \left[\bar{S}\right]} \Lambda_s$. Then, for all $k \ge \bar{k}_{\tau}$ and $s \in \left[\bar{S}\right]$, the sequence of iterates $\{x_{k,s}\}$ generated by Algorithm \ref{alg.sqp_svrg} (\textbf{Option II}) satisfy
\begin{equation}
     \mathbb{E}_{\tau_{\min}} \left[ \tfrac{1}{(k - \bar{k}_{\tau} + 1)S}\sum_{j =  \bar{k}_{\tau}}^{k} \sum_{s=0}^{S-1}  \Delta l(x_{j,s} ,\bar\tau_{\min},g_{j,s},d_{j,s})\right]   
     \le \tfrac{\mathbb{E}_{\tau,small}\left[ \phi(x_{\bar{k}_{\tau},0},\bar\tau_{\min})\right] -  \phi_{\inf}}{(k - \bar{k}_{\tau} + 1)S\Lambda_{\min}}.
\label{thm2:complexity_adaptive}
\end{equation}
\end{lemma}
\begin{proof}
Consider arbitrary $k \ge \bar{k}_{\tau}$ and $s \in \left[\bar{S}\right]$. By Lemmas~\ref{lemma_variance},   \ref{lemma:29'}, \ref{lemma:adaptive_phi_reduction3}, we have
\begin{align*}
 &\mathbb{E}_{k,s}[ \phi(x_{k,s+1},\bar{\tau}_{k,s}) ] \\
 \le \ &\mathbb{E}_{k,s}[ \phi(x_{k,s},\bar{\tau}_{k,s}) ]  
 - \alpha_l \beta   \Delta  l(x_{k,s} ,\bar\tau_{\min},g_{k,s},d_{k,s})\\ 
 &+ \tfrac12 \alpha_u \beta^2    \mathbb{E}_{k,s}[\Delta  l(x_{k,s} ,\bar\tau_{\min},\bar g_{k,s},\bar d_{k,s})] + \mathbb{E}_k[\bar\alpha_{k,s}\bar{\tau}_{k,s}   g_{k,s}^T (\bar{d}_{k,s} -{d}_{k,s} )] \\
 \le \ &\mathbb{E}_{k,s}[ \phi(x_{k,s},\bar{\tau}_{k,s}) ] - \beta \left(\alpha_l - \tfrac12\alpha_u \beta - \tfrac{\alpha_u  \kappa_H \kappa_d \beta}{2 \kappa_l}  \right) \Delta  l(x_{k,s} ,\bar\tau_{\min},g_{k,s},d_{k,s})\\ 
 &+ \tfrac{\alpha_u  \bar\tau_{\min} L^2}{2b} \left( \kappa_H \kappa_d + \tfrac{\beta^2}{\zeta} \right) \|x_{k,s} - x_{k,0}\|^2.
\end{align*} 
Moreover, similar to the proof of lemma \ref{lemma:1'}, we have
\begin{align*}
    &\mathbb{E}_{k,s}[\|x_{k,s+1} - x_{k,0}\|_2^2 ] \\
    = \ &\mathbb{E}_{k,s}[\|x_{k,s+1} -  x_{k,s} + x_{k,s} - x_{k,0}\|_2^2 ] \\ 
    = \ &\mathbb{E}_{k,s}[ \| \bar\alpha_{k,s}  \bar{d}_{k,s}  \|_2^2 +  \|x_{k,s} - x_{k,0}\|_2^2 + 2 \bar\alpha_{k,s} \bar{d}_{k,s}^T (x_{k,s} - x_{k,0}) ] \\
    \le \ &\mathbb{E}_{k,s}\left[ \| (1+\tfrac1z) \bar\alpha_{k,s}  \bar{d}_{k,s}  \|^2 +  (1+z) \|x_{k,s} - x_{k,0}\|^2 \right]  \\ 
    \le \ &(1+z) \tfrac{ a_u^2 \beta ^2 }{\bar{\tau}_{\min}\kappa_lz}   \mathbb{E}_{k,s}\left[   \Delta l(x_{k,s} ,\bar\tau_{\min},\bar{g}_{k,s} ,\bar{d}_{k,s} )\right] + (1+z) \|x_{k,s} - x_{k,0}\|^2 \\ 
     \le \ &(1+z) \tfrac{ a_u^2 \beta ^2 }{\bar{\tau}_{\min}\kappa_lz}   \Delta l(x_{k,s} ,\bar\tau_{\min},g_{k,s},d_{k,s}) + \left[1+z+ (1+z) \tfrac{ a_u^2 \beta ^2  L^2}{\kappa_l z b \zeta} \right] \|x_{k,s} - x_{k,0}\|^2.
\end{align*}

Taking total expectation conditioned on $E_{\tau_{\min}}$, for all $k \ge \bar{k}_{\tau}$ and $s \in \left[\bar{S}\right]$, combining the results above and the definitions of $\lambda_s,  \Lambda_s$, it follows that
\begin{align*}
    \mathbb{E}_{\tau_{\min}}[{R}_{k,s+1}] &= \mathbb{E}_{\tau_{\min}} [\phi(x_{k,s+1},\bar{\tau}_{k,s}) + \lambda_{s+1} \|x_{k,s+1} - x_{k,0}\|^2 ] \\
    & \le \mathbb{E}_{\tau_{\min}}[ \phi(x_{k,s},\bar{\tau}_{k,s}) ] + 
 \left[ \lambda_{s+1}(1+z)(1+ \tfrac{ a_u^2 \beta ^2   L^2}{\kappa_l z b \zeta}) \right. \\ 
    &\quad  + \left. \tfrac{\alpha_u \bar\tau_{\min} L^2}{2b} ( \kappa_H \kappa_d + \tfrac{\beta^2}{\zeta} )  \right] \mathbb{E}_{\tau_{\min}}[\|x_{k,s} - x_{k,0}\|^2] \\
 &\quad  -  \left( a_l \beta -\tfrac12 a_u \beta^2   -  \tfrac{\alpha_u  \kappa_H \kappa_d  \beta^2}{2 \kappa_l}- \lambda_{s+1}(1+z) \tfrac{ a_u^2 \beta^2 }{\bar{\tau}_{\min}\kappa_lz}\right) \\ 
 &\qquad \mathbb{E}_{\tau_{\min}}[  \Delta  l(x_{k,s} ,\bar\tau_{\min},g_{k,s},d_{k,s})] \\
 & \le \mathbb{E}_{\tau_{\min}}[{R}_{k,s}] - \Lambda_s \mathbb{E}_{\tau_{\min}}[\Delta  l(x_{k,s} ,\bar\tau_{\min},g_{k,s},d_{k,s})] \\
   & \le \mathbb{E}_{\tau_{\min}}[{R}_{k,s}] - \Lambda_{\min} \mathbb{E}_{\tau_{\min}}[\Delta  l(x_{k,s} ,\bar\tau_{\min},g_{k,s},d_{k,s})].
\end{align*}
Summing over all inner iterations  ($s \in \left[\bar{S}\right]$), we have 
\begin{align*}
    \sum_{s=0}^{S-1}  \mathbb{E}_{\tau_{\min}}[ \Delta l(x_{k,s} ,\bar\tau_{\min},g_{k,s},d_{k,s})] &\le \tfrac{\mathbb{E}_{\tau_{\min}}[R_{k,0}-R_{k,S} ] }{\Lambda_{\min}}  \\
    &= \tfrac{\mathbb{E}_{\tau_{\min}}[\phi(x_{k,0},\bar\tau_{\min})-\phi(x_{k+1,0},\bar\tau_{\min})] }{\Lambda_{\min}}.
\end{align*} 
The equality follows from the fact that $\lambda_S = 0$ and $x_{k,S} = x_{k+1,0}$. The desired conclusion \eqref{thm2:complexity_adaptive} then follows by
summing this inequality for $j \in \{  \bar{k}_{\tau},\bar{k}_{\tau}+1, \dots, k \}$.
\end{proof}


As a consequence of Lemma~\ref{lemma:core_adaptive}, in Theorem~\ref{theorem:constant_adaptive_prescribed} we present the main convergence result of this subsection.

\begin{theorem} Suppose Assumption \ref{ass.small_tau} holds. Let $\lambda_s$, $\Lambda_s$ and $\Lambda_{\min}$ be defined as in Lemma~\ref{lemma:core_adaptive}. 
Suppose $\beta = \tfrac{\mu_1 b}{(\bar\tau_{\min}L + \Gamma)N^{\gamma}} \in (0,1] $ with $\mu_1 \in (0,1]$,  $z = \tfrac{\bar\tau_{\min}L + \Gamma}{N^{\gamma/2}}$, and $S \le \left\lfloor \tfrac{N^{\gamma/2}}{(\bar\tau_{\min}L + \Gamma ) + \tfrac{\mu_1^2 a_u^2   L^2 b}{ (\bar\tau_{\min}L + \Gamma)^3 \kappa_l \zeta } + \tfrac{\mu_1^2 a_u^2   L^2 b}{ (\bar\tau_{\min}L + \Gamma)^2 \kappa_l \zeta }} \right\rfloor$. Define the quantity $\Lambda_{\min} = \min_s \Lambda_s$. Then for $b < N^{\gamma}$, there exists universal constants $\mu_1$, $\nu_1$ such that:  $\Lambda_{\min} \ge \tfrac{\nu_1 b}{ (\bar{\tau}_{\min} L +\Gamma ) N^{\gamma}}$ and
\begin{align*}
    \mathbb{E}_{\tau_{\min}}& \left[\tfrac{1}{(k - \bar{k}_{\tau} + 1)S} \sum_{j =  \bar{k}_{\tau}}^{k} \sum_{s=0}^{S-1}  \Delta l(x_{j,s}, \bar\tau_{\min}, g_{j,s}, d_{j,s})\right] 
    \\  
    &\qquad \le\tfrac{(\bar{\tau}_{\min} L +\Gamma ) N^{\gamma}(\mathbb{E}_{\tau_{\min}}[ \phi(x_{\bar{k}_{\tau},0},\bar\tau_{\min}) ] -  \phi_{\inf})}{(k - \bar{k}_{\tau} + 1)S\nu_1 b }.
\end{align*}
\label{theorem:constant_adaptive_prescribed}
\end{theorem} 

\begin{proof}
By the recursive definition of  $\lambda_s$ and the fact that $\lambda_S= 0$, we have that
  \begin{align}\label{eq.svrg_eta0_1}
    \lambda_0 = ( \kappa_H \kappa_d + \tfrac{\beta^2}{\zeta} )\tfrac{\alpha_u  \bar\tau_{\min} L^2}{2b}  \tfrac{(1+\rho)^S - 1}{\rho}, 
 \end{align}
with 
\begin{align*}
        \rho &= z + (1+z) \tfrac{ a_u^2 \beta ^2   L^2}{\kappa_l z b \zeta} \\ &=  (\bar\tau_{\min}L + \Gamma) N^{-\gamma/2} + (\tfrac{ N^{\gamma/2}}{\bar\tau_{\min}L + \Gamma}  + 1 ) \tfrac{\mu_1^2 a_u^2   L^2 b}{ (\bar\tau_{\min}L + \Gamma)^2 \kappa_l  N^{2\gamma} \zeta} \\
        & \le \left((\bar\tau_{\min}L + \Gamma ) + \tfrac{\mu_1^2 a_u^2   L^2 b}{ (\bar\tau_{\min}L + \Gamma)^3 \kappa_l \zeta } + \tfrac{\mu_1^2 a_u^2   L^2 b}{ (\bar\tau_{\min}L + \Gamma)^2 \kappa_l \zeta }\right) N^{-\gamma/2}
\end{align*} 
It follows that 
    \begin{align*}
        \lambda_0 &\le \tfrac{\alpha_u  \bar\tau_{\min} L^2}{2b} ( \kappa_H \kappa_d + \tfrac{\mu_1^2 b^2}{(\bar\tau_{\min} L + \Gamma)^2 N^{2\gamma}\zeta}) \tfrac{e - 1}{(\bar\tau_{\min}L + \Gamma) N^{-\gamma/2}}
    \end{align*}
    where the inequality is obtained by noticing that for $l > 0$,  $(1 + \tfrac{1}{l})^l$ is an increasing function and $(1 + \tfrac{1}{l})^l \to e$ as $l \to \infty$. Hence, $(1+\rho)^S \le e$ by the definition of $S$. Now, with the upper bound of $\lambda_0$,  the fact that $\lambda_s$ is decreasing as $s$ increases from 0 to $S$, and $\mu_0 \in (0,1]$ and $N \ge 1$, we can lower bound $\Lambda_{\min}$ as 
    \begin{align*}
        \Lambda_{\min} &= \min_{0\leq s \leq S-1}\left\{ \alpha_l \beta-\tfrac12 \alpha_u \beta^2    - \tfrac{\alpha_u\kappa_H \kappa_d  \beta^2}{2 \kappa_l} - \lambda_{s+1}(1+z) \tfrac{ \alpha_u^2 \beta ^2 }{\bar{\tau}_{\min}\kappa_lz} \right\}\\
        &>  \alpha_l \beta-\tfrac12 \alpha_u \beta^2    - \tfrac{\alpha_u\kappa_H \kappa_d  \beta^2}{2 \kappa_l} - \lambda_{0}(1+z) \tfrac{ \alpha_u^2 \beta ^2 }{\bar{\tau}_{\min}\kappa_lz}\\
        &\geq \beta \left[\alpha_l - \tfrac{\alpha_u  b \mu_1}{2(\bar\tau_{\min}L + \Gamma)N^{\gamma}} - \tfrac{\alpha_u\kappa_H \kappa_d  \mu_1}{2 \kappa_l (\bar\tau_{\min}L + \Gamma)N^{\gamma}} - \tfrac{\alpha_u^3  L^2 \kappa_H \kappa_d (e - 1)\mu_1}{2\kappa_l(\bar\tau_{\min}L + \Gamma)^3}\right. \\ &\left. \quad - \tfrac{\alpha_u^3  L^2 \kappa_H \kappa_d (e - 1)\mu_1}{2\kappa_l(\bar\tau_{\min}L + \Gamma)^2 N^{\gamma/2}}  - \tfrac{\alpha_u^3  \bar\tau_{\min} L^2  b^2 (e - 1)\mu_1}{2(\bar\tau_{\min} L + \Gamma)^5 N^{2\gamma}\zeta \kappa_l}- \tfrac{\alpha_u^3  \bar\tau_{\min} L^2 b^2 (e - 1)\mu_1 }{2(\bar\tau_{\min} L + \Gamma)^4 N^{5\gamma/2}\zeta \kappa_l} \right]
    \end{align*}
   Let $\nu_1 = \alpha_l - \tfrac{\alpha_u  b \mu_1}{2(\bar\tau_{\min}L + \Gamma)N^{\gamma}} - \tfrac{\alpha_u\kappa_H \kappa_d  \mu_1}{2 \kappa_l (\bar\tau_{\min}L + \Gamma)N^{\gamma}} - \tfrac{\alpha_u^3  L^2 \kappa_H \kappa_d (e - 1)\mu_1}{2\kappa_l(\bar\tau_{\min}L + \Gamma)^3} - \tfrac{\alpha_u^3  L^2 \kappa_H \kappa_d (e - 1)\mu_1}{2\kappa_l(\bar\tau_{\min}L + \Gamma)^2 N^{\gamma/2}}$ \\
   $- \tfrac{\alpha_u^3  \bar\tau_{\min} L^2  b^2 (e - 1)\mu_1}{2(\bar\tau_{\min} L + \Gamma)^5 N^{2\gamma}\zeta \kappa_l}- \tfrac{\alpha_u^3  \bar\tau_{\min} L^2 b^2 (e - 1)\mu_1 }{2(\bar\tau_{\min} L + \Gamma)^4 N^{5\gamma/2}\zeta \kappa_l}$. 
 By  choosing $\mu_1$ (independent of $N$) such that  $\nu_1>0$, it follows that 
   $\Lambda_{\min} \geq \tfrac{b\nu_1}{(\bar\tau_{\min}L + \Gamma) N^\gamma}$. Combining this lower bound with Lemma \ref{lemma:1'} yields the desired result. 
\end{proof}

We conclude this section by noting that an analogue of Corollary~\ref{cor.constant} can be proven for the case in which adaptive step sizes are utilized. For brevity we omit this corollary since it is identical to Corollary~\ref{cor.constant} up to constants.

\section{Numerical Results}\label{sec:experiments}


In this section, we demonstrate the empirical performance of a Matlab implementation of Algorithm~\ref{alg.sqp_svrg}, with both \textbf{Options I} and \textbf{II}, for solving equality constrained binary classification machine learning problems. Specifically, we consider constrained logistic regression problems
(datasets from the LIBSVM collection \cite{chang2011libsvm}) with linear equality constraints or an $\ell_2$   norm squared constraint. All experiments were run in Matlab R2021b on macOS 12.2 with an Apple M1 Pro chip and 16GB memory.

In order to illustrate the merits of our proposed algorithm, we compared two variants of the \SVRSQP{} method (constant step sizes \SVRSQPCONST{} and adaptive steps sizes  \SVRSQPADAPT{}) with  the stochastic SQP method from \cite{berahas2021sequential} (\StoSQP) and a Stochastic Subgradient method that utilizes SVRG-type variance reduced gradient approximations (\StoSubVR). The goals of this section can be summarized as follows: $(1)$ illustrate the power and robustness of the adaptive step size variant of the \SVRSQP{} method; $(2)$ show the advantages of utilizing variance reduced gradient approximations; $(3)$ demonstrate the advantage of the SQP paradigm over a simple stochastic subgradient method; and, $(4)$ show the robustness of the \SVRSQP{} method to user-defined parameters such as the inner iteration length ($S$) and the adaptive step size parameter ($\beta$).

\subsection{Problem Specification, Implementation Details and Evaluation Metrics}\label{sec.problem_evaluation}

Throughout this section we consider the following two constrained binary classification problems:
\begin{align}
  &\min_{x\in\mathbb{R}^n}\ f(x) = \frac{1}{N} \sum_{i=1}^N \log \left( 1 + e^{-y_i(X_i^Tx)}\right)\ \text{ s.t. }\ Ax=a_1 \label{eq.log_lin}\\
  &\min_{x\in\mathbb{R}^n}\ f(x) = \frac{1}{N} \sum_{i=1}^N \log \left( 1 + e^{-y_i(X_i^Tx)}\right)\ \text{ s.t. }\ \|x\|_2^2=a_2 \label{eq.log_el2} 
\end{align}
where $X \in \mathbb{R}^{N \times n}$ is the data matrix (containing feature data for N data points; $X_i$ representing the $i$th column of X) and $y \in \{-1,1\}^N$ are the labels (for each data point), and $A \in \mathbb{R}^{m \times n}$, $a_1 \in \mathbb{R}^m$ and $a_2 = 1 \in \mathbb{R}_{>0} 
$ define the constraints. We consider 10 datasets, listed in Table~\ref{tab:data}, from the LIBSVM collection \cite{chang2011libsvm}. For the linear constraints in \eqref{eq.log_lin}, we generated normal random 
$A$ and $a_1$ for each problem with $m=10$. 

\begin{table}[ht]
\caption{\label{tab:data}Binary classification data set details. For more information see \cite{chang2011libsvm}. }
  \centering
  {\footnotesize
\begin{tabular}{lcc}\toprule
\textbf{dataset}    & \textbf{dimension ($\pmb{n}$)} & \textbf{ datapoints ($\pmb{N}$)} \\ \midrule
\texttt{a9a}             & $123$                          & $32,561$                                                                             \\ \hdashline
\texttt{australian}      & $14$                           & $621$                                                                                \\ \hdashline
\texttt{heart}           & $13$                           & $270$                                                                                \\ \hdashline
\texttt{ijcnn1}          & $22$                           & $35,000$                                                                             \\ \hdashline
\texttt{ionosphere}      & $34$                           & $351$                                                                                \\ \hdashline
\texttt{mushroom}       & $112$                          & $5,500$                                                                              \\ \hdashline
\texttt{phising}         & $68$                           & $11,055$                                                                             \\ \hdashline
\texttt{sonar}           & $60$                           & $208$                                                                                \\ \hdashline
\texttt{splice}          & $60$                           & $3,175$                                                                              \\ \hdashline
\texttt{w8a}             & $300$                          & $49,749$       \\
\bottomrule                                                                     
\end{tabular}}
\end{table}

A budget of 30 epochs  (i.e., number of effective passes over the dataset; equivalent to the number of   gradient evaluations of the objective function) was used for all methods. For all problems and algorithms, the initial primal iterate ($x_0$) was set to a normal random vector scaled to have norm $0.1$, and the multipliers were initialized as $y_0 = \arg\min_{y\in\mathbb{R}^m}\ \|g_0 + J_0^Ty\|^2$. For each method, we considered two batch sizes $b=16$ (small batch) and $b=128$ (large batch). For each problem, dataset, algorithm and batch size, we ran 10 instances with different random seeds. With regards to the constraint Lipschitz constant estimates, we used the true constants $\Gamma_{k,s} = 0$ (for \eqref{eq.log_lin}) and $\Gamma_{k,s} = 2$ (for \eqref{eq.log_el2}) for all $k\in\mathbb{N}$ and $s\in\left[\bar{S}\right]$ for all algorithms. We set $L_{k,s} = L$ for all $k\in\mathbb{N}$ and $s\in\left[\bar{S}\right]$ for all algorithms, where $L$ was estimated by differences of gradients around the initial iterate. The details of all parameter settings are given below.
\begin{itemize}
	\item \SVRSQPCONST{} and \SVRSQPADAPT{}: $\sigma = 0.5$, $\theta = 10^4$, $\bar{\tau}_{-1, 0} = 0.1$, and $\epsilon_\tau = 10^{-6}$ 
	\begin{itemize}
		\item[--] \SVRSQPCONST{} [Algorithm~\ref{alg.sqp_svrg}; constant step sizes]: \\ $\bar\alpha_{k,s} = \alpha\in\{10^{-3},10^{-2},10^{-1},10^{0},10^{1}\}$ for all $(k,s) \in \mathbb{N} \times \left[\bar{S}\right]$;
		\item[--] \SVRSQPADAPT{} [Algorithm~\ref{alg.sqp_svrg}; adaptive step sizes]:  $\bar\alpha_{k,s}$ computed via \eqref{eq:step size1}--\eqref{eq.step size} for all $(k,s) \in \mathbb{N} \times \left[\bar{S}\right]$, $\alpha_u = 10^{6}$, $\beta = 1$.
	\end{itemize}
	\item \StoSQP{} \cite[Algorithm 3.1]{berahas2021sequential}: $\theta = 10^4$, $\bar\tau_{-1} = 0.1$, $\epsilon_{\tau} = 10^{-6}$, $\bar\xi_{-1} = 0.1$, $\epsilon_{\xi} = 10^{-2}$, $\sigma = 0.5$, and $\beta_k = \beta \in \{ 10^{-3},10^{-2},10^{-1},10^{0},10^{1}\}$ for all $k\in\mathbb{N}$.
	\item \StoSubVR{}: $\bar\alpha_{k,s} = \tfrac{\alpha}{\tau L + \Gamma}$  with  
	$\alpha\in\{10^{-3},10^{-2},10^{-1},10^{0},10^{1}\}$ and $\tau_{k,s} = \tau \in \{10^{-10},10^{-9},\dots,10^{0}\}$ for all $(k,s) \in \mathbb{N} \times \left[\bar{S}\right]$. We should note that \StoSubVR{} applies the SVRG algorithm \cite{johnson2013accelerating} to directly minimize the nonsmooth merit function \eqref{def.merit}; SVRG directly applied to the smooth part of merit function with the subgradient of the nonsmooth part added.
\end{itemize}
For all algorithms with inner outer iterations, the inner itereation length was set as $S= \left\lfloor \tfrac{N}{2b} \right\rfloor$, unless otherwise specified.

In all of our experiments, results are given in terms of feasibility and stationarity errors discussed below. We present the evolution of these measures with respect to epochs in Figures~\ref{fig.svrsqra_svrsqpa}, \ref{fig.sensitivity}, \ref{fig.sensitivity2} and \ref{fig.best}. Moreover, in Figure~\ref{fig.svrsqra_svrsqpa_summary} and Tables~\ref{tab.best_linear} and \ref{tab.best_el2}, we report the error metrics at the best iterate found within the budget defined as follows. Given a fixed epoch budget, assume we have $k\in\{0,\ldots,K\}$ for some $K\in\mathbb{N}$. If 
\begin{equation*}
    \min \{\|c_k\|_{\infty}:k = 0,\ldots,K\} > 10^{-6},\ \ \text{then}\ \ \xbest \leftarrow \argmin_{x_k\in\{x_0,\ldots,x_K\}}\ \|c_k\|_{\infty}.
\end{equation*}
Otherwise, if $\min \{\|c_k\|_{\infty}:k = 0,\ldots,K\} \leq 10^{-6}$, then we set
\begin{equation*}
    \xbest \leftarrow \argmin_{x_k\in\{x_0,\ldots,x_K\}}\ \|\nabla f_k + J_k^Ty_{k,\texttt{ls}}\|_{\infty} \quad\text{s.t.}\ \ \|c_k\|_{\infty} \leq 10^{-6},
\end{equation*}
where $y_{k,\texttt{ls}}$ is the least-squares multiplier at $x_k$. Given $\xbest$ and the corresponding dual variables $y_{\texttt{best,ls}}$, we report 
feasibility error ($\|c(\xbest)\|_{\infty}$) and stationarity error ($\|\nabla f(\xbest) + \nabla c(\xbest)y_{\texttt{best,ls}}\|_{\infty}$).

\subsection{Comparison: \SVRSQPCONST{} and \SVRSQPADAPT{}}\label{sec.numerical_adaptiveconstant}

In this section, we compare the performance of \SVRSQPCONST{} and \SVRSQPADAPT{} on \eqref{eq.log_lin} and \eqref{eq.log_el2}. We ran all methods for 30 epochs with two different batch sizes. For   \SVRSQPCONST{} we tuned the step size $\bar\alpha_{k,s} = \alpha\in\{10^{-3},10^{-2},10^{-1},10^{0},10^{1}\}$ for all $(k,s) \in \mathbb{N} \times \left[\bar{S}\right]$. For the \SVRSQPADAPT{} we set $\beta=1$. For both methods we used $S= \left\lfloor \tfrac{N}{2b} \right\rfloor$. 

\begin{figure}[ht]
   \centering
     \begin{subfigure}[b]{1\textwidth}

  \includegraphics[width=0.24\textwidth,clip=true,trim=30 180 50 200]{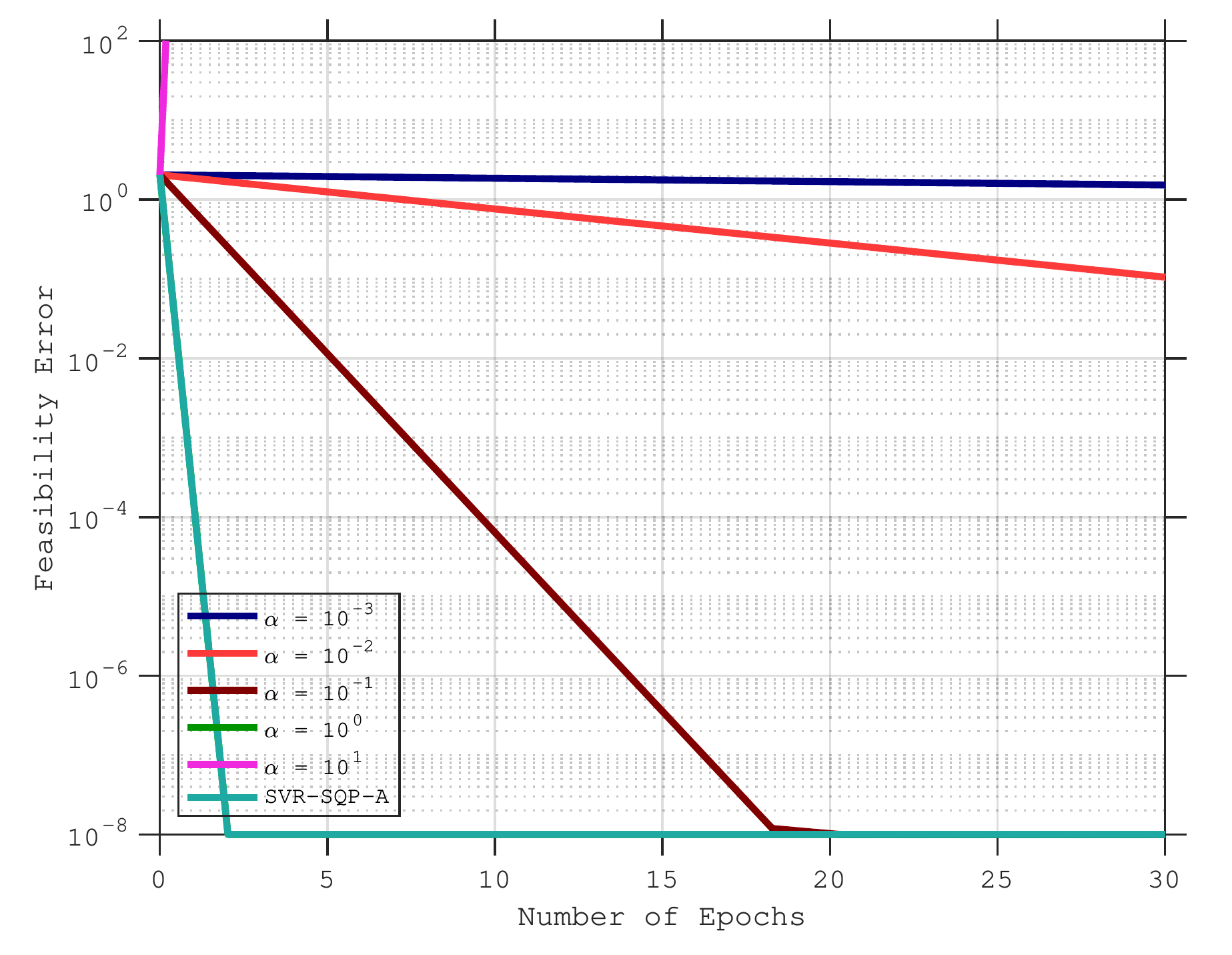}
  \includegraphics[width=0.24\textwidth,clip=true,trim=30 180 50 200]{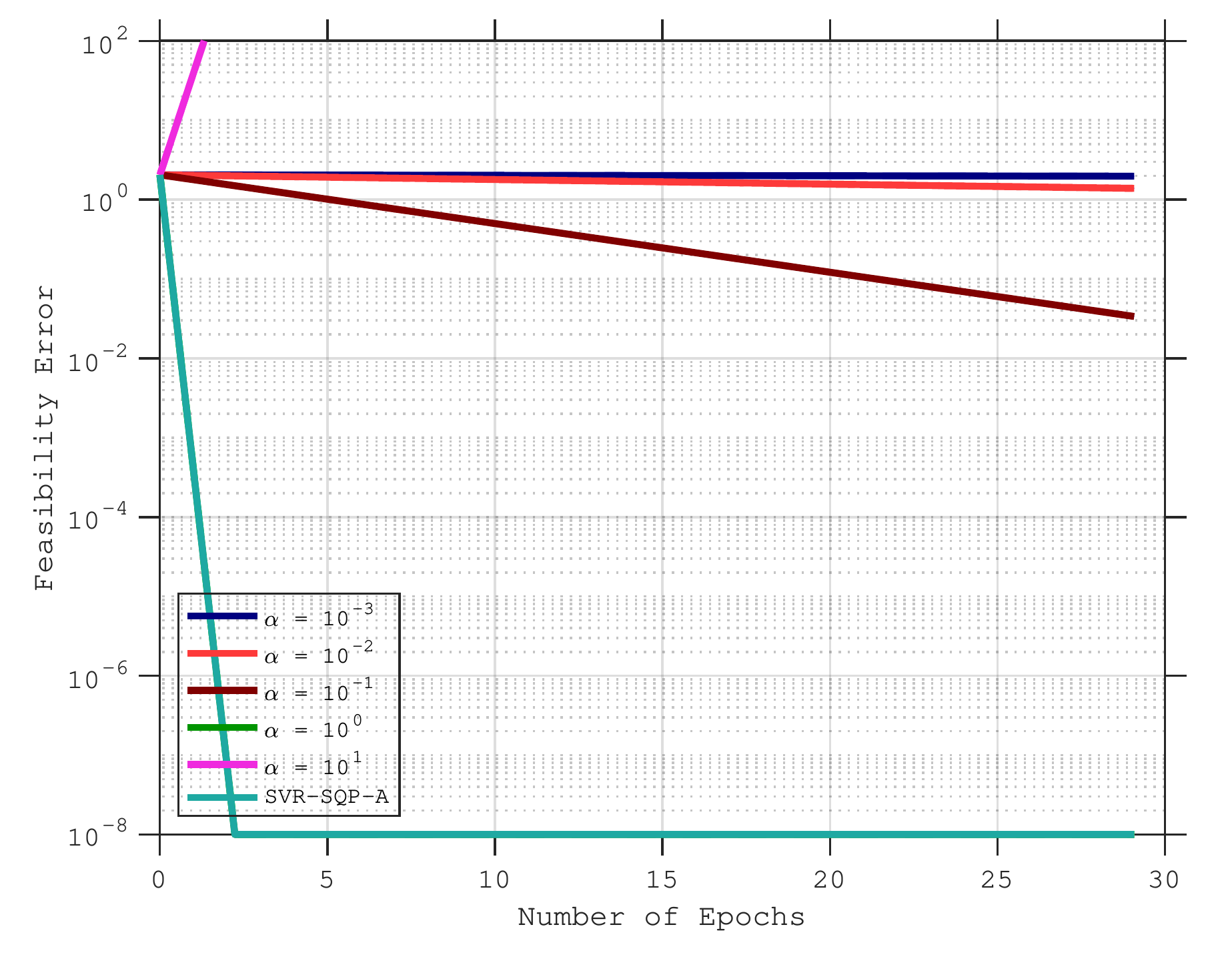}
  \includegraphics[width=0.24\textwidth,clip=true,trim=30 180 50 200]{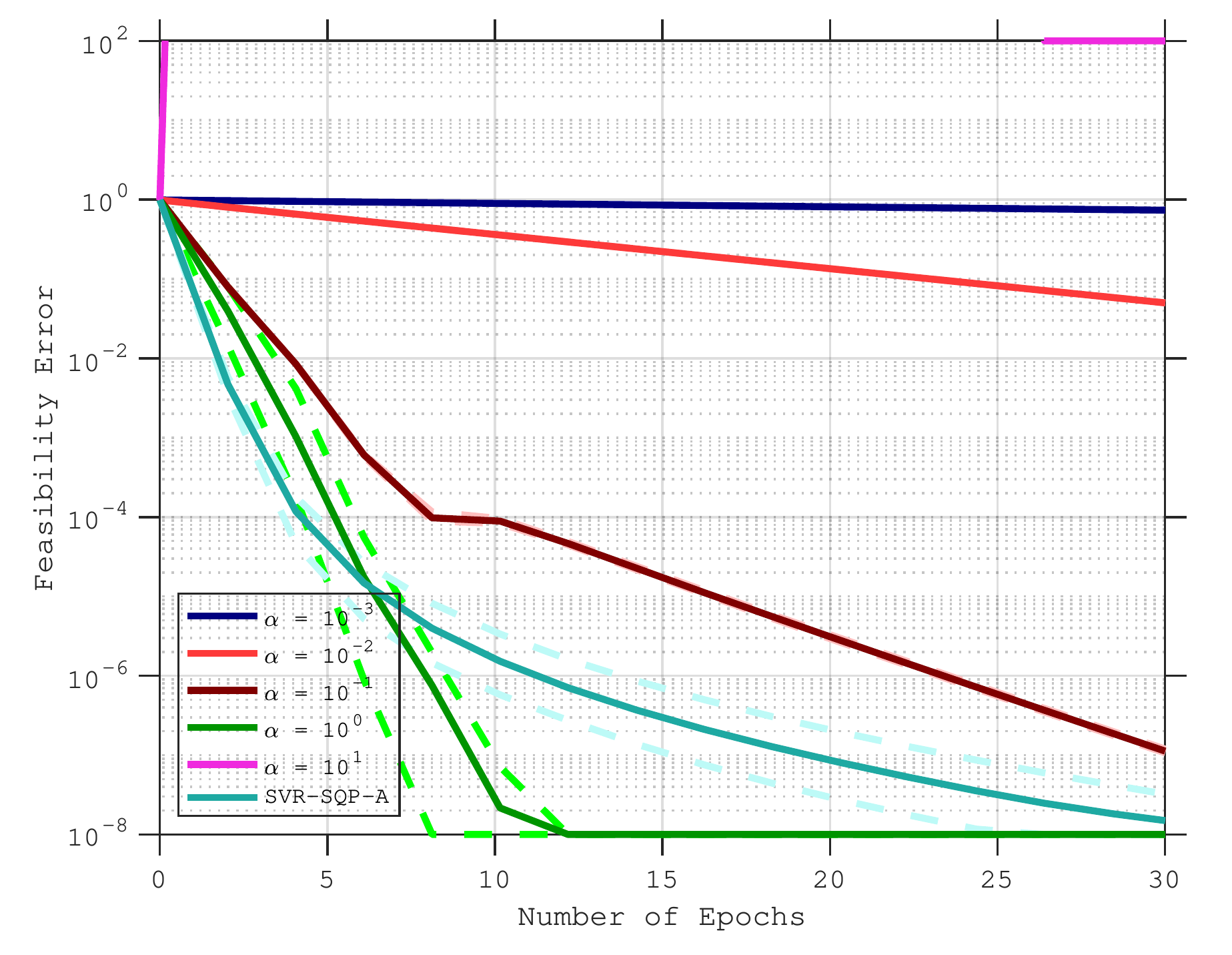}
  \includegraphics[width=0.24\textwidth,clip=true,trim=30 180 50 200]{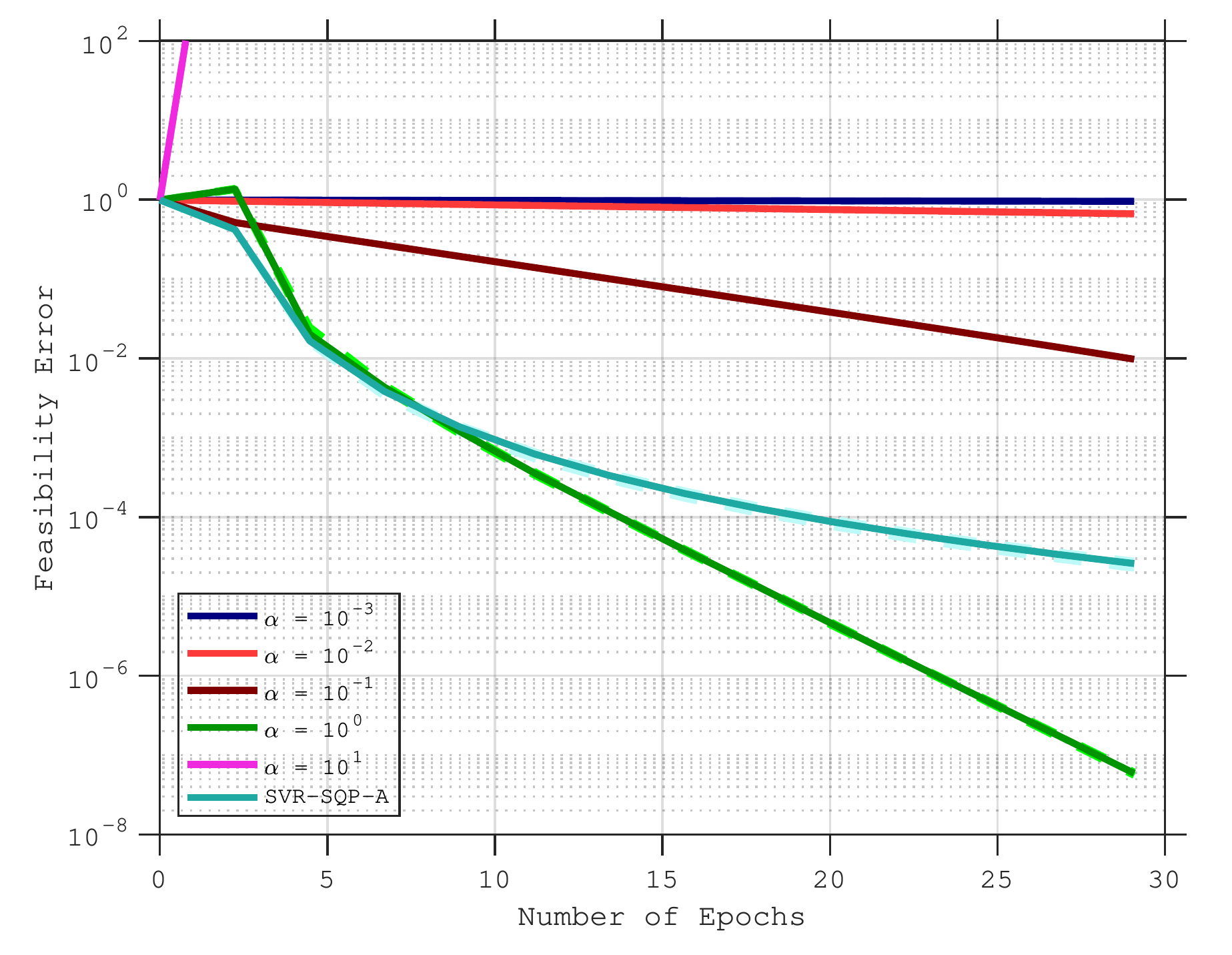}
  
  \includegraphics[width=0.24\textwidth,clip=true,trim=30 180 50 200]{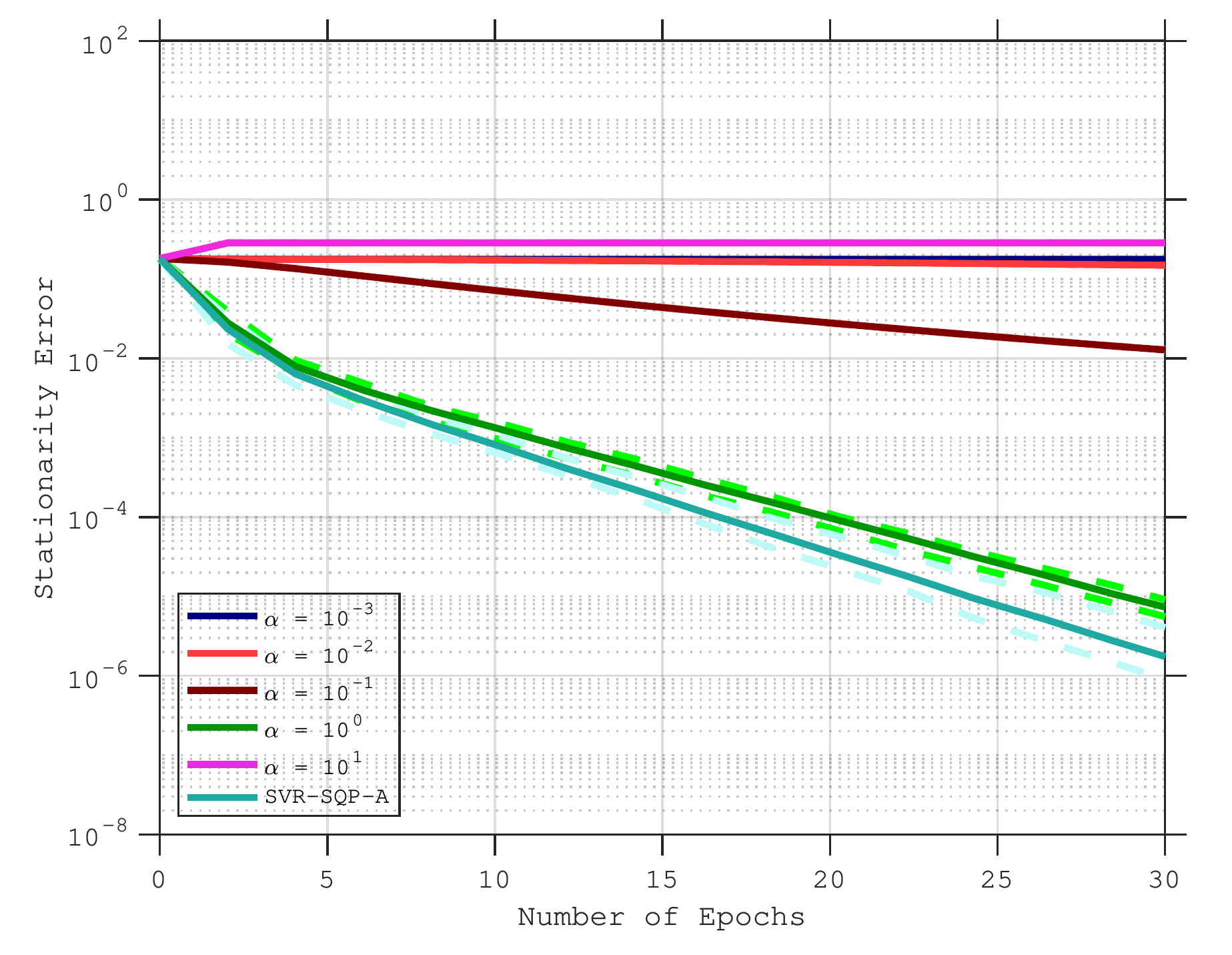}
  \includegraphics[width=0.24\textwidth,clip=true,trim=30 180 50 200]{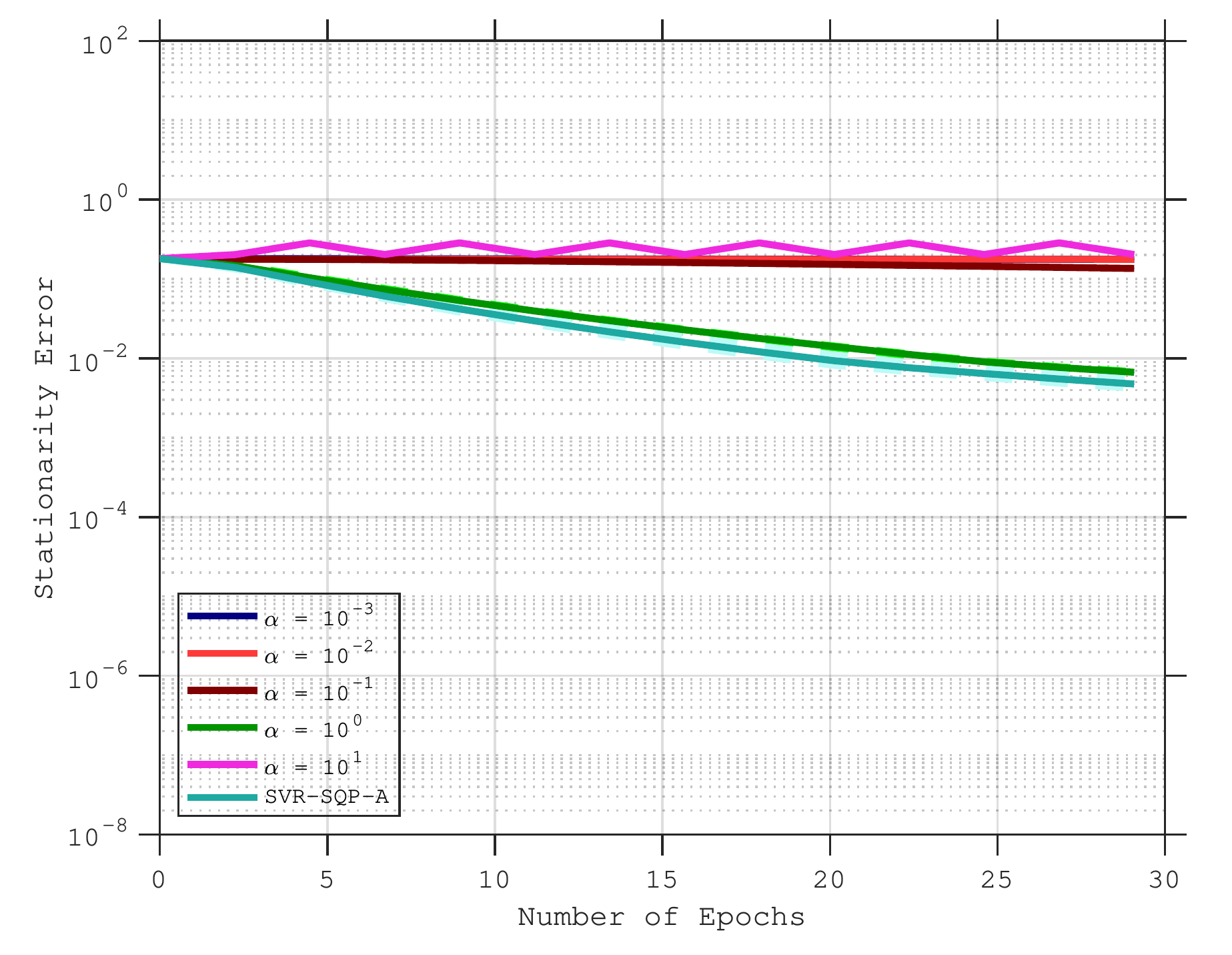}
  \includegraphics[width=0.24\textwidth,clip=true,trim=30 180 50 200]{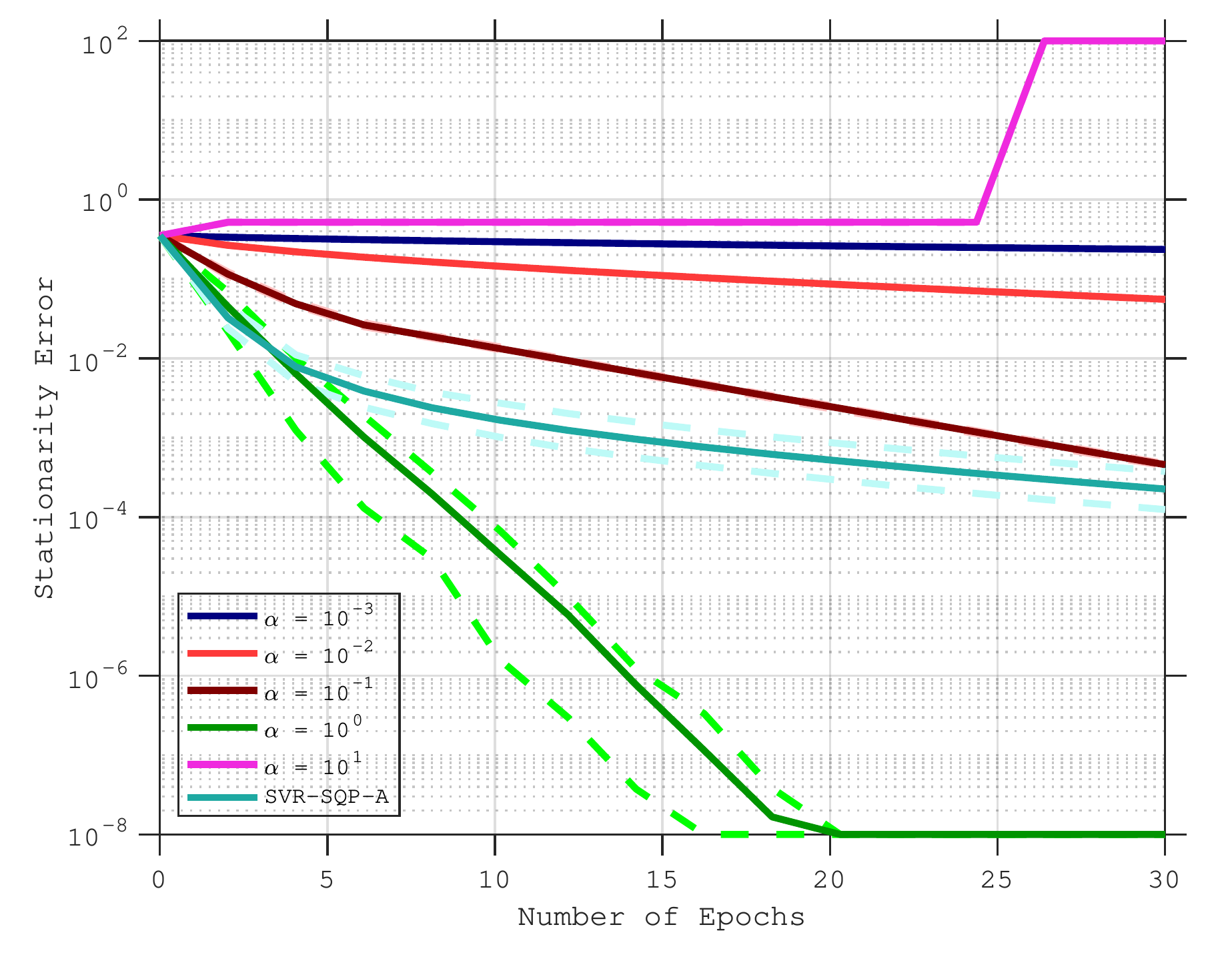}
  \includegraphics[width=0.24\textwidth,clip=true,trim=30 180 50 200]{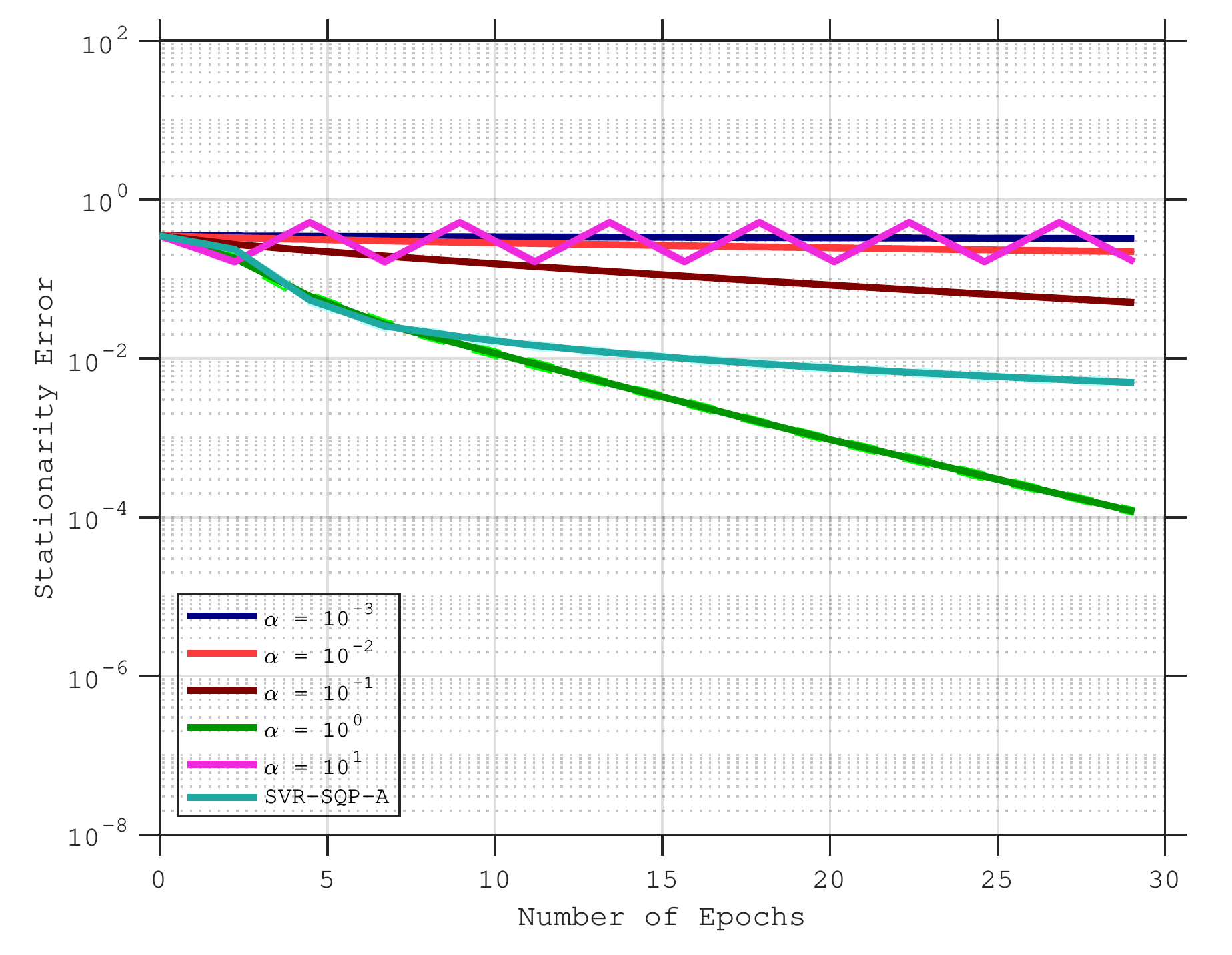}
  \caption{\texttt{australian} dataset. Top row: feasibility error; Bottom row: stationarity
  error.}
  \label{fig.svrsqra_svrsqpa_1}
    \end{subfigure}

  \begin{subfigure}[b]{\textwidth}
  \includegraphics[width=0.24\textwidth,clip=true,trim=30 180 50 200]{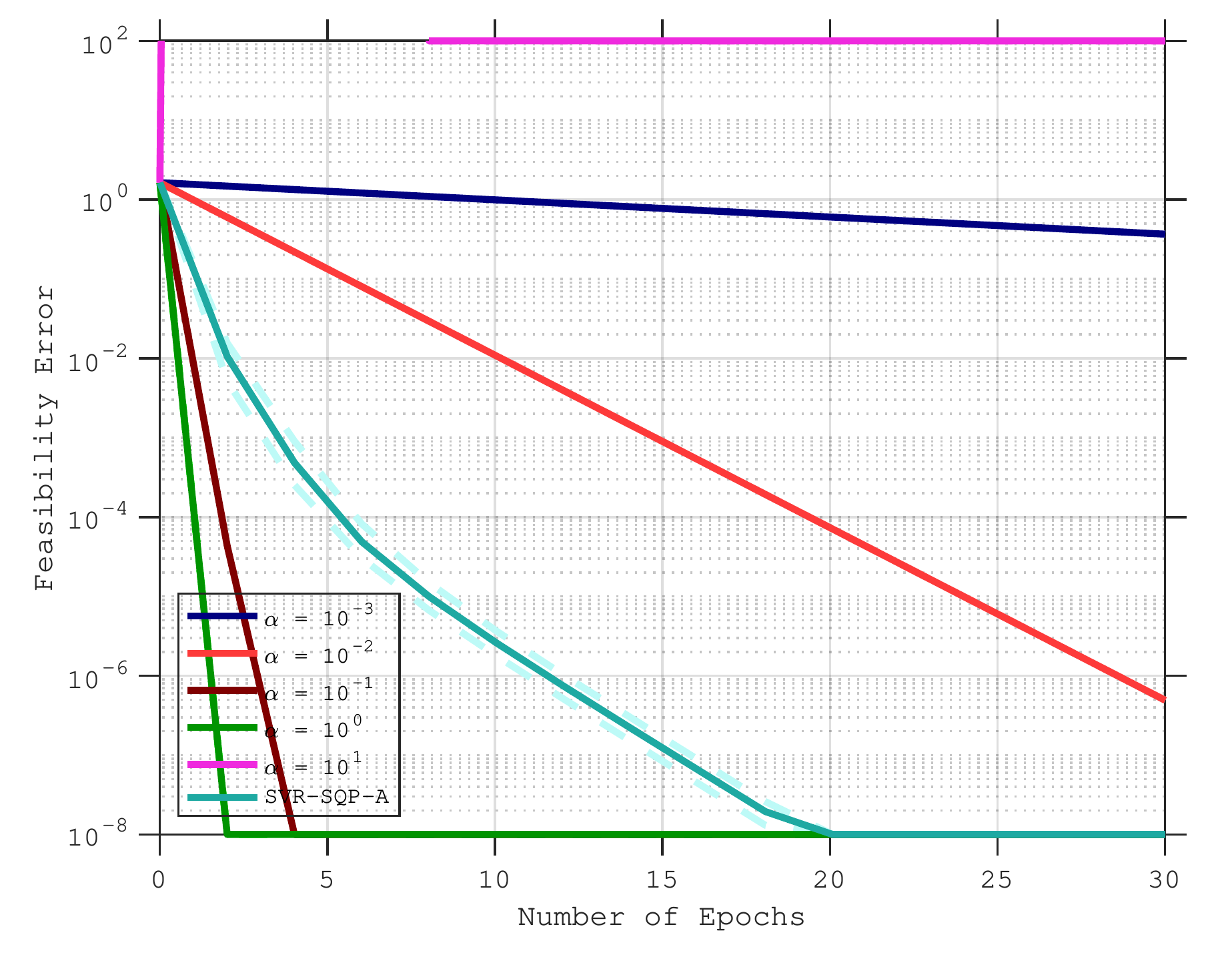}
  \includegraphics[width=0.24\textwidth,clip=true,trim=30 180 50 200]{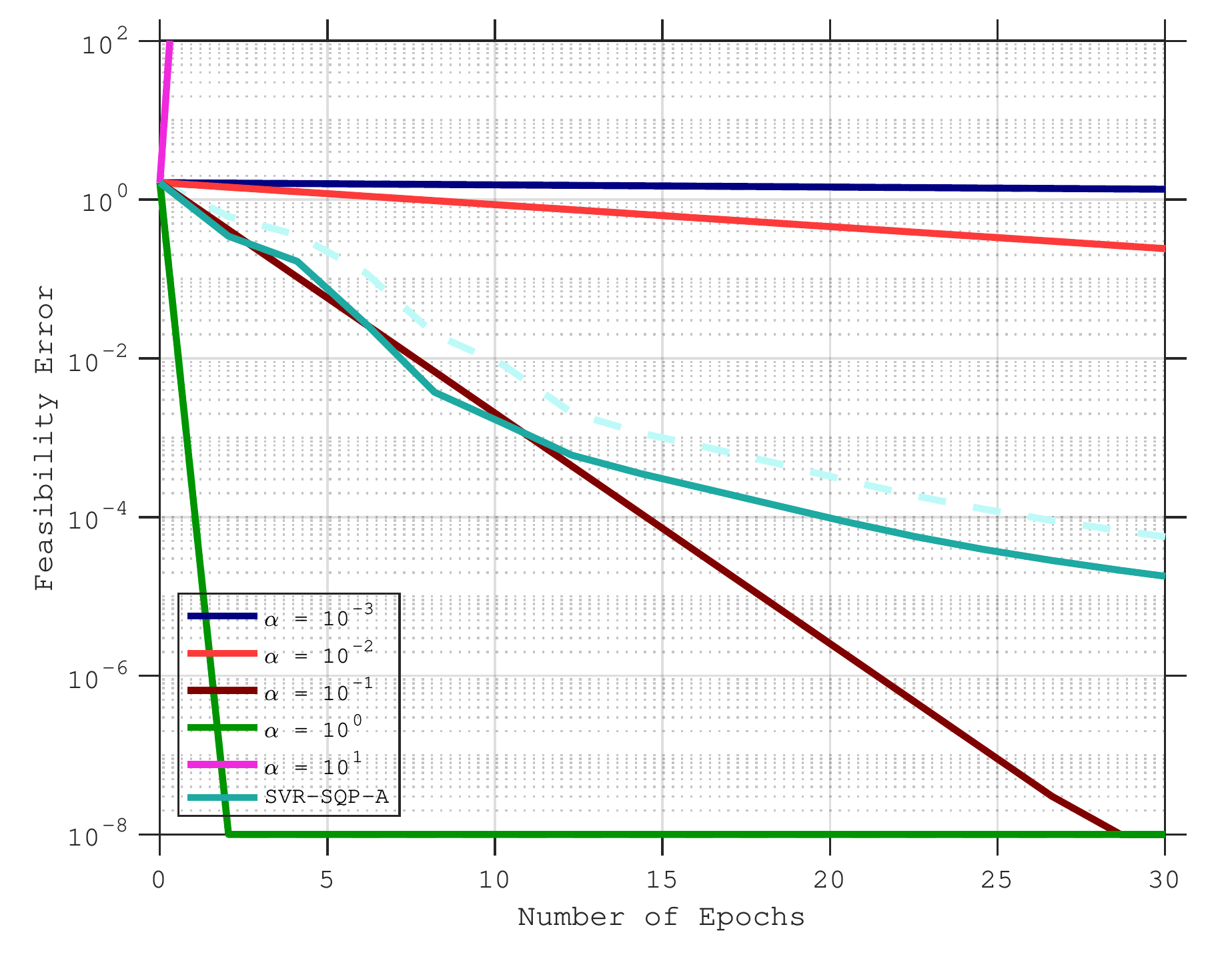}
  \includegraphics[width=0.24\textwidth,clip=true,trim=30 180 50 200]{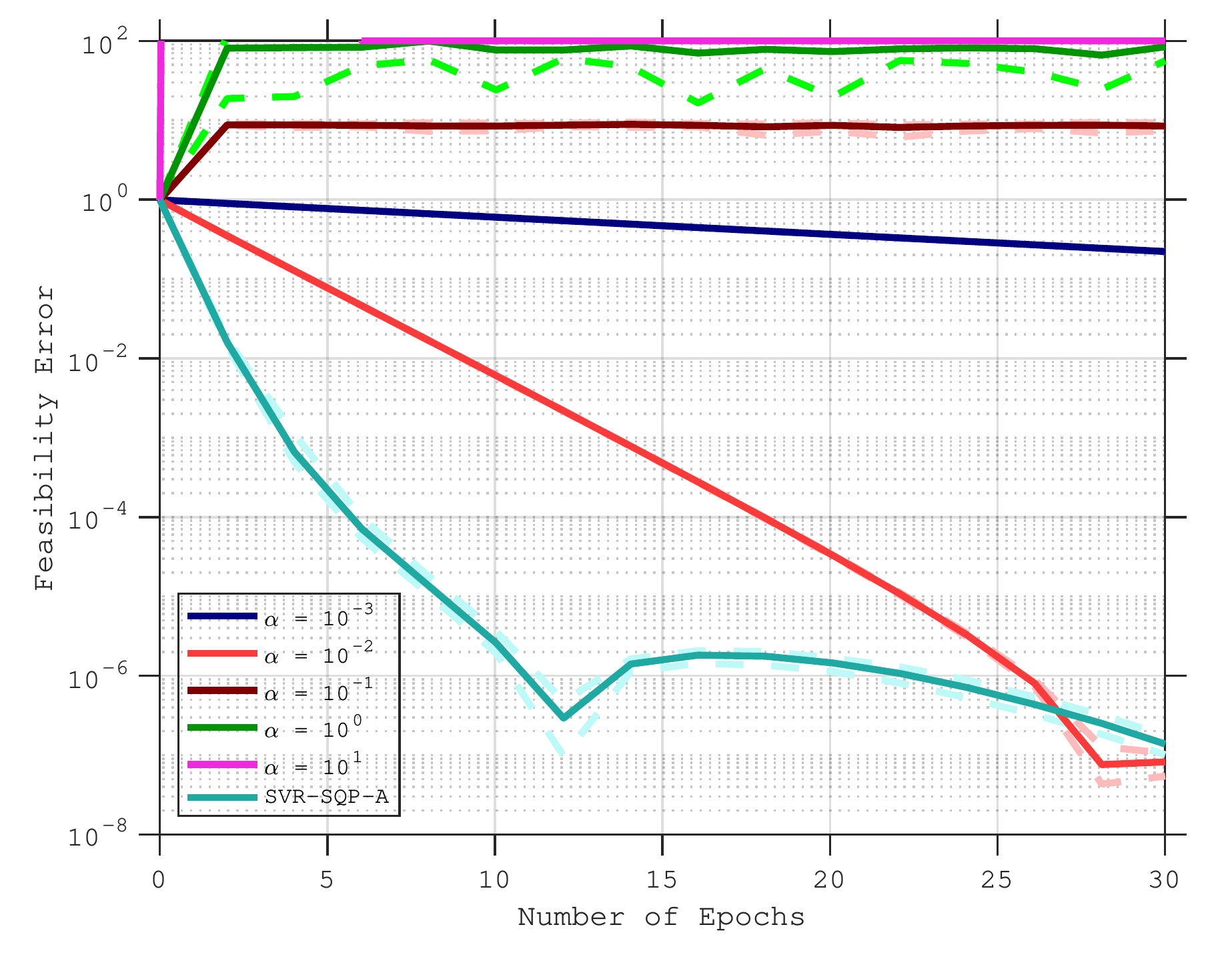}
  \includegraphics[width=0.24\textwidth,clip=true,trim=30 180 50 200]{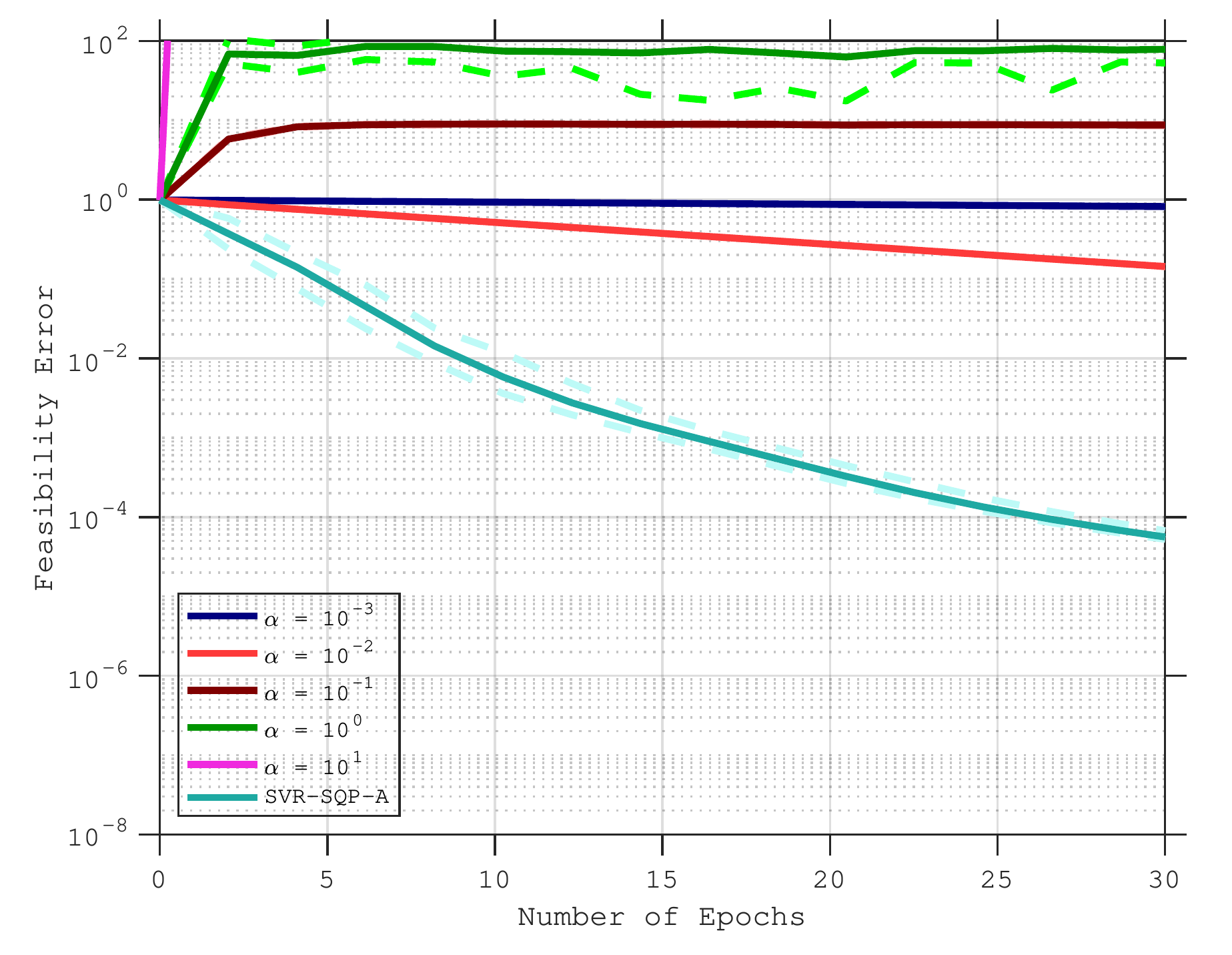}
  
  \includegraphics[width=0.24\textwidth,clip=true,trim=30 180 50 200]{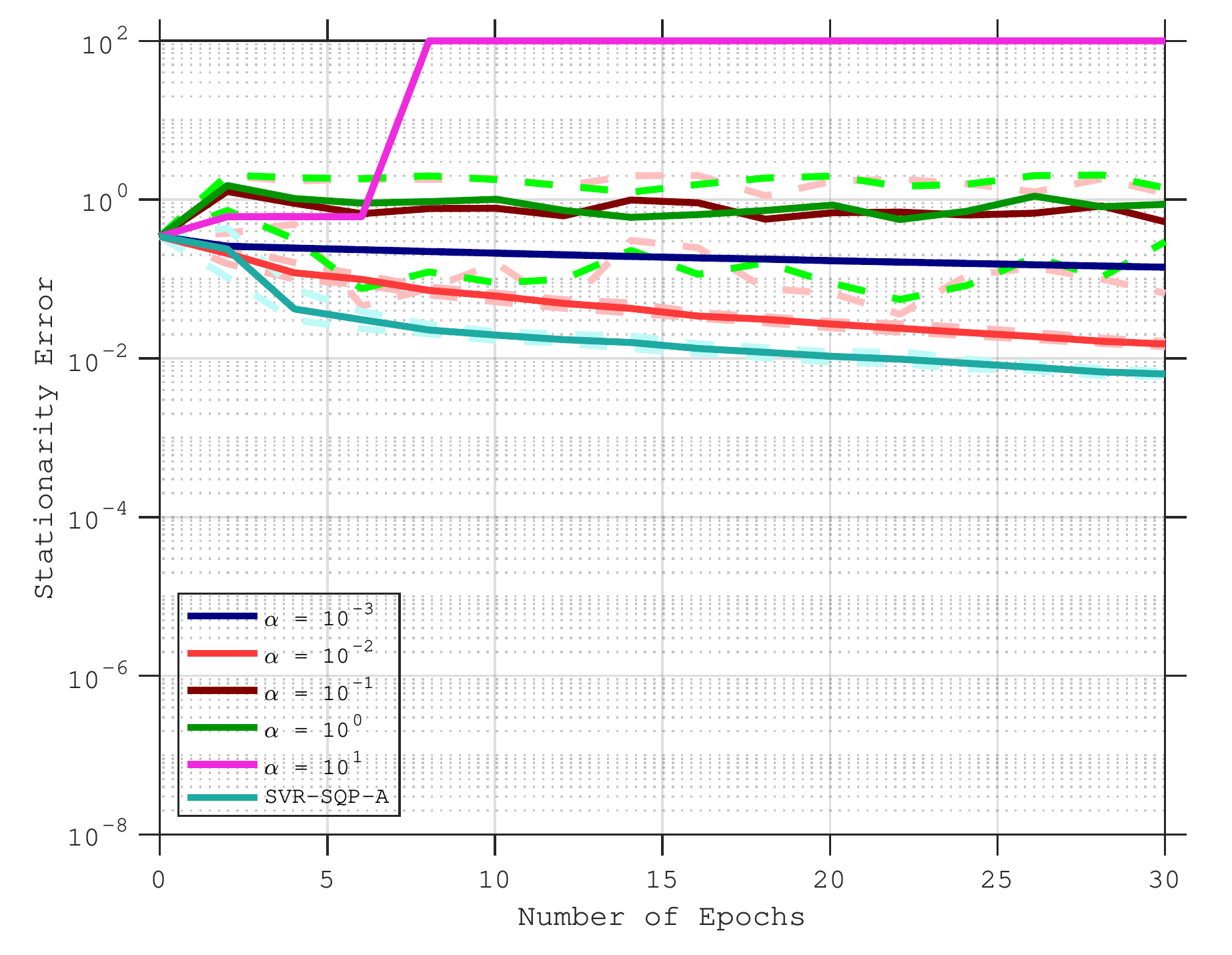}
  \includegraphics[width=0.24\textwidth,clip=true,trim=30 180 50 200]{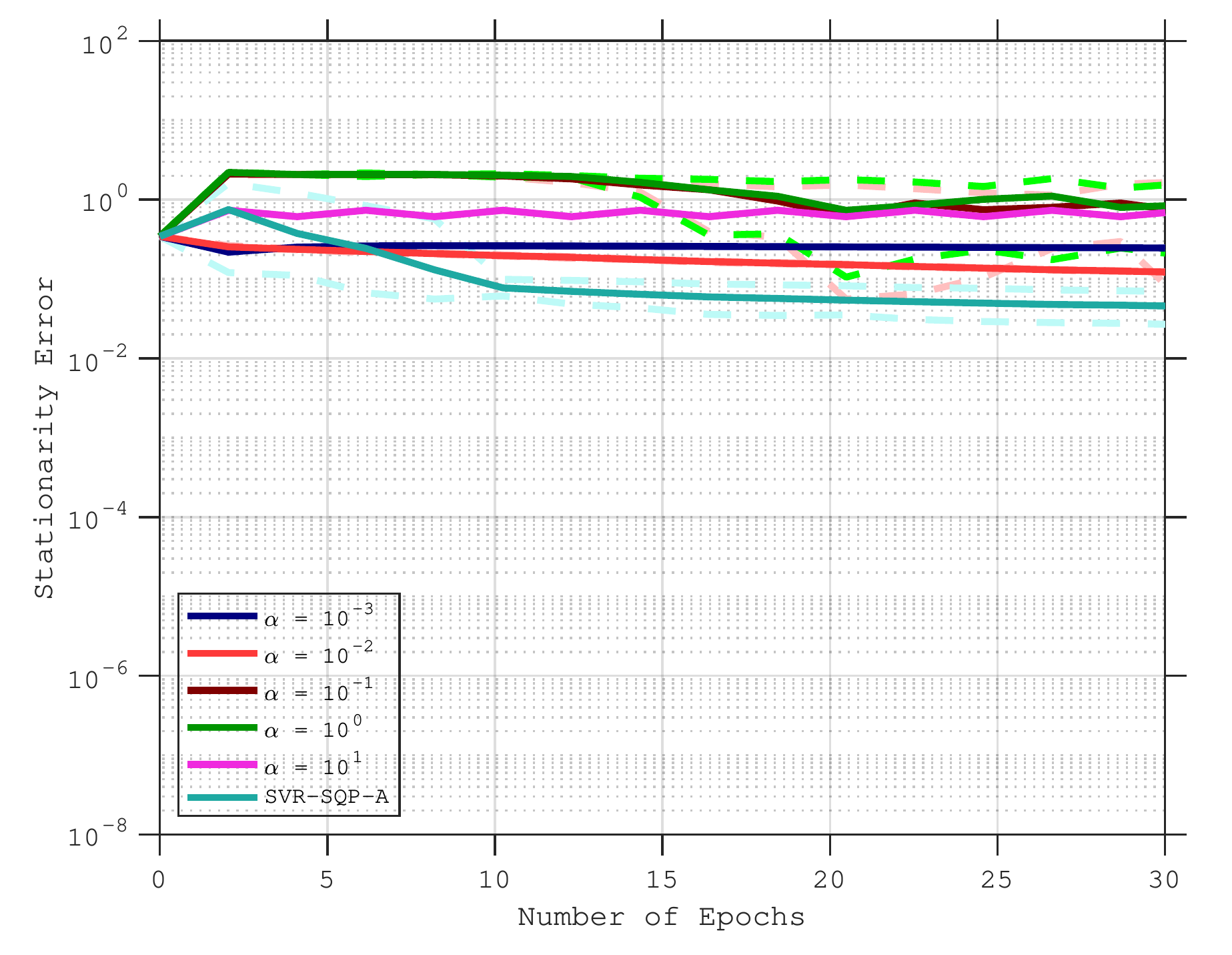}
  \includegraphics[width=0.24\textwidth,clip=true,trim=30 180 50 200]{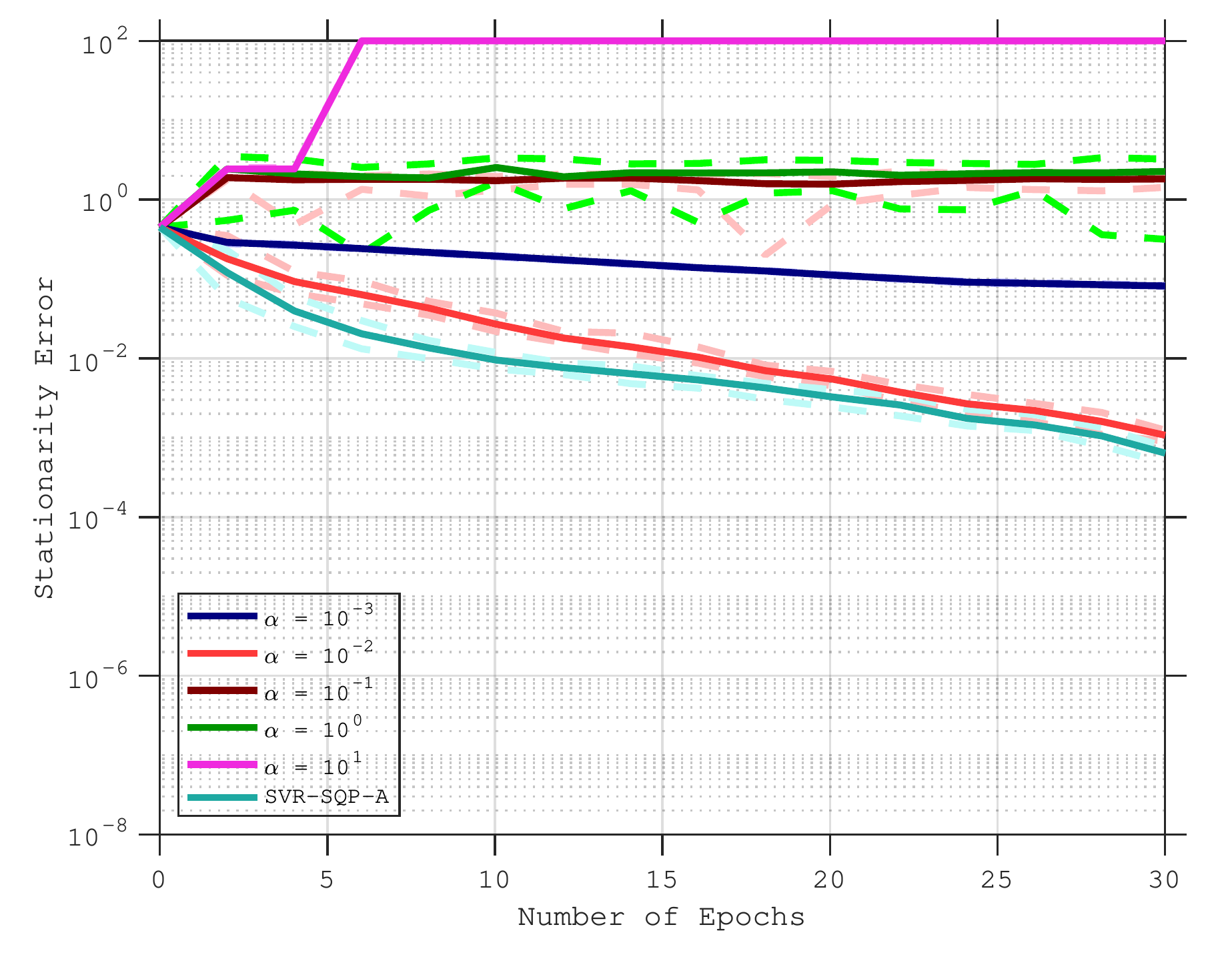}
  \includegraphics[width=0.24\textwidth,clip=true,trim=30 180 50 200]{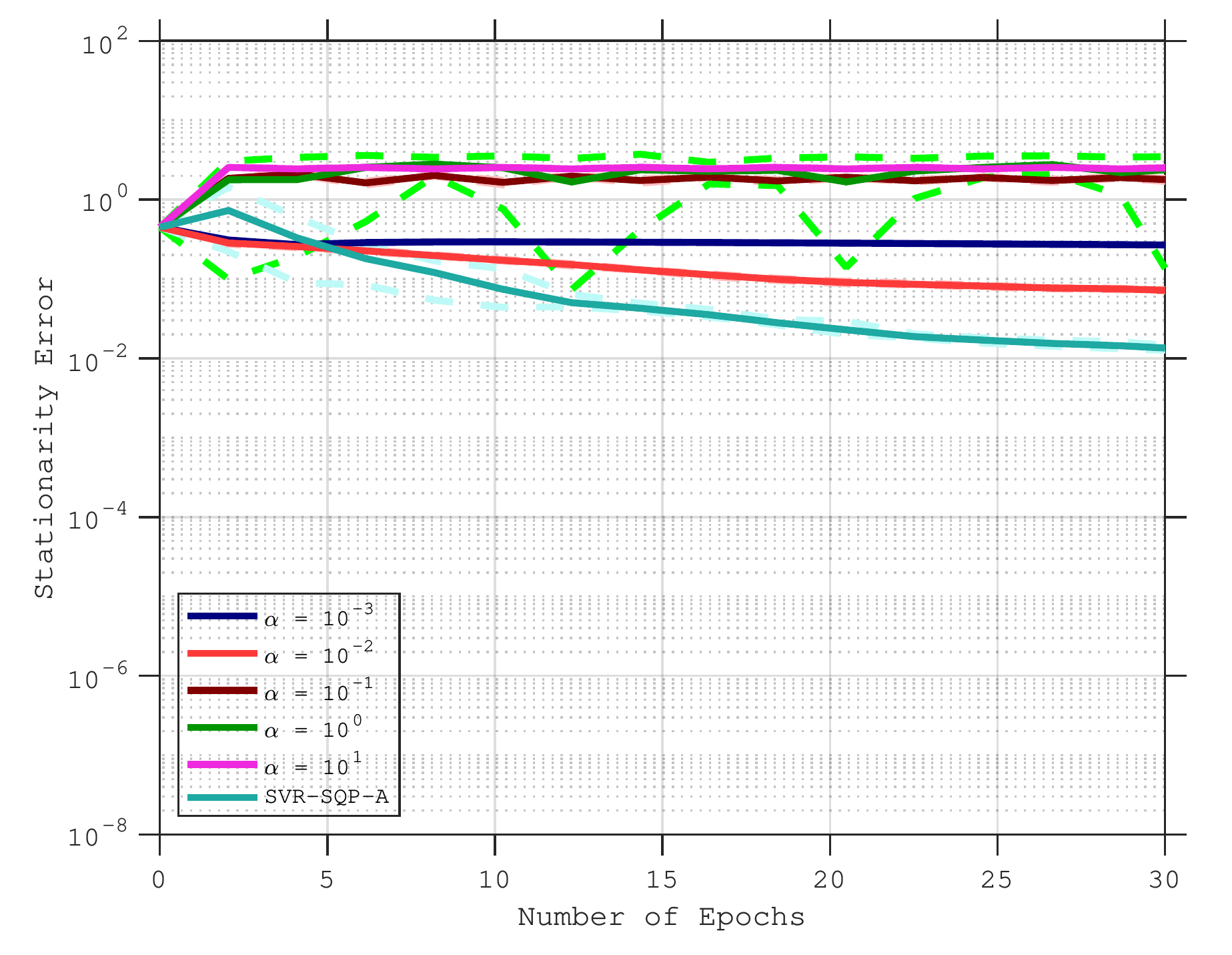}
  \caption{ \texttt{splice} dataset. Top row: feasibility error; Bottom row: stationarity
  error.} 
  \label{fig.svrsqra_svrsqpa_2}
  \end{subfigure}
  \caption{Performance of \SVRSQPCONST{} with different step sizes $\alpha \in \{ 10^{-3},10^{-2},10^{-1},10^{0},10^{1}\}$ and \SVRSQPCONST{} on logistic regression problems with linear (columns 1 and 2) and $\ell_2$ norm (columns 3 and 4) constraints. First and third columns: batch size 16; Second and fourth columns: batch size 128. \label{fig.svrsqra_svrsqpa}}
\end{figure}

Figs.~\ref{fig.svrsqra_svrsqpa_1} and \ref{fig.svrsqra_svrsqpa_2} show the stationarity and feasibility errors versus epochs for two datasets (\texttt{australian} and \texttt{splice}) for the \SVRSQPCONST{} (different values of $\alpha$) and \SVRSQPADAPT{} ($\beta=1$)  methods with different batch sizes. For each method, the figure shows the average trajectory (solid line) over the $10$ random seeds of the measures with respect to epochs, and the $95\%$ confidence interval (dashed lines). As is clear, \SVRSQPADAPT{} appears to be competitive with the best tuned version of the \SVRSQPCONST{} method. 

Similar behavior was observed on other datasets. Fig.~\ref{fig.svrsqra_svrsqpa_summary} presents feasibility and stationarity errors for all datasets in Table~\ref{tab:data} for the best iterates found by four variants of \SVRSQPCONST{} and \SVRSQPADAPT{}. For each problem, we report as boxplots the feasibility and stationarity errors for the best iterate found by each method for two different batch sizes (4 box plots per problem per method). From Fig. \ref{fig.svrsqra_svrsqpa_summary}, we observe that for both batch size options and for both constraints types, \SVRSQPADAPT{} performs as good as (if not better than) \SVRSQPCONST{} with the best tuned step size in terms of stationarity and feasibility.

\begin{figure}
     \centering
     \begin{subfigure}[b]{0.48\textwidth}
         \centering \includegraphics[width=0.48\textwidth,clip=true,trim=30 180 70 200]{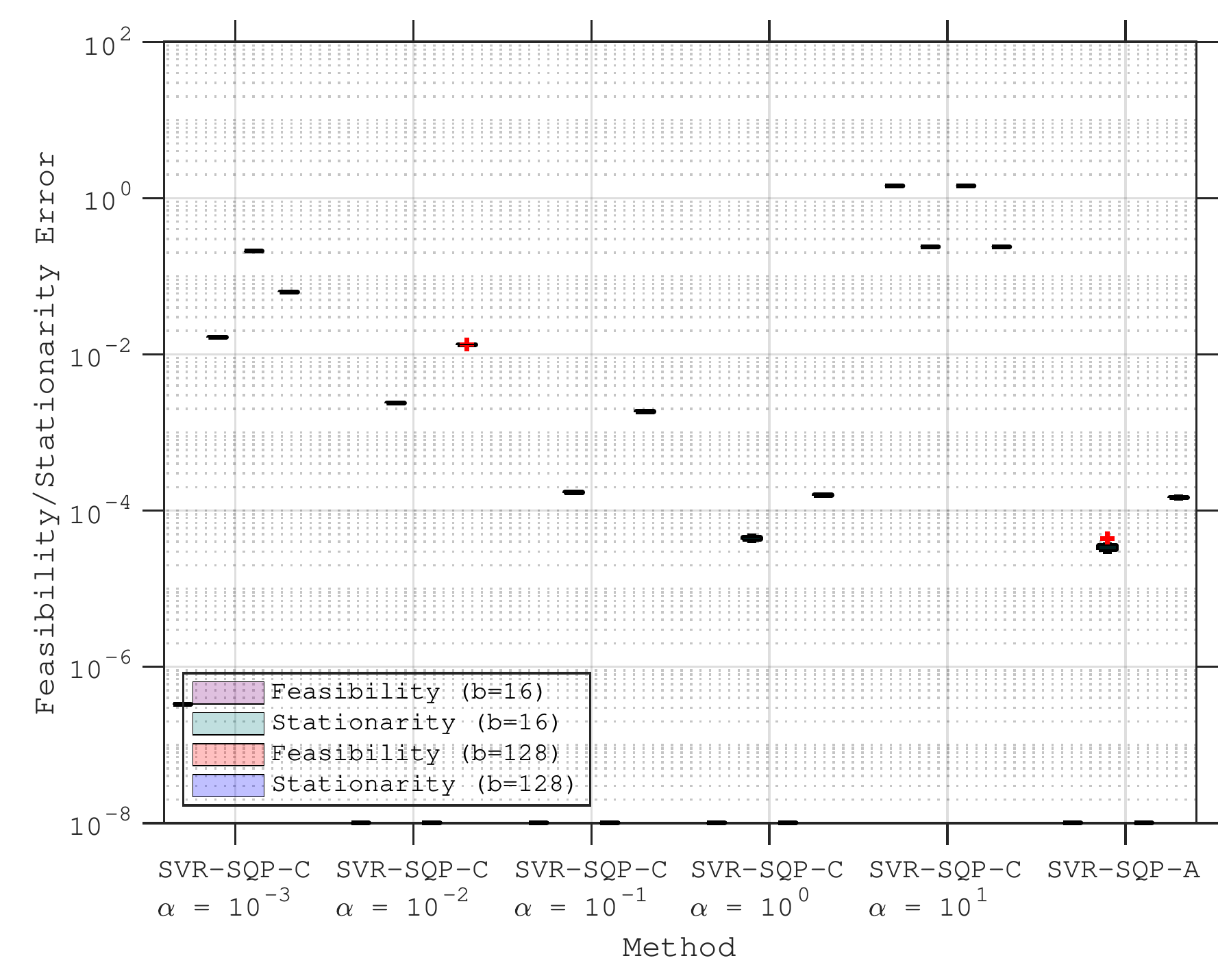}
         \includegraphics[width=0.48\textwidth,clip=true,trim=30 180 70 200]{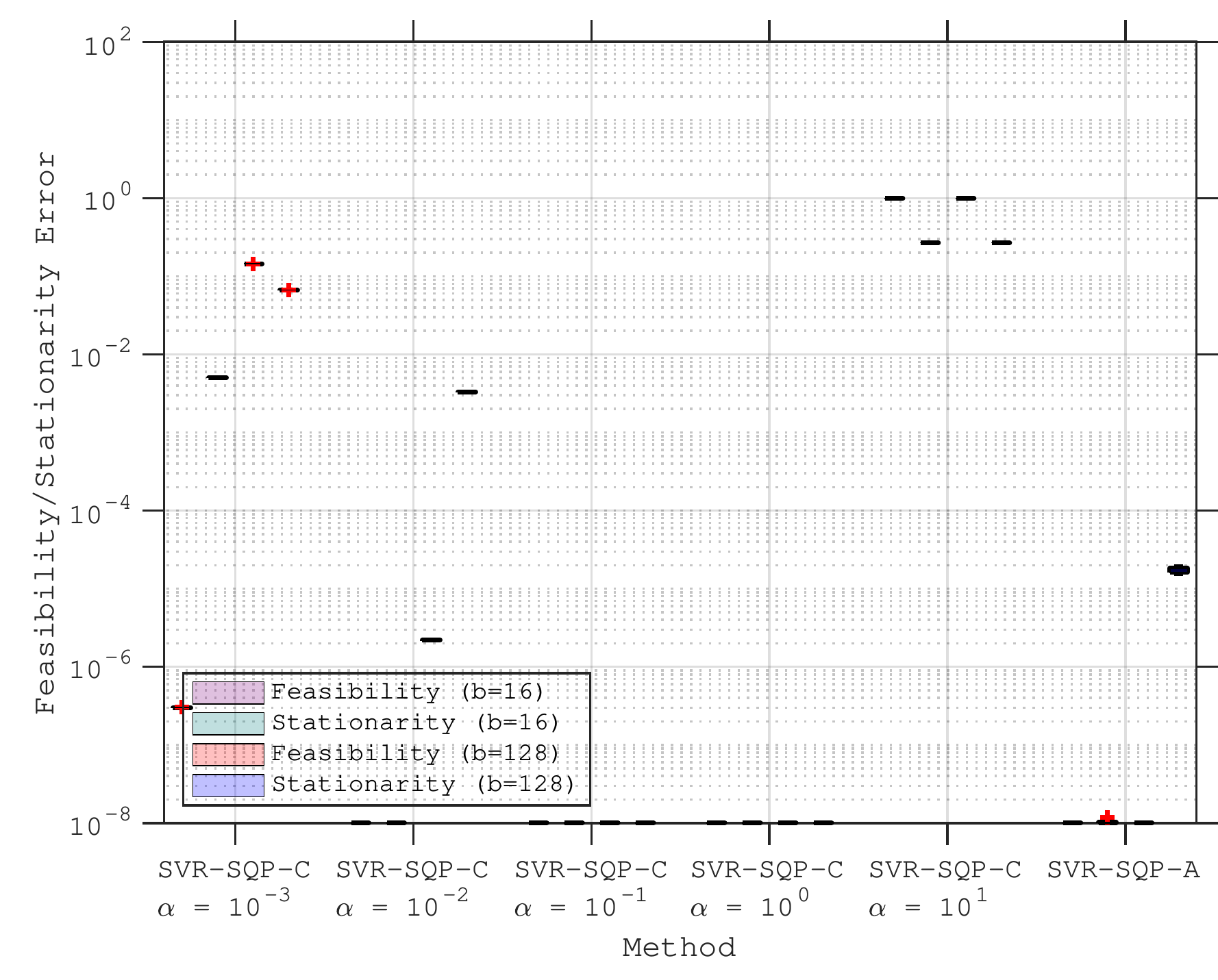}
         \caption{\texttt{a9a}; left \eqref{eq.log_lin}, right \eqref{eq.log_el2}}
         \label{fig:1}
     \end{subfigure}
     \hfill
     \begin{subfigure}[b]{0.48\textwidth}
         \centering \includegraphics[width=0.48\textwidth,clip=true,trim=30 180 70 200]{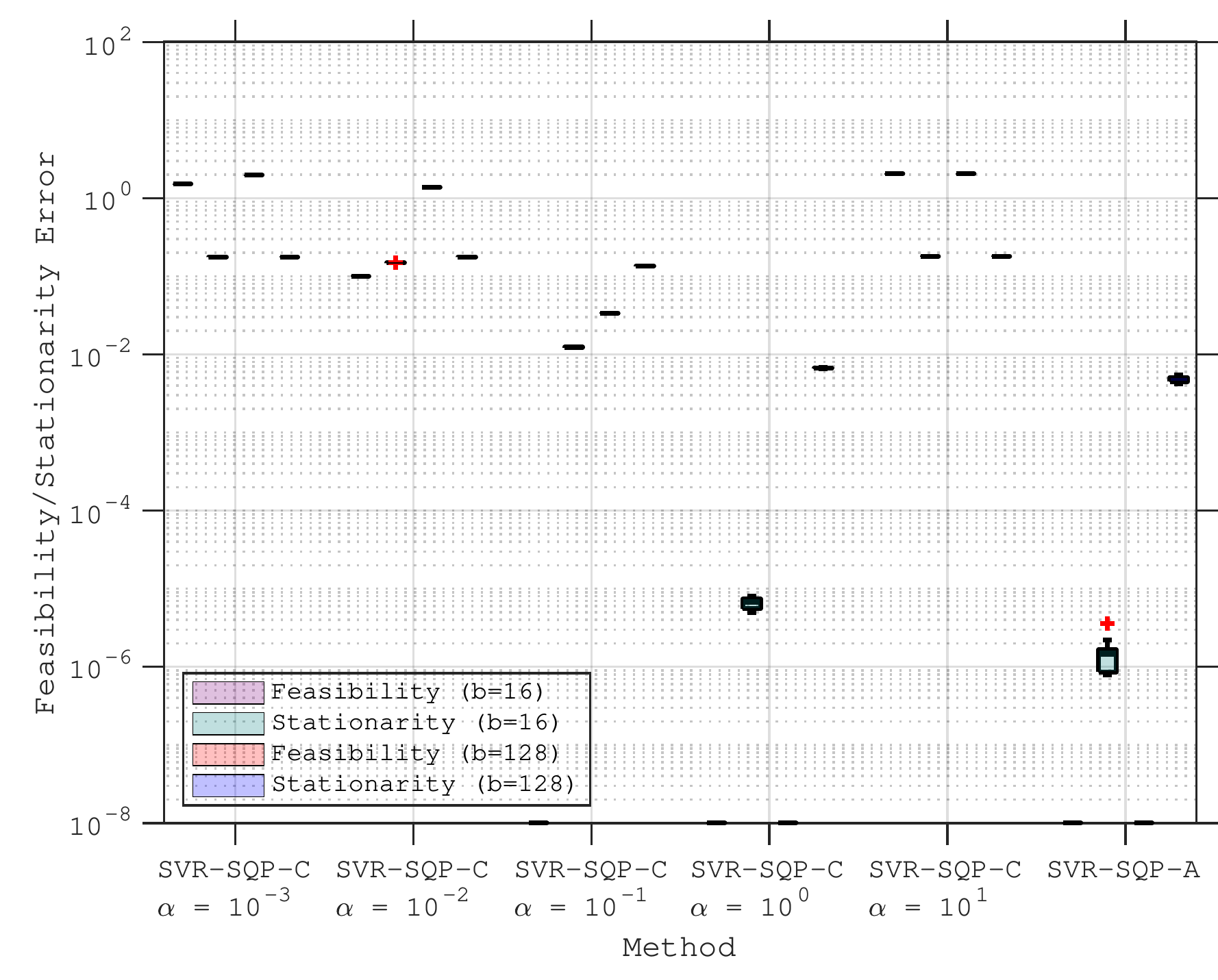}
         \includegraphics[width=0.48\textwidth,clip=true,trim=30 180 70 200]{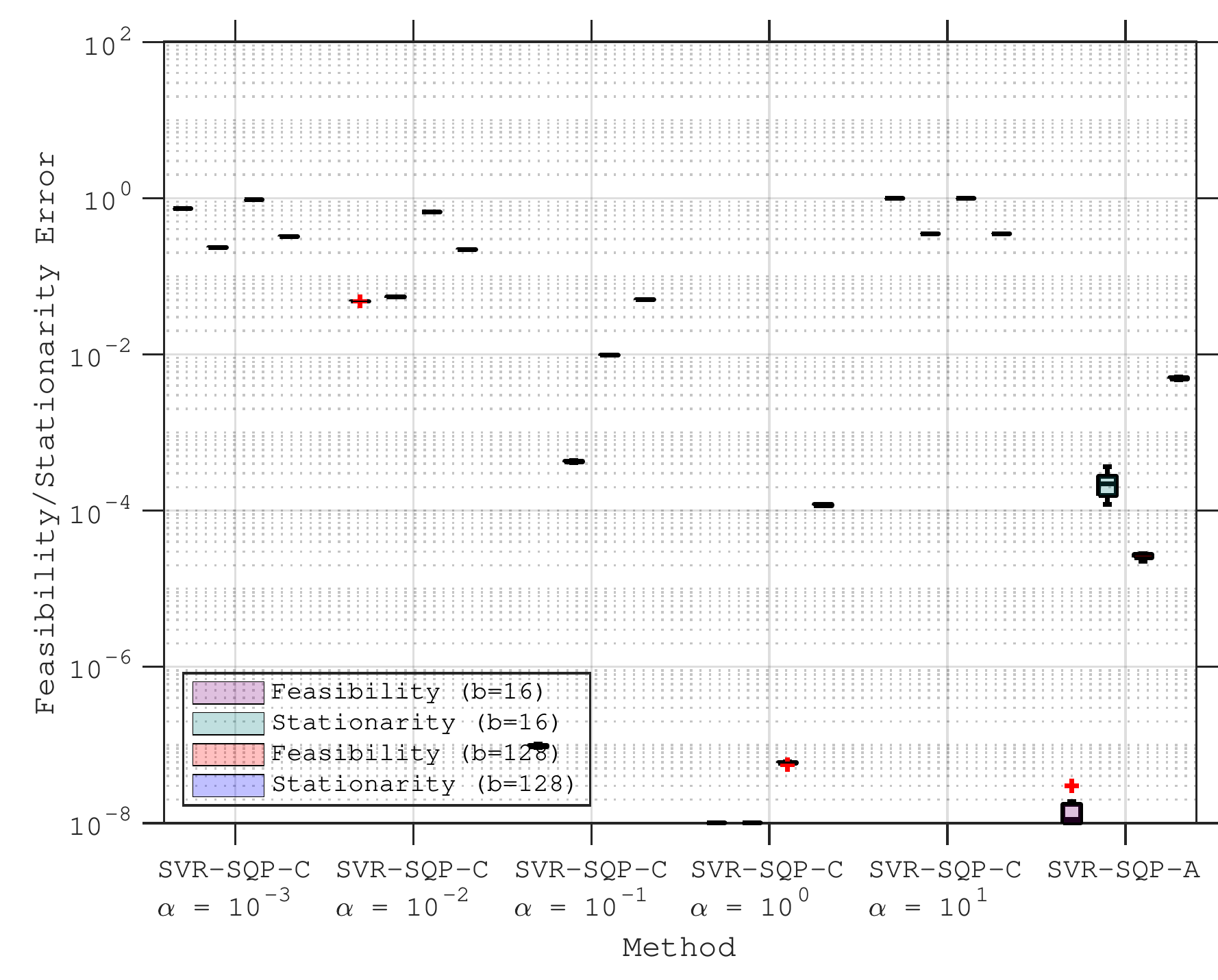} \caption{\texttt{australian}; left \eqref{eq.log_lin}, right \eqref{eq.log_el2}}
         \label{fig:2}
     \end{subfigure}
     
     \begin{subfigure}[b]{0.48\textwidth}
         \centering \includegraphics[width=0.48\textwidth,clip=true,trim=30 180 70 200]{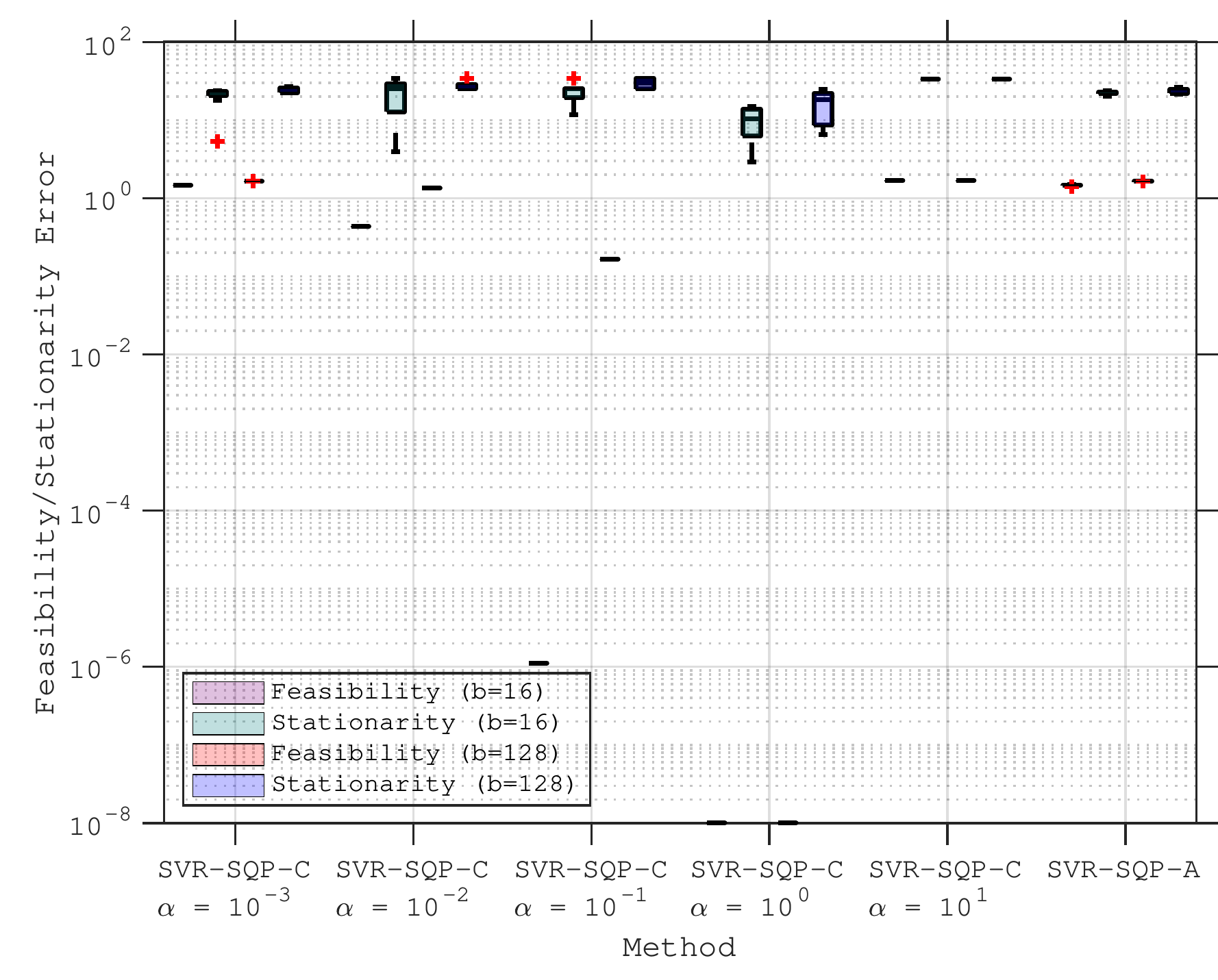}
         \includegraphics[width=0.48\textwidth,clip=true,trim=30 180 70 200]{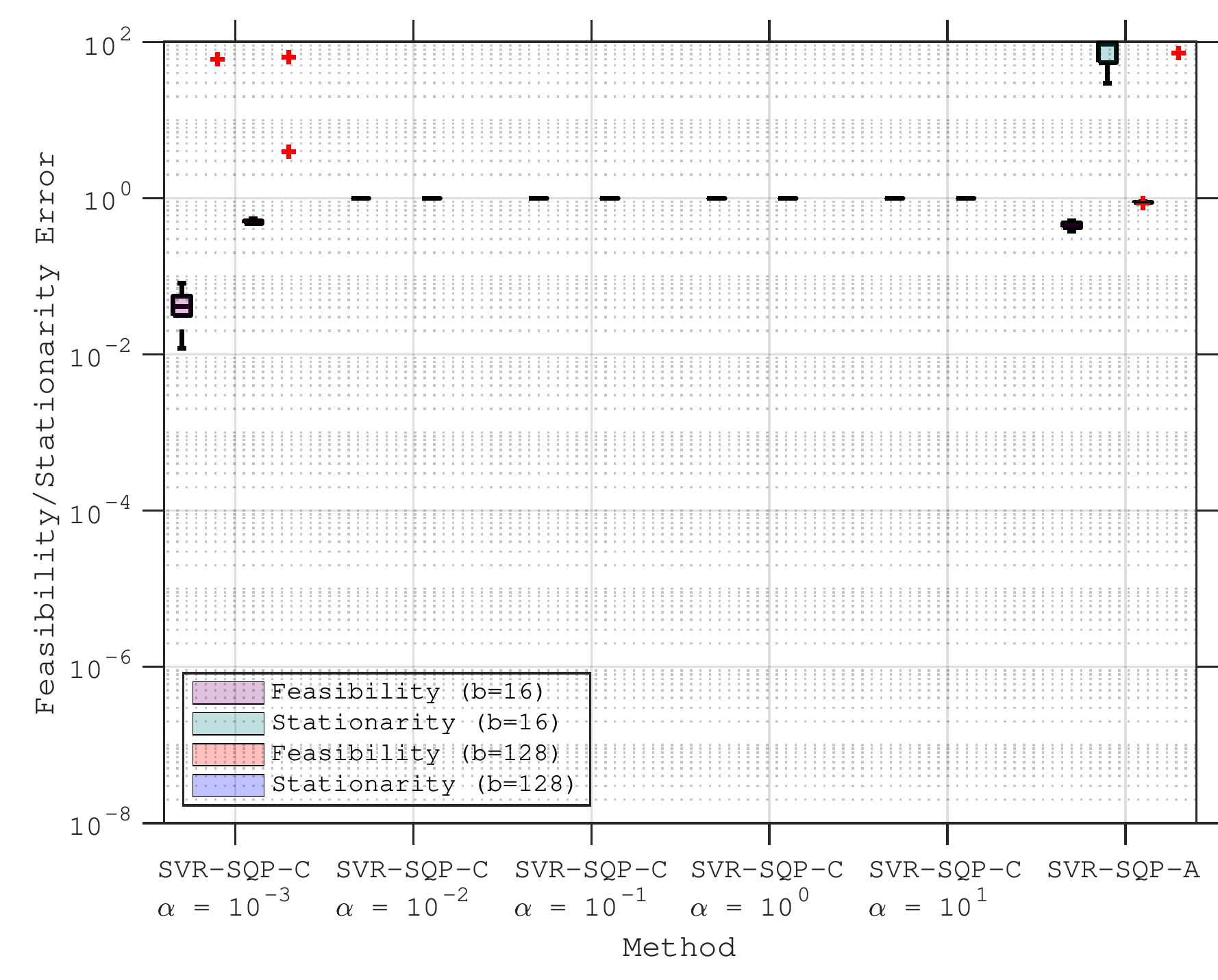}
         \caption{\texttt{heart}; left \eqref{eq.log_lin}, right \eqref{eq.log_el2}}
         \label{fig:3}
     \end{subfigure}
     \hfill
     \begin{subfigure}[b]{0.48\textwidth}
         \centering \includegraphics[width=0.48\textwidth,clip=true,trim=30 180 70 200]{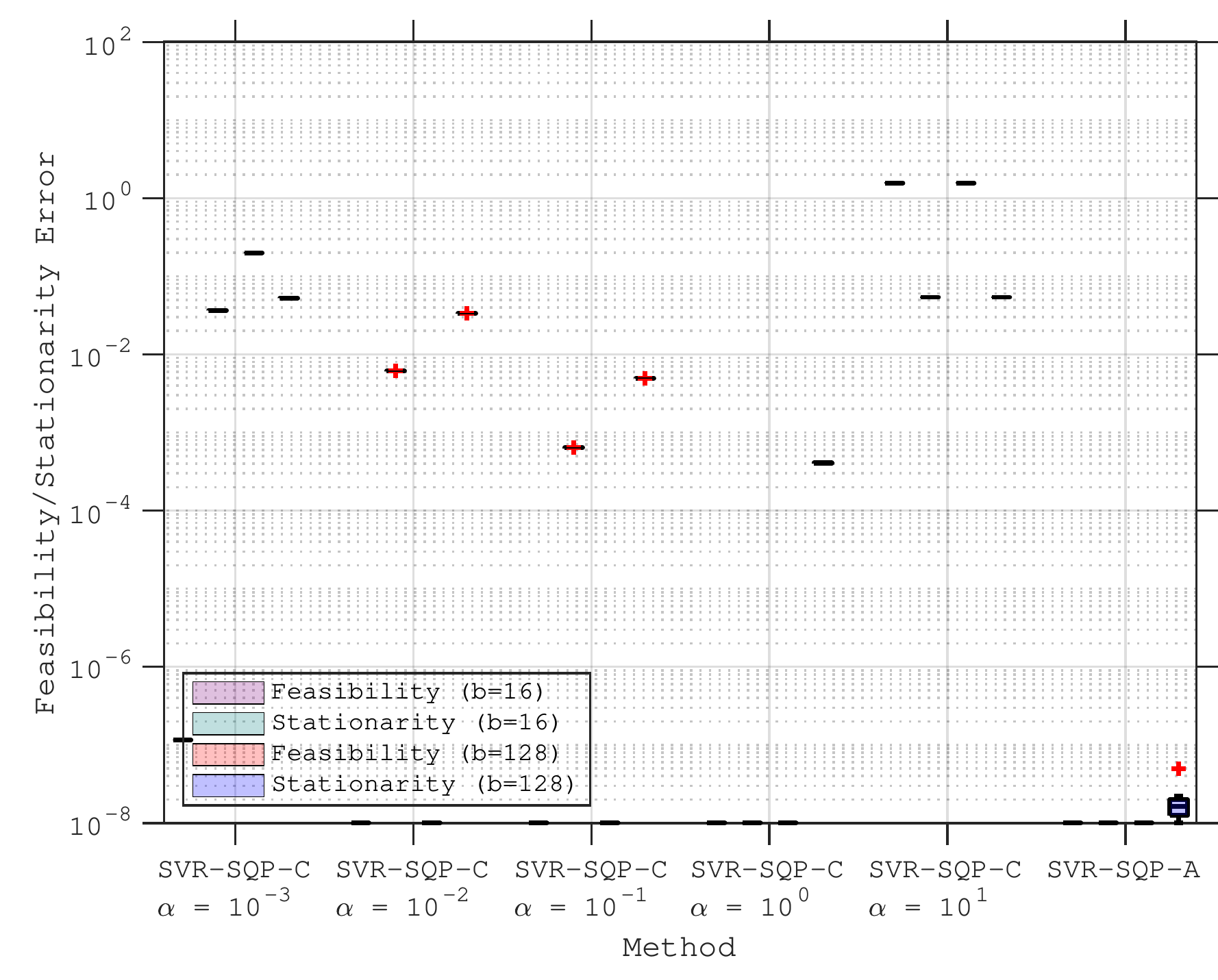}
         \includegraphics[width=0.48\textwidth,clip=true,trim=30 180 70 200]{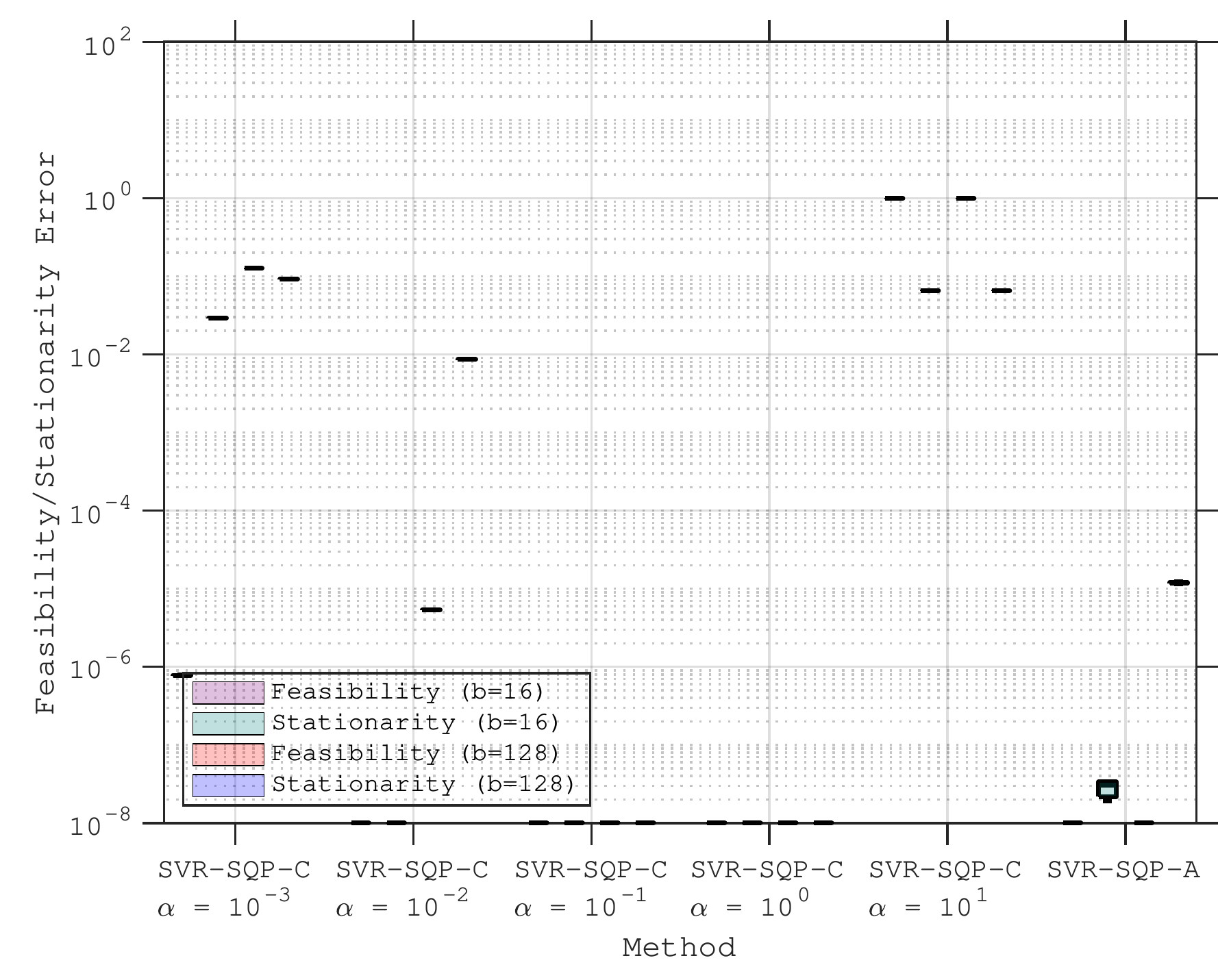} \caption{\texttt{ijcnn1}; left \eqref{eq.log_lin}, right \eqref{eq.log_el2}}
         \label{fig:4}
     \end{subfigure}
     
     \begin{subfigure}[b]{0.48\textwidth}
         \centering \includegraphics[width=0.48\textwidth,clip=true,trim=30 180 70 200]{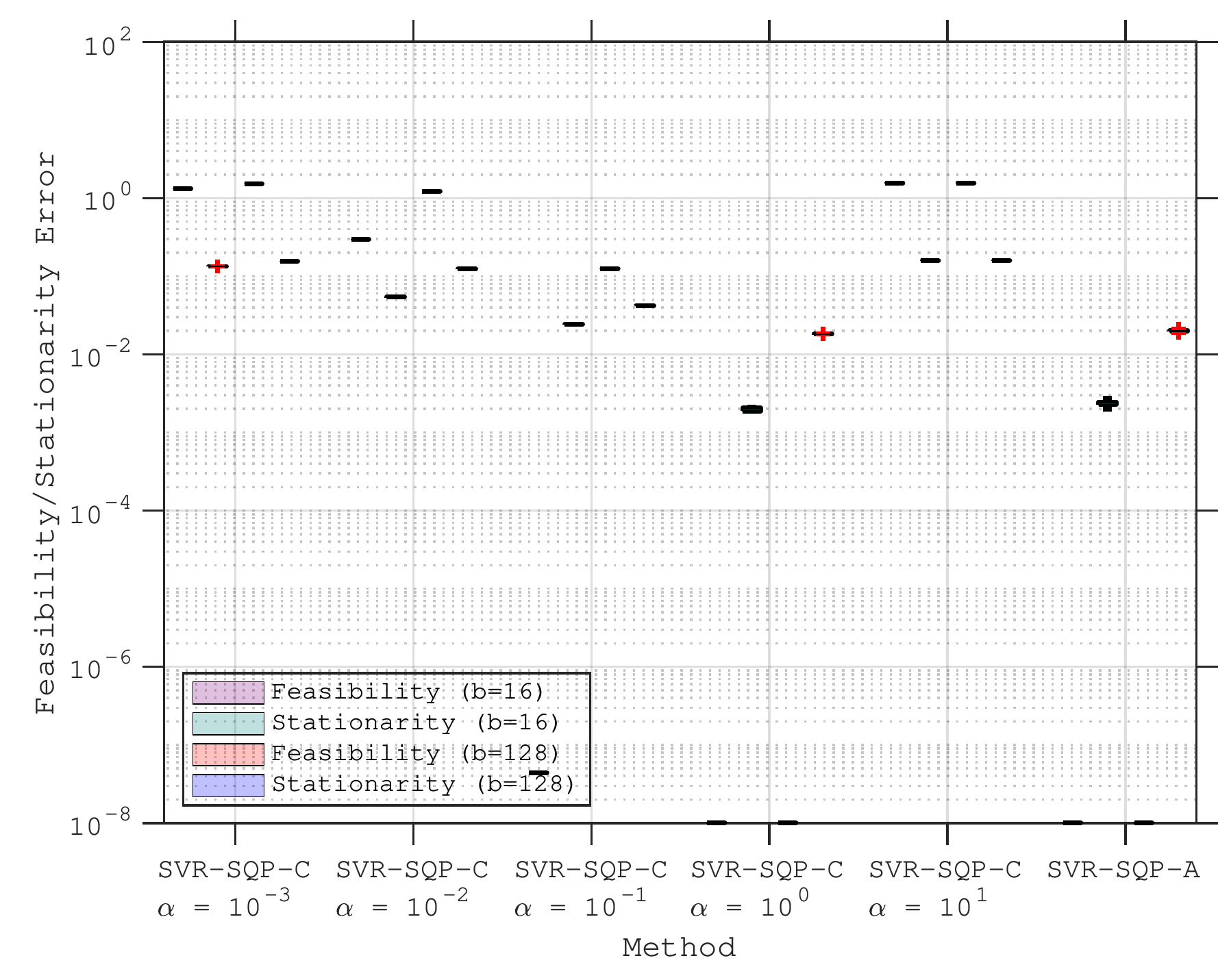}
         \includegraphics[width=0.48\textwidth,clip=true,trim=30 180 70 200]{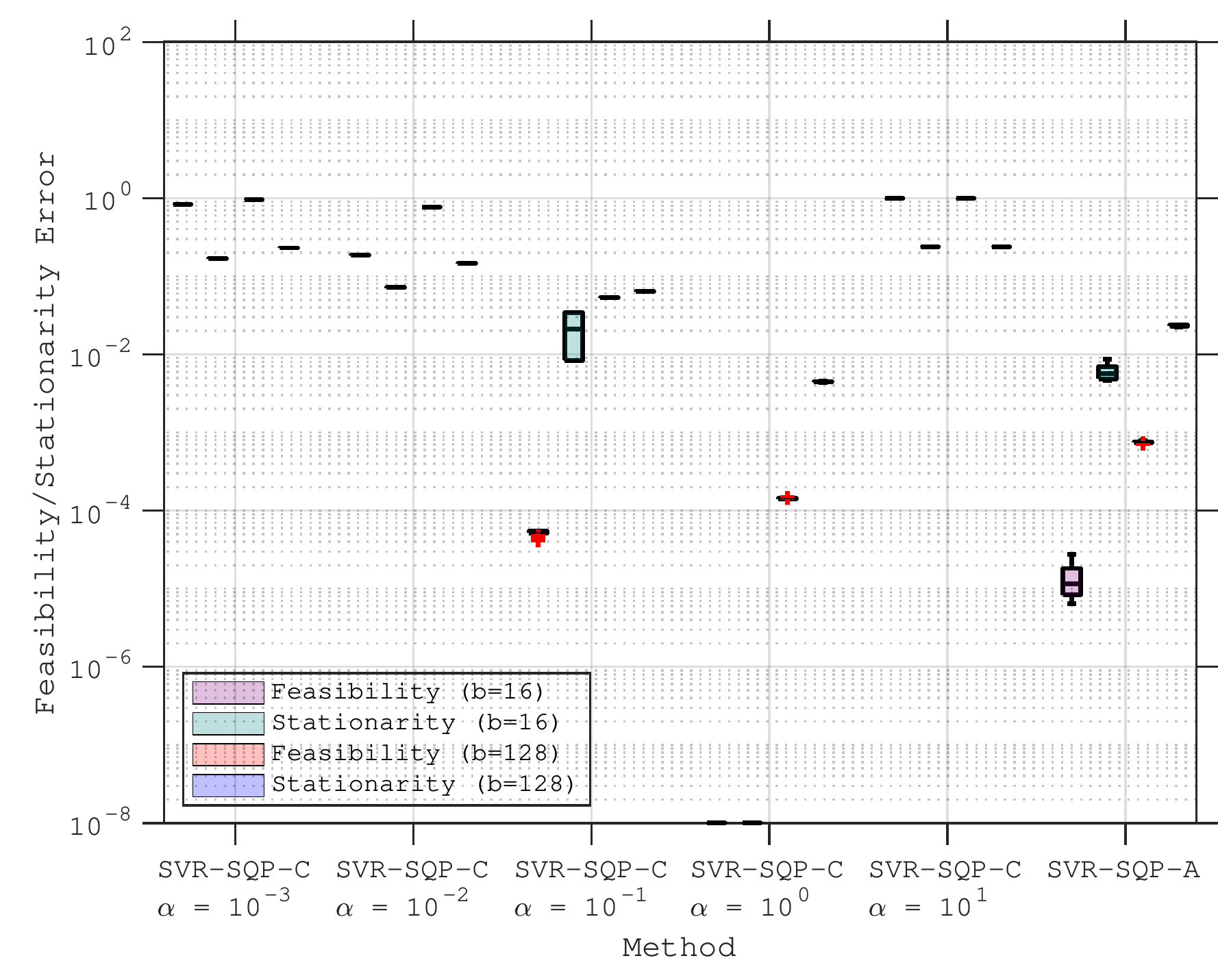}
         \caption{\texttt{ionosphere}; left \eqref{eq.log_lin}, right \eqref{eq.log_el2}}
         \label{fig:5}
     \end{subfigure}
     \hfill
     \begin{subfigure}[b]{0.48\textwidth}
         \centering \includegraphics[width=0.48\textwidth,clip=true,trim=30 180 70 200]{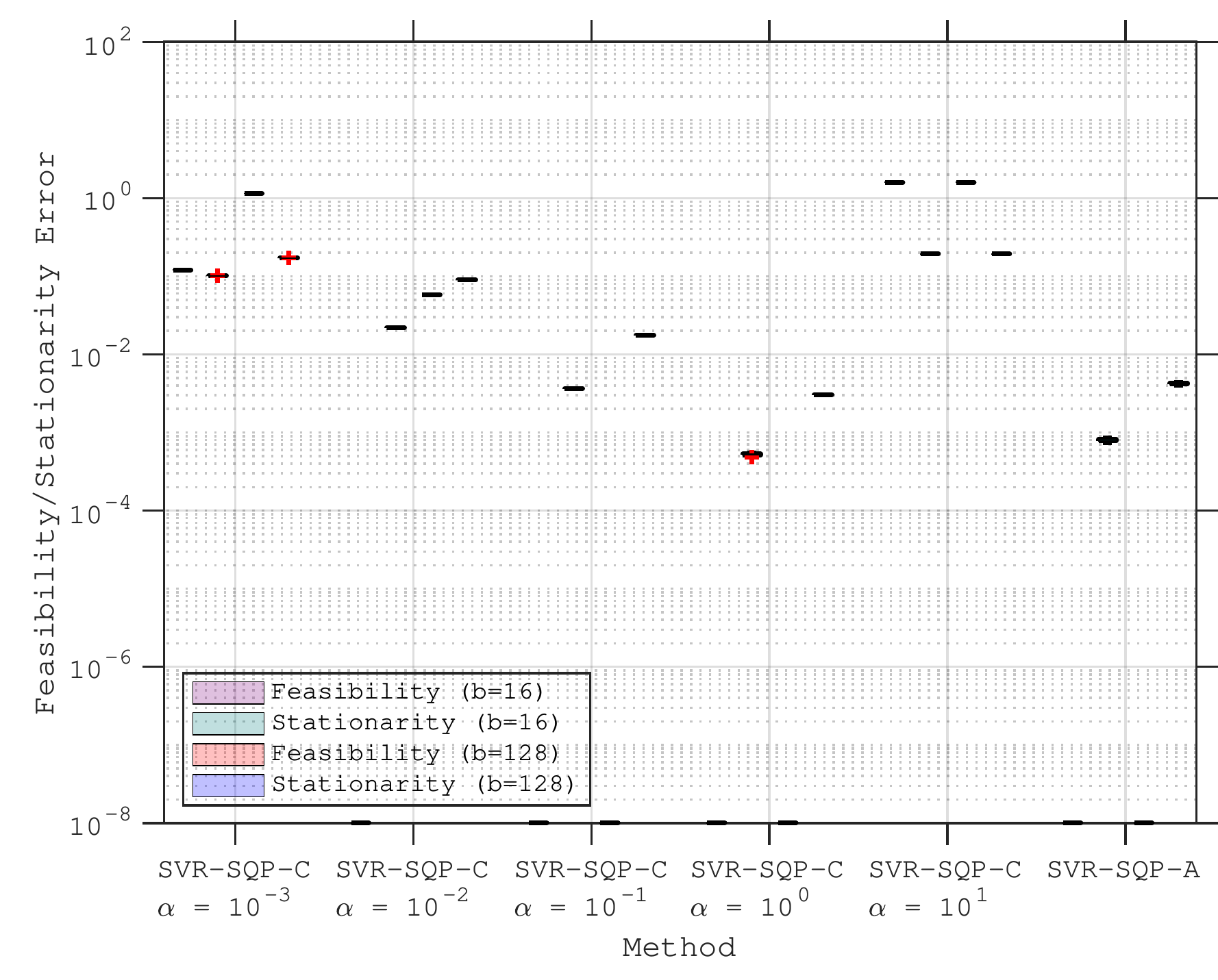}
         \includegraphics[width=0.48\textwidth,clip=true,trim=30 180 70 200]{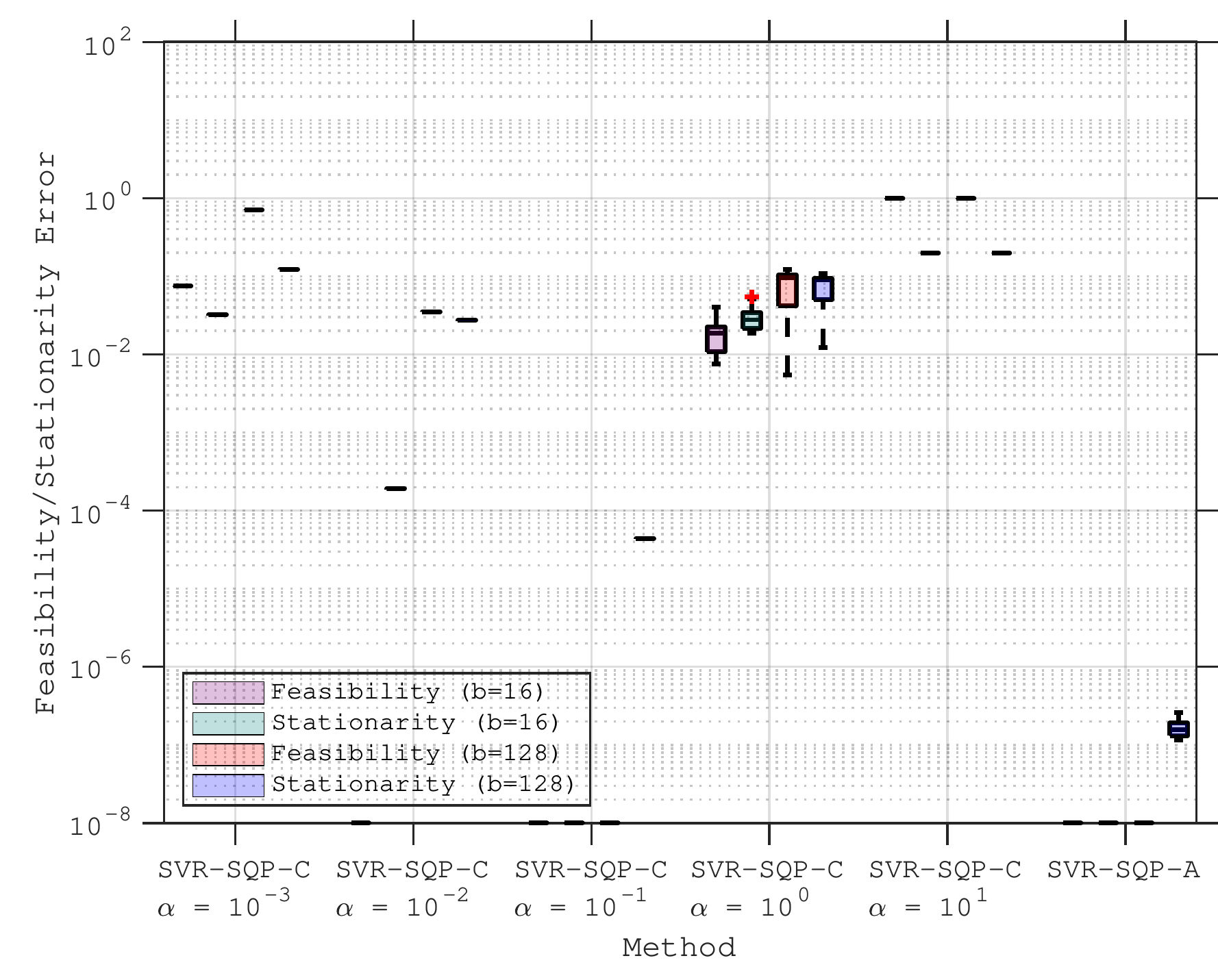} \caption{\texttt{mushrooms}; left \eqref{eq.log_lin}, right \eqref{eq.log_el2}}
         \label{fig:6}
     \end{subfigure}
     
     \begin{subfigure}[b]{0.48\textwidth}
         \centering \includegraphics[width=0.48\textwidth,clip=true,trim=30 180 70 200]{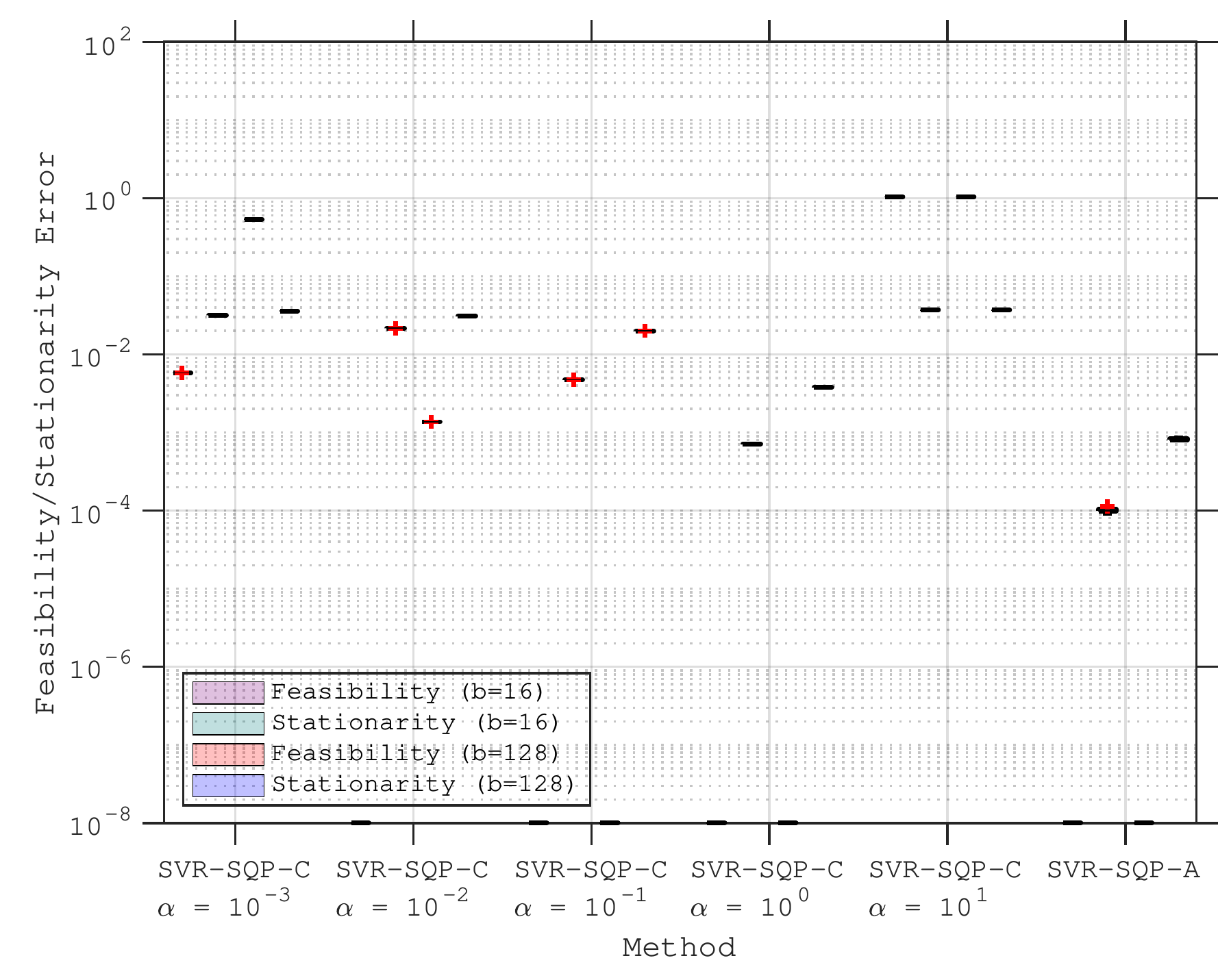}
         \includegraphics[width=0.48\textwidth,clip=true,trim=30 180 70 200]{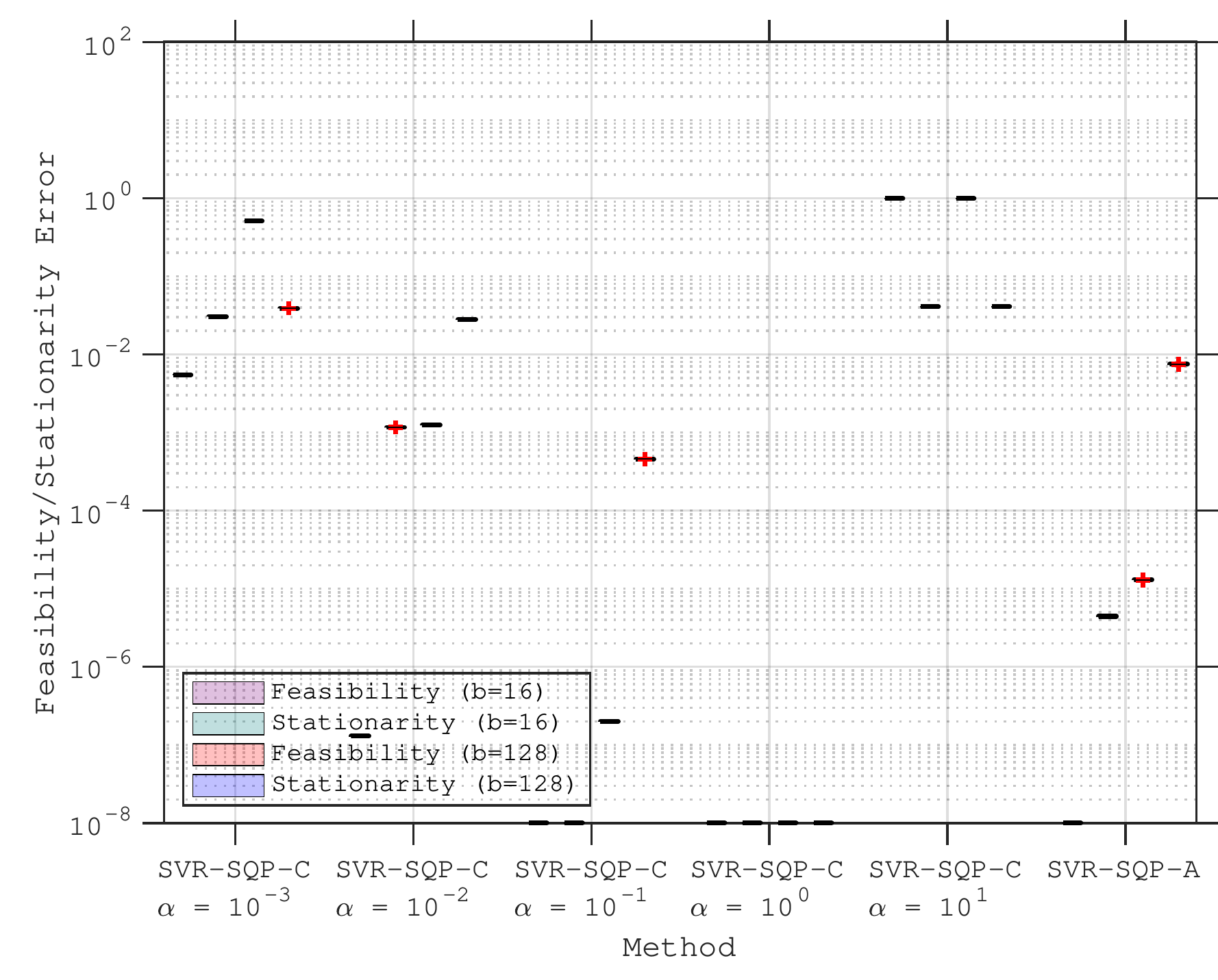}
         \caption{\texttt{phishing}; left \eqref{eq.log_lin}, right \eqref{eq.log_el2}}
         \label{fig:7}
     \end{subfigure}
     \hfill
     \begin{subfigure}[b]{0.48\textwidth}
         \centering \includegraphics[width=0.48\textwidth,clip=true,trim=30 180 70 200]{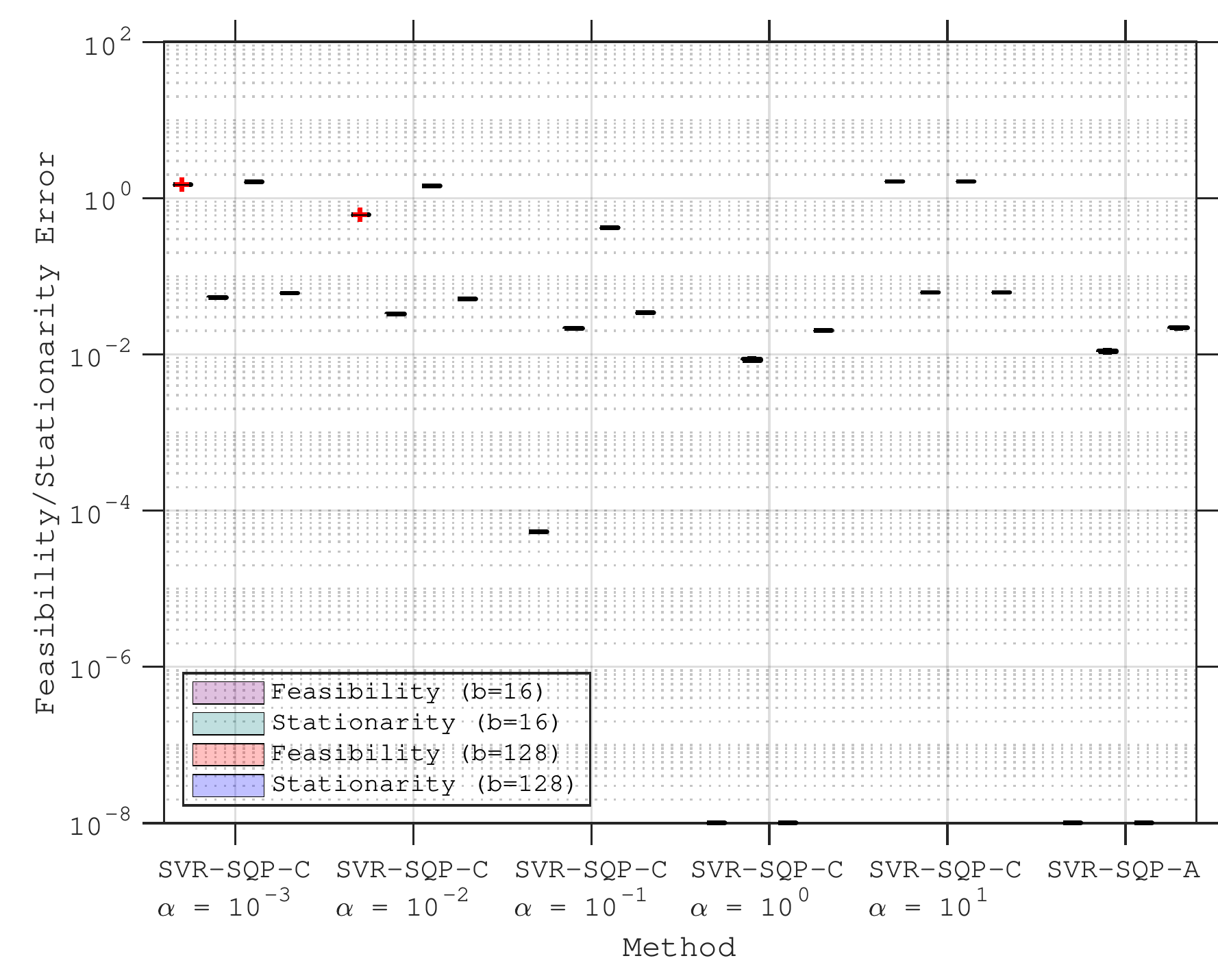}
         \includegraphics[width=0.48\textwidth,clip=true,trim=30 180 70 200]{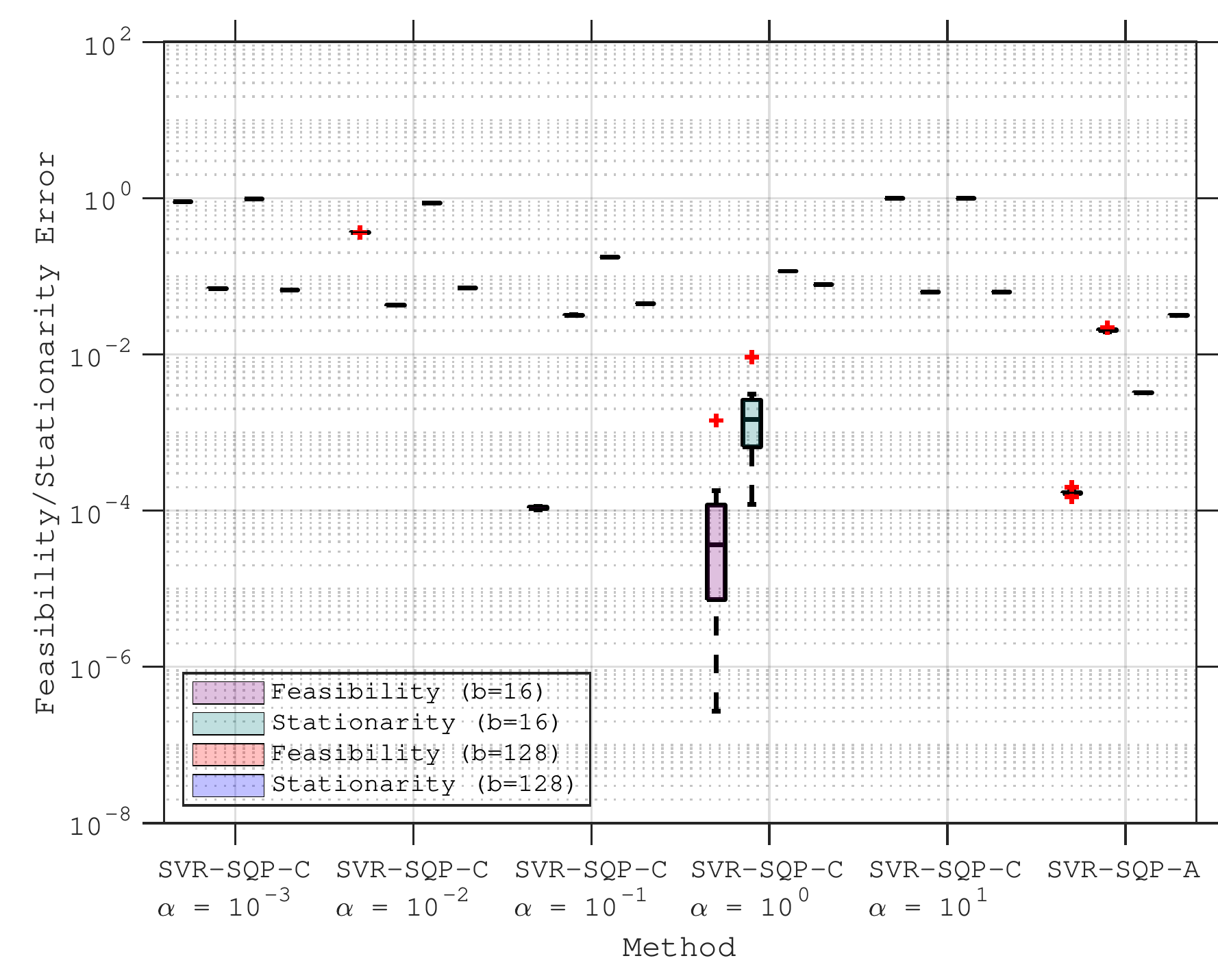} \caption{\texttt{sonar}; left \eqref{eq.log_lin}, right \eqref{eq.log_el2}}
         \label{fig:8}
     \end{subfigure}
     
     \begin{subfigure}[b]{0.48\textwidth}
         \centering \includegraphics[width=0.48\textwidth,clip=true,trim=30 180 70 200]{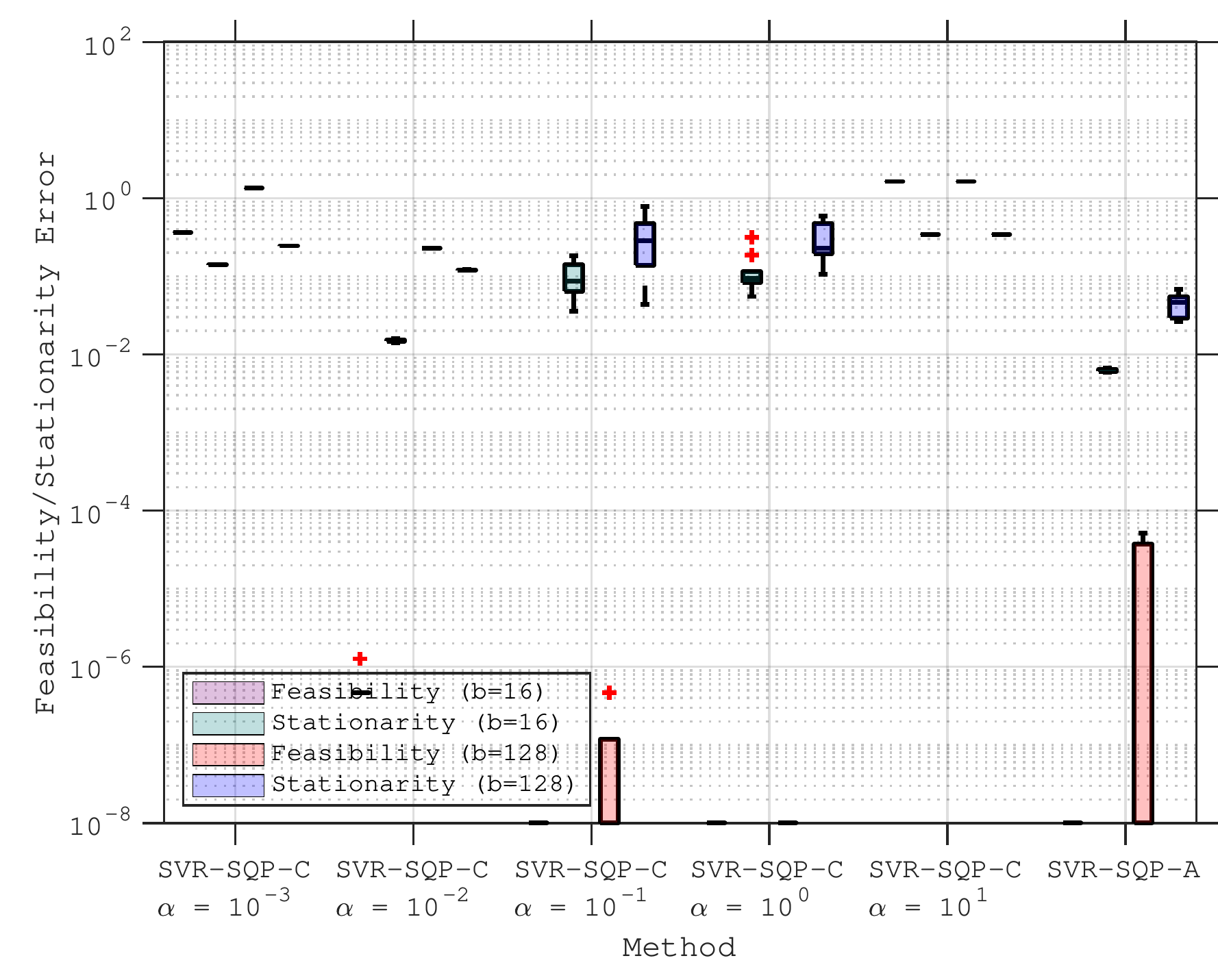}
         \includegraphics[width=0.48\textwidth,clip=true,trim=30 180 70 200]{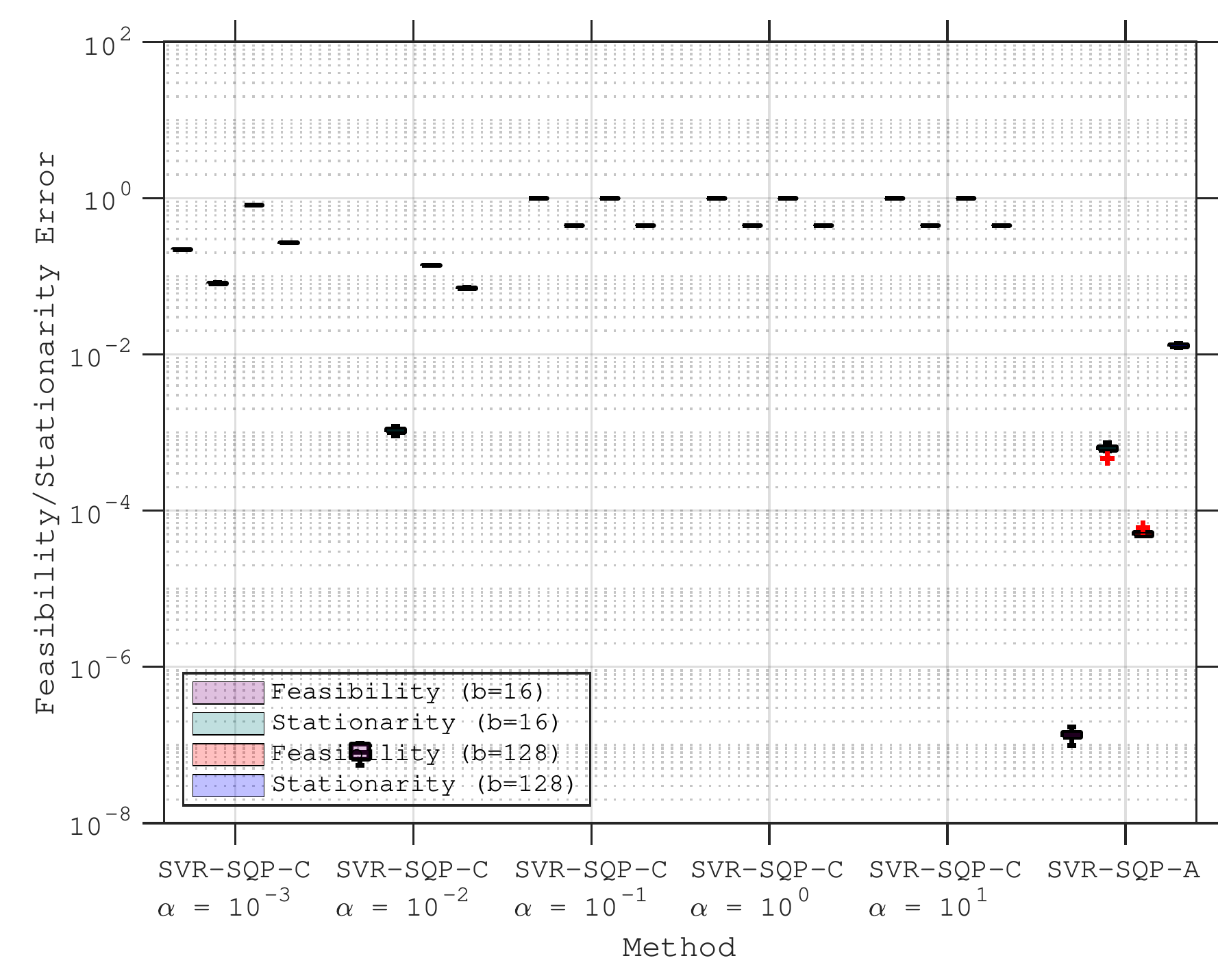}
         \caption{\texttt{splice}; left \eqref{eq.log_lin}, right \eqref{eq.log_el2}}
         \label{fig:9}
     \end{subfigure}
     \hfill
     \begin{subfigure}[b]{0.48\textwidth}
         \centering \includegraphics[width=0.48\textwidth,clip=true,trim=30 180 70 200]{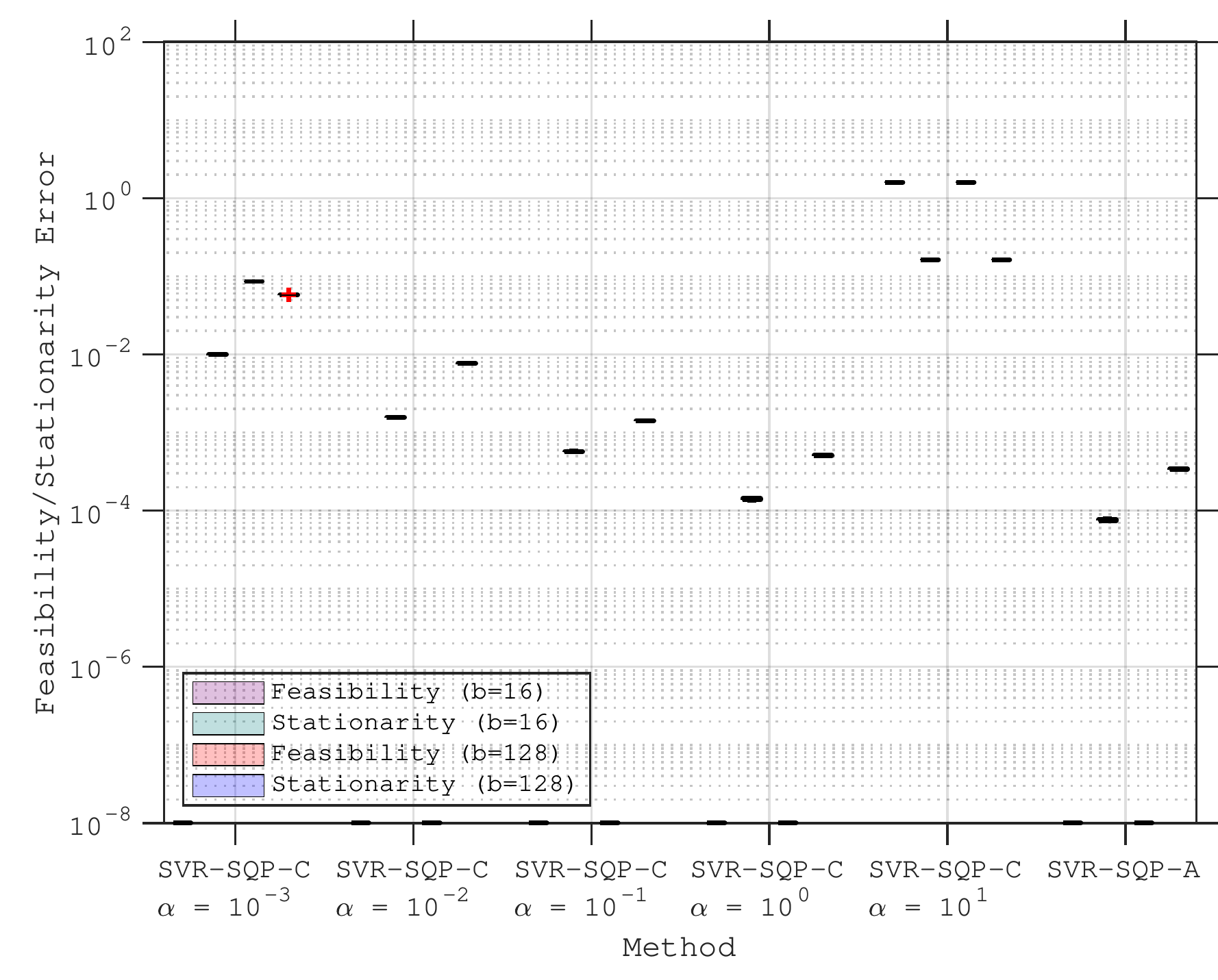}
         \includegraphics[width=0.48\textwidth,clip=true,trim=30 180 70 200]{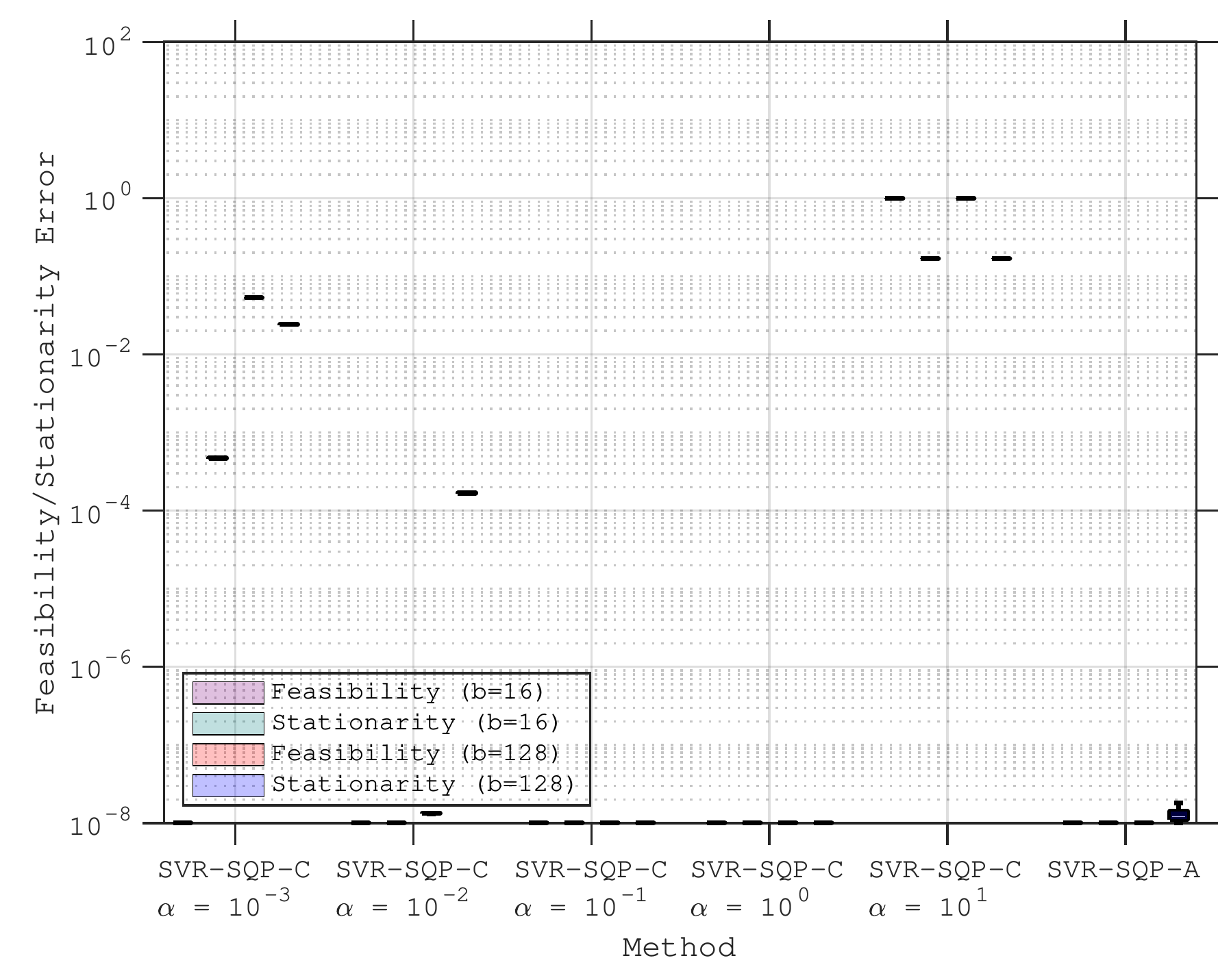} \caption{\texttt{w8a}; left \eqref{eq.log_lin}, right \eqref{eq.log_el2}}
         \label{fig:10}
     \end{subfigure}
        \caption{Best feasibility and stationarity errors,  for \SVRSQPCONST{} and \SVRSQPADAPT{} on \eqref{eq.log_lin} and \eqref{eq.log_el2}.}
        \label{fig.svrsqra_svrsqpa_summary}
\end{figure}

\clearpage

\newpage

\subsection{Sensitivity to user-defined parameters}

Given the encouraging numerical results for \SVRSQPADAPT{} (Section~\ref{sec.numerical_adaptiveconstant}), in this subsection we investigate the robustness of \SVRSQPADAPT{} to two user-defined parameters: (1) the step size parameter $\beta \in \{ 10^{-3},10^{-2},10^{-1},10^{0},10^{1}\}$ (Fig.~\ref{fig.sensitivity}), and (2) the number of inner iterations $S \in \left\{ \left\lfloor \tfrac{N}{b} \right\rfloor,\left\lfloor \tfrac{N}{2b} \right\rfloor,\left\lfloor \tfrac{N}{4b} \right\rfloor\right\}$ (Fig.~\ref{fig.sensitivity2}) for two datasets (\texttt{australian} and \texttt{splice}). Overall, the results on these two datasets suggest that $\beta = 1$ is often the best choice.  Moreover, our results in Fig.~\ref{fig.sensitivity2} illustrate the robustness of \SVRSQPADAPT{} to the choice of the number of inner iterations.

\begin{figure}[ht]
   \centering
     \begin{subfigure}[b]{1\textwidth}
     \includegraphics[width=0.24\textwidth,clip=true,trim=30 180 50 200]{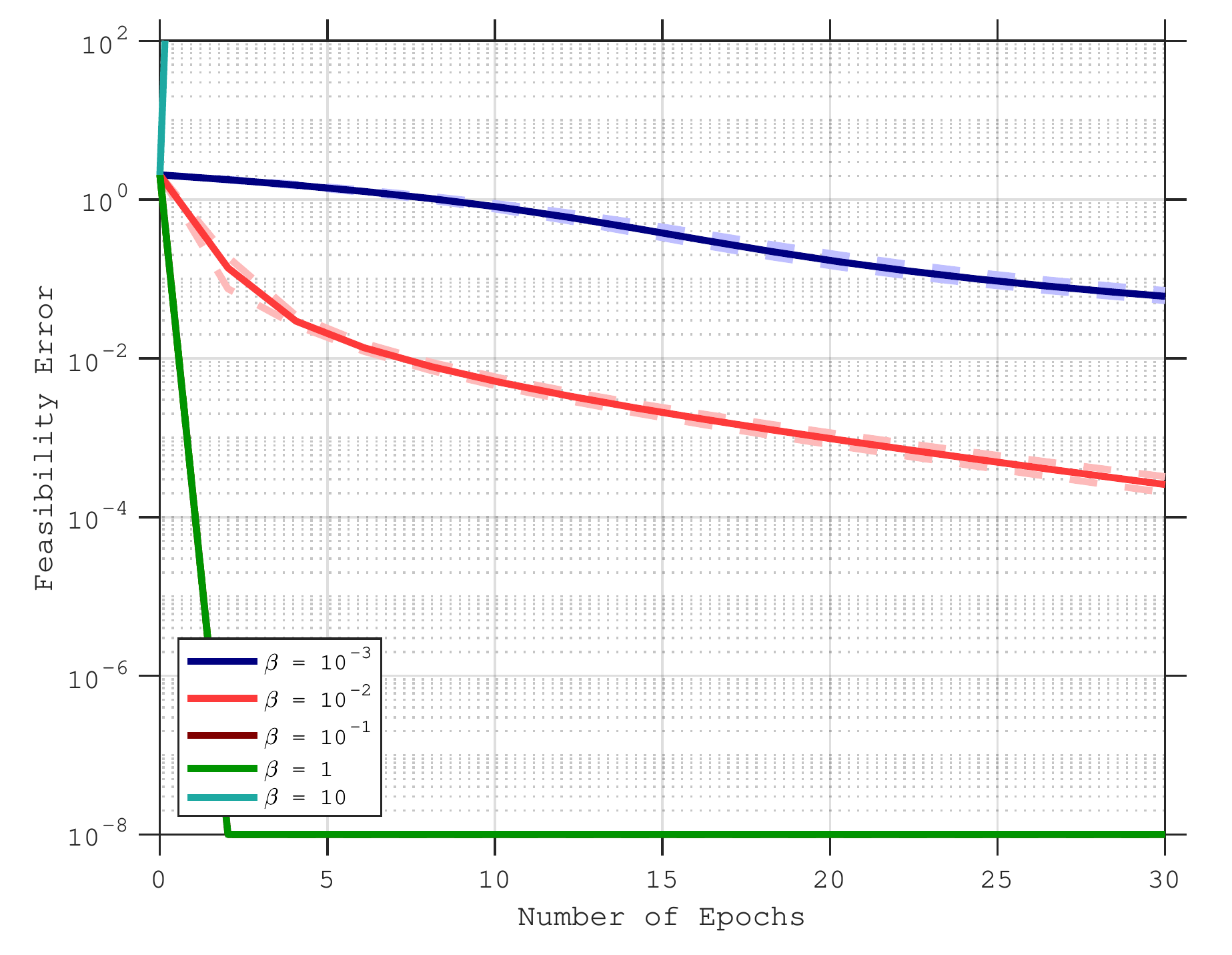}
  \includegraphics[width=0.24\textwidth,clip=true,trim=30 180 50 200]{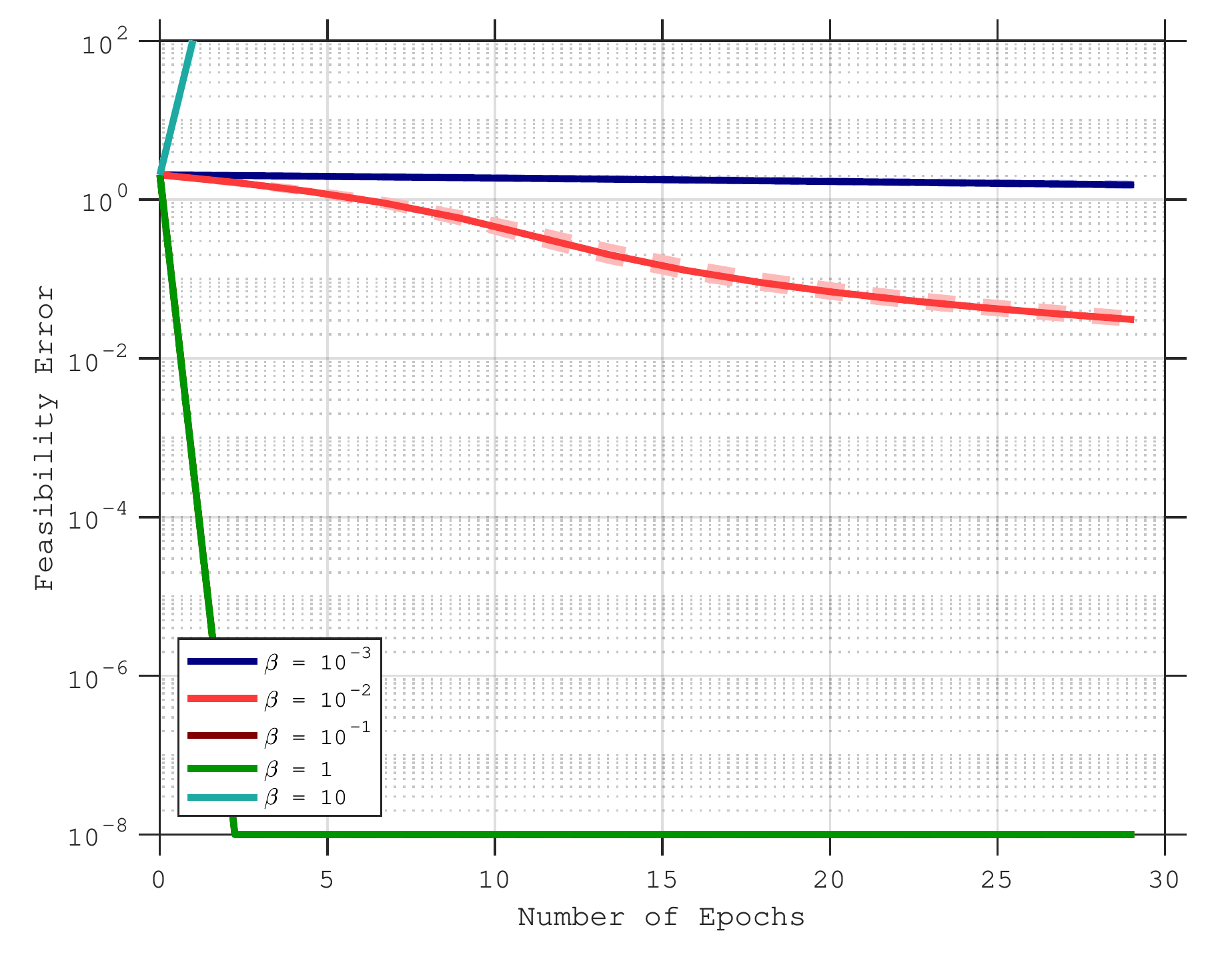}
  \includegraphics[width=0.24\textwidth,clip=true,trim=30 180 50 200]{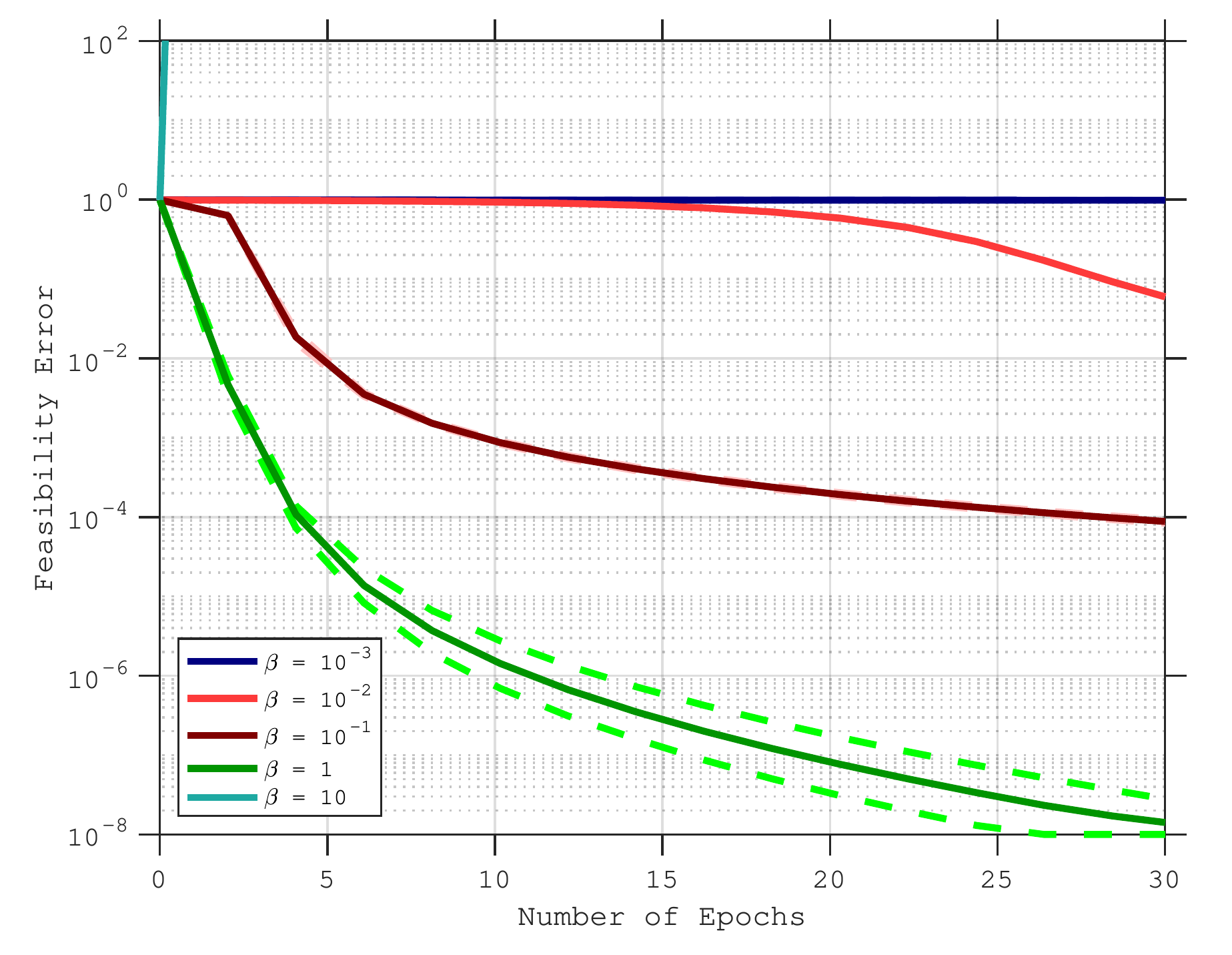}
  \includegraphics[width=0.24\textwidth,clip=true,trim=30 180 50 200]{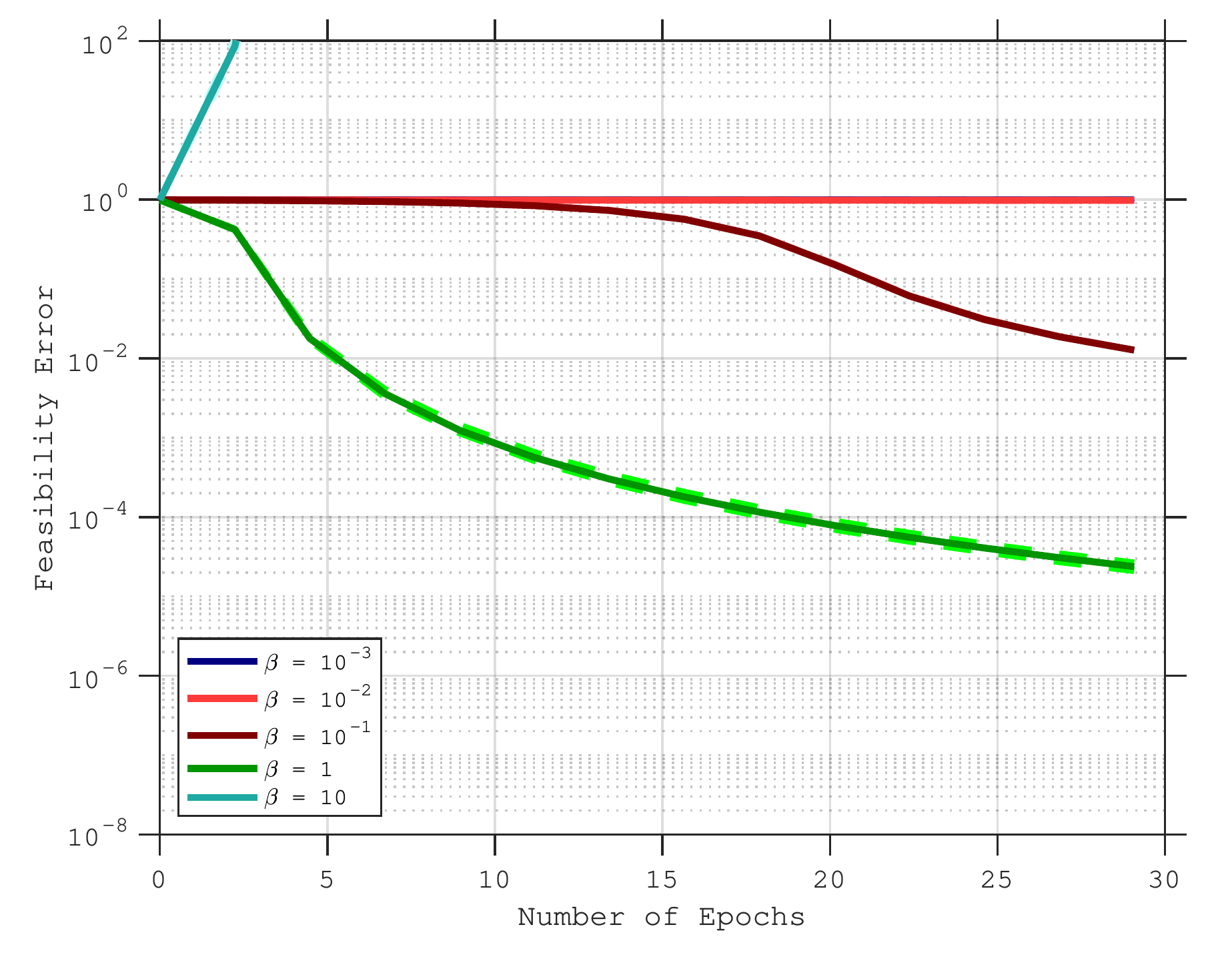}
  
  \includegraphics[width=0.24\textwidth,clip=true,trim=30 180 50 200]{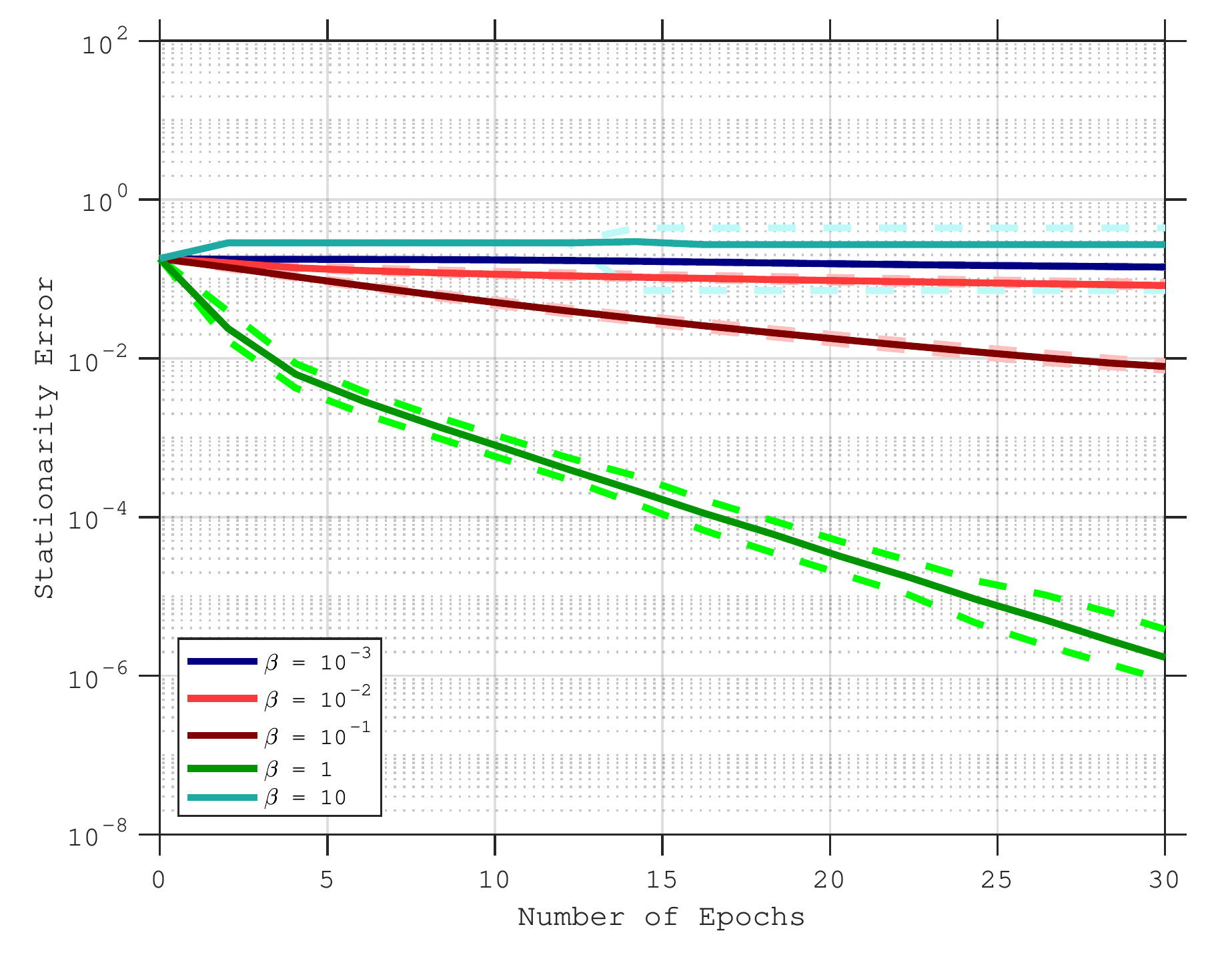}
  \includegraphics[width=0.24\textwidth,clip=true,trim=30 180 50 200]{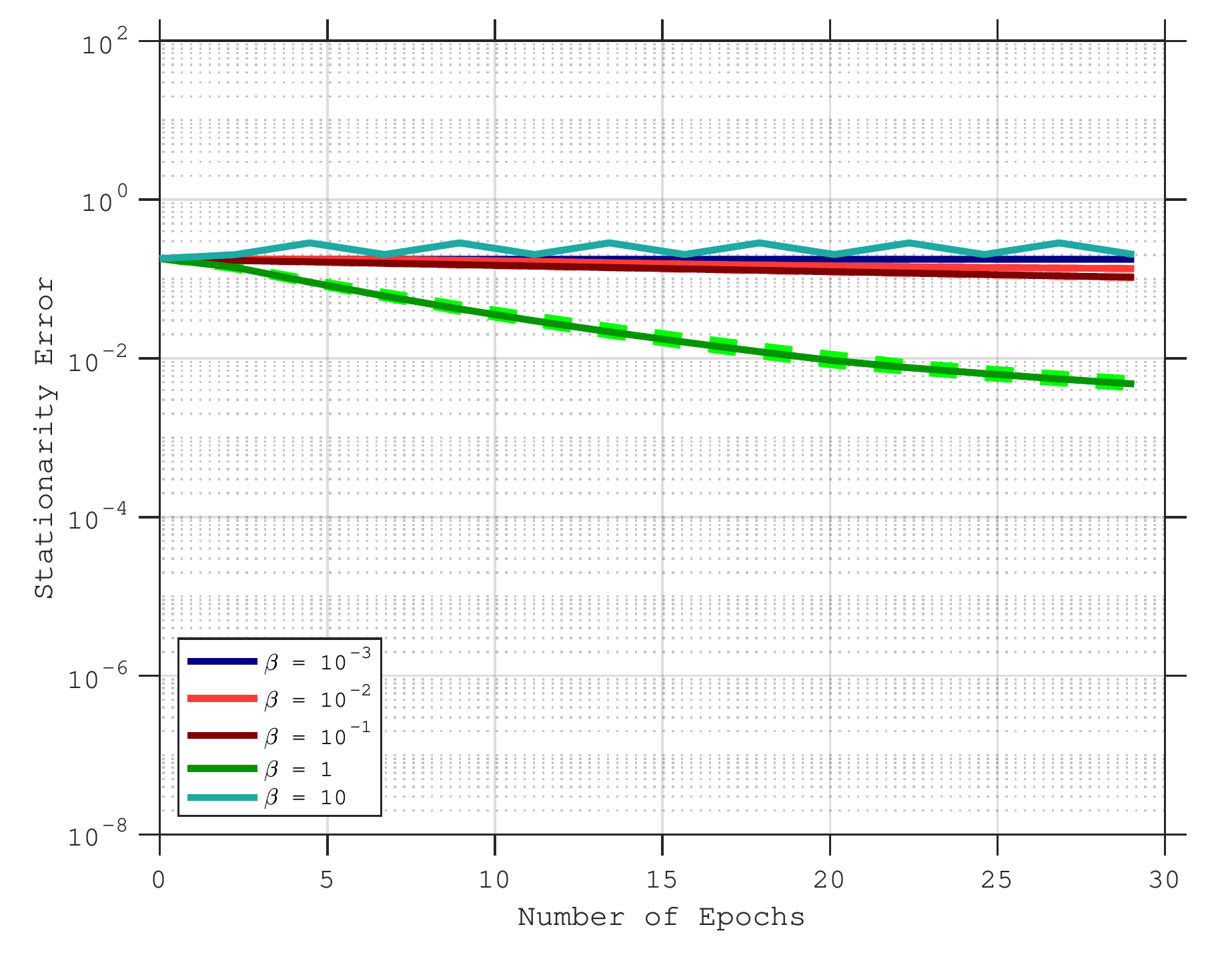}
  \includegraphics[width=0.24\textwidth,clip=true,trim=30 180 50 200]{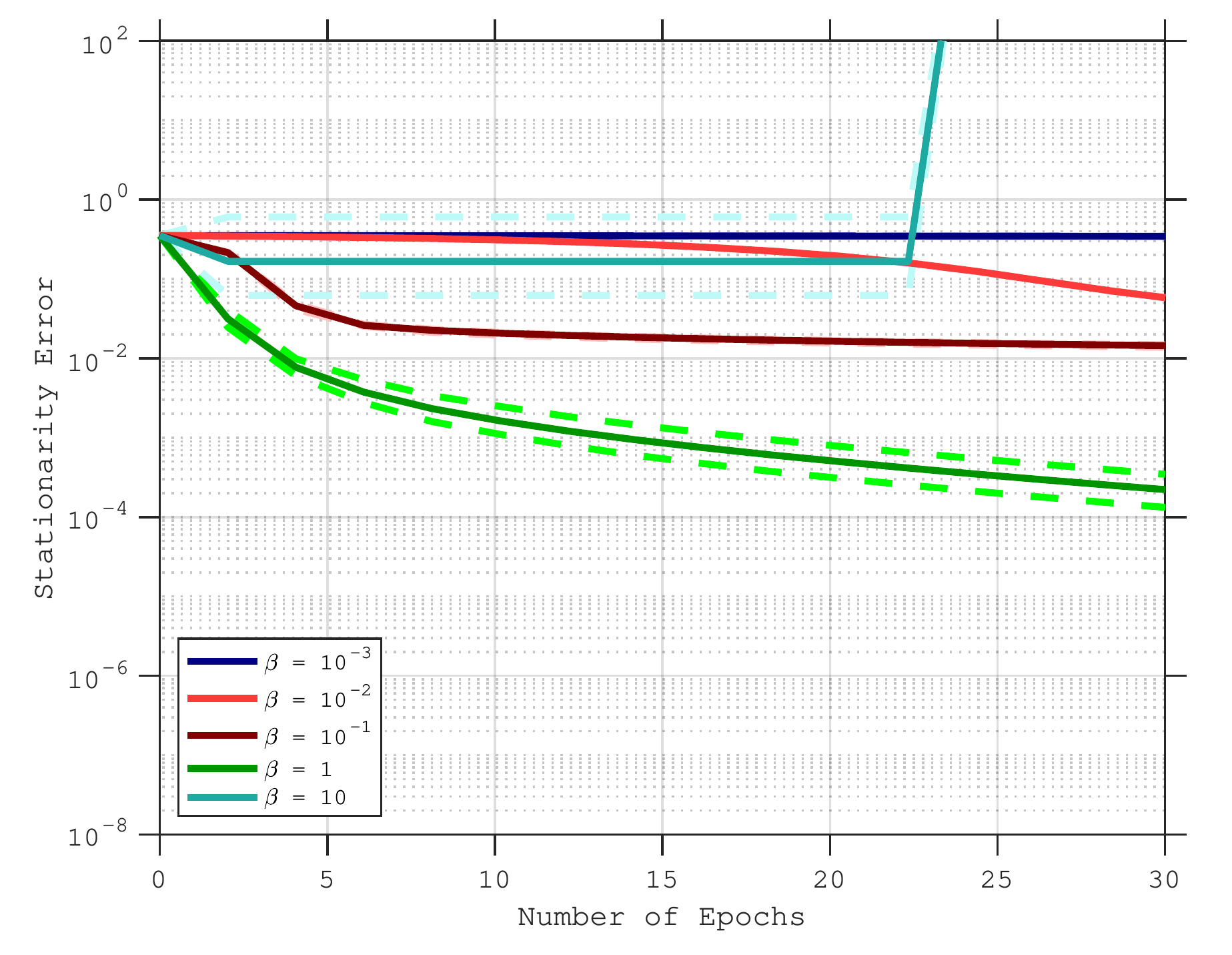}
  \includegraphics[width=0.24\textwidth,clip=true,trim=30 180 50 200]{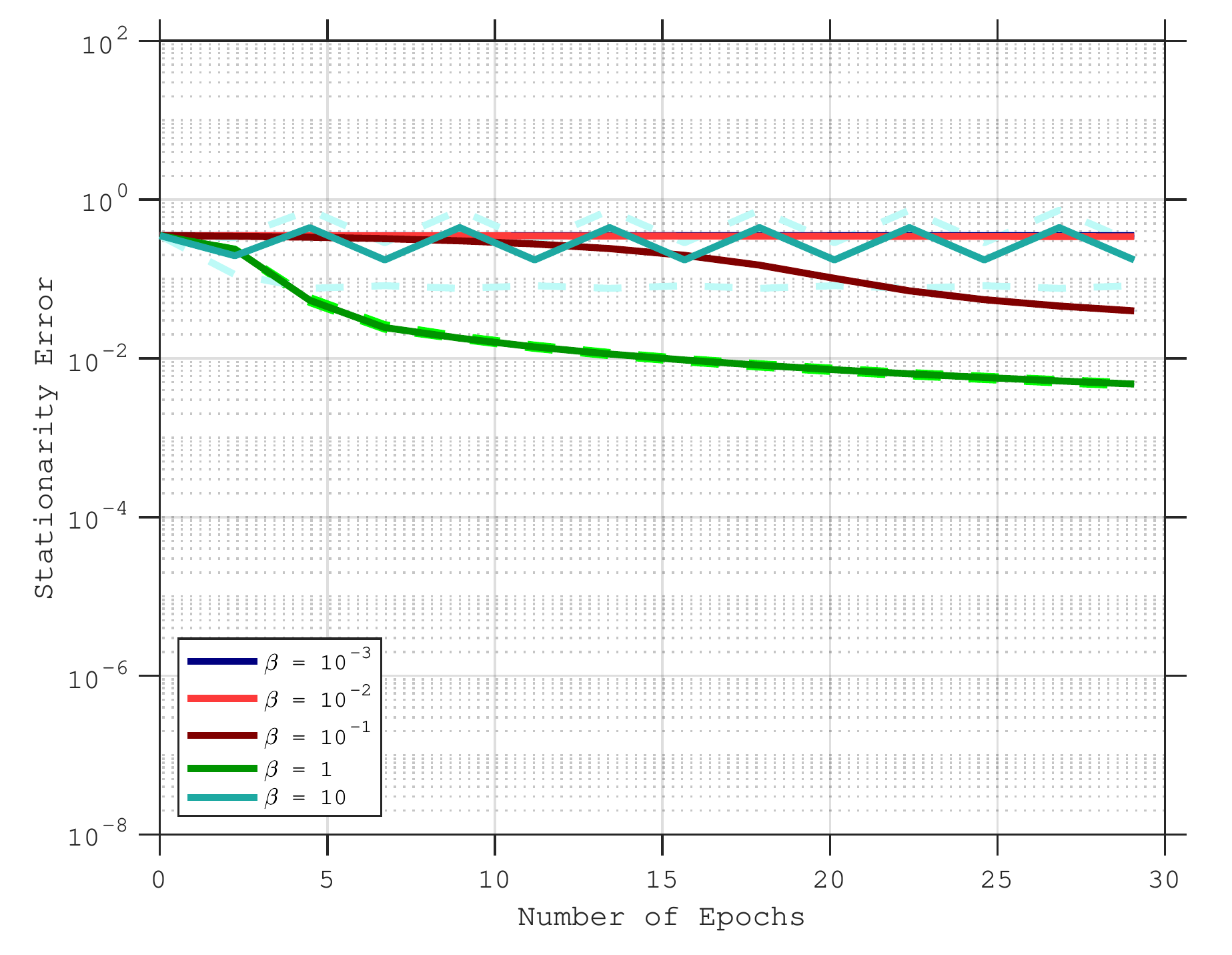}
  \caption{\texttt{australian} dataset. Top row: feasibility error; Bottom row: stationarity
  error.}
  \label{fig.sensitivity1a}
    \end{subfigure}

  \begin{subfigure}[b]{\textwidth}
  \includegraphics[width=0.24\textwidth,clip=true,trim=30 180 50 200]{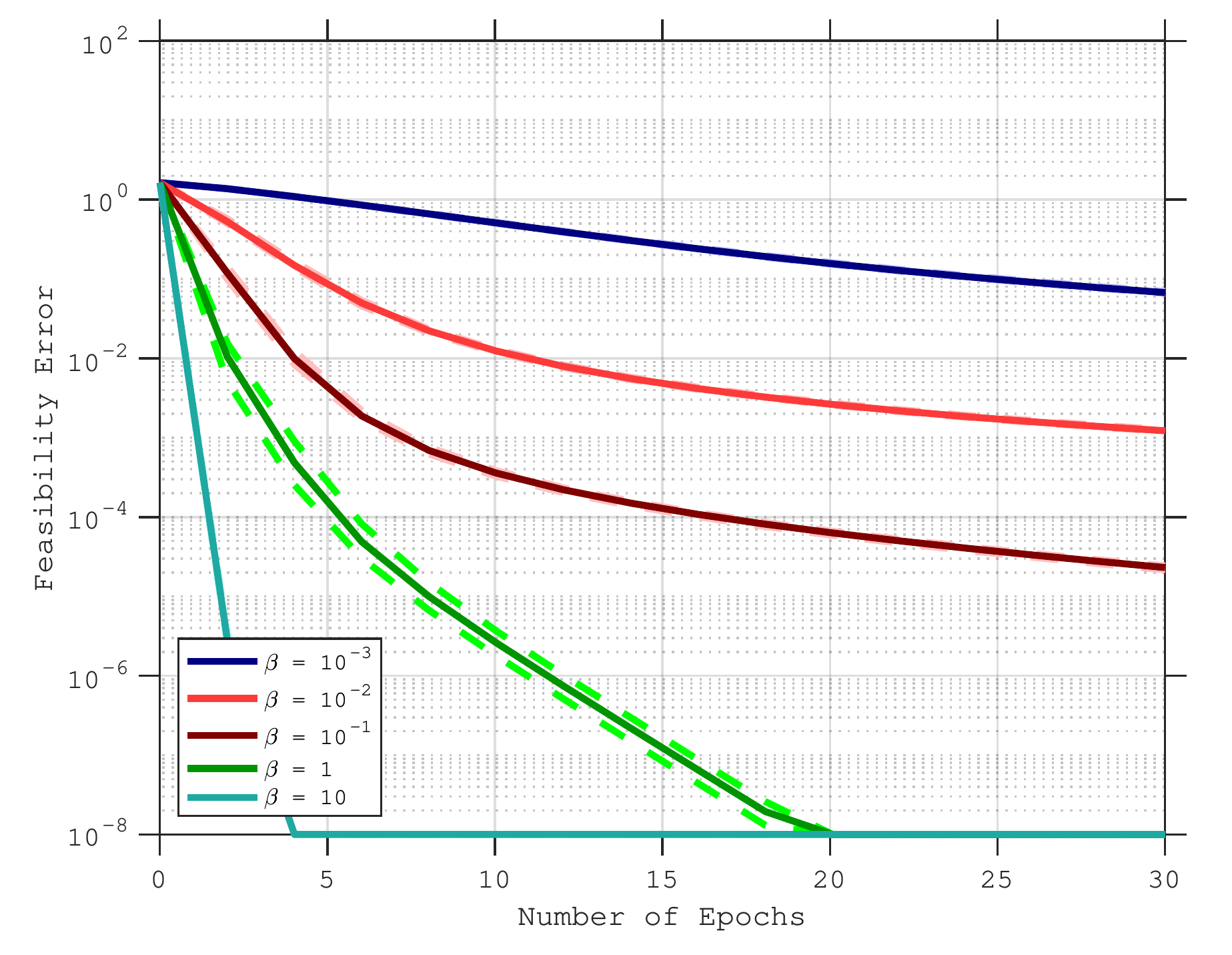}
  \includegraphics[width=0.24\textwidth,clip=true,trim=30 180 50 200]{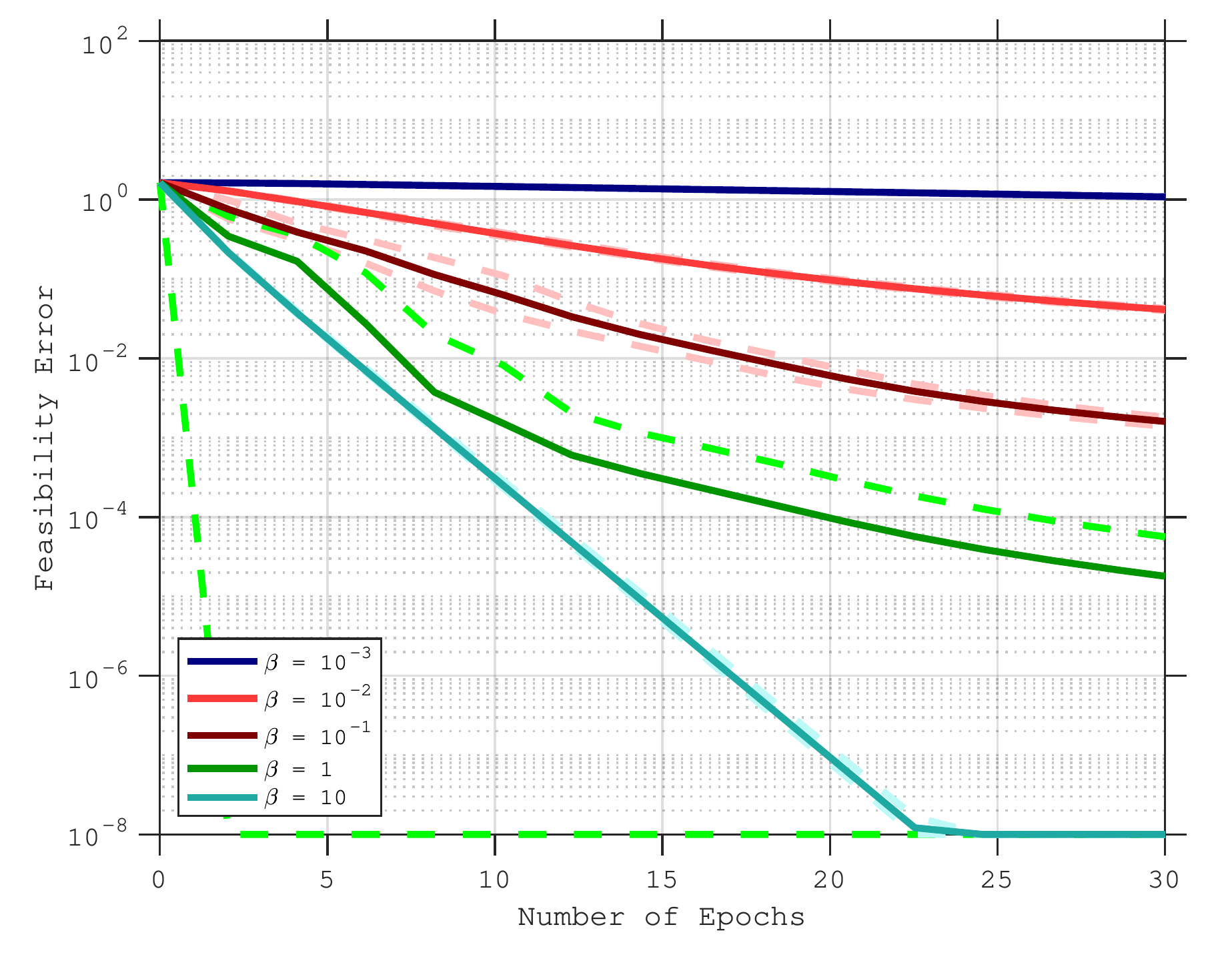}
  \includegraphics[width=0.24\textwidth,clip=true,trim=30 180 50 200]{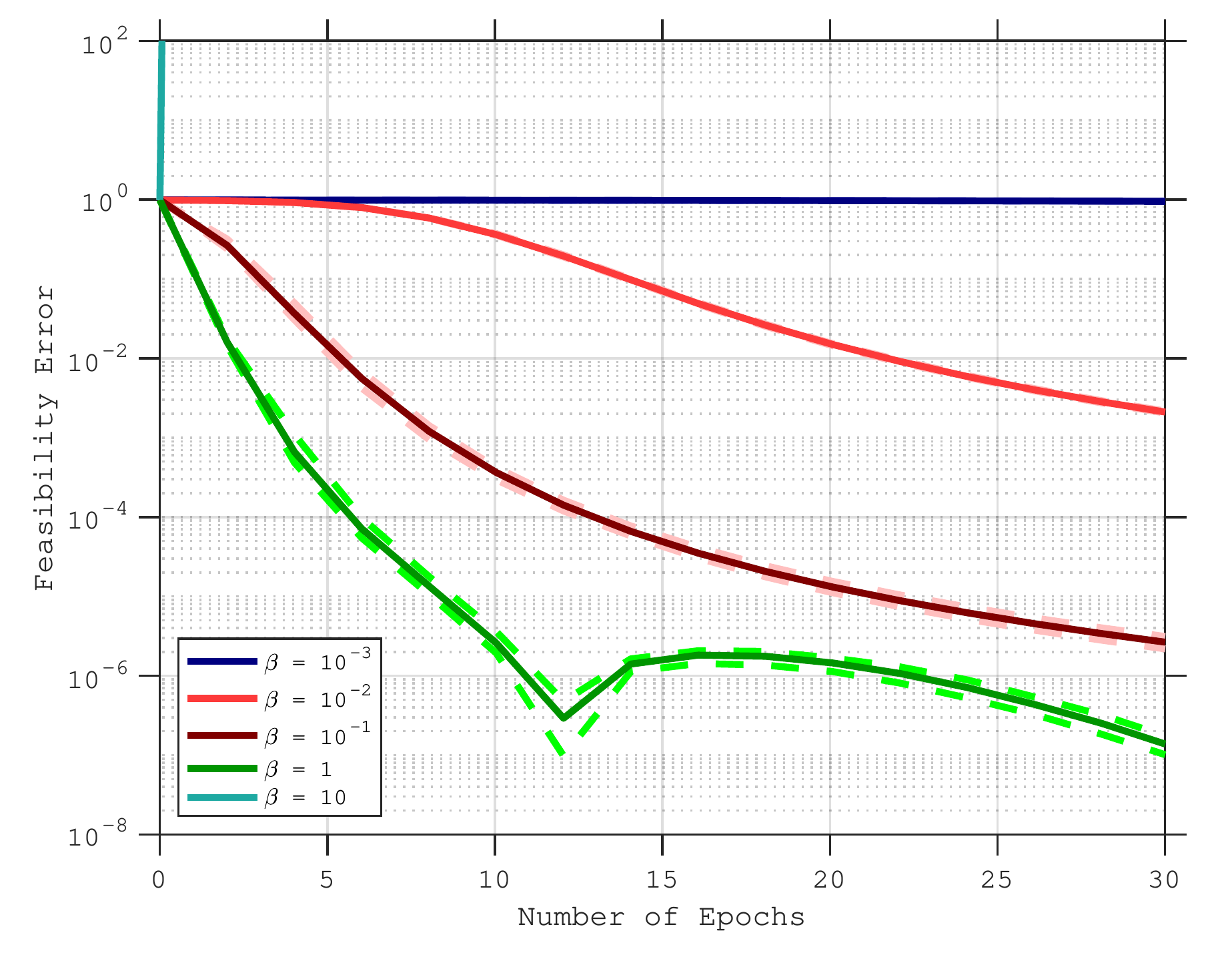}
  \includegraphics[width=0.24\textwidth,clip=true,trim=30 180 50 200]{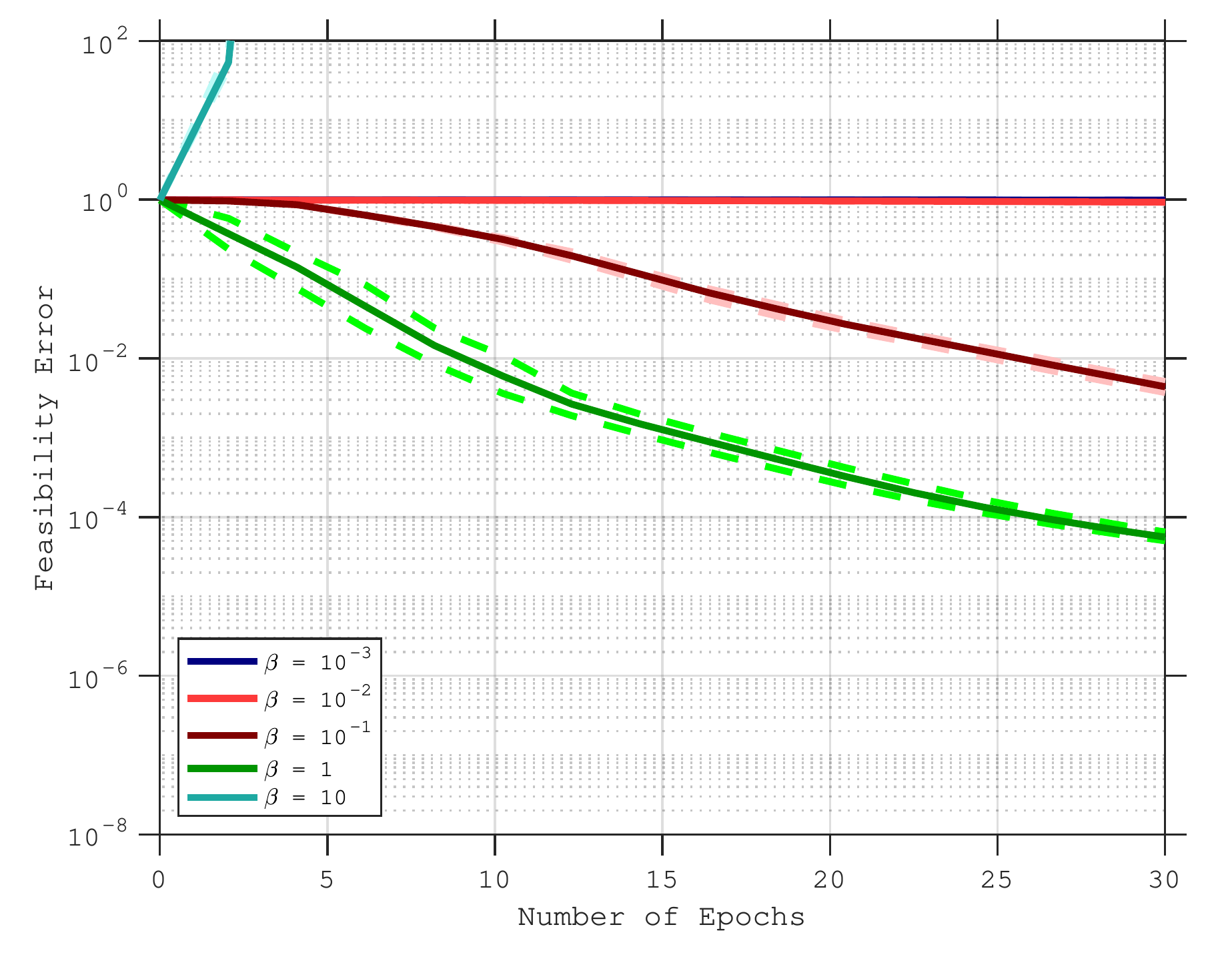}
  
  \includegraphics[width=0.24\textwidth,clip=true,trim=30 180 50 200]{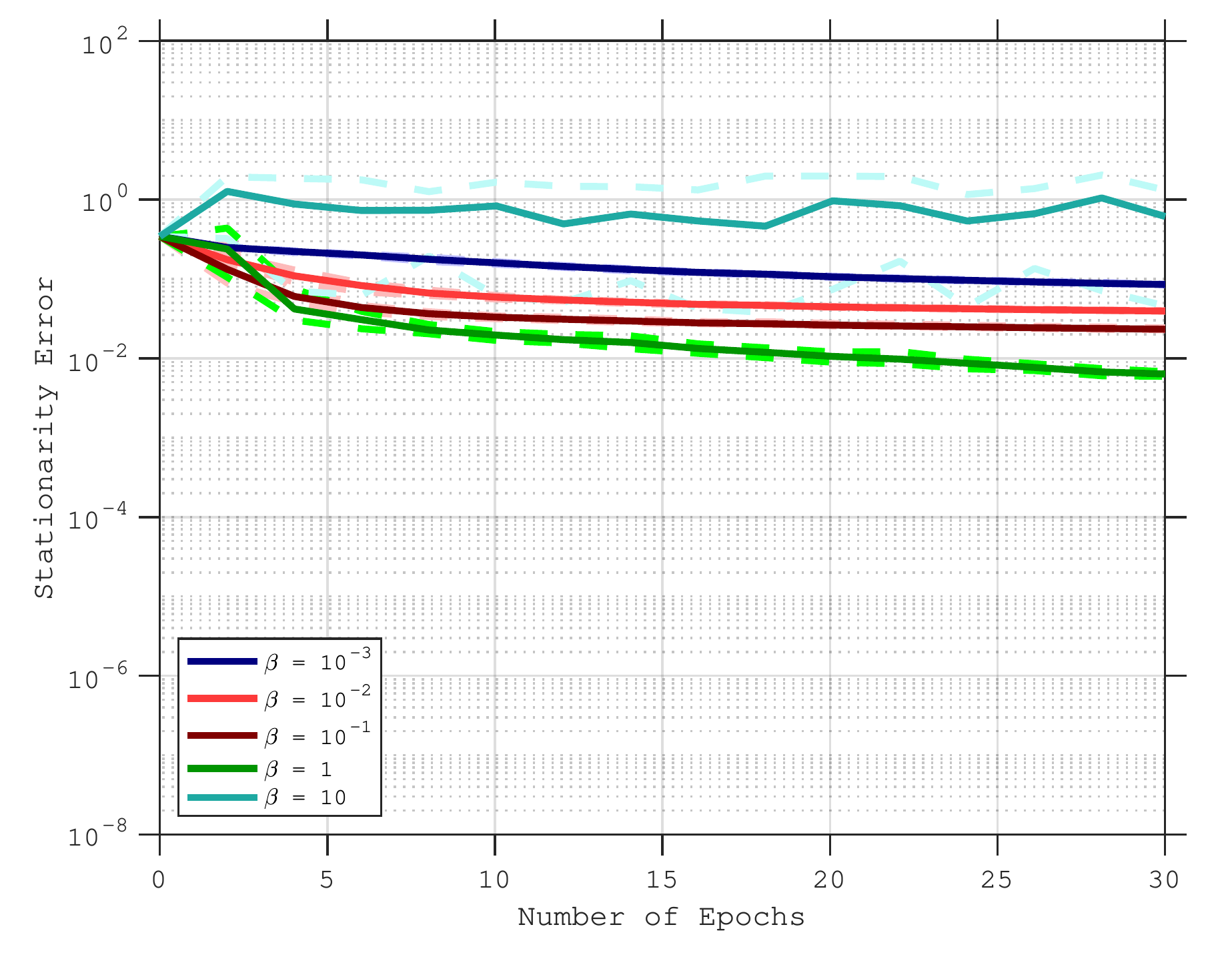}
  \includegraphics[width=0.24\textwidth,clip=true,trim=30 180 50 200]{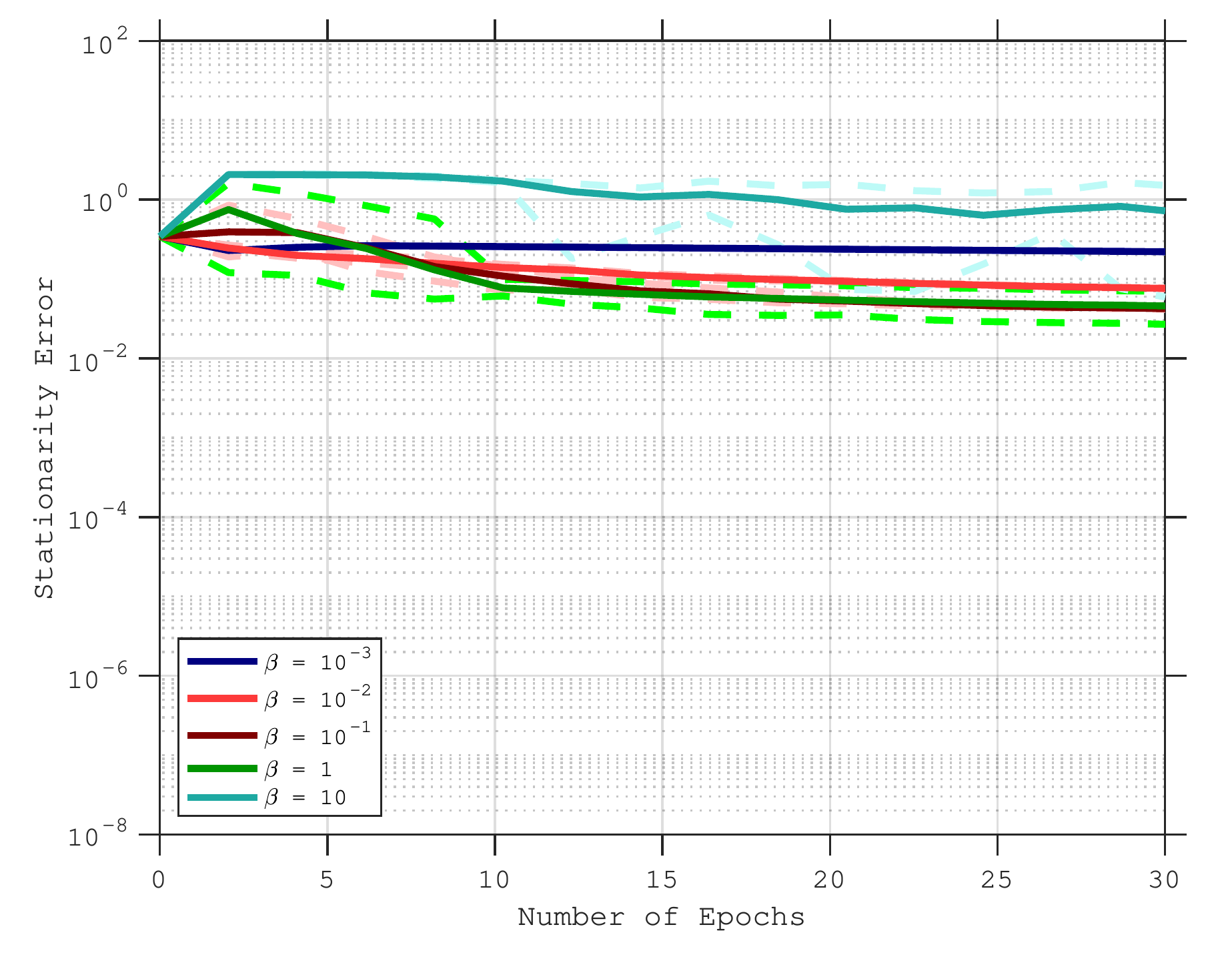}
  \includegraphics[width=0.24\textwidth,clip=true,trim=30 180 50 200]{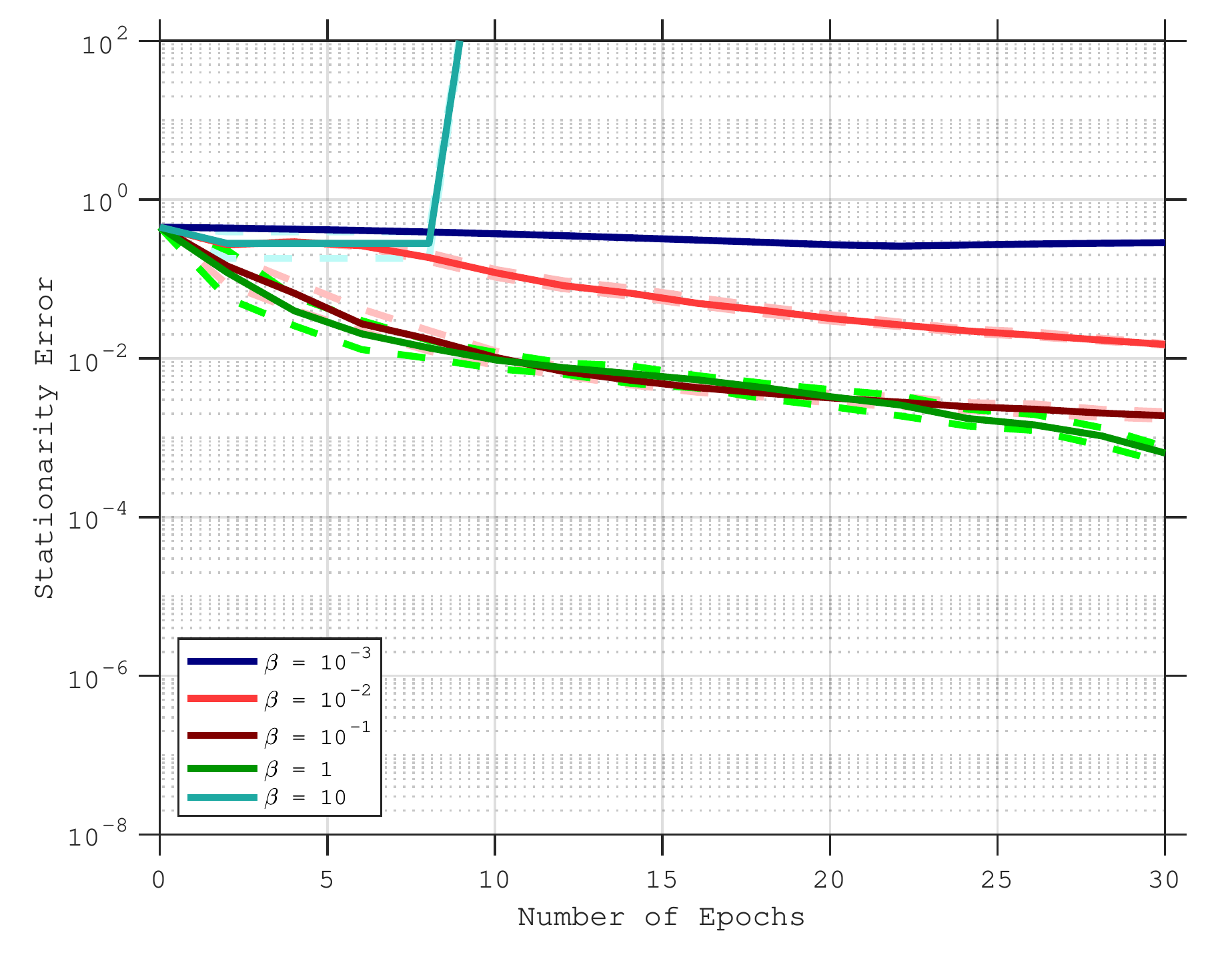}
  \includegraphics[width=0.24\textwidth,clip=true,trim=30 180 50 200]{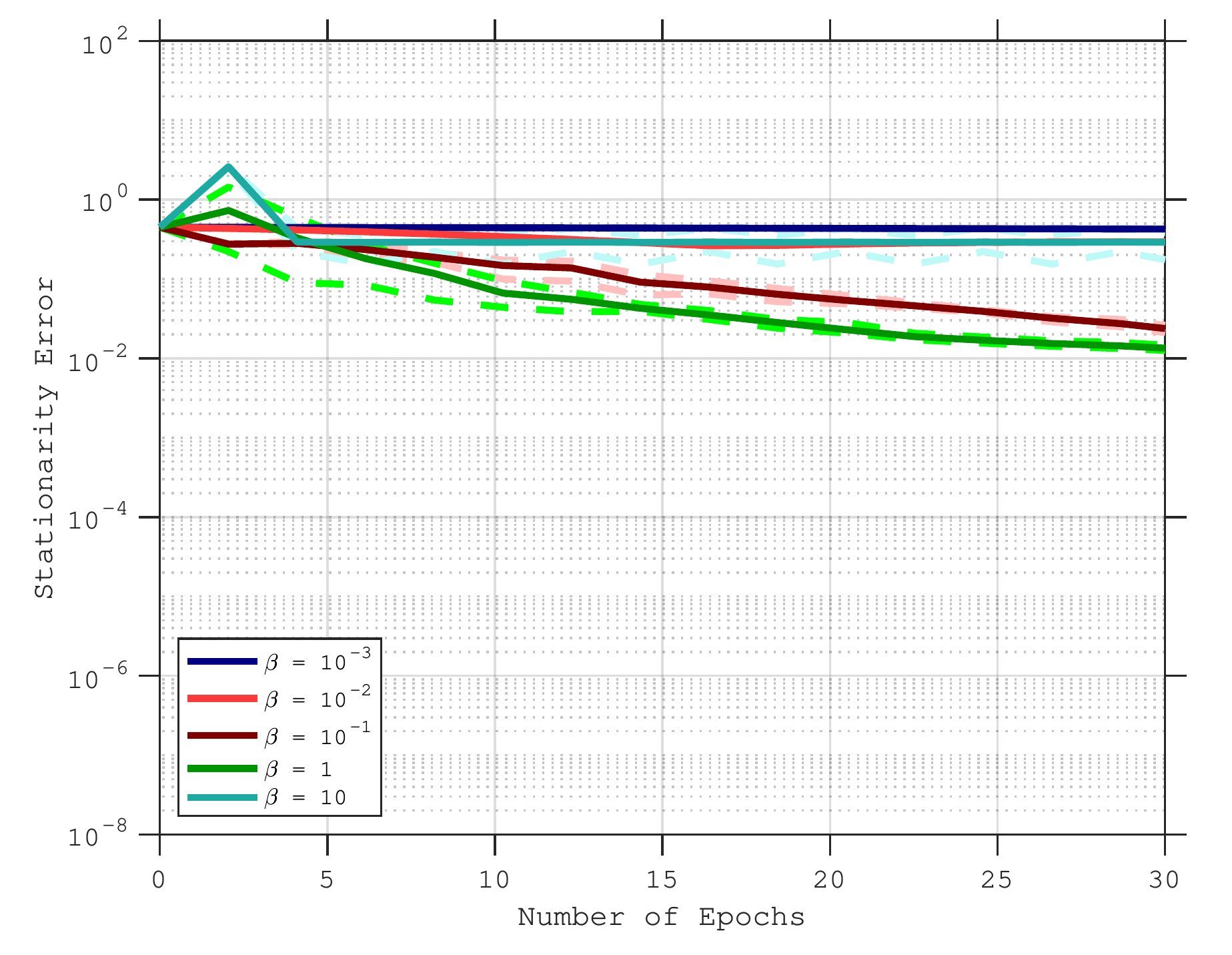}
  \caption{ \texttt{splice} dataset. Top row: feasibility error; Bottom row: stationarity
  error.}
  \label{fig.sensitivity1b}
  \end{subfigure}
  \caption{Performance of \SVRSQPADAPT{} with different step sizes parameter values $\beta \in \{ 10^{-3},10^{-2},10^{-1},10^{0},10^{1}\}$ on logistic regression problems with linear (columns 1 and 2) and $\ell_2$ norm (columns 3 and 4) constraints. First and third columns: batch size 16; Second and fourth columns: batch size 128. \label{fig.sensitivity}}
\end{figure}

\begin{figure}[ht]
   \centering
     \begin{subfigure}[b]{1\textwidth}
     \includegraphics[width=0.24\textwidth,clip=true,trim=30 180 50 200]{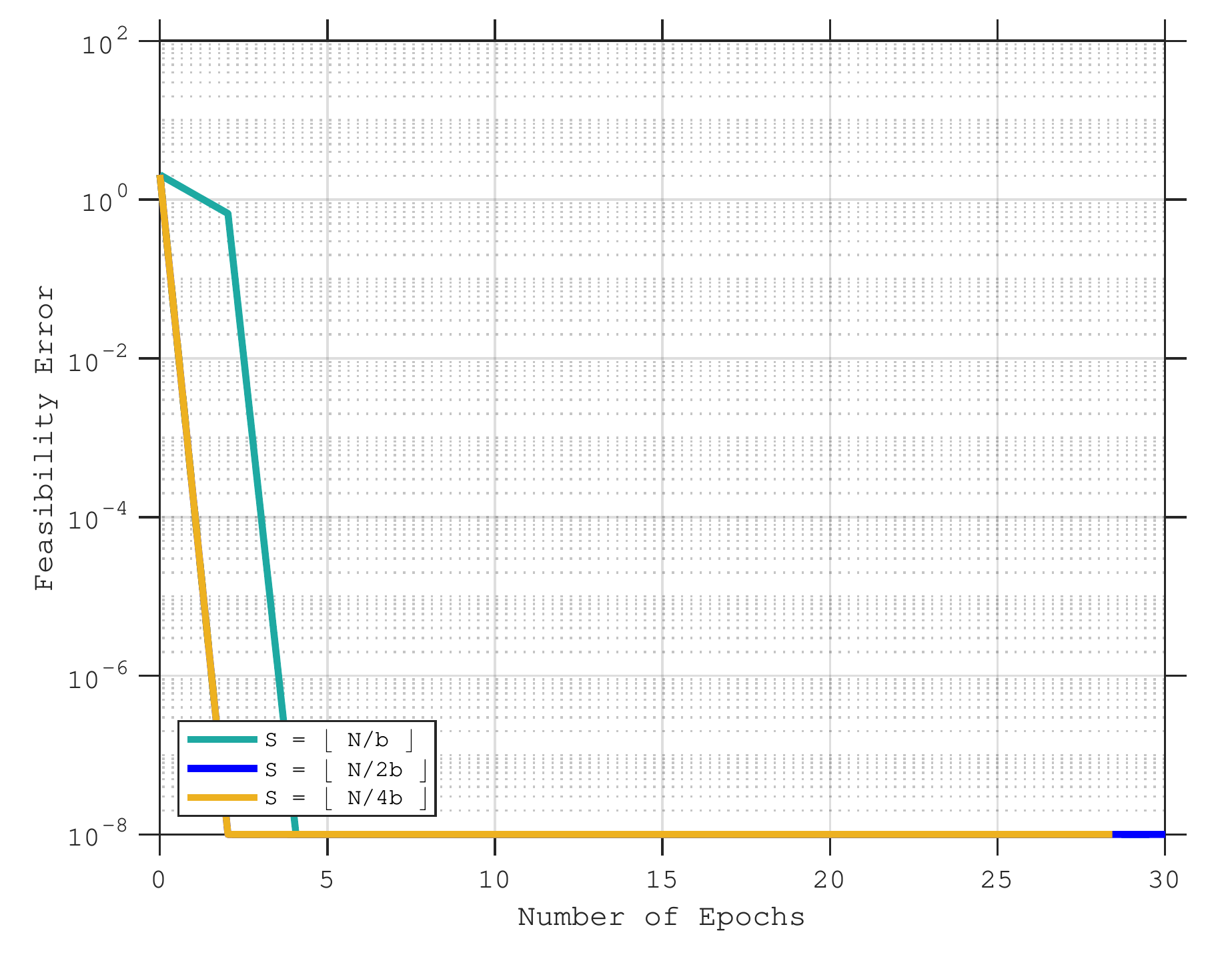}
  \includegraphics[width=0.24\textwidth,clip=true,trim=30 180 50 200]{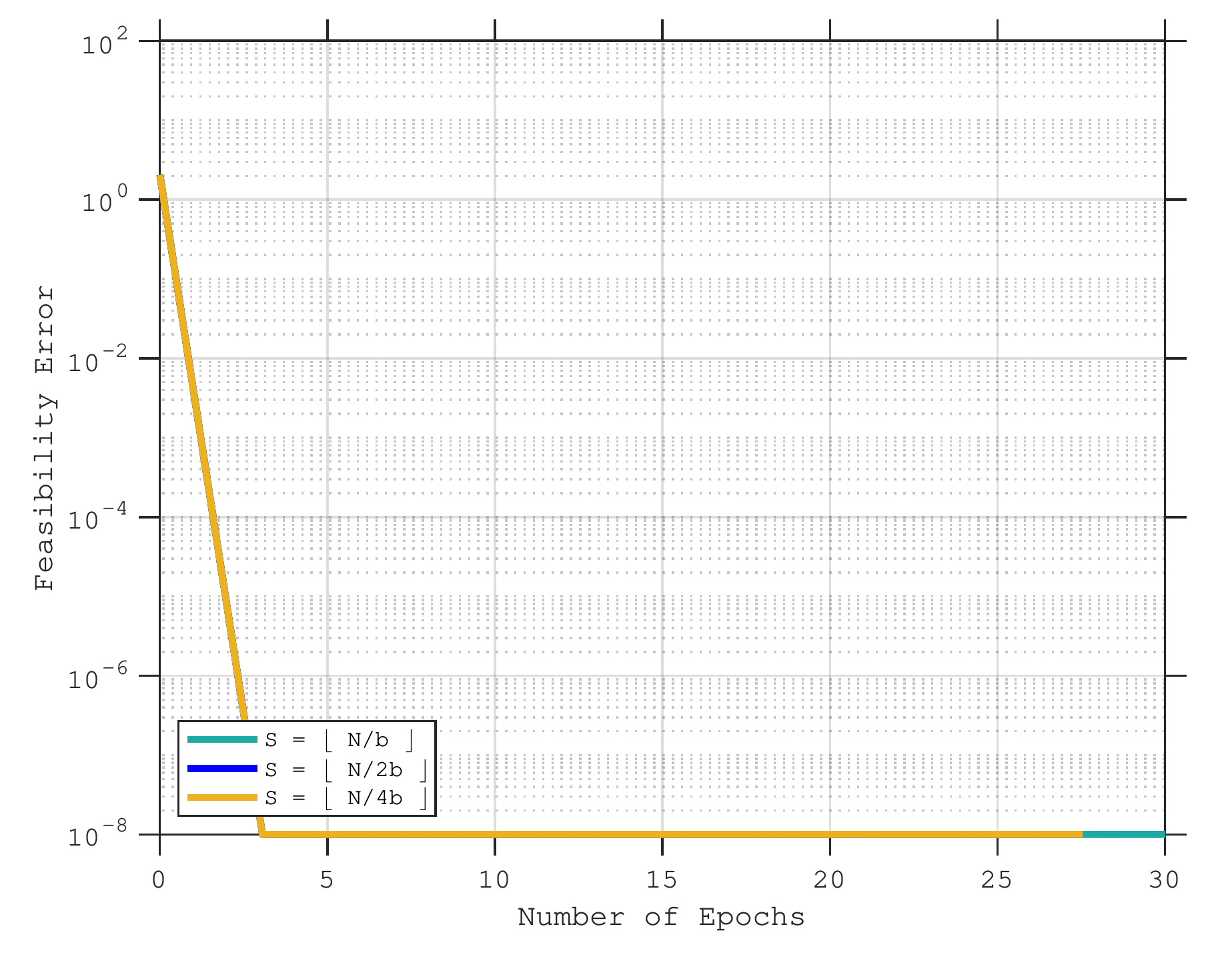}
  \includegraphics[width=0.24\textwidth,clip=true,trim=30 180 50 200]{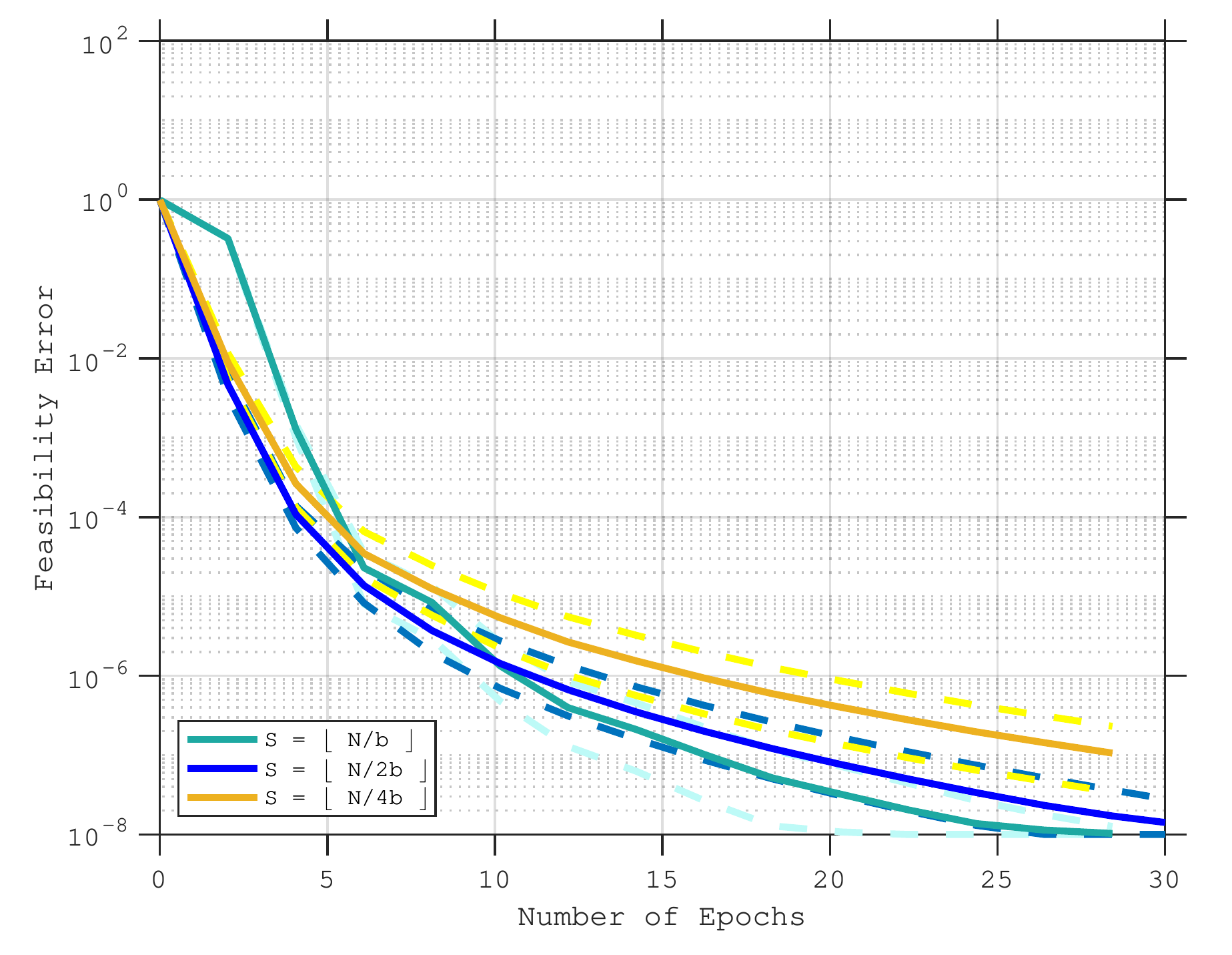}
  \includegraphics[width=0.24\textwidth,clip=true,trim=30 180 50 200]{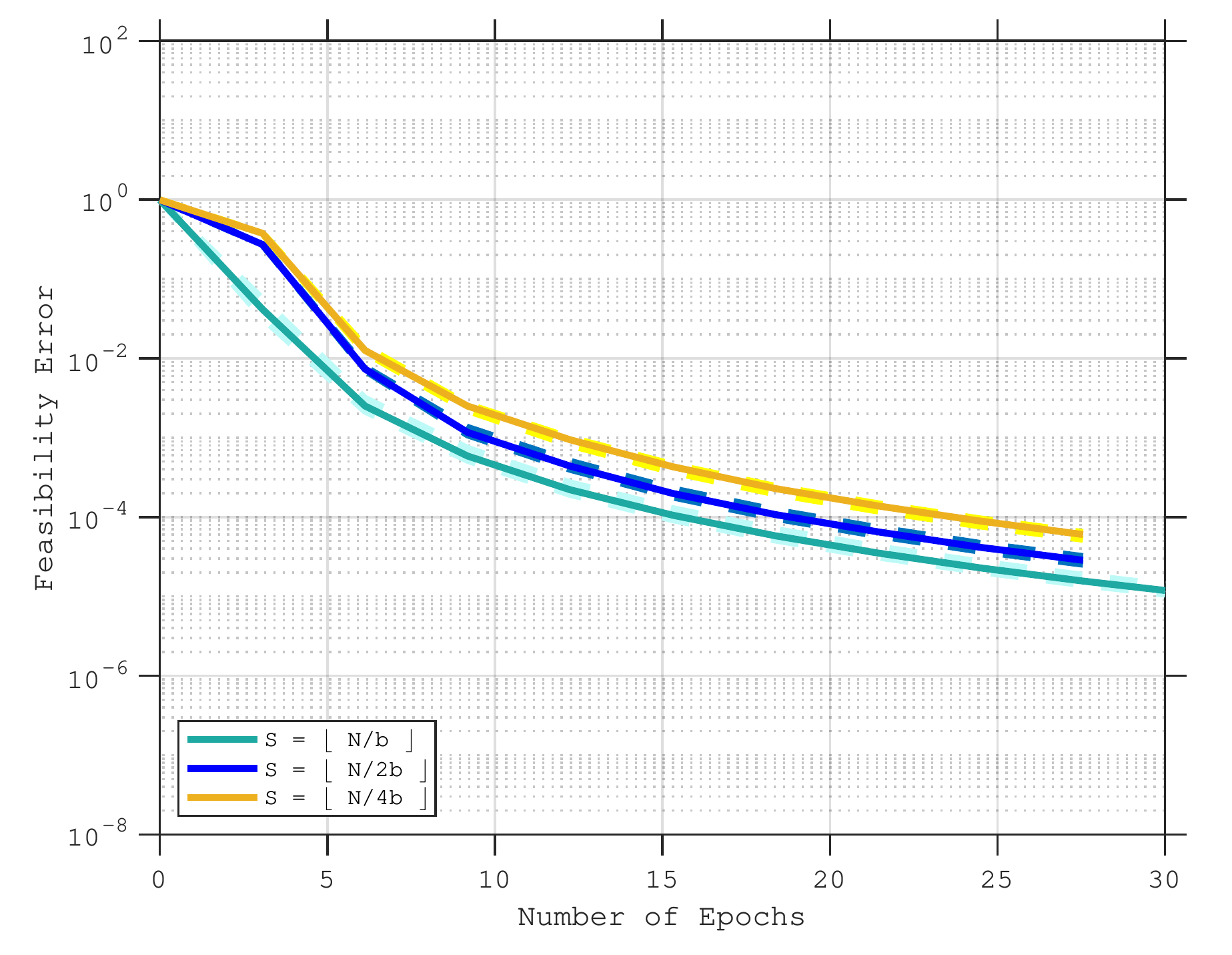}
  
  \includegraphics[width=0.24\textwidth,clip=true,trim=30 180 50 200]{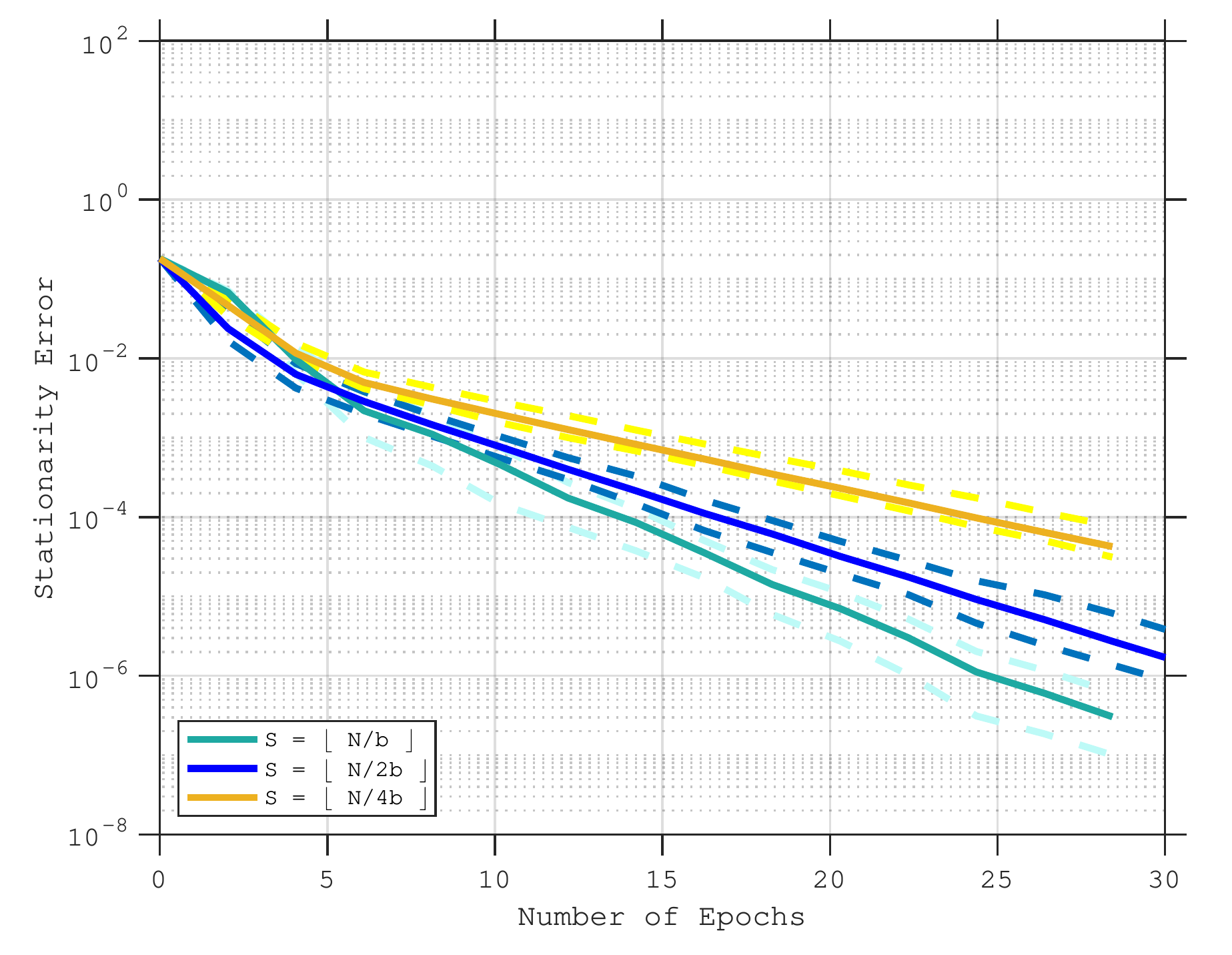}
  \includegraphics[width=0.24\textwidth,clip=true,trim=30 180 50 200]{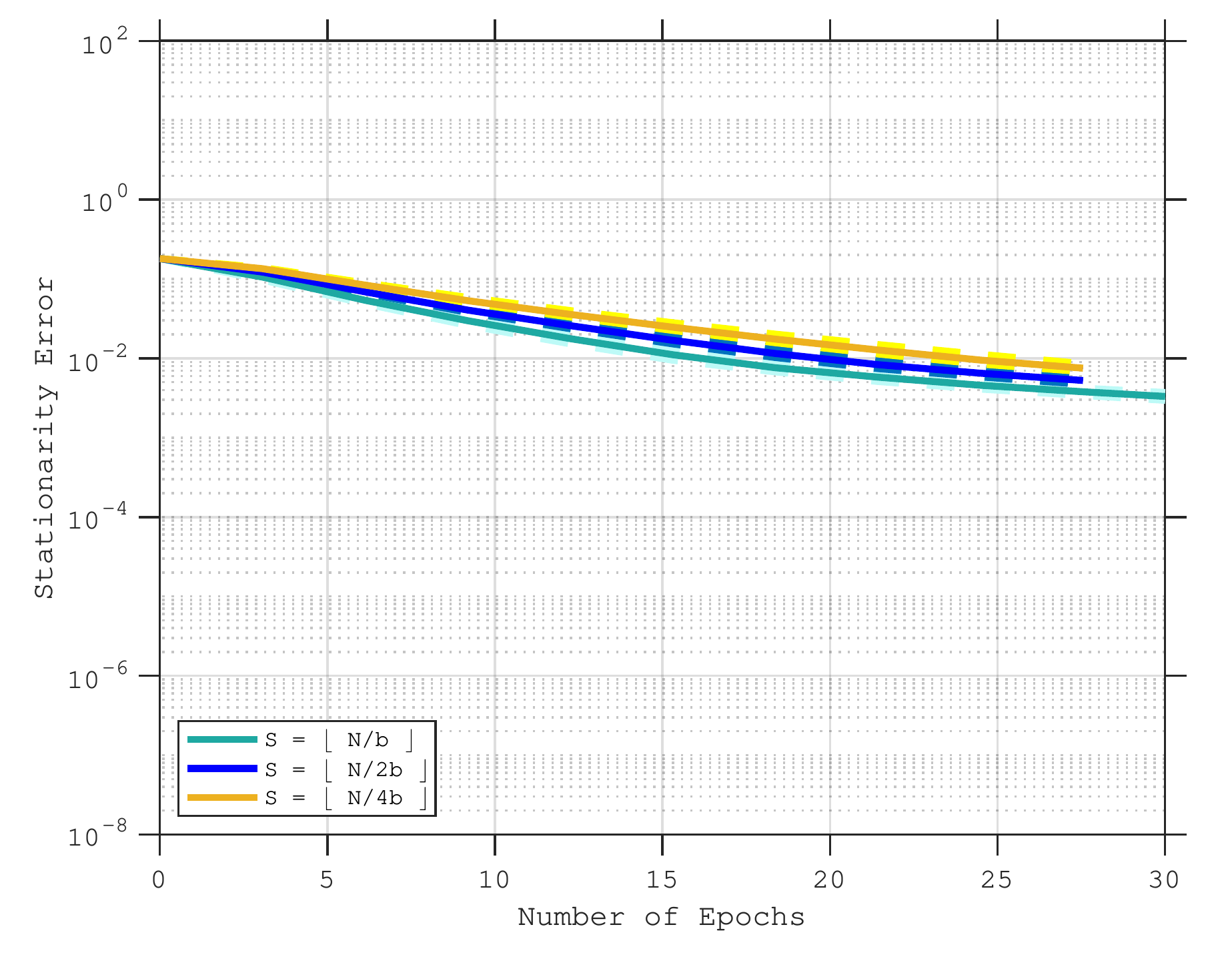}
  \includegraphics[width=0.24\textwidth,clip=true,trim=30 180 50 200]{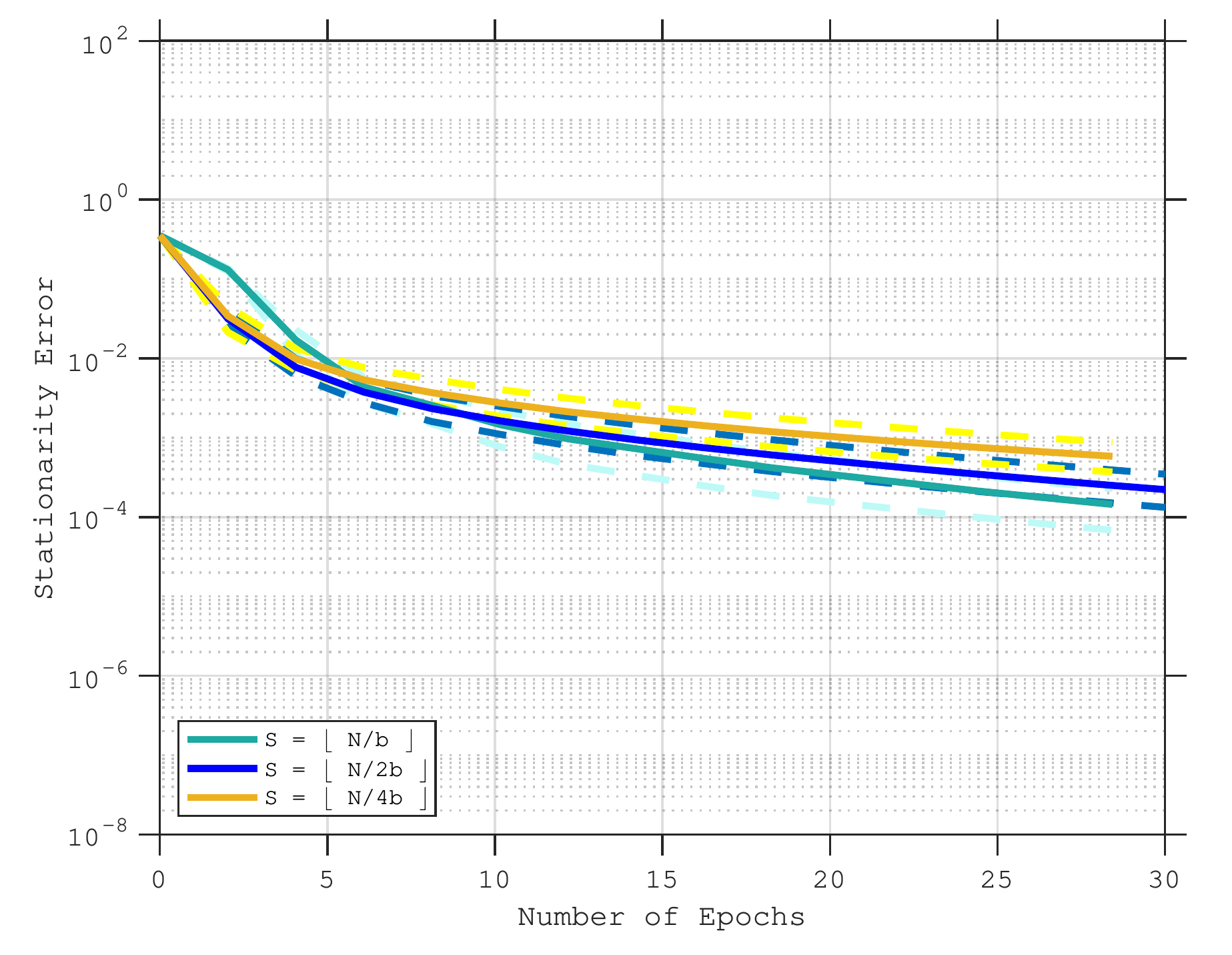}
  \includegraphics[width=0.24\textwidth,clip=true,trim=30 180 50 200]{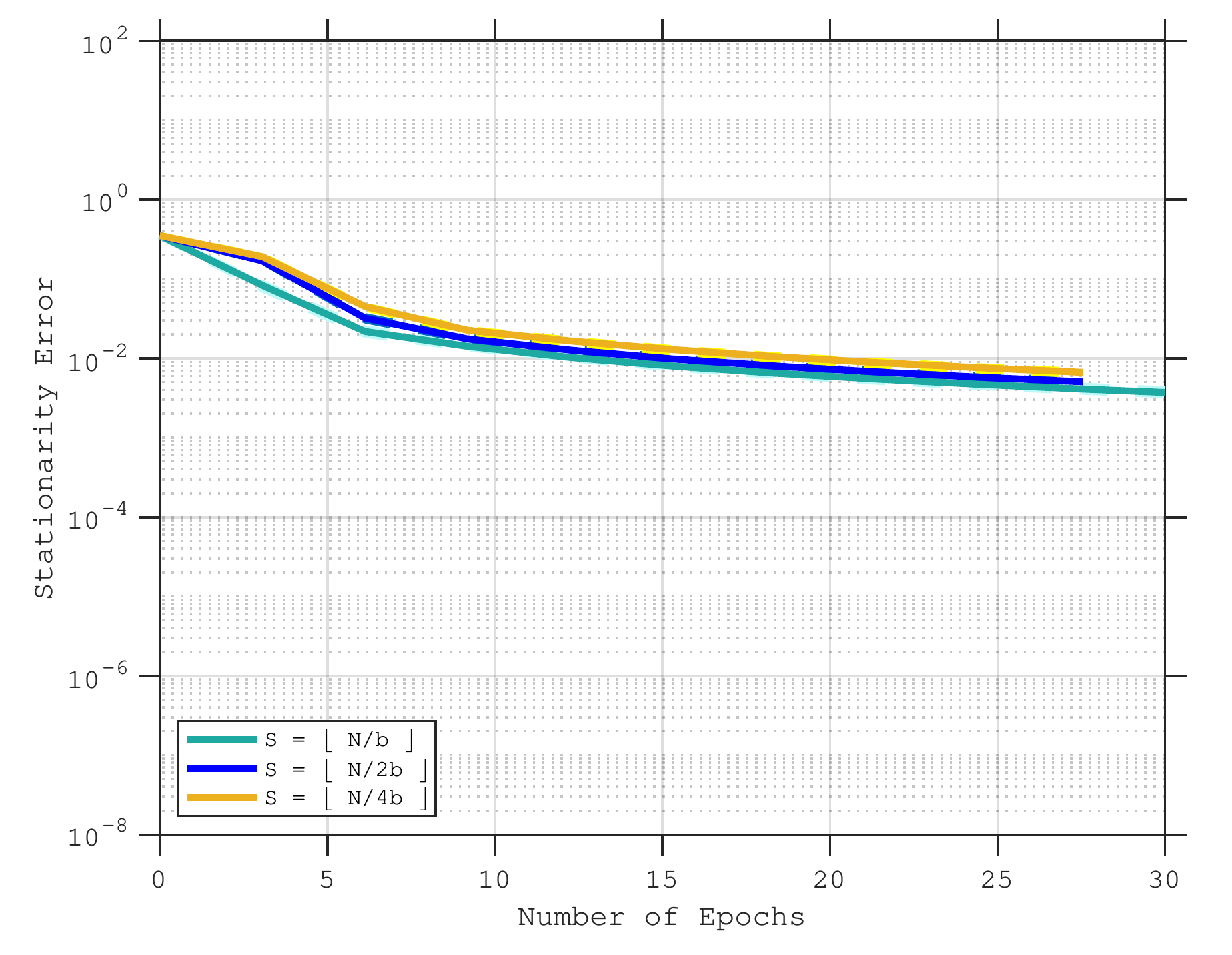}
  \caption{\texttt{australian} dataset. Top row: feasibility error; Bottom row: stationarity
  error.}
  \label{fig.sensitivity2a}
    \end{subfigure}

  \begin{subfigure}[b]{\textwidth}
  \includegraphics[width=0.24\textwidth,clip=true,trim=30 180 50 200]{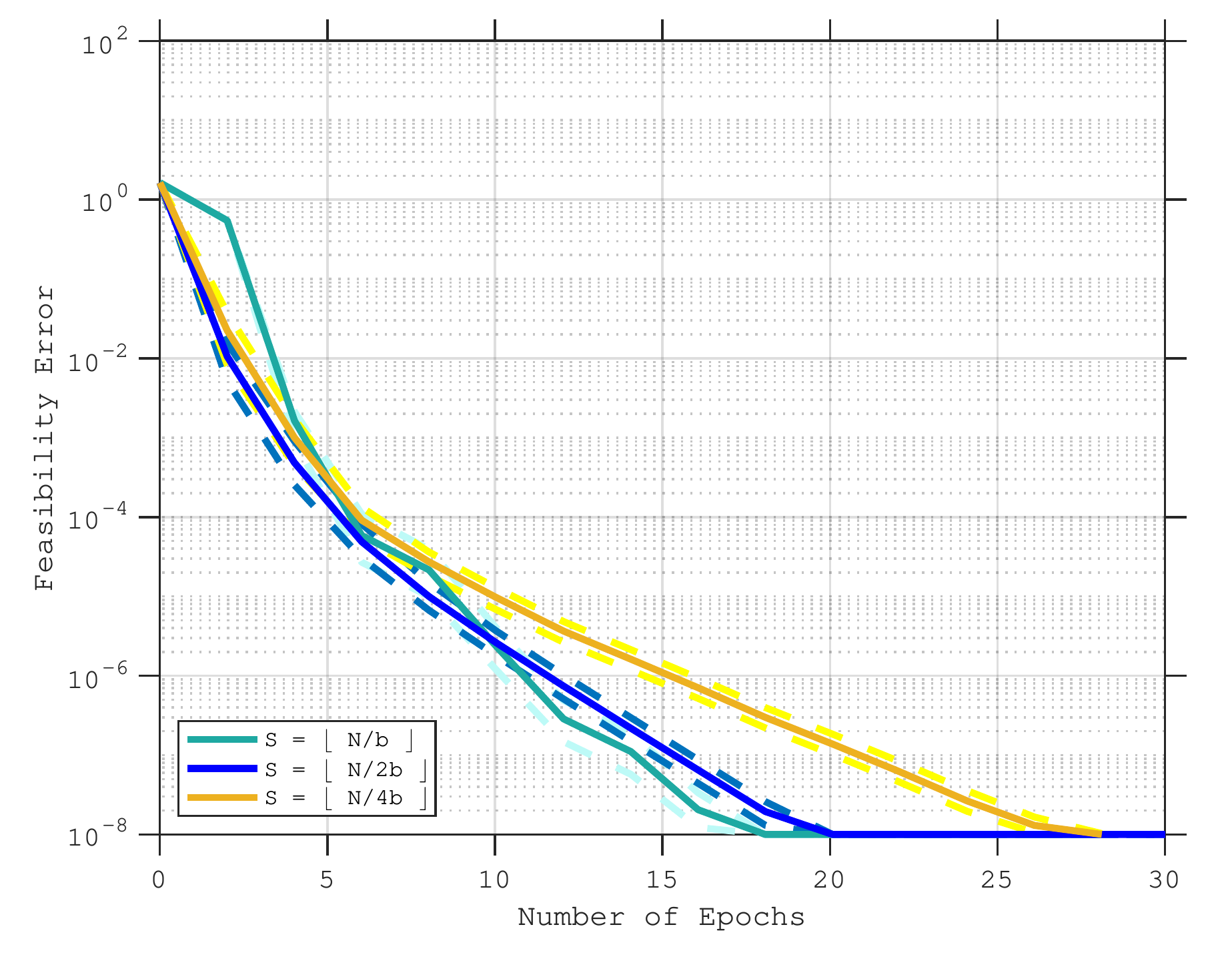}
  \includegraphics[width=0.24\textwidth,clip=true,trim=30 180 50 200]{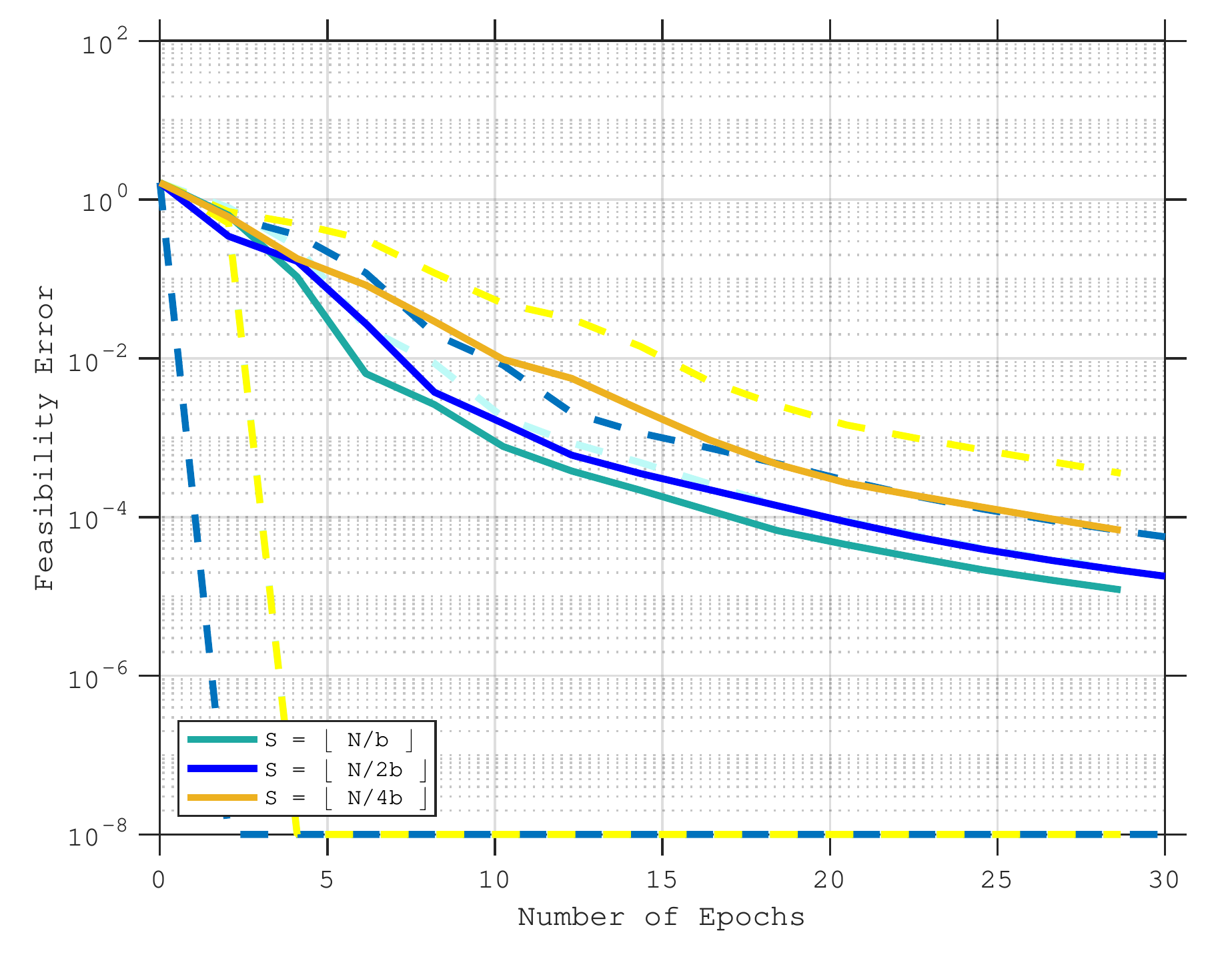}
  \includegraphics[width=0.24\textwidth,clip=true,trim=30 180 50 200]{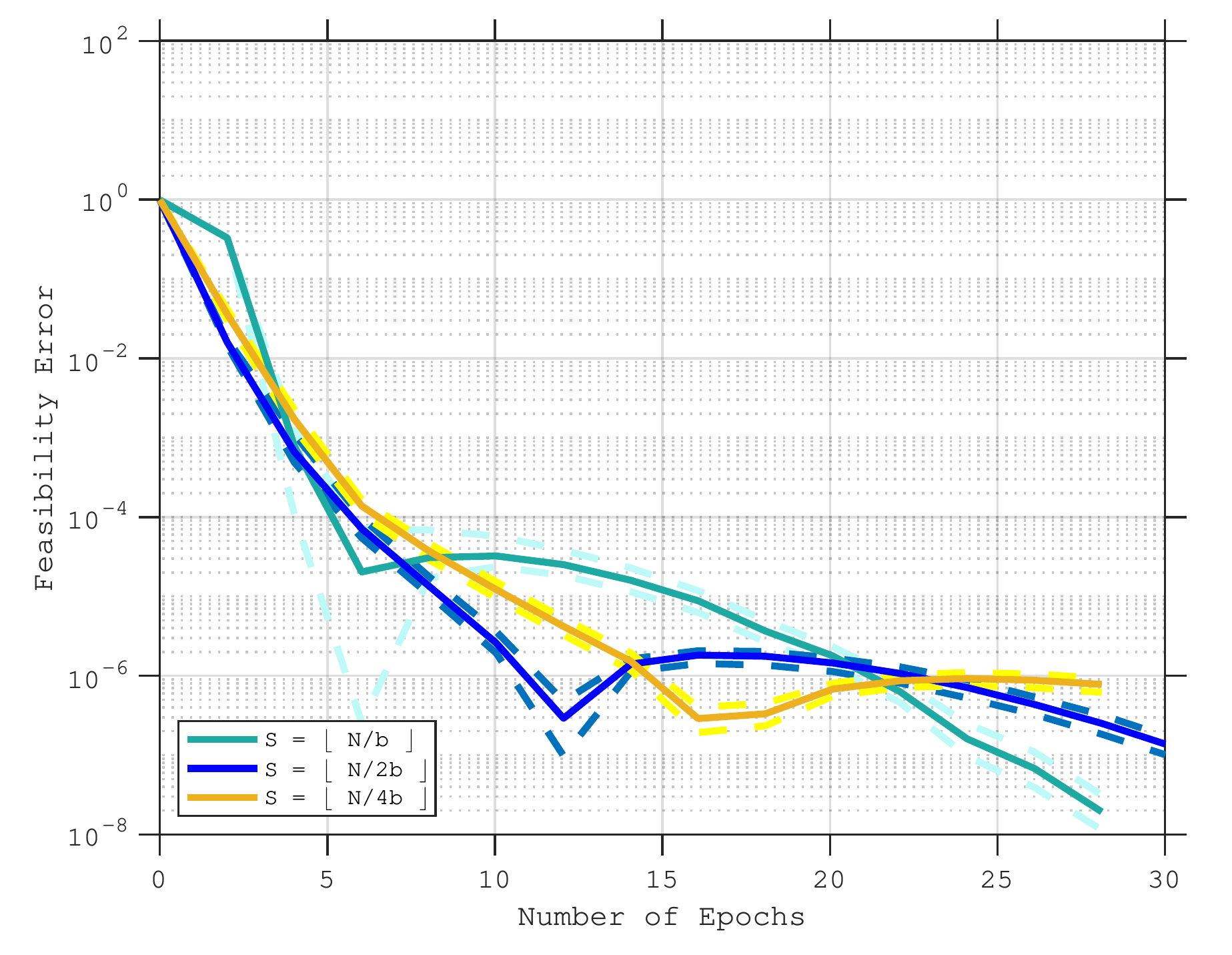}
  \includegraphics[width=0.24\textwidth,clip=true,trim=30 180 50 200]{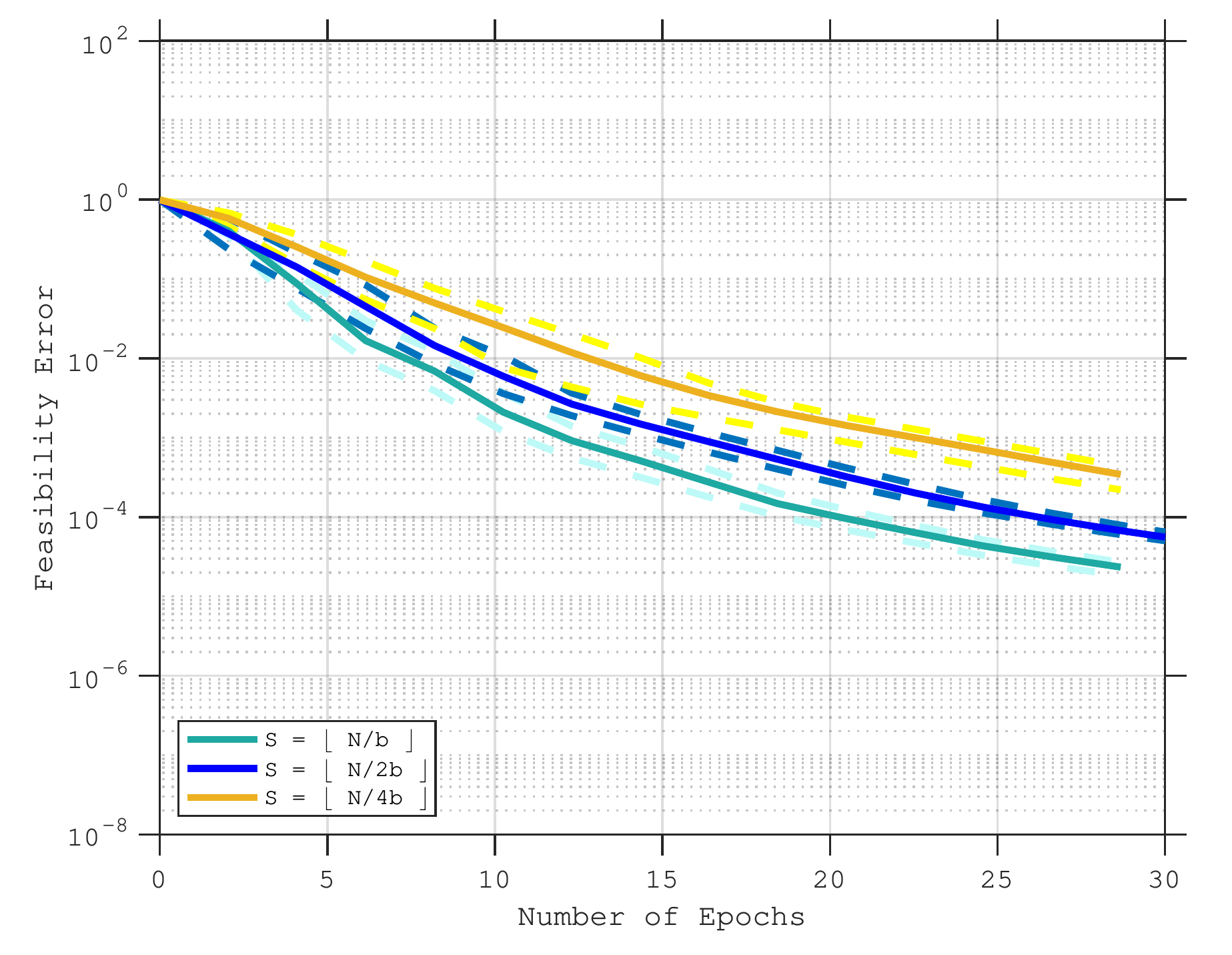}
  
  \includegraphics[width=0.24\textwidth,clip=true,trim=30 180 50 200]{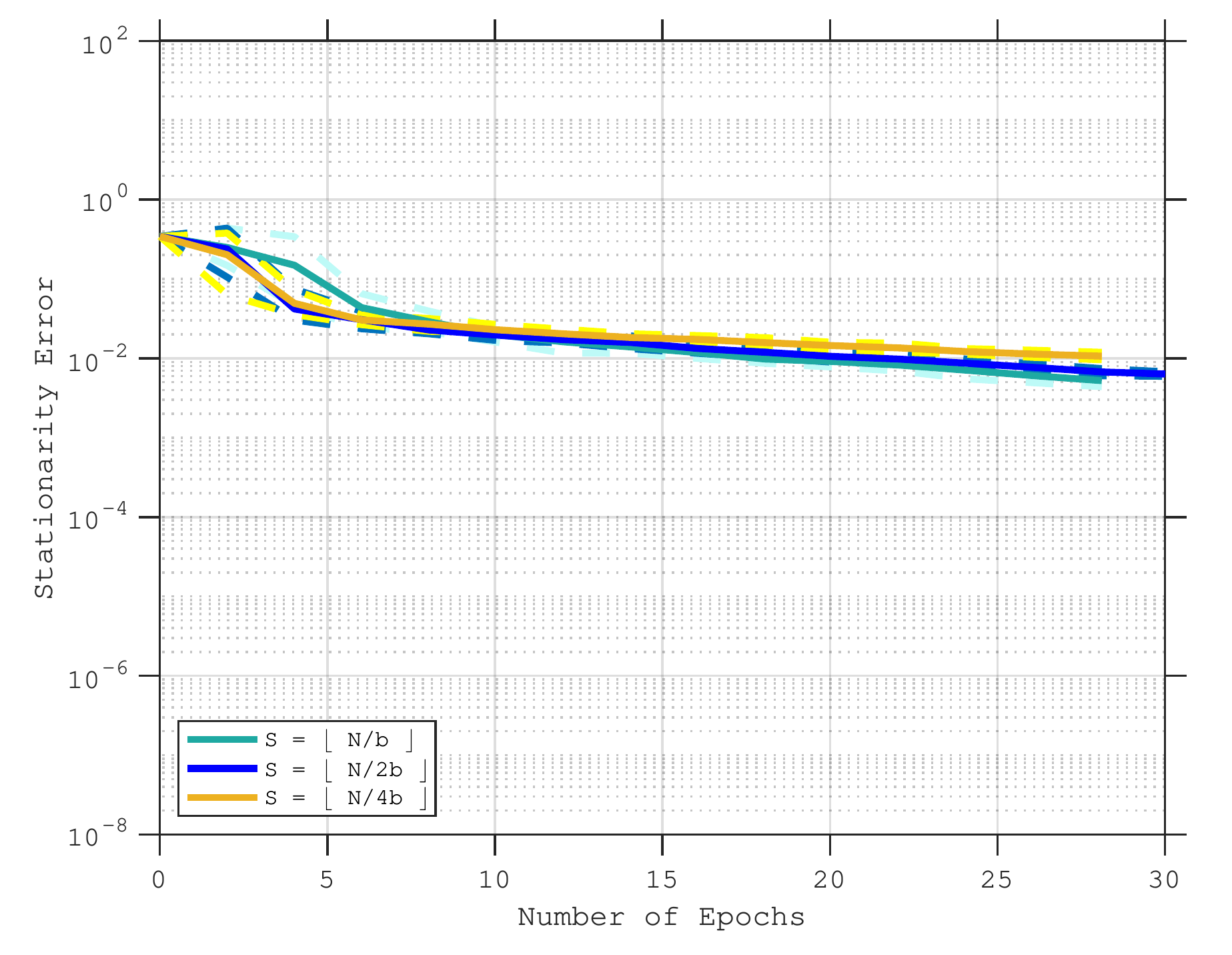}
  \includegraphics[width=0.24\textwidth,clip=true,trim=30 180 50 200]{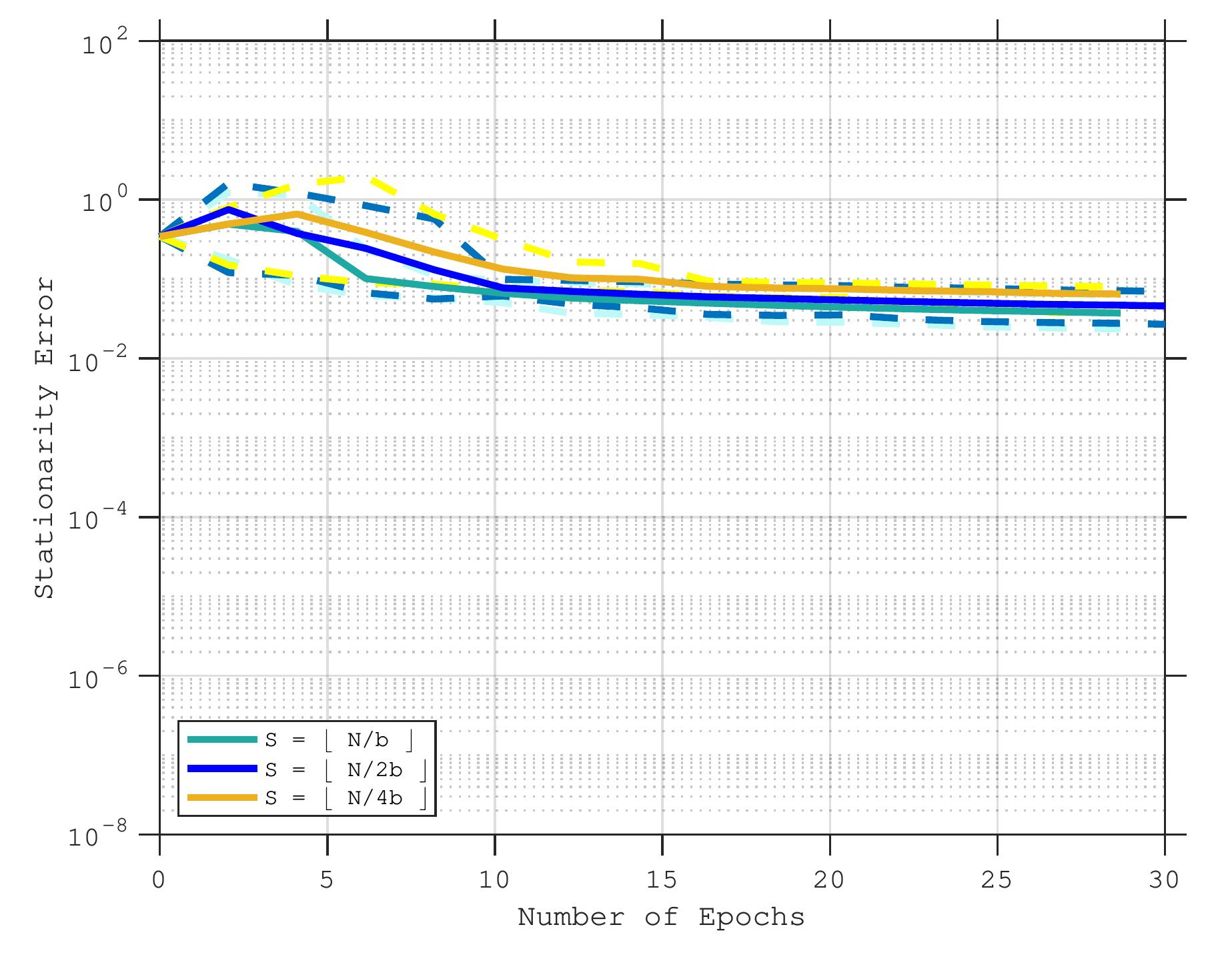}
  \includegraphics[width=0.24\textwidth,clip=true,trim=30 180 50 200]{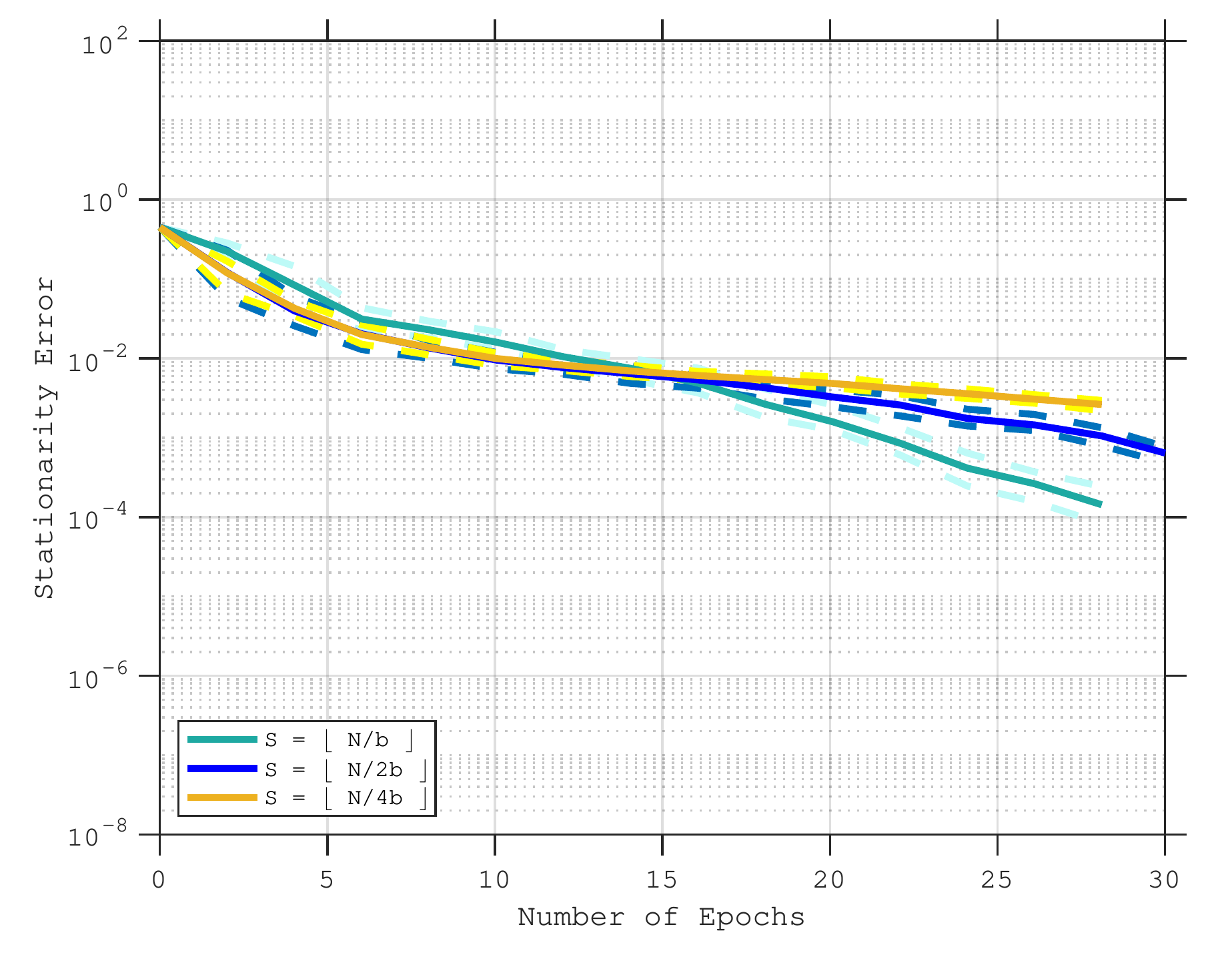}
  \includegraphics[width=0.24\textwidth,clip=true,trim=30 180 50 200]{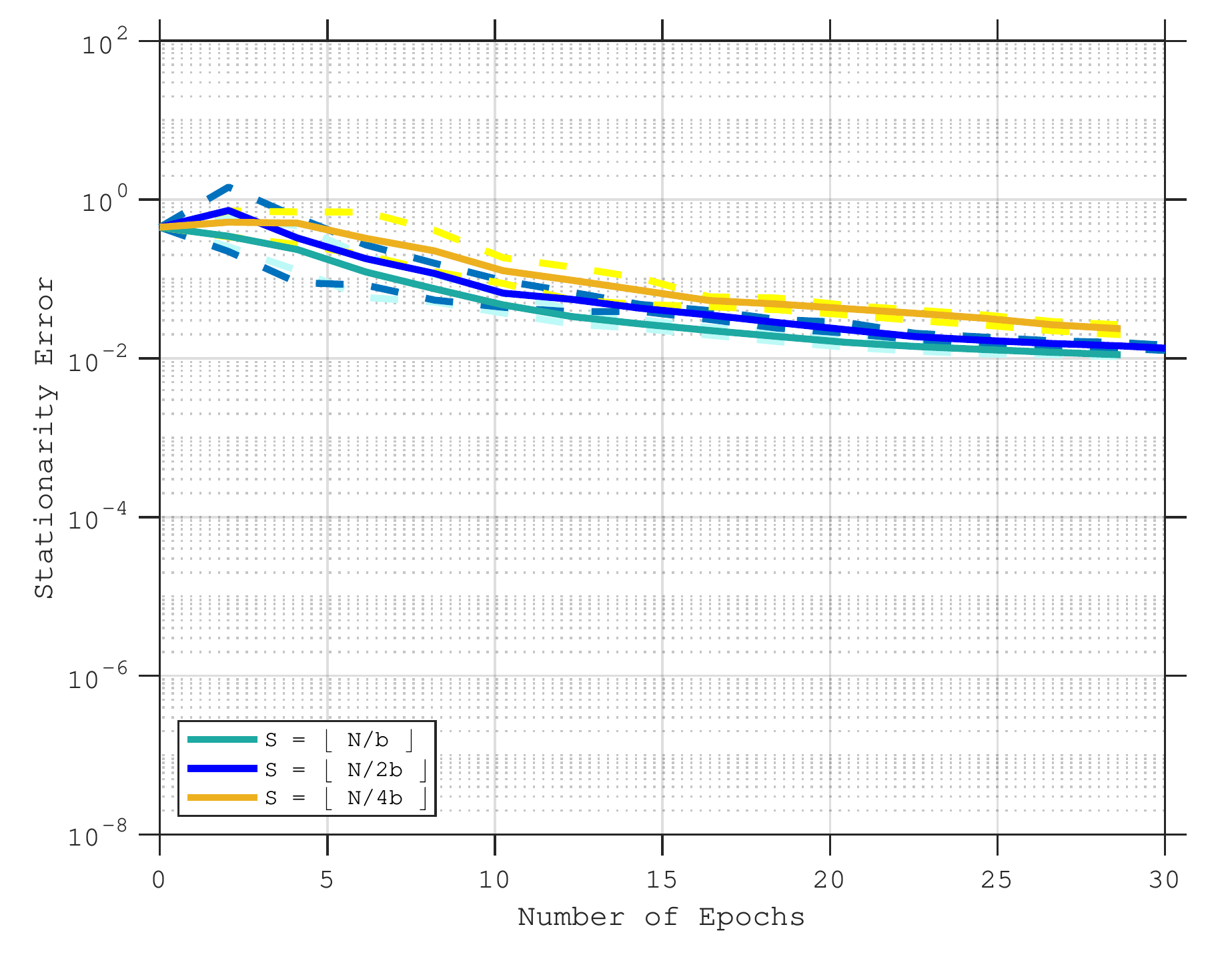}
  \caption{ \texttt{splice} dataset. Top row: feasibility error; Bottom row: stationarity
  error.}
  \label{fig.sensitivity2b}
  \end{subfigure}
  \caption{Performance of \SVRSQPADAPT{} with different step sizes parameter values $S \in \left\{ \left\lfloor \tfrac{N}{b} \right\rfloor,\left\lfloor \tfrac{N}{2b} \right\rfloor,\left\lfloor \tfrac{N}{4b} \right\rfloor\right\}$ on logistic regression problems with linear (columns 1 and 2) and $\ell_2$ norm (columns 3 and 4) constraints. First and third columns: batch size 16; Second and fourth columns: batch size 128. \label{fig.sensitivity2}}
\end{figure}

\subsection{Comparison:  \SVRSQPADAPT{}, \StoSQP{} and \StoSubVR{}}

In this final subsection, we compare the performance of \SVRSQPADAPT{}  to that of \StoSQP{} \cite[Algorithm 2]{berahas2021sequential} and \StoSubVR{}. A budget of 30 epochs was used for all methods. For all methods, the inner iteration length was set to $S = \left\lfloor \tfrac{N}{2b} \right\rfloor$. For the \StoSQP{} method the step size parameter was tuned $\beta \in \{ 10^{-3},10^{-2},10^{-1},10^{0},10^{1}\}$ for all $k\in\mathbb{N}$, and for the \StoSubVR{} method the step size parameter and the merit parameter were tuned $\bar\alpha_{k,s} = \tfrac{\alpha}{\tau L + \Gamma}$ for all $(k,s) \in \mathbb{N} \times \left[\bar{S}\right]$ where  $\alpha\in\{10^{-3},10^{-2},10^{-1},10^{0},10^{1}\}$, $\tau_{k,s} = \tau \in \{10^{-10},10^{-9},\dots,10^{0}\}$. For the \SVRSQPADAPT{}, we set $\beta=1$. Overall, this meant that the \StoSQP{}
and \StoSubVR{} methods were effectively run for 5 and 55 times the number of epochs,
respectively, than were allowed for our method.

The results of these experiments are reported in Figs.~\ref{fig.best_1} and \ref{fig.best_2} and in Tables~\ref{tab.best_linear} and \ref{tab.best_el2}. For each batch size and dataset, we report the average feasibility and stationarity errors for the best iterates generated (defined in Section~\ref{sec.problem_evaluation}) for the best hyper-parameter settings for each method in Tables~\ref{tab.best_linear} and \ref{tab.best_el2}. The results suggest that, when small batch sizes are employed (i.e., $b = 16$), \SVRSQPADAPT{} consistently outperforms the other methods for both sets of constraints. When large batch sizes are used (i.e., $b=128$), \SVRSQPADAPT{} is competitive with \StoSQP{}, even though the adaptive step size parameter $\beta$ is well-tuned for \StoSQP{} whereas for \SVRSQPADAPT{} we simply set $\beta = 1$. We should note again that $5$ and $55$ times the tuning effort was allocated to \StoSQP{} and \StoSubVR{}, respectively, as compared to \SVRSQPADAPT{}.

\begin{figure}[ht]
   \centering
     \begin{subfigure}[b]{1\textwidth}
  \includegraphics[width=0.24\textwidth,clip=true,trim=30 180 50 200]{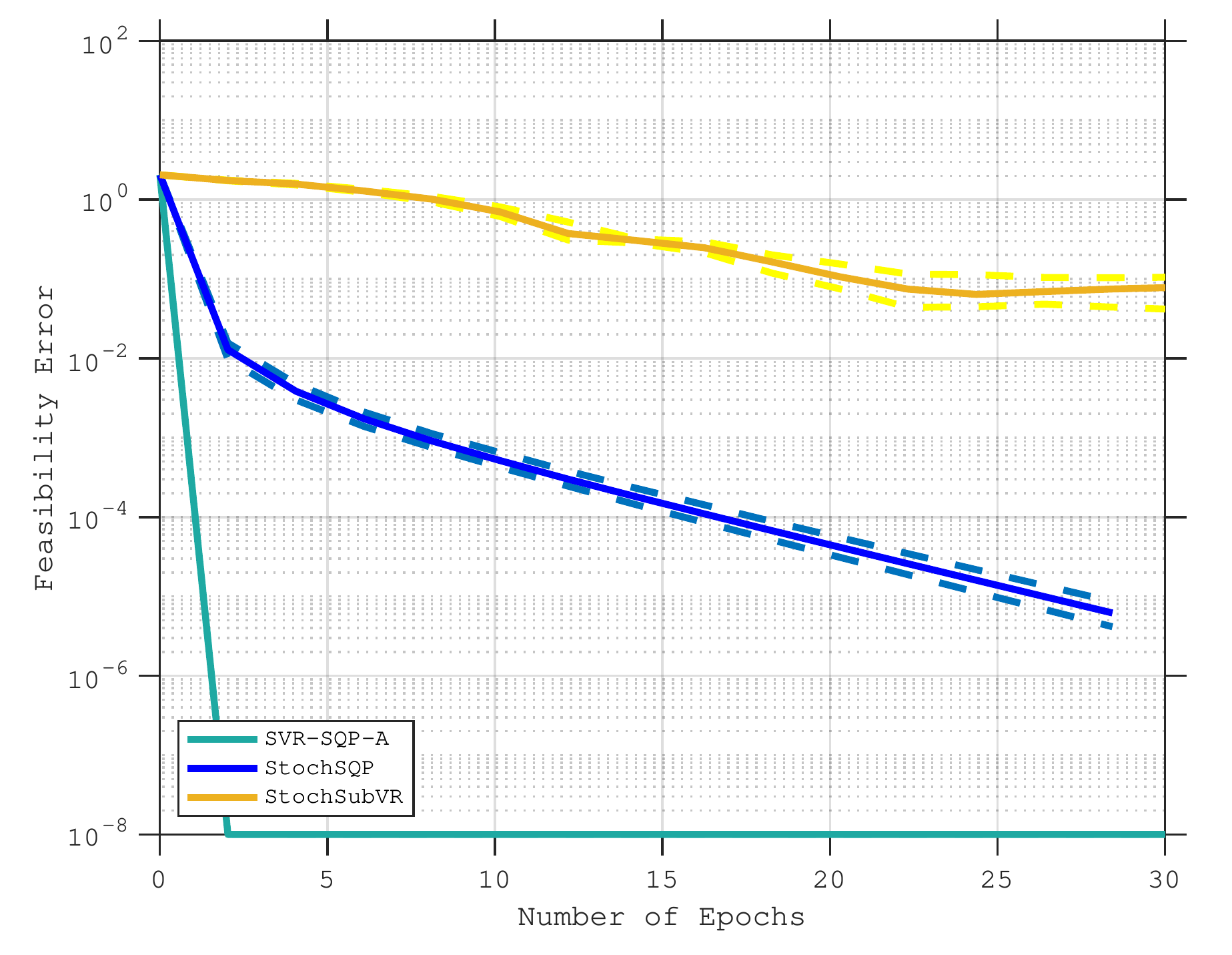}
  \includegraphics[width=0.24\textwidth,clip=true,trim=30 180 50 200]{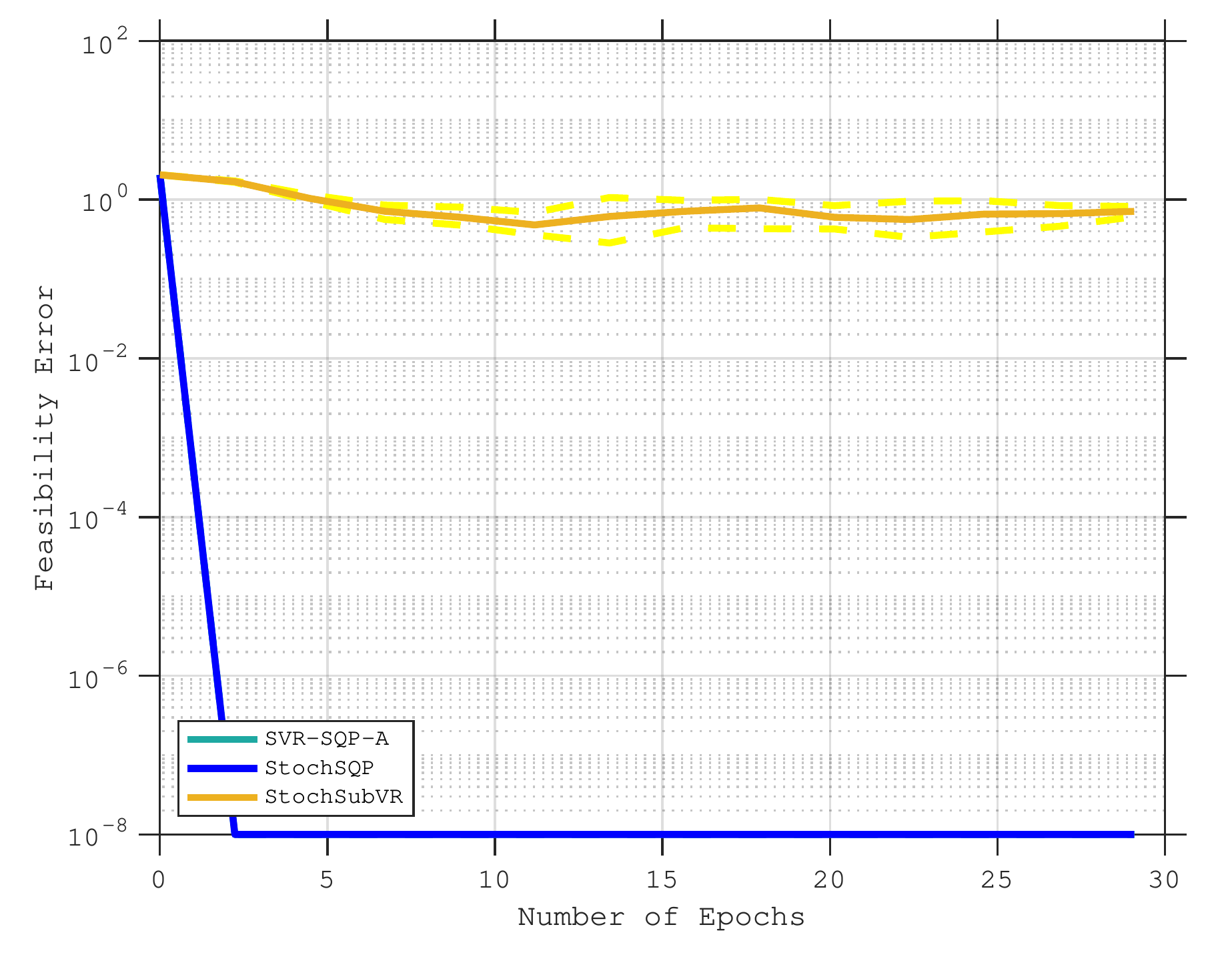}
  \includegraphics[width=0.24\textwidth,clip=true,trim=30 180 50 200]{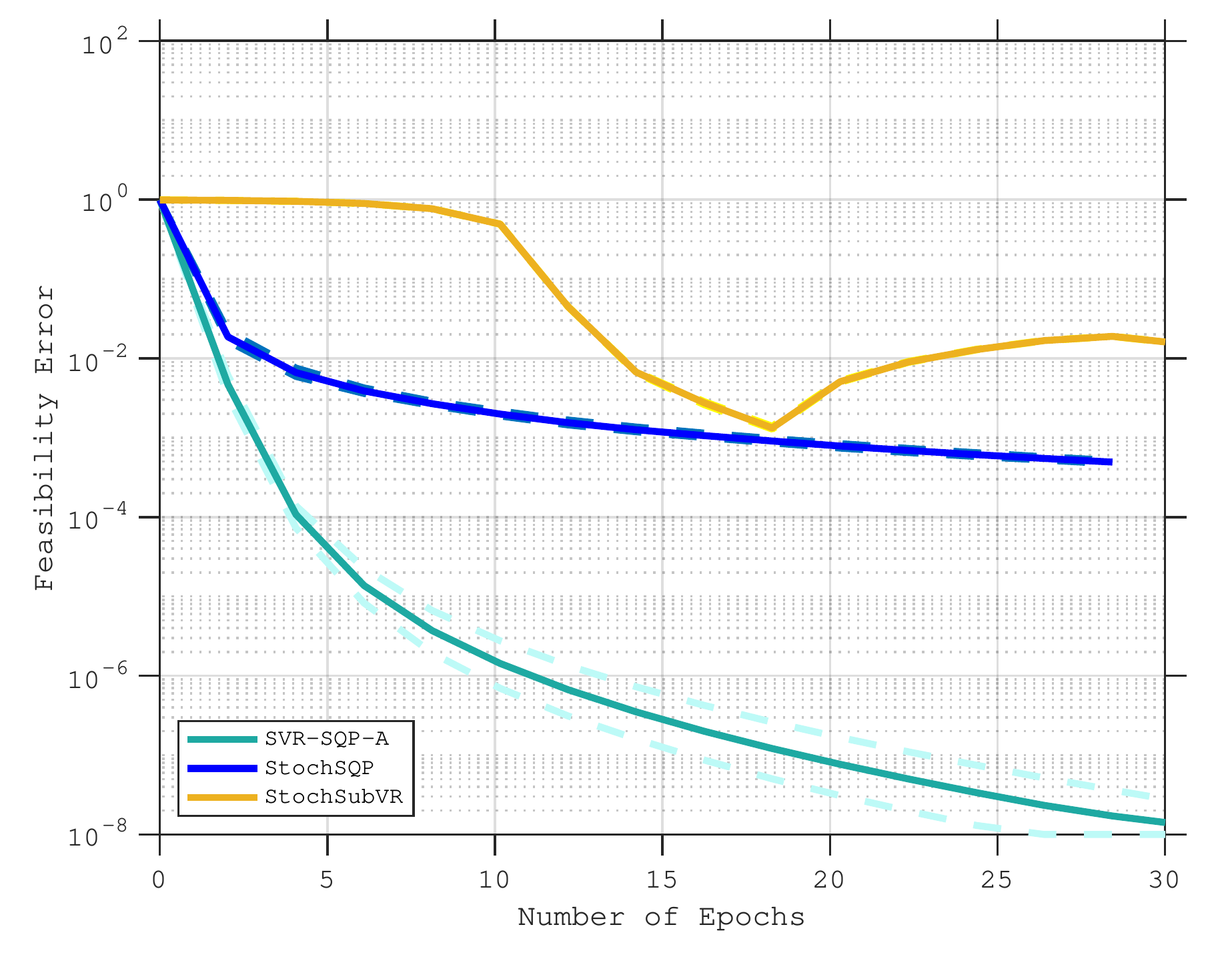}
  \includegraphics[width=0.24\textwidth,clip=true,trim=30 180 50 200]{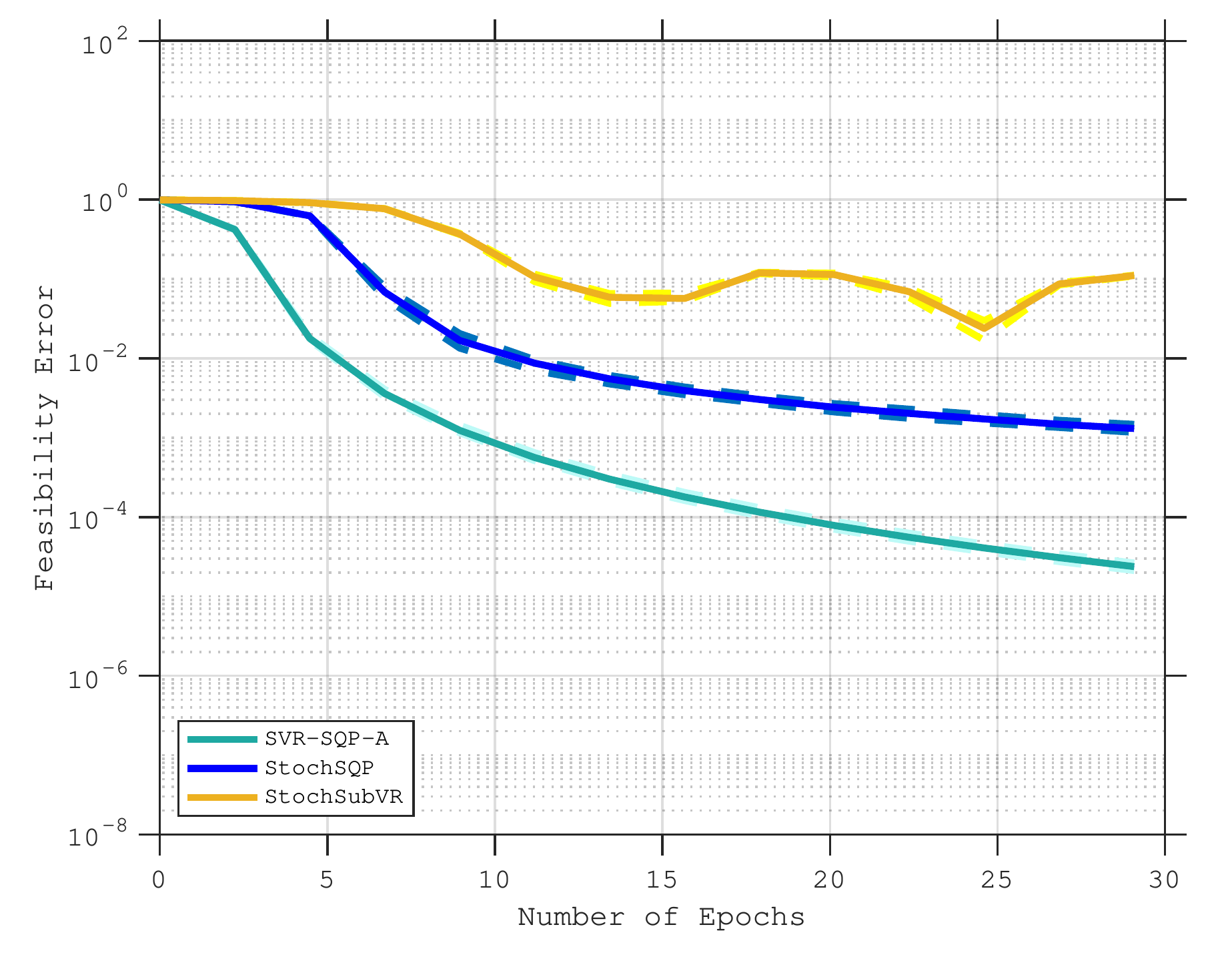}
  
  \includegraphics[width=0.24\textwidth,clip=true,trim=30 180 50 200]{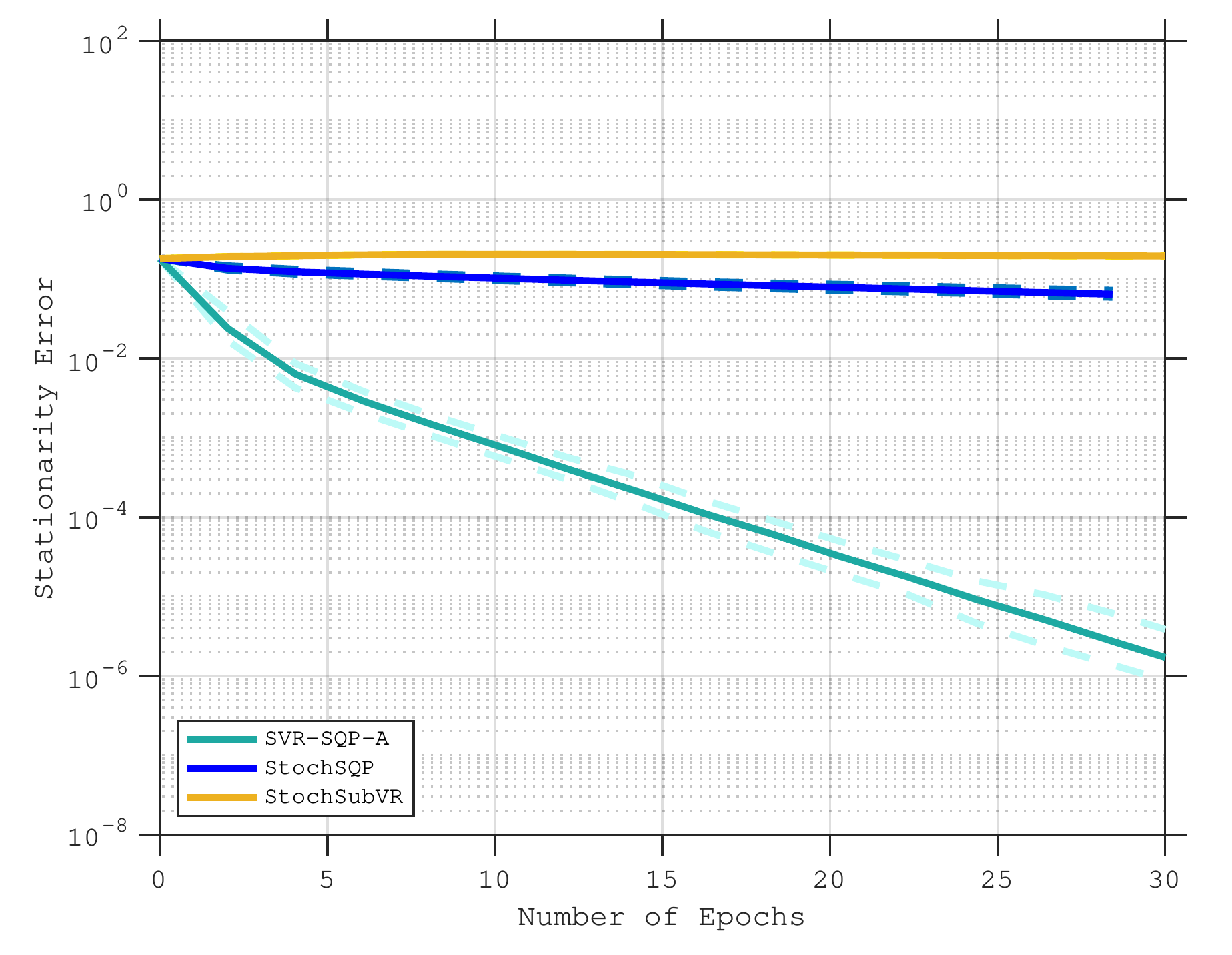}
  \includegraphics[width=0.24\textwidth,clip=true,trim=30 180 50 200]{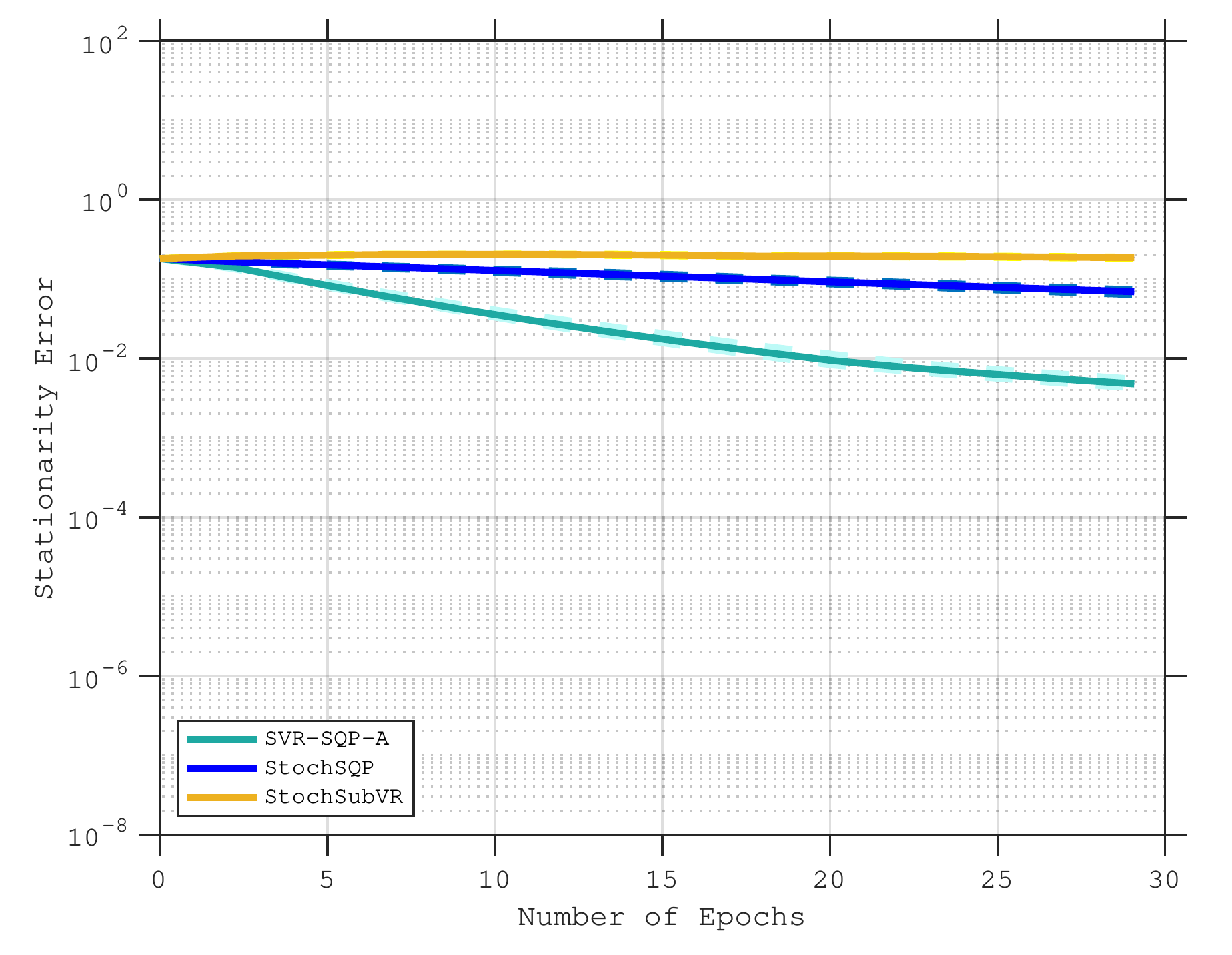}
  \includegraphics[width=0.24\textwidth,clip=true,trim=30 180 50 200]{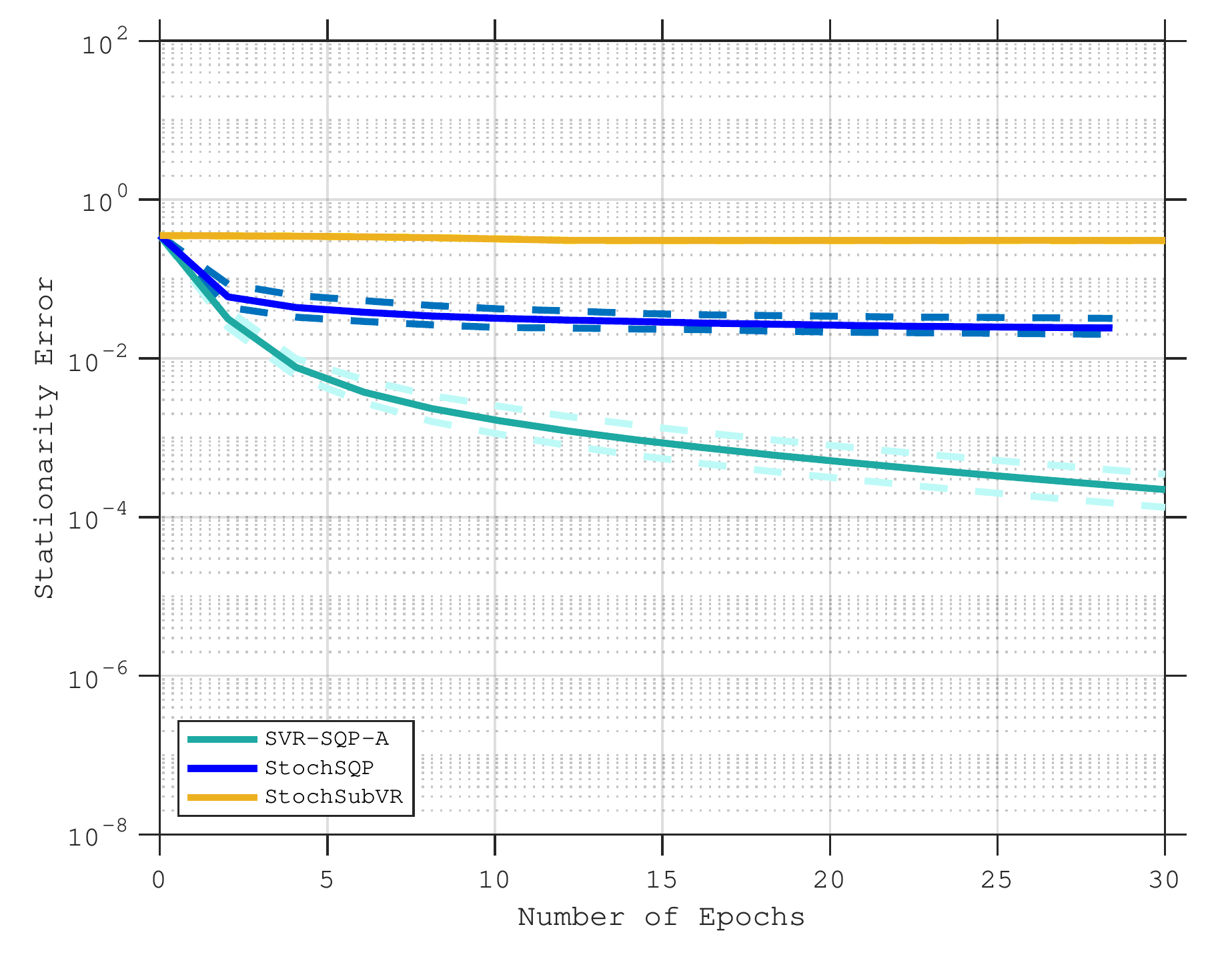}
  \includegraphics[width=0.24\textwidth,clip=true,trim=30 180 50 200]{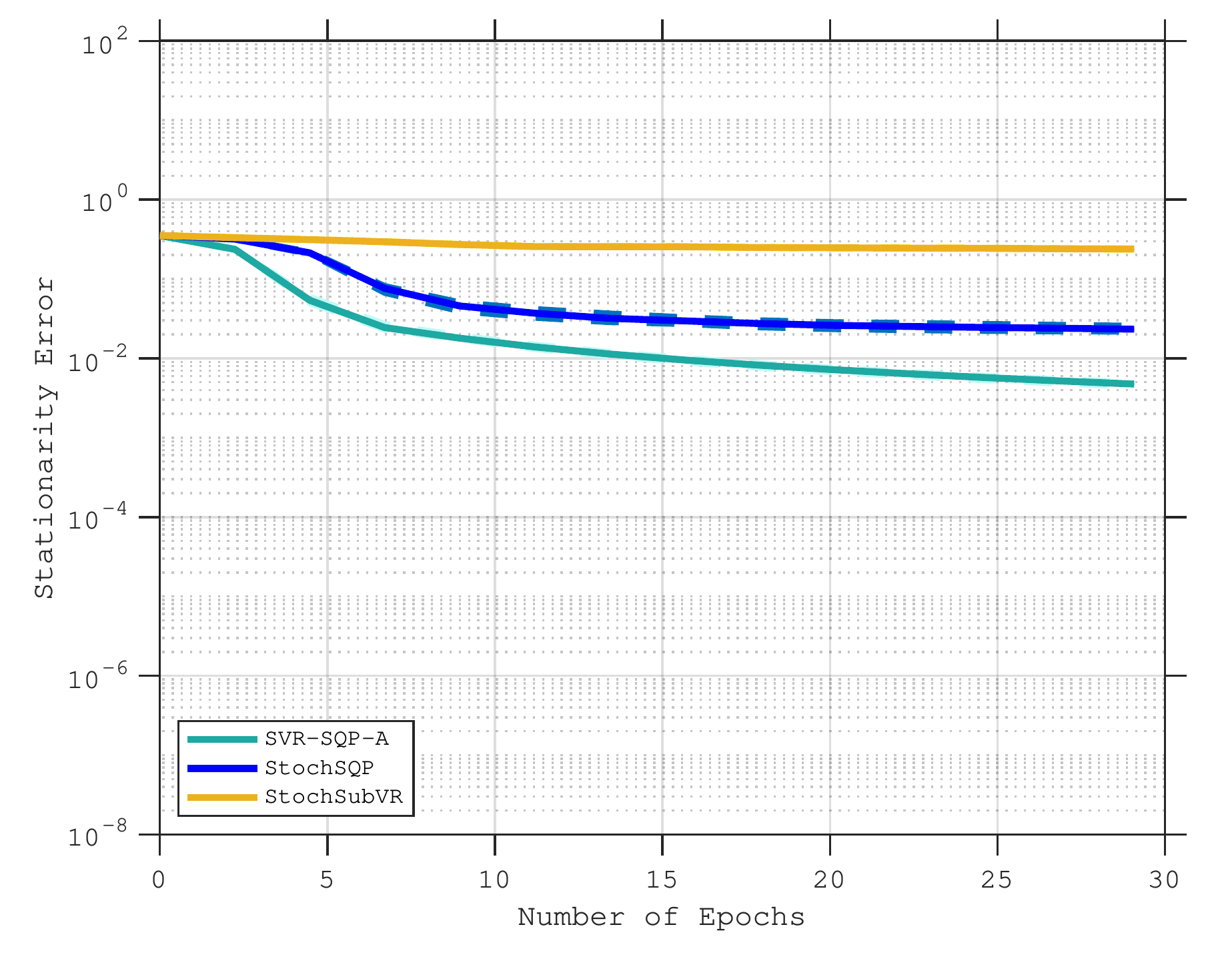}
  \caption{\texttt{australian} dataset. Top row: feasibility error; Bottom row: stationarity
  error.}
  \label{fig.best_1}
    \end{subfigure}

  \begin{subfigure}[b]{\textwidth}
  \includegraphics[width=0.24\textwidth,clip=true,trim=30 180 50 200]{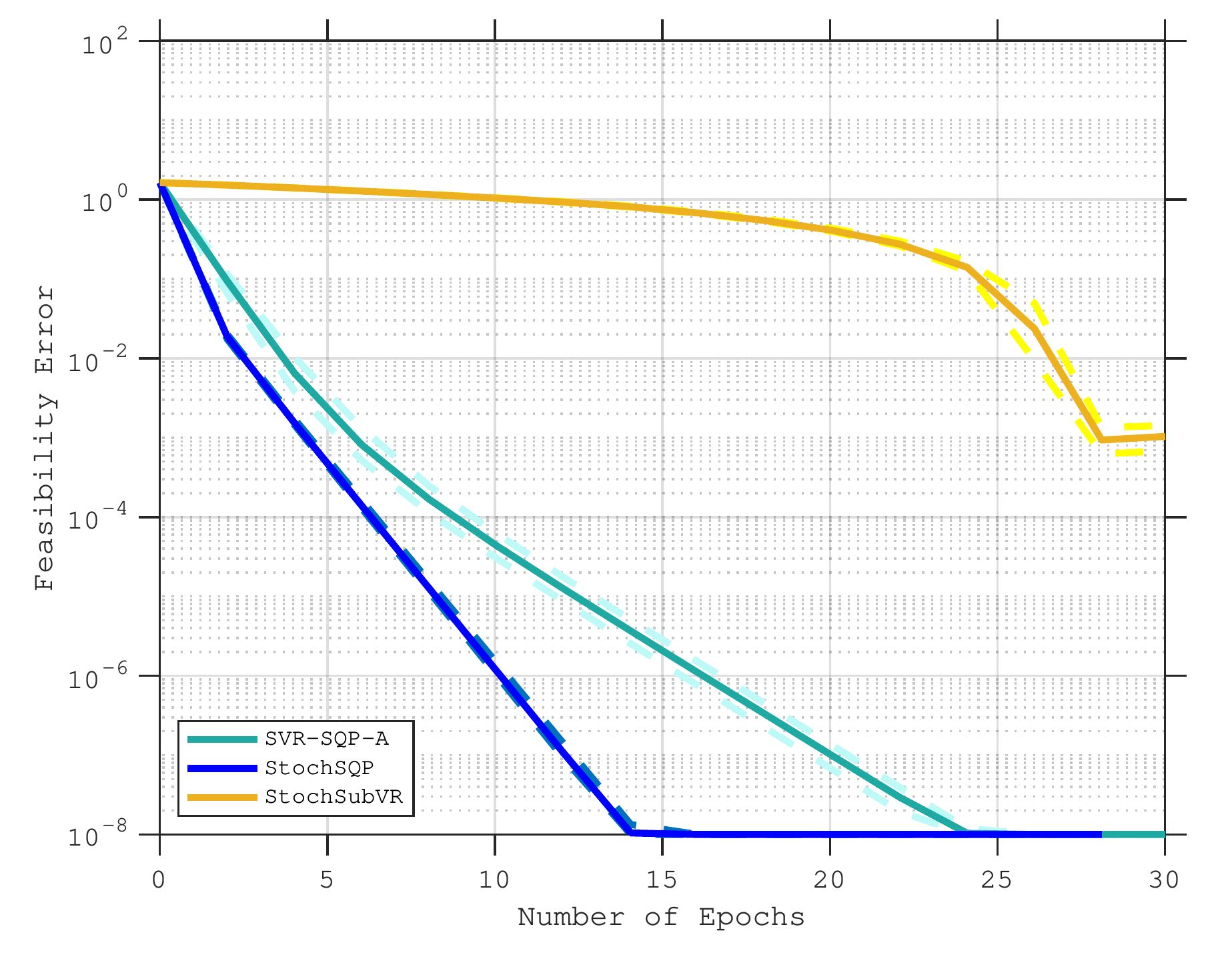}
  \includegraphics[width=0.24\textwidth,clip=true,trim=30 180 50 200]{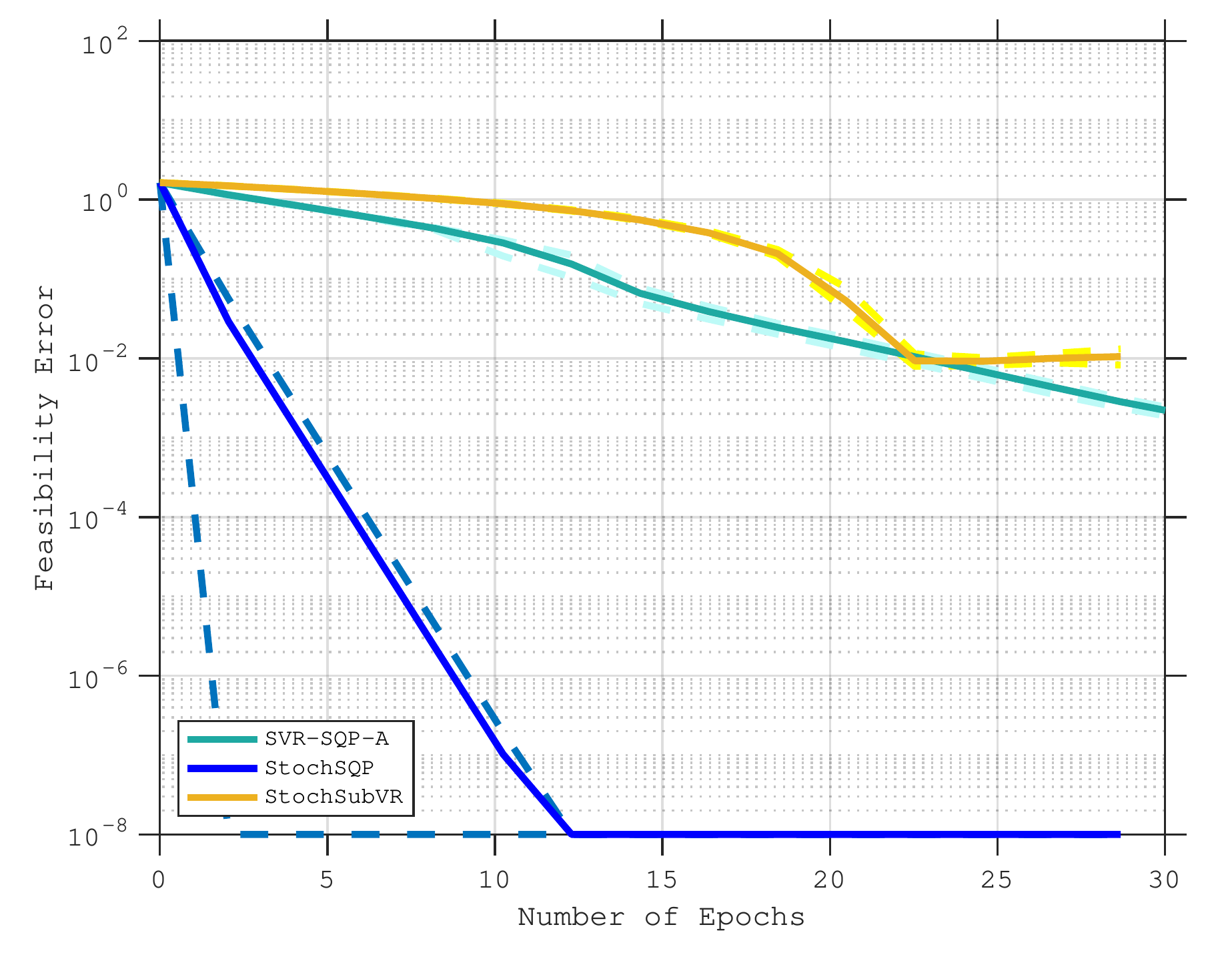}
  \includegraphics[width=0.24\textwidth,clip=true,trim=30 180 50 200]{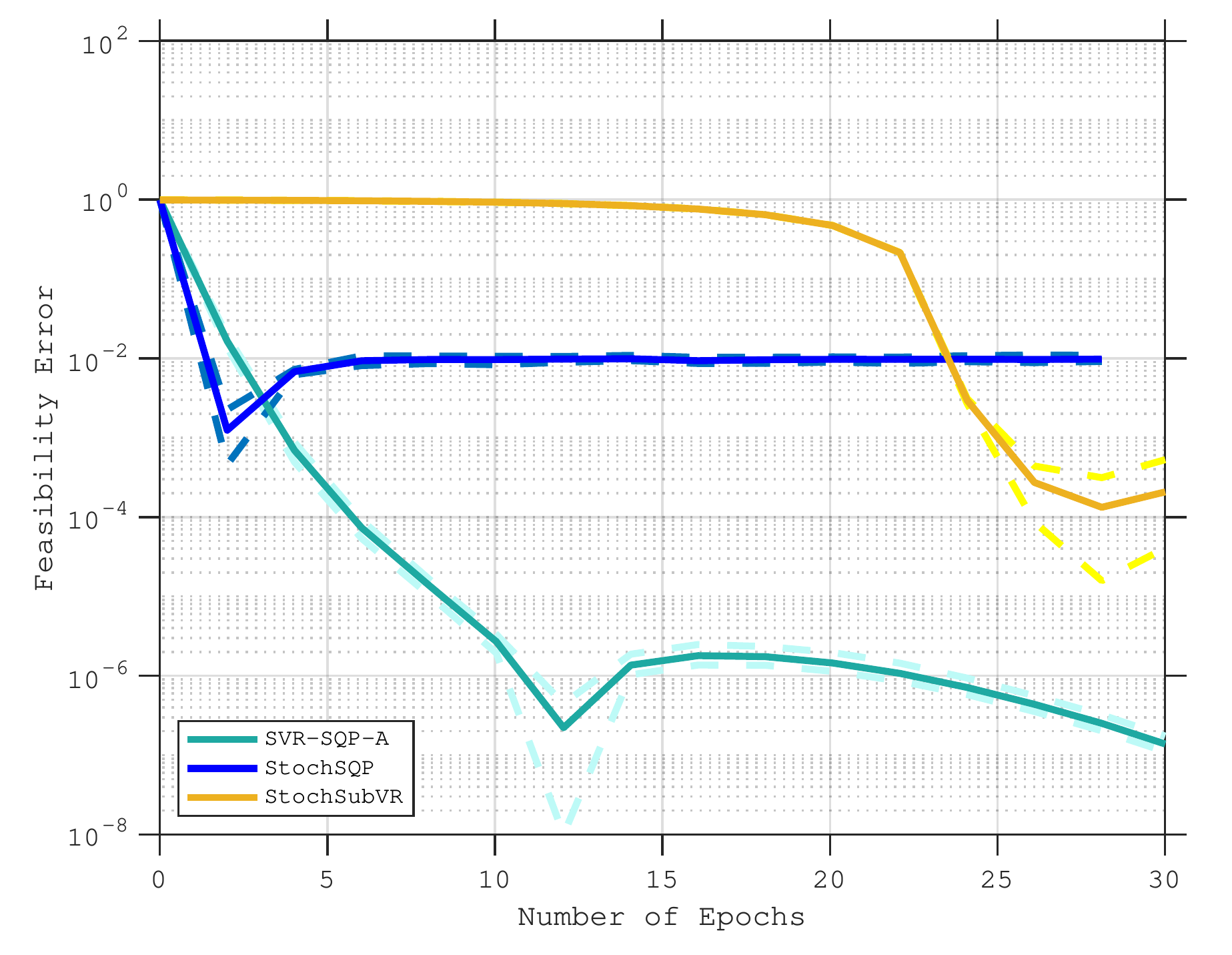}
  \includegraphics[width=0.24\textwidth,clip=true,trim=30 180 50 200]{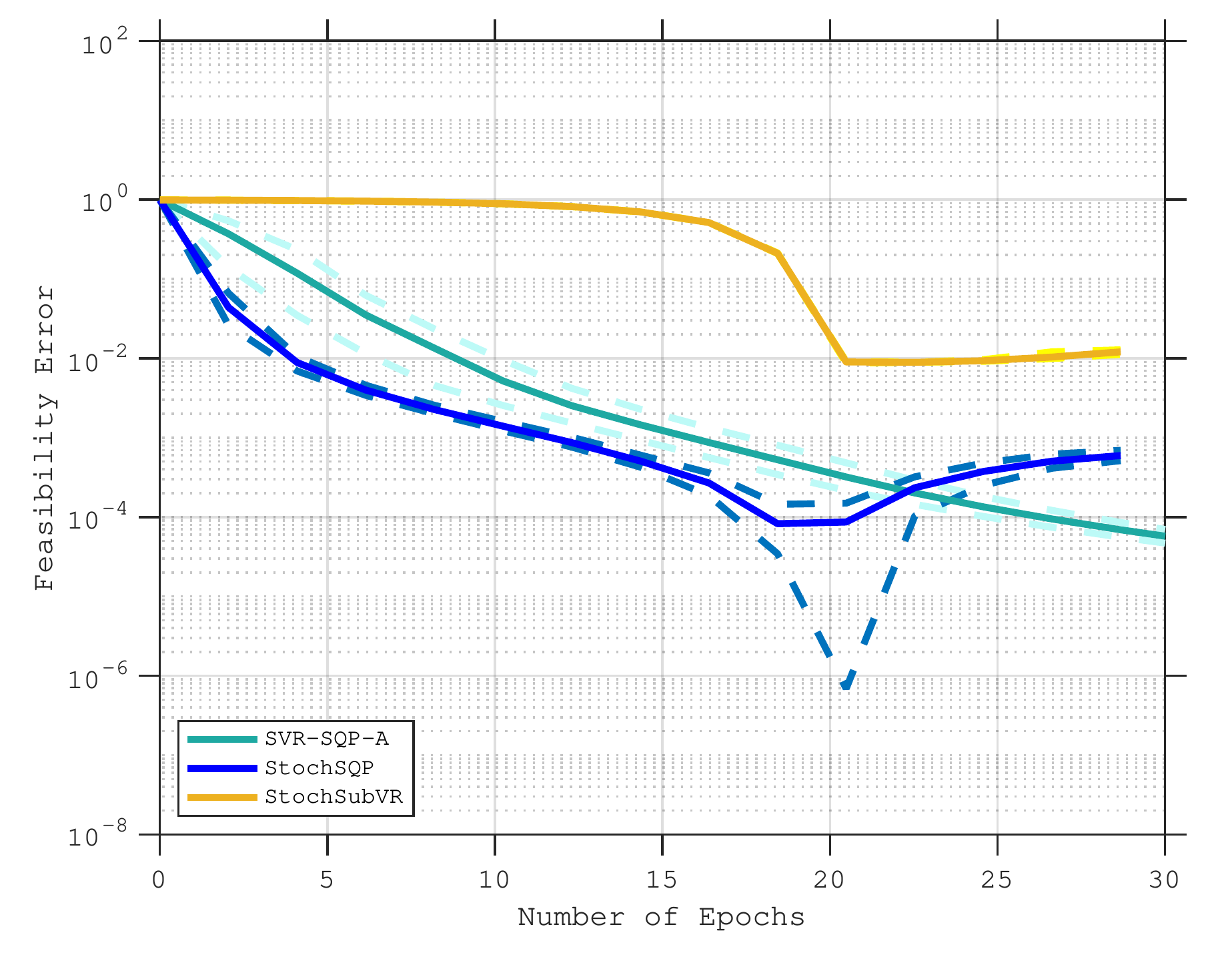}
  
  \includegraphics[width=0.24\textwidth,clip=true,trim=30 180 50 200]{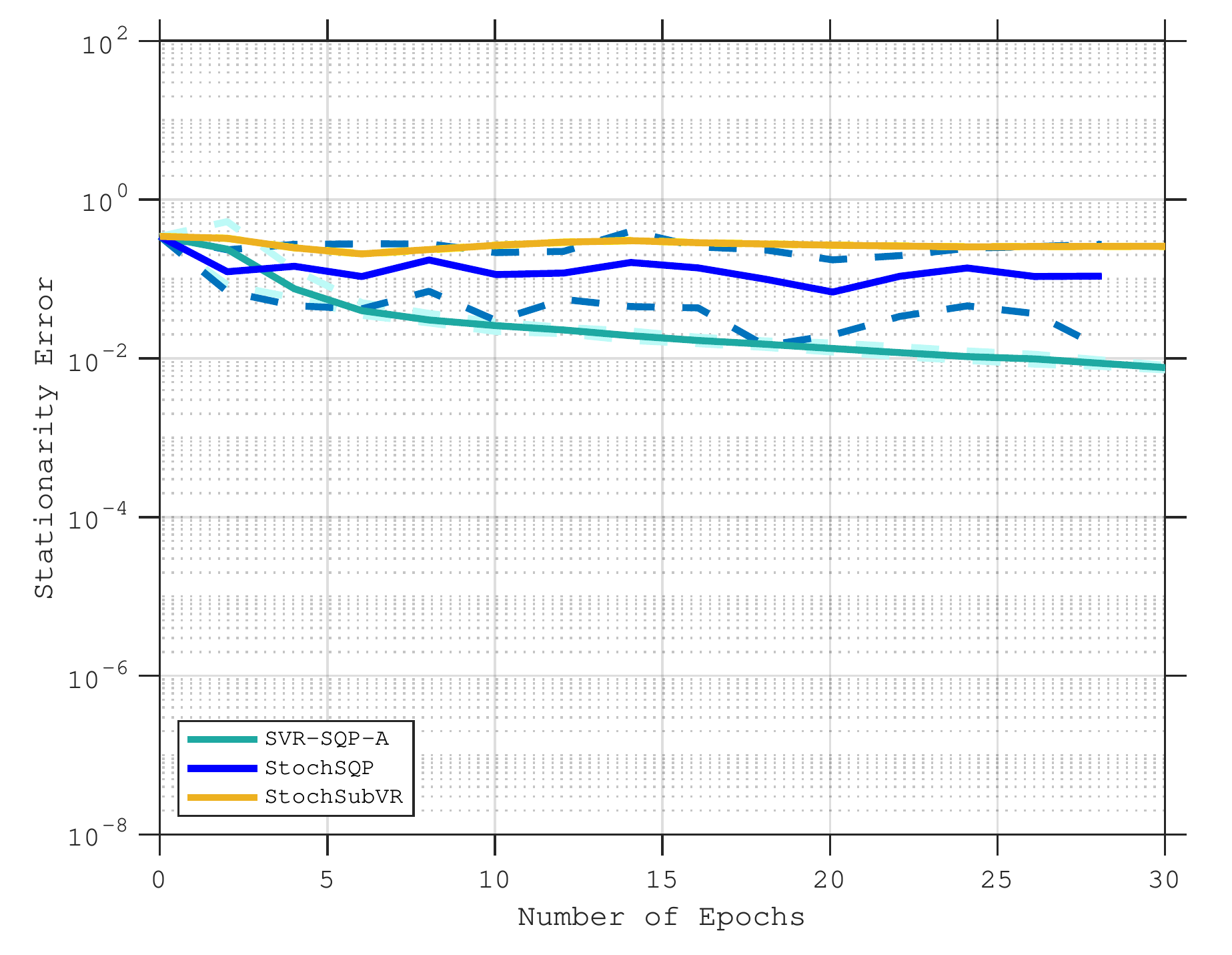}
  \includegraphics[width=0.24\textwidth,clip=true,trim=30 180 50 200]{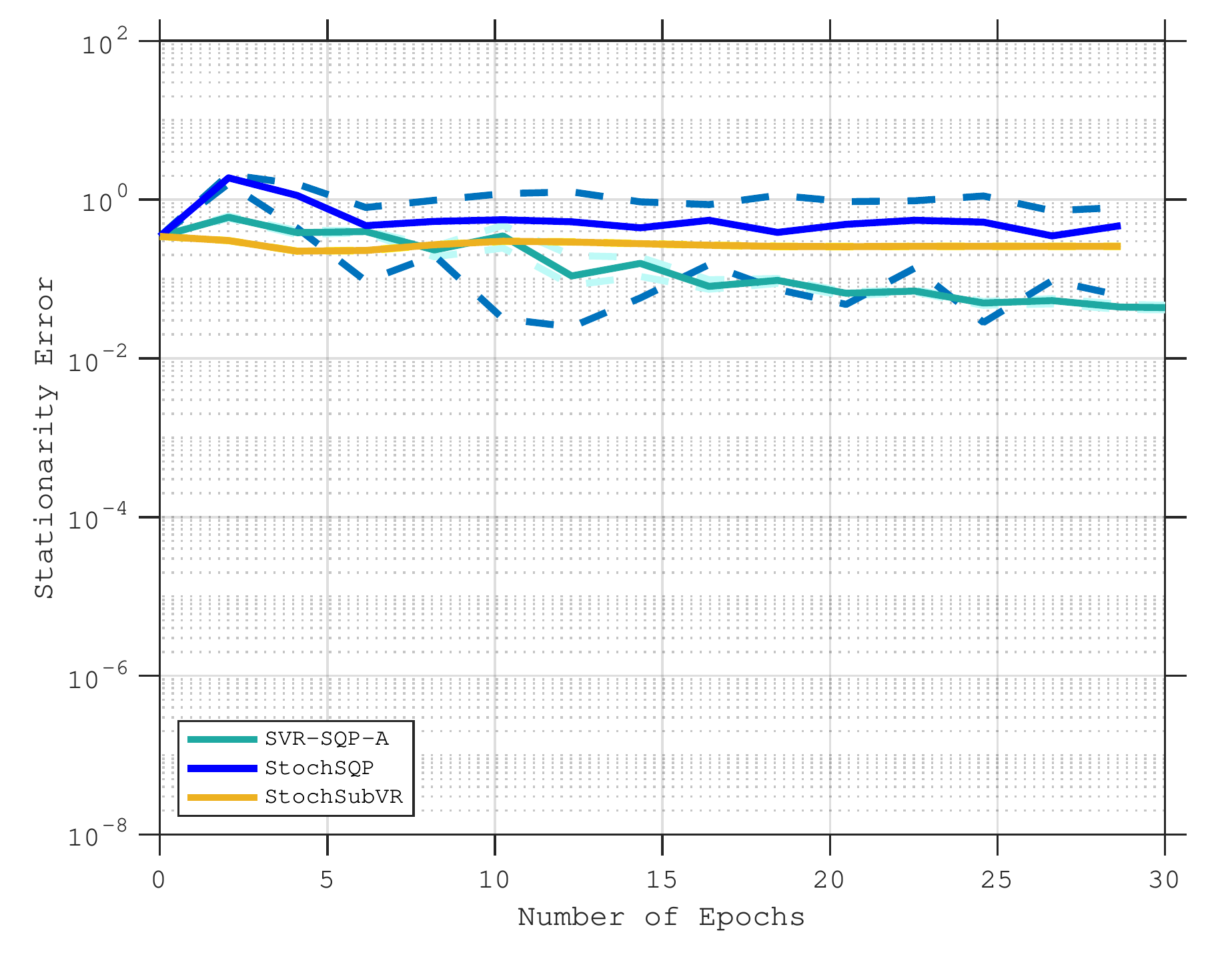}
  \includegraphics[width=0.24\textwidth,clip=true,trim=30 180 50 200]{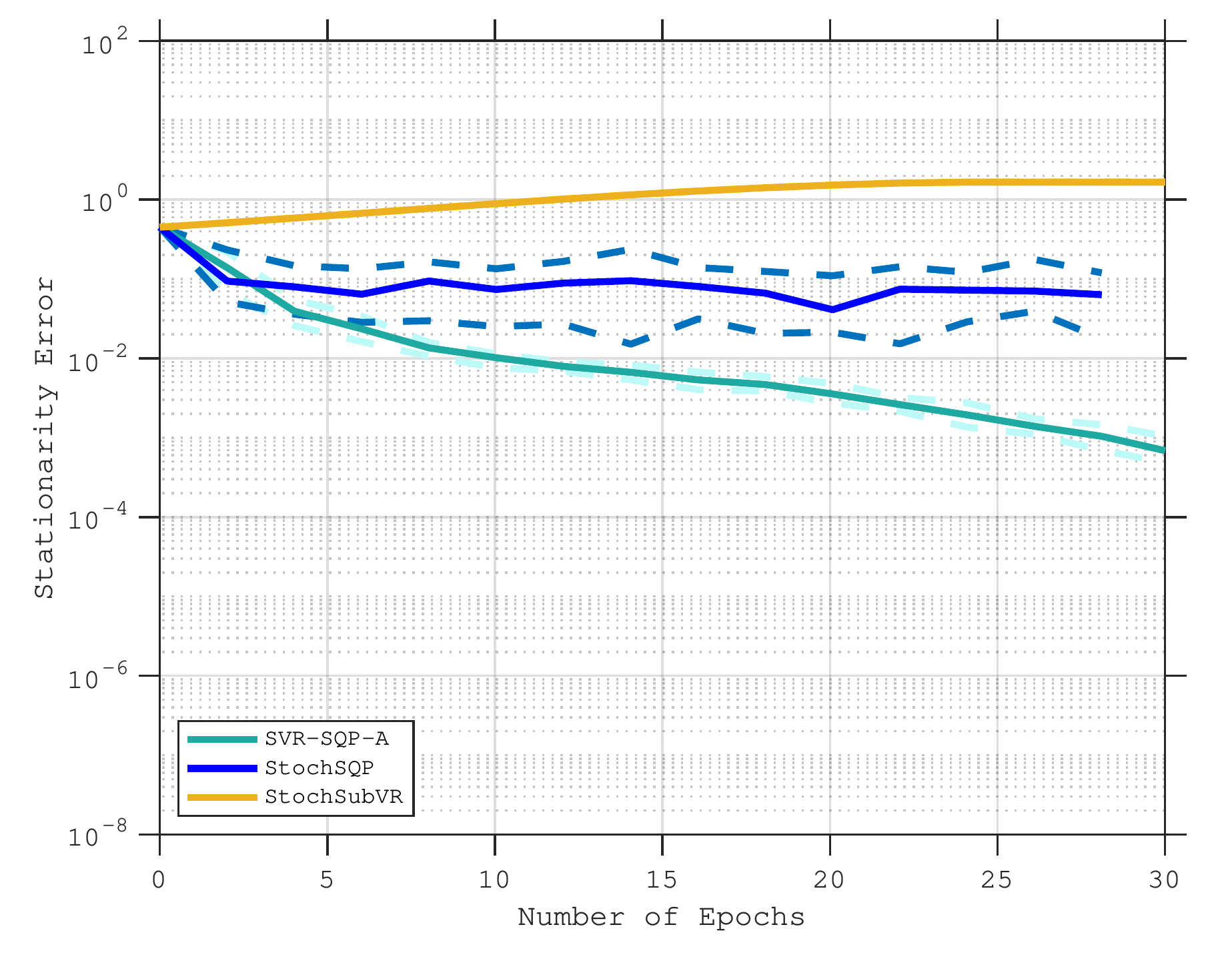}
  \includegraphics[width=0.24\textwidth,clip=true,trim=30 180 50 200]{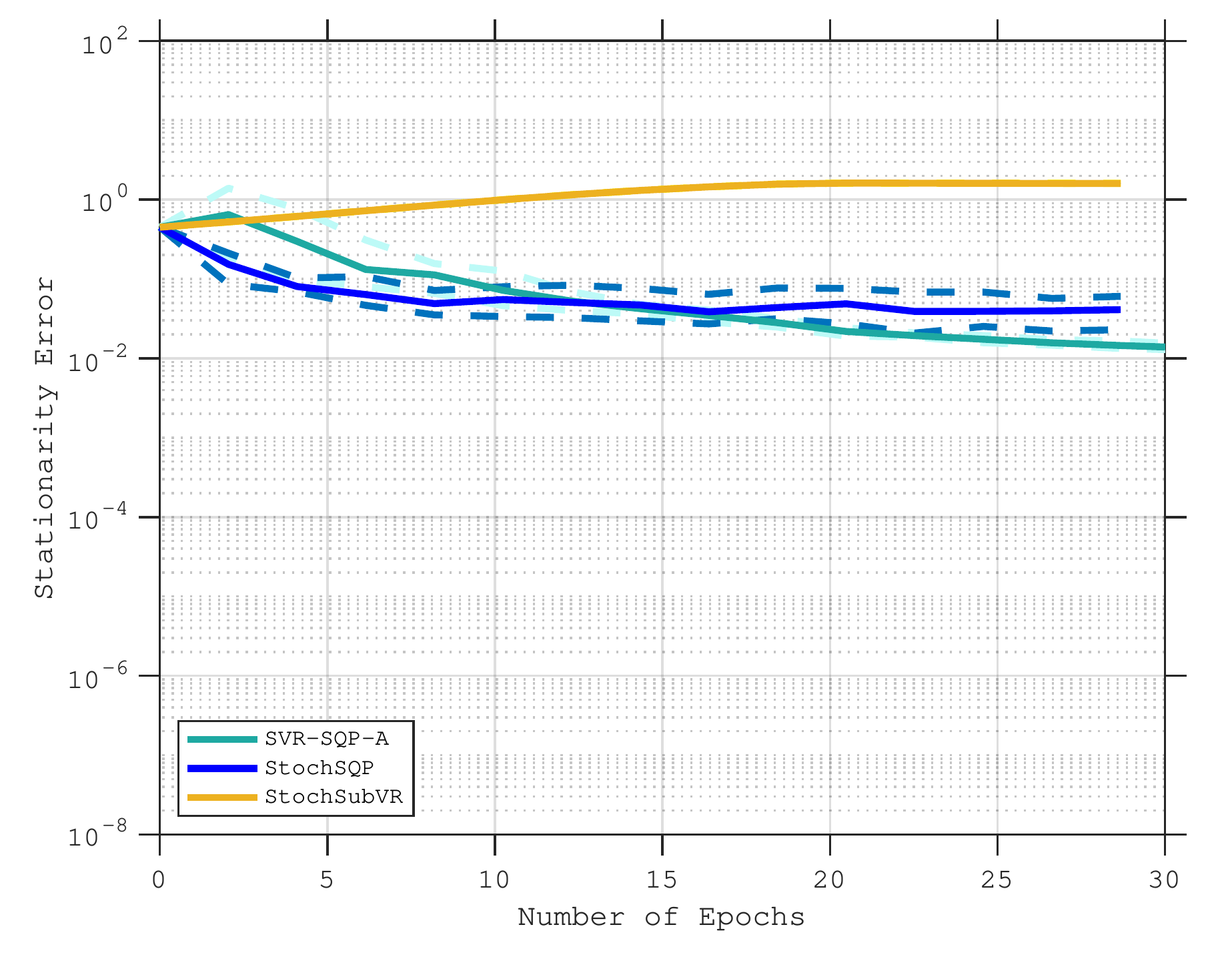}
  \caption{ \texttt{splice} dataset. Top row: feasibility error; Bottom row: stationarity
  error.}
  \label{fig.best_2}
  \end{subfigure}
  \caption{Performance of \textbf{best tuned} variants of \StoSubVR{} and \StoSQP{}, and \SVRSQPADAPT{} with $\beta=1$ and $S = \left\lfloor \tfrac{N}{2b} \right\rfloor$  on logistic regression problems with linear (columns 1 and 2) and $\ell_2$ norm (columns 3 and 4) constraints. First and third columns: batch size 16; Second and fourth columns: batch size 128. \label{fig.best}}
\end{figure}

\renewcommand{\arraystretch}{0.5}
\setlength{\tabcolsep}{0.25pt}

\begin{table}[ht]
\caption{Average feasibility and stationarity errors over 10 independent runs for each experiment, along with $95\%$ confidence intervals represented by `$\pm$', of \textbf{best tuned} variants of \StoSubVR{} and \StoSQP{}, and \SVRSQPADAPT{} with $\beta=1$ and $S = \left\lfloor \tfrac{N}{2b} \right\rfloor$ on  logistic regression problems with linear constraints \eqref{eq.log_lin}. The results for the best-performing
algorithm for each batch size are shown in \textbf{bold}. The symbol $\star$ indicates that all the runs for a given method converged to $\min \{\|c_k\|_{\infty}:k = 0,\ldots,K\} \leq 10^{-6}$.}
\label{tab.best_linear}
{\footnotesize
\resizebox{\columnwidth}{!}{
\begin{tabular}{lccccccc}
\toprule
& & \multicolumn{2}{c}{\StoSubVR{}} & \multicolumn{2}{c}{\StoSQP{}} & \multicolumn{2}{c}{\SVRSQPADAPT{}} \\  
       \cmidrule(lr){3-4}
       \cmidrule(lr){5-6}
       \cmidrule(lr){7-8}
Dataset             & Batch & Feasibility    & Stationarity   & Feasibility   & Stationarity  & Feasibility    & Stationarity   \\ \midrule
\texttt{a9a}        & 16                                                    &     \scalebox{1}[1]{\textbf{${1.2\times10^{-1}\pm7.2\times10^{-3}}$}}         &  \scalebox{1}[1]{\textbf{${2.6\times10^{-2}\pm1.1\times10^{-2}}$}}  &  \scalebox{1}[1]{  \textbf{   $\pmb{\star }$}}  &    \scalebox{1}[1]{\textbf{${8.4\times10^{-4}\pm2.9\times10^{-5}}$}}         & \scalebox{1}[1]{  \textbf{   $\pmb{\star }$}}  &  \scalebox{1}[1]{\textbf{$\pmb{3.5\times10^{-5}\pm4.5\times10^{-6}}$}}  \\ \hdashline
\texttt{a9a}        & 128                                                   &     \scalebox{1}[1]{\textbf{${1.1\times10^{-1}\pm8.4\times10^{-3}}$}}         &   \scalebox{1}[1]{\textbf{${8.7\times10^{-2}\pm1.6\times10^{-2}}$}}            &  \scalebox{1}[1]{  \textbf{   $\pmb{\star }$}}             &           \scalebox{1}[1]{\textbf{${7.0\times10^{-4}\pm2.8\times10^{-5}}$}}    &          \scalebox{1}[1]{  \textbf{   $\pmb{\star }$}}       &         \scalebox{1}[1]{\textbf{$\pmb{1.5\times10^{-4}\pm1.1\times10^{-6}}$}}        \\ \midrule
\texttt{australian} & 16                            &   \scalebox{1}[1]{    $4.1\times10^{-2}\pm1.1\times10^{-2}$    }     &           \scalebox{1}[1]{$2.0\times10^{-1}\pm1.5\times10^{-3}$}     &        \scalebox{1}[1]{  \textbf{   $\pmb{\star }$}}    &   \scalebox{1}[1]{   $1.2\times10^{-3}\pm2.2\times10^{-4}$    }     &   \scalebox{1}[1]{  \textbf{   $\pmb{\star }$}}     &  \scalebox{1}[1]{\textbf{ $\pmb{1.5\times10^{-6}\pm5.6\times10^{-7}}$    } }        \\ \hdashline
\texttt{australian} & 128                                                   &    \scalebox{1}[1]{    $3.1\times10^{-1}\pm3.5\times10^{-2}$    }     &           \scalebox{1}[1]{$2.0\times10^{-1}\pm4.7\times10^{-3}$}     &        \scalebox{1}[1]{  \textbf{   $\pmb{\star }$}}    &   \scalebox{1}[1]{   $\pmb{2.3\times10^{-3}\pm7.3\times10^{-4}}$    }     &   \scalebox{1}[1]{  \textbf{   $\pmb{\star }$}}    &  \scalebox{1}[1]{\textbf{ $4.8\times10^{-3}\pm2.7\times10^{-4}$    } }  \\ \midrule
\texttt{heart} & 16                                                    &  \scalebox{1}[1]{\textbf{${2.3\times10^{-1}\pm4.8\times10^{-2}}$}}       &      \scalebox{1}[1]{\textbf{$\pmb{1.2\times10^{1}\pm3.9\times10^{0}}$}}         &    \scalebox{1}[1]{\textbf{$\pmb{7.9\times10^{-2}\pm1.2\times10^{-2}}$}}         &     \scalebox{1}[1]{\textbf{${2.3\times10^{1}\pm5.1\times10^{0}}$}}         &   \scalebox{1}[1]{\textbf{${1.5\times10^{0}\pm1.6\times10^{-2}}$}}           &           \scalebox{1}[1]{\textbf{${2.2\times10^{1}\pm1.3\times10^{0}}$}}     \\ \hdashline
\texttt{heart} & 128                                                   &     \scalebox{1}[1]{\textbf{$\pmb{8.7\times10^{-1}\pm1.6\times10^{-2}}$}}        &      \scalebox{1}[1]{\textbf{$\pmb{2.3\times10^{1}\pm3.0\times10^{0}}$}}         & \scalebox{1}[1]{\textbf{${1.2\times10^{0}\pm4.1\times10^{-2}}$}}             &   \scalebox{1}[1]{\textbf{${2.4\times10^{1}\pm2.0\times10^{0}}$}}     &     \scalebox{1}[1]{\textbf{${1.6\times10^{0}\pm2.3\times10^{-3}}$}}           &    \scalebox{1}[1]{\textbf{${2.6\times10^{1}\pm1.8\times10^{0}}$}}   \\ \midrule
\texttt{ijcnn1} & 16                                                    &      \scalebox{1}[1]{\textbf{${9.9\times10^{-1}\pm1.1\times10^{-1}}$}}            &   \scalebox{1}[1]{\textbf{${6.6\times10^{-3}\pm5.5\times10^{-3}}$}}            &          \scalebox{1}[1]{  \textbf{   $\pmb{\star }$}}     &   \scalebox{1}[1]{\textbf{${1.8\times10^{-4}\pm1.9\times10^{-5}}$}}            &   \scalebox{1}[1]{  \textbf{   $\pmb{\star }$}}             &          \scalebox{1}[1]{\textbf{$\pmb{1.3\times10^{-8}\pm7.8\times10^{-9}}$}}       \\ \hdashline
\texttt{ijcnn1} & 128                                                   &  \scalebox{1}[1]{\textbf{${9.7\times10^{-1}\pm1.8\times10^{-1}}$}}            &         \scalebox{1}[1]{\textbf{${1.6\times10^{-2}\pm6.5\times10^{-3}}$}}      &    \scalebox{1}[1]{  \textbf{   $\pmb{\star }$}}           &    \scalebox{1}[1]{\textbf{${2.1\times10^{-4}\pm2.2\times10^{-5}}$}}           &   \scalebox{1}[1]{  \textbf{   $\pmb{\star }$}}             &  \scalebox{1}[1]{\textbf{$\pmb{1.0\times10^{-8}\pm1.8\times10^{-10}}$}}     \\ \midrule
\texttt{ionosphere} & 16                                                    &     \scalebox{1}[1]{\textbf{${3.2\times10^{-2}\pm4.9\times10^{-3}}$}}        &  \scalebox{1}[1]{\textbf{${1.4\times10^{-1}\pm5.3\times10^{-3}}$}}            &    \scalebox{1}[1]{  \textbf{   $\pmb{\star }$}}           &   \scalebox{1}[1]{\textbf{${4.2\times10^{-3}\pm4.3\times10^{-4}}$}}            &  \scalebox{1}[1]{  \textbf{   $\pmb{\star }$}}                &         \scalebox{1}[1]{\textbf{$\pmb{2.4\times10^{-3}\pm1.8\times10^{-4}}$}}       \\ \hdashline
\texttt{ionosphere} & 128                                                   &     \scalebox{1}[1]{\textbf{${4.4\times10^{-1}\pm5.4\times10^{-2}}$}}           &  \scalebox{1}[1]{\textbf{${1.4\times10^{-1}\pm1.1\times10^{-2}}$}}              &       \scalebox{1}[1]{  \textbf{   $\pmb{\star }$}}        &  \scalebox{1}[1]{\textbf{$\pmb{1.2\times10^{-2}\pm9.8\times10^{-4}}$}}             &   \scalebox{1}[1]{  \textbf{   $\pmb{\star }$}}             &      \scalebox{1}[1]{\textbf{${2.0\times10^{-2}\pm4.0\times10^{-4}}$}}   \\ \midrule
\texttt{mushrooms} & 16                                                    &       \scalebox{1}[1]{\textbf{${4.8\times10^{-2}\pm6.0\times10^{-3}}$}}       &      
\scalebox{1}[1]{\textbf{${1.3\times10^{-1}\pm2.0\times10^{-2}}$}}          &  \scalebox{1}[1]{  \textbf{   $\pmb{\star }$}}             &         \scalebox{1}[1]{\textbf{$\pmb{4.2\times10^{-4}\pm9.0\times10^{-6}}$}}      &          \scalebox{1}[1]{  \textbf{   $\pmb{\star }$}}        &   \scalebox{1}[1]{\textbf{${8.0\times10^{-4}\pm2.4\times10^{-5}}$}}             \\ \hdashline
\texttt{mushrooms} & 128                                                   &     \scalebox{1}[1]{\textbf{${5.7\times10^{-2}\pm5.5\times10^{-3}}$}}          &    \scalebox{1}[1]{\textbf{${1.8\times10^{-1}\pm4.5\times10^{-3}}$}}           &    \scalebox{1}[1]{  \textbf{   $\pmb{\star }$}}           &   \scalebox{1}[1]{\textbf{$\pmb{2.5\times10^{-3}\pm9.2\times10^{-5}}$}}            &   \scalebox{1}[1]{  \textbf{   $\pmb{\star }$}}             &    \scalebox{1}[1]{\textbf{${4.2\times10^{-3}\pm7.7\times10^{-5}}$}}   \\ \midrule
\texttt{phising} & 16                                                    &    \scalebox{1}[1]{\textbf{${3.7\times10^{-1}\pm5.4\times10^{-2}}$}}         & \scalebox{1}[1]{\textbf{${2.6\times10^{-2}\pm2.5\times10^{-4}}$}}             &  \scalebox{1}[1]{  \textbf{   $\pmb{\star }$}}             &   \scalebox{1}[1]{\textbf{${4.6\times10^{-4}\pm1.2\times10^{-5}}$}}            &    \scalebox{1}[1]{  \textbf{   $\pmb{\star }$}}            &    \scalebox{1}[1]{\textbf{$\pmb{9.9\times10^{-5}\pm3.2\times10^{-6}}$}}            \\ \hdashline
\texttt{phising} & 128                                                   &  \scalebox{1}[1]{\textbf{${6.0\times10^{-1}\pm2.2\times10^{-2}}$}}            &  \scalebox{1}[1]{\textbf{${3.6\times10^{-2}\pm1.4\times10^{-4}}$}}            &      \scalebox{1}[1]{  \textbf{   $\pmb{\star }$}}         &     \scalebox{1}[1]{\textbf{${2.7\times10^{-3}\pm3.6\times10^{-5}}$}}          &    \scalebox{1}[1]{  \textbf{   $\pmb{\star }$}}            & \scalebox{1}[1]{\textbf{$\pmb{8.1\times10^{-4}\pm1.3\times10^{-5}}$}}      \\ \midrule
\texttt{sonar} & 16                                                    &    \scalebox{1}[1]{\textbf{${4.1\times10^{-2}\pm2.8\times10^{-3}}$}}           &   \scalebox{1}[1]{\textbf{${3.9\times10^{-2}\pm1.2\times10^{-3}}$}}             &  \scalebox{1}[1]{  \textbf{   $\pmb{\star }$}}             &  \scalebox{1}[1]{\textbf{$\pmb{7.5\times10^{-3}\pm4.4\times10^{-4}}$}}             &   \scalebox{1}[1]{  \textbf{   $\pmb{\star }$}}             &  \scalebox{1}[1]{\textbf{${1.1\times10^{-2}\pm3.1\times10^{-4}}$}}              \\ \hdashline
\texttt{sonar} & 128                                                   &    \scalebox{1}[1]{\textbf{${4.1\times10^{-1}\pm3.2\times10^{-2}}$}}            &   \scalebox{1}[1]{\textbf{${4.1\times10^{-2}\pm3.0\times10^{-3}}$}}             &   \scalebox{1}[1]{  \textbf{   $\pmb{\star }$}}            & \scalebox{1}[1]{\textbf{$\pmb{1.9\times10^{-2}\pm2.9\times10^{-4}}$}}              &  \scalebox{1}[1]{  \textbf{   $\pmb{\star }$}}              &  \scalebox{1}[1]{\textbf{${2.2\times10^{-2}\pm7.8\times10^{-5}}$}}     \\ \midrule
\texttt{splice} & 16                                                    &  \scalebox{1}[1]{    $6.8\times10^{-4}\pm6.2\times10^{-5}$    }              &    \scalebox{1}[1]{    $2.6\times10^{-1}\pm2.2\times10^{-4}$    }            &      \scalebox{1}[1]{  \textbf{   $\pmb{\star }$}}        &       \scalebox{1}[1]{    $\pmb{4.1\times10^{-3}\pm3.7\times10^{-4}}$    }        &  \scalebox{1}[1]{  \textbf{   $\pmb{\star }$}}             &       \scalebox{1}[1]{    $7.6\times10^{-3}\pm2.0\times10^{-4}$    }         \\ \hdashline
\texttt{splice} & 128                                                   &     \scalebox{1}[1]{    $7.0\times10^{-3}\pm3.3\times10^{-4}$    }            &    \scalebox{1}[1]{    $2.6\times10^{-1}\pm7.6\times10^{-4}$    }             &     \scalebox{1}[1]{  \textbf{   $\pmb{\star }$}}            &     \scalebox{1}[1]{    $\pmb{1.8\times10^{-2}\pm8.0\times10^{-4}}$    }           &              \scalebox{1}[1]{    $1.9\times10^{-3}\pm9.4\times10^{-5}$    }   &  \scalebox{1}[1]{    $4.3\times10^{-2}\pm1.2\times10^{-3}$    }      \\\midrule
\texttt{w8a} & 16                                                    &  \scalebox{1}[1]{\textbf{${5.1\times10^{-1}\pm8.6\times10^{-3}}$}}              &  \scalebox{1}[1]{\textbf{${4.8\times10^{-3}\pm1.1\times10^{-4}}$}}              & \scalebox{1}[1]{  \textbf{   $\pmb{\star }$}}              &  \scalebox{1}[1]{\textbf{${2.6\times10^{-4}\pm1.9\times10^{-5}}$}}             &   \scalebox{1}[1]{  \textbf{   $\pmb{\star }$}}             &    \scalebox{1}[1]{\textbf{$\pmb{7.5\times10^{-5}\pm7.9\times10^{-7}}$}}       \\ \hdashline
\texttt{w8a} & 128                                                   &   \scalebox{1}[1]{\textbf{${6.9\times10^{-1}\pm2.5\times10^{-3}}$}}           &  \scalebox{1}[1]{\textbf{${2.6\times10^{-2}\pm1.8\times10^{-4}}$}}              &  \scalebox{1}[1]{  \textbf{   $\pmb{\star }$}}             &   \scalebox{1}[1]{\textbf{$\pmb{1.8\times10^{-4}\pm5.0\times10^{-6}}$}}            & \scalebox{1}[1]{  \textbf{   $\pmb{\star }$}}              & \scalebox{1}[1]{\textbf{${3.4\times10^{-4}\pm2.1\times10^{-6}}$}}      \\\bottomrule
\end{tabular}}}
\end{table}

\begin{table}[ht]
\caption{Average feasibility and stationarity errors over 10 independent runs for each experiment, along with $95\%$ confidence intervals represented by `$\pm$', of \textbf{best tuned} variants of \StoSubVR{} and \StoSQP{}, and \SVRSQPADAPT{} with $\beta=1$ and $S = \left\lfloor \tfrac{N}{2b} \right\rfloor$ on  logistic regression problems with $\ell_2$  constraint \eqref{eq.log_el2}. The results for the best-performing
algorithm for each batch size are shown in bold. The symbol $\star$ indicates that all the runs for a given method converged to $\min \{\|c_k\|_{\infty}:k = 0,\ldots,K\} \leq 10^{-6}$.}
\label{tab.best_el2}
{\footnotesize
\resizebox{\columnwidth}{!}{
\begin{tabular}{lccccccc}
\toprule
& & \multicolumn{2}{c}{\StoSubVR{}} & \multicolumn{2}{c}{\StoSQP{}} & \multicolumn{2}{c}{\SVRSQPADAPT{}} \\  
       \cmidrule(lr){3-4}
       \cmidrule(lr){5-6}
       \cmidrule(lr){7-8}
Dataset             & Batch  & Feasibility    & Stationarity   & Feasibility   & Stationarity  & Feasibility    & Stationarity   \\ \midrule
\scalebox{1}[1]{\texttt{a9a}}        & 16                                                    &  \scalebox{1}[1]{\textbf{${3.0\times10^{-6}\pm9.7\times10^{-7}}$}}             &  \scalebox{1}[1]{\textbf{${2.5\times10^{-1}\pm7.4\times10^{-7}}$}}  &  \scalebox{1}[1]{  \textbf{   $\pmb{\star }$}} &     \scalebox{1}[1]{\textbf{${1.1\times10^{-2}\pm7.8\times10^{-4}}$}}        & \scalebox{1}[1]{  \textbf{   $\pmb{\star }$}} &   \scalebox{1}[1]{\textbf{$\pmb{1.5\times10^{-8}\pm1.9\times10^{-9}}$}} \\ \hdashline
\scalebox{1}[1]{\texttt{a9a}}        & 128                                                   &    \scalebox{1}[1]{\textbf{${2.9\times10^{-5}\pm3.8\times10^{-6}}$}}            &   \scalebox{1}[1]{\textbf{${2.5\times10^{-1}\pm3.6\times10^{-8}}$}}             &   \scalebox{1}[1]{  \textbf{   $\pmb{\star }$}}           &  \scalebox{1}[1]{\textbf{${8.1\times10^{-3}\pm2.2\times10^{-4}}$}}             &    \scalebox{1}[1]{  \textbf{   $\pmb{\star }$}}            &      \scalebox{1}[1]{\textbf{$\pmb{1.7\times10^{-5}\pm1.2\times10^{-6}}$}}          \\ \midrule
\scalebox{1}[1]{\texttt{australian}} & 16                            &   \scalebox{1}[1]{    $2.9\times10^{-4}\pm5.8\times10^{-5}$    }     &           \scalebox{1}[1]{$3.1\times10^{-1}\pm1.3\times10^{-6}$}     &        \scalebox{1}[1]{$2.9\times10^{-4}\pm1.4\times10^{-4}$  }     &   \scalebox{1}[1]{   $2.3\times10^{-2}\pm1.9\times10^{-2}$    }     &   \scalebox{1}[1]{  \textbf{   $\pmb{\star }$}}     &  \scalebox{1}[1]{\textbf{ $\pmb{2.1\times10^{-4}\pm4.9\times10^{-5}}$    } }        \\ \hdashline
\scalebox{1}[1]{\texttt{australian}} & 128                                                   &    \scalebox{1}[1]{    $2.6\times10^{-3}\pm1.1\times10^{-3}$    }     &           \scalebox{1}[1]{$2.5\times10^{-1}\pm3.2\times10^{-3}$}     &        \scalebox{1}[1]{$1.3\times10^{-4}\pm4.1\times10^{-6}$  }     &   \scalebox{1}[1]{   $9.1\times10^{-3}\pm1.1\times10^{-3}$    }     &   \scalebox{1}[1]{  \textbf{   $\pmb{2.4\times10^{-5}\pm9.6\times10^{-7}}$  } }     &  \scalebox{1}[1]{\textbf{ $\pmb{4.7\times10^{-3}\pm9.2\times10^{-5}}$    } }  \\ \midrule
\scalebox{1}[1]{\texttt{heart}} & 16                                                    &     \scalebox{1}[1]{\textbf{$\pmb{1.2\times10^{-3}\pm5.2\times10^{-4}}$}}           &  \scalebox{1}[1]{\textbf{$\pmb{2.5\times10^{0}\pm2.3\times10^{-1}}$}}              &    \scalebox{1}[1]{\textbf{${3.3\times10^{-1}\pm6.1\times10^{-2}}$}}           &     \scalebox{1}[1]{\textbf{${9.1\times10^{1}\pm2.2\times10^{1}}$}}          &   \scalebox{1}[1]{\textbf{${4.7\times10^{-1}\pm3.8\times10^{-2}}$}}             &    \scalebox{1}[1]{\textbf{${8.8\times10^{1}\pm7.6\times10^{0}}$}}            \\ \hdashline
\scalebox{1}[1]{\texttt{heart}} & 128                                                   &    \scalebox{1}[1]{\textbf{$\pmb{2.9\times10^{-2}\pm1.5\times10^{-2}}$}}            &    \scalebox{1}[1]{\textbf{${1.1\times10^{2}\pm2.5\times10^{1}}$}}            &    \scalebox{1}[1]{\textbf{${9.9\times10^{-1}\pm5.7\times10^{-5}}$}}           &   \scalebox{1}[1]{\textbf{$\pmb{5.9\times10^{0}\pm1.1\times10^{0}}$}}            &     \scalebox{1}[1]{\textbf{${8.8\times10^{-1}\pm8.3\times10^{-3}}$}}           &  \scalebox{1}[1]{\textbf{${1.1\times10^{2}\pm1.4\times10^{1}}$}}     \\ \midrule
\scalebox{1}[1]{\texttt{ijcnn1}} & 16                                            &     \scalebox{1}[1]{\textbf{${1.2\times10^{-5}\pm8.1\times10^{-6}}$}}           &      \scalebox{1}[1]{\textbf{${6.8\times10^{-2}\pm8.3\times10^{-5}}$}}          &     \scalebox{1}[1]{\textbf{${1.0\times10^{-6}\pm2.6\times10^{-9}}$}}          &   \scalebox{1}[1]{\textbf{${3.8\times10^{-2}\pm7.9\times10^{-4}}$}}            &   \scalebox{1}[1]{  \textbf{   $\pmb{\star }$}}             &      \scalebox{1}[1]{\textbf{$\pmb{3.1\times10^{-8}\pm3.8\times10^{-9}}$}}          \\ \hdashline
\scalebox{1}[1]{\texttt{ijcnn1}} & 128                                                   &   \scalebox{1}[1]{\textbf{${4.4\times10^{-6}\pm1.8\times10^{-7}}$}}             &    \scalebox{1}[1]{\textbf{${6.8\times10^{-2}\pm2.1\times10^{-10}}$}}            &    \scalebox{1}[1]{\textbf{${1.0\times10^{-6}\pm3.2\times10^{-10}}$}}           &  \scalebox{1}[1]{\textbf{${3.1\times10^{-2}\pm5.8\times10^{-4}}$}}             &    \scalebox{1}[1]{  \textbf{   $\pmb{\star }$}}            &   \scalebox{1}[1]{\textbf{$\pmb{1.2\times10^{-5}\pm1.1\times10^{-7}}$}}    \\ \midrule
\scalebox{1}[1]{\texttt{ionosphere}} & 16                                                    &   \scalebox{1}[1]{\textbf{${1.4\times10^{-3}\pm5.4\times10^{-4}}$}}             &     \scalebox{1}[1]{\textbf{${1.1\times10^{-1}\pm6.4\times10^{-4}}$}}           &    \scalebox{1}[1]{\textbf{${4.3\times10^{-4}\pm4.0\times10^{-4}}$}}           &  \scalebox{1}[1]{\textbf{${5.2\times10^{-2}\pm1.7\times10^{-2}}$}}             &     \scalebox{1}[1]{\textbf{$\pmb{1.4\times10^{-5}\pm3.2\times10^{-6}}$}}           &     \scalebox{1}[1]{\textbf{$\pmb{6.1\times10^{-3}\pm7.0\times10^{-4}}$}}           \\ \hdashline
\scalebox{1}[1]{\texttt{ionosphere}} & 128                                                   &    \scalebox{1}[1]{\textbf{${8.0\times10^{-3}\pm7.1\times10^{-3}}$}}            &       \scalebox{1}[1]{\textbf{${1.1\times10^{-1}\pm1.5\times10^{-3}}$}}         &    \scalebox{1}[1]{\textbf{$\pmb{5.8\times10^{-4}\pm1.9\times10^{-5}}$}}           &      \scalebox{1}[1]{\textbf{$\pmb{2.0\times10^{-2}\pm1.2\times10^{-3}}$}}         &     \scalebox{1}[1]{\textbf{${7.6\times10^{-4}\pm1.4\times10^{-5}}$}}           & \scalebox{1}[1]{\textbf{${2.3\times10^{-2}\pm3.9\times10^{-4}}$}}      \\ \midrule
\scalebox{1}[1]{\texttt{mushrooms}} & 16                                                    &   \scalebox{1}[1]{\textbf{${3.4\times10^{-5}\pm9.1\times10^{-6}}$}}             &   \scalebox{1}[1]{\textbf{${1.8\times10^{-1}\pm6.4\times10^{-8}}$}}             &     \scalebox{1}[1]{  \textbf{   $\pmb{\star }$}}         &     \scalebox{1}[1]{\textbf{${7.1\times10^{-3}\pm1.2\times10^{-4}}$}}          &   \scalebox{1}[1]{  \textbf{   $\pmb{\star }$}}             &      \scalebox{1}[1]{\textbf{$\pmb{1.0\times10^{-8}\pm2.1\times10^{-11}}$}}          \\ \hdashline
\scalebox{1}[1]{\texttt{mushrooms}} & 128                                                   &    \scalebox{1}[1]{\textbf{${4.0\times10^{-4}\pm2.4\times10^{-4}}$}}            &     \scalebox{1}[1]{\textbf{${2.9\times10^{-2}\pm1.6\times10^{-3}}$}}           &    \scalebox{1}[1]{  \textbf{   $\pmb{\star }$}}           &        \scalebox{1}[1]{\textbf{${2.2\times10^{-2}\pm5.8\times10^{-4}}$}}       &    \scalebox{1}[1]{  \textbf{   $\pmb{\star }$}}            &   \scalebox{1}[1]{\textbf{$\pmb{1.6\times10^{-7}\pm2.7\times10^{-8}}$}}    \\ \midrule
\scalebox{1}[1]{\texttt{phising}} & 16                                                    &   \scalebox{1}[1]{\textbf{${3.9\times10^{-5}\pm2.1\times10^{-5}}$}}             &     \scalebox{1}[1]{\textbf{${4.0\times10^{-2}\pm4.1\times10^{-6}}$}}           &    \scalebox{1}[1]{\textbf{${8.4\times10^{-5}\pm1.2\times10^{-6}}$}}             &       \scalebox{1}[1]{\textbf{${1.0\times10^{-3}\pm4.0\times10^{-4}}$}}           &      \scalebox{1}[1]{  \textbf{   $\pmb{\star }$}}          &    \scalebox{1}[1]{\textbf{$\pmb{4.4\times10^{-6}\pm5.2\times10^{-8}}$}}            \\ \hdashline
\scalebox{1}[1]{\texttt{phising}} & 128                                                   &     \scalebox{1}[1]{\textbf{${7.1\times10^{-4}\pm3.5\times10^{-4}}$}}           &    \scalebox{1}[1]{\textbf{${2.9\times10^{-2}\pm1.7\times10^{-3}}$}}            &    \scalebox{1}[1]{\textbf{$\pmb{1.2\times10^{-5}\pm3.4\times10^{-7}}$}}         &    \scalebox{1}[1]{\textbf{$\pmb{3.0\times10^{-3}\pm3.4\times10^{-4}}$}}        &     \scalebox{1}[1]{\textbf{${1.3\times10^{-5}\pm6.2\times10^{-8}}$}}           &   \scalebox{1}[1]{\textbf{${7.5\times10^{-3}\pm2.0\times10^{-5}}$}}    \\ \midrule
\scalebox{1}[1]{\texttt{sonar}} & 16                                                    &     \scalebox{1}[1]{\textbf{${2.5\times10^{-3}\pm6.5\times10^{-4}}$}}           &  \scalebox{1}[1]{\textbf{${1.4\times10^{-1}\pm2.0\times10^{-5}}$}}              &    \scalebox{1}[1]{\textbf{${7.4\times10^{-4}\pm1.9\times10^{-4}}$}}           &  \scalebox{1}[1]{\textbf{${2.3\times10^{-2}\pm3.9\times10^{-3}}$}}             &    \scalebox{1}[1]{\textbf{$\pmb{1.7\times10^{-4}\pm1.1\times10^{-5}}$}}            &      \scalebox{1}[1]{\textbf{$\pmb{2.0\times10^{-2}\pm5.6\times10^{-4}}$}}          \\ \hdashline
\scalebox{1}[1]{\texttt{sonar}} & 128                                                   &    \scalebox{1}[1]{\textbf{${3.3\times10^{-2}\pm2.1\times10^{-2}}$}}            &   \scalebox{1}[1]{\textbf{${2.9\times10^{-2}\pm2.5\times10^{-3}}$}}             &   \scalebox{1}[1]{\textbf{$\pmb{8.9\times10^{-4}\pm6.0\times10^{-5}}$}}            &     \scalebox{1}[1]{\textbf{$\pmb{2.7\times10^{-2}\pm1.8\times10^{-3}}$}}          &   \scalebox{1}[1]{\textbf{${3.2\times10^{-3}\pm3.9\times10^{-7}}$}}             &   \scalebox{1}[1]{\textbf{${3.2\times10^{-2}\pm3.5\times10^{-6}}$}}    \\ \midrule
\scalebox{1}[1]{\texttt{splice}} & 16                                                    &  \scalebox{1}[1]{    $3.0\times10^{-5}\pm3.9\times10^{-4}$    }              &    \scalebox{1}[1]{    $1.7\times10^{0}\pm1.6\times10^{-4}$    }            &      \scalebox{1}[1]{    $1.3\times10^{-5}\pm8.4\times10^{-6}$    }         &       \scalebox{1}[1]{    $1.0\times10^{-1}\pm3.8\times10^{-2}$    }        &  \scalebox{1}[1]{  \textbf{   $\pmb{\star }$}}           &       \scalebox{1}[1]{    $\pmb{6.5\times10^{-4}\pm9.0\times10^{-5}}$    }         \\ \hdashline
\scalebox{1}[1]{\texttt{splice}} & 128                                                   &     \scalebox{1}[1]{    $7.1\times10^{-4}\pm2.8\times10^{-4}$    }            &    \scalebox{1}[1]{    $1.6\times10^{0}\pm2.1\times10^{-3}$    }             &     \scalebox{1}[1]{    $\pmb{1.0\times10^{-6}\pm5.4\times10^{-7}}$    }           &     \scalebox{1}[1]{    $4.5\times10^{-2}\pm2.2\times10^{-2}$    }           &              \scalebox{1}[1]{    $5.2\times10^{-5}\pm4.1\times10^{-6}$    }   &  \scalebox{1}[1]{    $\pmb{1.4\times10^{-2}\pm5.2\times10^{-4}}$    }      \\\midrule
\scalebox{1}[1]{\texttt{w8a}} & 16                                                    &      \scalebox{1}[1]{\textbf{${3.5\times10^{-5}\pm2.2\times10^{-5}}$}}          &     \scalebox{1}[1]{\textbf{${1.6\times10^{-1}\pm8.4\times10^{-5}}$}}           &    \scalebox{1}[1]{\textbf{${1.3\times10^{-6}\pm5.4\times10^{-8}}$}}           &    \scalebox{1}[1]{\textbf{${3.6\times10^{-3}\pm4.6\times10^{-5}}$}}           &   \scalebox{1}[1]{  \textbf{   $\pmb{\star }$}}             &   \scalebox{1}[1]{\textbf{$\pmb{1.4\times10^{-8}\pm2.1\times10^{-9}}$}}             \\ \hdashline
\scalebox{1}[1]{\texttt{w8a}} & 128                                                   &      \scalebox{1}[1]{\textbf{${1.5\times10^{-5}\pm1.0\times10^{-6}}$}}          &     \scalebox{1}[1]{\textbf{${1.6\times10^{-1}\pm2.2\times10^{-9}}$}}           &   \scalebox{1}[1]{\textbf{${1.0\times10^{-6}\pm3.6\times10^{-10}}$}}            &    \scalebox{1}[1]{\textbf{${3.4\times10^{-3}\pm8.9\times10^{-5}}$}}           &   \scalebox{1}[1]{  \textbf{   $\pmb{\star }$}}             &     \scalebox{1}[1]{\textbf{$\pmb{2.4\times10^{-8}\pm2.4\times10^{-9}}$}}  \\\bottomrule
\end{tabular}}}
\end{table}

\clearpage

\newpage

\section{Final Remarks}\label{sec:final_remarks}

We have designed and analyzed an adaptive variance reduced SQP method for minimizing general smooth finite-sum optimization problems
with deterministic nonlinear equality constraints.  Under common assumptions, with constant or adaptive (non-diminishing) step sizes, we presented comprehensive convergence guarantees for our proposed method. Specifically, we proved that the \SVRSQP{} method generates a sequence of iterates whose first-order stationarity measure converges to zero in expectation. Our theoretical results can be viewed as analogues of those of the SVRG method on general unconstrained nonconvex finite-sum optimization problems \cite{reddi2016stochastic}. The numerical experiments presented on classification problems from the LIBSVM collection \cite{chang2011libsvm} demonstrated the efficiency, efficacy and robustness of the proposed method.

\bibliographystyle{plain}
\bibliography{references}


\end{document}